\documentclass[fontsize=10pt,a4paper]{scrartcl}
\usepackage[a4paper, top=4cm, bottom=4cm]{geometry}
\usepackage[utf8x]{inputenc}
\usepackage[T1]{fontenc}
\usepackage{authblk}
\usepackage{tikz}
\usetikzlibrary{matrix,arrows}
\usepackage{enumitem}
\usepackage{hyperref}
\usepackage{amsmath}
\usepackage{amsthm}
\usepackage{amssymb}
\usepackage{amsfonts}
\usepackage{mathrsfs}
\usepackage{euscript}
\usepackage{bbm}
\usepackage{abstract}
\usepackage{graphicx}
\usepackage{placeins}
\usepackage{subfigure}
\usepackage{nicefrac}
\usepackage{booktabs}
\usepackage{bookmark}
\usepackage{multirow}
\usepackage{dsfont}
\usepackage{longtable}
\usepackage{braket}
\usepackage{tikz}
\usepackage{scalerel}
\usepackage{listings}
\usepackage{color}
\usepackage[all]{xy}
\usepackage{tensor}
\usepackage{faktor}
\usepackage{lipsum}
\usepackage{comment}

\theoremstyle{definition}
\newtheorem{definition}{Definition}[section] 
 
\theoremstyle{definition}
\newtheorem{remark}[definition]{Remark}

\theoremstyle{definition}
\newtheorem{example}[definition]{Example}

\theoremstyle{theorem}
\newtheorem{theorem}[definition]{Theorem}
\theoremstyle{prop}
\newtheorem{prop}[definition]{Proposition}
\theoremstyle{cor}
\newtheorem{cor}[definition]{Corollary}
\theoremstyle{lemma}
\newtheorem{lemma}[definition]{Lemma}

\newcommand{\Exterior}{\mathchoice{{\textstyle\bigwedge}}%
    {{\bigwedge}}%
    {{\textstyle\wedge}}%
    {{\scriptstyle\wedge}}}

\begin{document}
\numberwithin{equation}{section}
\title{\LARGE Super fiber bundles, connection forms and parallel transport}
\author{
  \large Konstantin Eder\thanks{konstantin.eder@gravity.fau.de} 
}
\affil{Friedrich-Alexander-Universit\"at Erlangen-N\"urnberg (FAU),\\
Institute for Quantum Gravity (IQG),\\Staudtstr. 7,D-91058 Erlangen, Germany}
\maketitle

\begin{abstract}
The  present work provides a mathematically rigorous account on super fiber bundle theory, connection forms and their parallel transport, that ties together various approaches. We begin with a detailed introduction to super fiber bundles. We then introduce the concept of so-called relative supermanifolds as well as bundles and connections defined in these categories. Studying these objects turns out to be of utmost importance in order to, among other things, model anticommuting classical fermionic fields in mathematical physics. We then construct the parallel transport map corresponding to such connections and compare the results with those found by other means in the mathematical literature. Finally, applications of these methods to supergravity will be discussed such as the Cartan geometric formulation of Poincaré supergravity as well as the description of Killing vector fields and Killing spinors of super Riemannian manifolds arising from metric reductive super Cartan geometries.
\end{abstract}

\newpage

\section{Introduction}
Over the last forty years, many different approaches have been developed in order to formulate the notion of a supermanifold. The first and probably most popular one is the so-called \emph{algebro-geometric approach} introduced by Berezin, Kostant and Leites \cite{Berezin:1976,Kostant:1975qe} which borrows techniques from algebraic geometry. It is based on the interesting observation that ordinary smooth manifolds can equivalently be described in terms of the structure sheaf of smooth functions defined on the underlying topological space. This approach is very elegant and, in particular, avoids the introduction of superfluous (unphysical) degrees of freedom. Nevertheless, its definition turns out to be very abstract since, roughly speaking, points in this framework are implicitly encoded in the underlying structure sheaf of supersmooth functions. This makes this approach less accessible for physicists for concrete applications.\\
Hence, another approach to supermanifolds, the so-called \emph{concrete approach}, was initiated by DeWitt \cite{DeWitt:1984} and Rogers \cite{Rogers:1980}, and studied even more systematically by Tuynman in \cite{Tuynman:2004}, defining them similar to ordinary smooth manifolds in terms of a topological space of points, i.e., a topological manifold that locally looks a flat superspace. However, as it turns out, this definition has various ambiguities in formulating the notion of a point which, in contrast to the algebro-geometric approach, leads to too many unphysical degrees of freedom.\\
It was then found by Molotkov \cite{Mol:10} and further developed by Sachse \cite{Sac:08} that both approaches can be regarded as two sides of the same coin. In this framework, at least in the finite dimensional setting, it follows that Rogers-DeWitt supermanfolds can be interpreted in terms of a particular kind of a functor constructed out of a algebro-geometric supermanifold. This functorial intepretation then resolved the ambiguities arising in the Rogers-DeWitt approach and also opened the way towards a generalization of the theory to infinite dimensional supermanifolds.\\
Another caveat, both in the algebraic and Rogers-DeWitt approach, is the appropriate description of anticommuting (fermionic) degrees freedom. In fact, it turns out that the pullback of superfields to the underlying ordindary smooth manifold are purely commutative (bosonic). This seems, however, incompatible in various constructions in physics. For instance, in the Castellani-D'Auria-Fré approach \cite{DAuria:1982uck,Castellani:1991et}, a geometric approach to supergravity, by the so-called \emph{rheonomy principle}, physical degrees of freedom are completey determined by their pullback.\\
Furthermore, as we will also see in section \ref{section:parallel}, from a mathematical point of view, this issue also appears in the context of the parallel transport map corresponding to super connections. In fact, it turns out that in ordinary category of supermanifolds, both in the algebraic and concrete approach, the parallel transport cannot be used in order to compare different fibers of the bundle, in contrast to the classical theory.\\
A resolution has been proposed by Schmitt \cite{Schmitt:1996hp}. There, motivated by the Molotkov-Sachse approach to supermanifold theory \cite{Mol:10,Sac:08}, superfields on parametrized supermanifolds are considered. Since, a priori, this additional parametrizing supermanifold is chosen arbitrarily, one then has to ensure that these superfields transform covariantly under change of parametrization. This idea has then been implemented rigorously for instance by Hack et. al. in \cite{Hack:2015vna} studying \emph{relative supermanifolds} which are well-known in the algebraic approach \cite{Deligne:1999qp}. As it turns out, superfields on these supermanifolds indeed have the required properties, i.e., in the sense of Molotkov-Sachse, they behave functorially under change of parametrization. Moreover, in this framework, it turns out that fermionic fields have the interpretation in terms of functionals on supermanifolds which is in strong similarity to other approaches such as in the context of algebraic quantum field theory (pAQFT) \cite{Rejzner:2011au,Rejzner:2016hdj}.\\
\\
In this paper, we want to provide the mathematical rigorous foundations for the study of gauge theories on (relative) supermanifolds. In particular, we will study the parallel transport map corresponding to super connection forms defined on (relative) principal super fiber bundles. The parallel transport map, or the associated holonomies, have been considered in the context of covariant derivatives on super vector bundles in the algebro-geometric approach by Dumitrescu \cite{Florin:2008} and Groeger \cite{Groeger:2013aja}.\\ 
The theory of super fiber bundles and connection forms in the concrete approach has been developed in \cite{Tuynman:2004}. In the algebraic category, a precise definition of principal bundles and connection forms has been given in \cite{Stavracou:1996qb}. In this article, we will generalize the considerations of \cite{Tuynman:2004} to the relative category and, in particular, define super connection forms on relative principal super fiber bundles. We will then use this formalism in order to construct the corresponding parallel transport map and study some of its important properties. Moreover, we will analyze the precise relation between the algebraic and concrete approach and show explicitly that both approaches are in fact equivalent. We will therefore employ the functor of points technique as discussed in detail in \cite{Eder:2020erq}. Moreover, studying the induced parallel transport map on associated super vector bundles, this enables us to compare the results with those obtained in \cite{Florin:2008,Groeger:2013aja} in the algebraic setting.\\
\\
The structure of this paper is as follows: At the beginning, we will give a detailed introduction to super fiber bundle theory in the concrete approach to supermanifold theory. In contrast to \cite{Tuynman:2004}, we will therefore use the concept of \emph{formal bundle atlases} which is very well-known in the classical theory (see e.g. \cite{Baum,Hamilton:2017}  and references therein) and, in fact, turns out to be even applicable in the context of supermanifolds.\\
In section \ref{section:relative}, we will then introduce the concept of relative supermanifolds and define principal connections and super connection one-forms. In this context, in section \ref{section:Cartan}, we will also discuss super Cartan geometries and super Cartan connections as defined and studied in detail in \cite{Eder:2020erq} and study their relations to principal connections in the sense of Ehresmann. In section \ref{section:graded}, we will compare the theory of principal bundles and connection forms both in the algebraic and concrete approach and show that both approaches are in fact equivalent.\\
These results will then be used in section \ref{section:parallel} and \ref{section:parallelNEU} in order to construct the parallel transport map. Moreover, for a particular subclass of super Lie groups, a concrete formula for this map will be derived making it easier accessible for physical applications.\\
In the last section \ref{section:Application}, we will give some concrete examples for applications of supermanifold techniques in mathematical physics. First, we will describe $\mathcal{N}=1$, $D=4$ Poincaré supergravity in terms of a metric reductive super Cartan geometry. In this context, we will also sketch a possible embedding of the Castellani-D'Auria-Fré approach into the present formalism. Furthermore, we will discuss a link between superfields on parametrized supermanifolds and the description of anticommutative fermionic fields in pAQFT. Finally, Killing vector fields on super Riemannian manifolds will be discussed arsing from metric reductive super Cartan geometries as well as their relation to Killing spinors as they typically appear in physics. In this context, let us note that this paper can also be read as a companion paper of \cite{Eder:2020erq} studying a mathematical rigorous approach towards \emph{geometric supergravity}.


\section{Super fiber bundles}
In this section, we want to give a detailed account on super fiber bundles in the category of $H^{\infty}$-supermanifolds as this will provide us with the necessary mathematical tools needed in the subsequent sections (see appendix \ref{appendix:Super} for a review on $H^{\infty}$-supermanifold theory and our choice of conventions; for a relation to \emph{algebro-geometric supermanifolds} see section \ref{section:graded}). If not stated explicitly otherwise, in what follows, we will always work in the category of $H^{\infty}$-supermanifolds so that smoothness and related notions are always referred to this particular category. Let us start with the basic definition of a super fiber bundle.    
\begin{definition}[Super fiber bundle]\label{prop:2.1}
A \emph{super fiber bundle} ($\mathcal{E},\pi,\mathcal{M},\mathcal{F}$), also simply denoted by $\mathcal{F}\rightarrow\mathcal{E}\stackrel{\pi}{\rightarrow}\mathcal{M}$, consists of supermanifolds $\mathcal{E}$, $\mathcal{M}$ and $\mathcal{F}$ called \emph{total space}, \emph{base} and \emph{typical fiber}, respectively, as well as a smooth surjective map $\pi:\,\mathcal{E}\rightarrow\mathcal{M}$, called \emph{projection}, satisfying the local triviality property:\\
For any $p\in\mathcal{E}$ there exists an open subset $U\subset\mathcal{M}$ which is an open neighborhood of $\pi(p)$ and a homeomorphism $\phi:\,\pi^{-1}(U)\rightarrow U\times\mathcal{F}$ called \emph{local trivialization} such that the following diagram commutes
	\begin{displaymath}
	 \phi:\xymatrix{
         \pi^{-1}(U)\ar[r] \ar[d]_{\pi}  &    U\times\mathcal{F} \ar[ld]^{\mathrm{pr}_1}\\
            U   &         
     }
		\label{eq:2.1}
 \end{displaymath}
i.e. $\mathrm{pr}_1\circ\phi=\pi$ where $\mathrm{pr}_1$ denotes the projection onto the first factor. 
\end{definition}
\begin{prop}\label{prop:2.2}
Let $\mathcal{F}\rightarrow\mathcal{E}\stackrel{\pi}{\rightarrow}\mathcal{M}$ be a super fiber bundle. Then $\mathbf{B}(\mathcal{F})\rightarrow \mathbf{B}(\mathcal{E})\stackrel{\bar{\pi}}{\rightarrow}\mathbf{B}(\mathcal{M})$, with $\bar{\pi}:=\mathbf{B}(\pi)$ and $\mathbf{B}:\,\mathbf{SMan}_{H^{\infty}}\rightarrow\mathbf{Man}$ the body functor, defines a smooth fiber bundle in the category $\mathbf{Man}$ of ordinary $C^{\infty}$-smooth manifolds.
\end{prop}
\begin{proof}
This is an immediate consequence of Prop. \ref{Prop:A.10} as well as the fact that $\mathbf{B}:\,\mathbf{SMan}_{H^{\infty}}\rightarrow\mathbf{Man}$ is a functor. 
\end{proof}

\begin{definition}\label{prop:2.3}
Let $\mathcal{M}$ and $\mathcal{F}$ be supermanifolds, $\mathcal{E}$ an abstract set and $\pi:\,\mathcal{E}\rightarrow\mathcal{M}$ a surjective map.
\begin{enumerate}[label=(\roman*)]
	\item Let $U\subset\mathcal{M}$ be open and $\phi:\,\pi^{-1}(U)\rightarrow U\times\mathcal{F}$ a bijective map such that $\mathrm{pr}_1\circ\phi=\pi\big{|}_{\pi^{-1}(U)}$, then $(U,\phi_{U})$ is called a \emph{formal bundle chart}.
	\item A family $\{(U_{\alpha},\phi_{\alpha})\}_{\alpha\in\Upsilon}$ of formal bundle charts is called a \emph{(smooth) formal bundle atlas} of $\mathcal{E}$ (w.r.t. $\pi$) iff $\{U_{\alpha}\}_{\alpha\in\Upsilon}$ is an open covering of $\mathcal{M}$ and for any $\alpha,\beta\in\Upsilon$ with $U_{\alpha}\cap U_{\beta}\neq\emptyset$, the transition functions
	\begin{equation}
	\phi_{\beta}\circ\phi^{-1}_{\alpha}:\,(U_{\alpha}\cap U_{\beta})\times\mathcal{F}\rightarrow(U_{\alpha}\cap U_{\beta})\times\mathcal{F}
	\label{eq:2.4}
	\end{equation} 
are smooth.
\end{enumerate} 
\end{definition}

\begin{theorem}\label{prop:2.4}
Let $\mathcal{M}$ and $\mathcal{F}$ be supermanifolds of dimensions $\mathrm{dim}\,\mathcal{M}=(m,n)$ and $\mathrm{dim}\,\mathcal{F}=(p,q)$, respectively, $\mathcal{E}$ an abstract set and $\pi:\,\mathcal{E}\rightarrow\mathcal{M}$ a surjective map. Let furthermore $\{(U_{\alpha},\phi_{\alpha})\}_{\alpha\in\Upsilon}$ be a smooth formal bundle atlas of $\mathcal{E}$ (w.r.t. $\pi$). Then, there exists a unique topology and smooth structure on $\mathcal{E}$ such that $\mathcal{E}$ becomes a supermanifold of dimension $\mathrm{dim}\,\mathcal{E}=(m+p,n+q)$ and $(\mathcal{E},\pi,\mathcal{M},\mathcal{F})$ a super fiber bundle which is locally trivial w.r.t. the bundle atlas $\{(U_{\alpha},\phi_{\alpha})\}_{\alpha\in\Upsilon}$.
\end{theorem}
\begin{proof}
We define a topology on $\mathcal{E}$ by declaring a subset $O\subseteq\mathcal{E}$ to be open if and only if for any $\alpha\in\Upsilon$ the image
\begin{equation}
\phi_{\alpha}(O\cap\pi^{-1}(U_{\alpha}))\subseteq U_{\alpha}\times\mathcal{F}
\end{equation}
is an open subset in $\mathcal{M}\times\mathcal{F}$ (note that this condition is mandatory in order for $\phi_{\alpha}$ to define a homeomorphism). Since the formal bundle charts are bijective, it follows immediately that arbitrary unions and finite intersections of open sets are open. Moreover, the empty set and $\mathcal{E}$ are open as well, so that this indeed defines a topology on $\mathcal{E}$.\\
By intersection, one may assume that the $\{U_{\alpha}\}_{\alpha}$ are coordinate neighborhoods of $\mathcal{M}$. Let then $\{(U_{\alpha},\varphi_{\alpha})\}_{\alpha\in\Upsilon}$ and $\{(V_{\beta},\psi_{\beta})\}_{\beta\in\Sigma}$ be smooth atlases of $\mathcal{M}$ and $\mathcal{F}$, respectively. For $(\alpha,\beta)\in\Upsilon\times\Sigma$ we define open sets $W_{\alpha\beta}:=\phi_{\alpha}^{-1}(U_{\alpha}\times V_{\beta})\subseteq\mathcal{E}$ as well as bijective maps
\begin{equation}
\theta_{\alpha\beta}:=(\varphi_{\alpha}\times\psi_{\beta})\circ\phi_{\alpha}:\,W_{\alpha\beta}\rightarrow\varphi_{\alpha}(U_{\alpha})\times\psi_{\alpha}(V_{\beta})\subseteq\Lambda^{m+p,n+q}
\label{eq:2.6}
\end{equation}
For $(\alpha',\beta')\in\Upsilon\times\Sigma$ such that $W_{\alpha\beta}\cap W_{\alpha'\beta'}\neq\emptyset$, it then follows 
\begin{equation}
\theta_{\alpha'\beta'}\circ\theta_{\alpha\beta}^{-1}=(\varphi_{\alpha'}\times\psi_{\beta'})\circ\phi_{\alpha'}\circ\phi_{\alpha}^{-1}\circ(\varphi^{-1}_{\alpha}\times\psi^{-1}_{\beta})
\label{eq:2.7}
\end{equation}
on $\theta_{\alpha\beta}(W_{\alpha\beta}\cap W_{\alpha'\beta'})$, which is smooth by definition of a formal bundle atlas. Hence, $\{(W_{\alpha\beta},\theta_{\alpha\beta})\}$ defines a smooth atlas of $\mathcal{E}$ turning it into a proto $H^{\infty}$-supermanifold of dimension $\mathrm{dim}\,\mathcal{E}=(m+p,n+q)$.\\
Let $U\subseteq\mathcal{M}$ be open. Then, for any $\alpha\in\Upsilon$, $\phi_{\alpha}(\pi^{-1}(U)\cap\pi^{-1}(U_{\alpha}))=\phi_{\alpha}(\pi^{-1}(U\cap U_{\alpha}))=(U\cap U_{\alpha})\times\mathcal{F}$ is open in $U_{\alpha}\times\mathcal{F}$ and thus $\pi^{-1}(U)\subseteq\mathcal{E}$ is open proving that $\pi:\,\mathcal{E}\longrightarrow\mathcal{M}$ is continuous. To see that is also smooth, let $(\alpha,\beta)\in\Upsilon\times\Sigma$ and $\alpha'\in\Upsilon$ such that $U_{\alpha'}\cap U_{\alpha}\neq\emptyset$. It follows
\begin{equation}
\varphi_{\alpha'}\circ\pi\circ\theta_{\alpha\beta}^{-1}=\varphi_{\alpha'}\circ\mathrm{pr}_1\circ(\varphi^{-1}_{\alpha}\times\psi^{-1}_{\beta}):\,\varphi_{\alpha}(U_{\alpha'}\cap U_{\alpha})\times\psi_{\beta}(V_{\beta})\rightarrow\varphi_{\alpha}(U_{\alpha'}\cap U_{\alpha})
\label{eq:2.8}
\end{equation}
which is obviously smooth. It remains to show that $\mathcal{E}$ is in fact a $H^{\infty}$-supermanifold, i.e., $\mathbf{B}(\mathcal{E})$ is a second countable Hausdorff topological space.\\
To prove that it is Hausdorff, consider first $p,q\in\mathbf{B}(\mathcal{E})$ with $\mathbf{B}(\pi)(p)\neq\mathbf{B}(\pi)(q)$. Since $\mathbf{B}(\mathcal{M})\times\mathbf{B}(\mathcal{F})$ is Hausdorff by assumption, there are open subsets $\mathbf{B}(\pi)(p)\in U_p$ and $\mathbf{B}(\pi)(q)\in U_q$ in $\mathbf{B}(\mathcal{M})\times\mathbf{B}(\mathcal{F})$ with $U_p\cap U_q\neq\emptyset$. Since, $\mathbf{B}(\pi):\,\mathbf{B}(\mathcal{E})\rightarrow\mathbf{B}(\mathcal{M})$ is smooth, these yield disjoint open subsets $p\in\mathbf{B}(\pi)^{-1}(U_p)$ and $q\in\mathbf{B}(\pi)^{-1}(U_q)$ in $\mathbf{B}(\mathcal{E})$ separating $p$ and $q$. If $\mathbf{B}(\pi)(p)=\mathbf{B}(\pi)(q)$, consider $\alpha\in\Upsilon$ with $p,q\in\mathbf{B}(\pi)^{-1}(\mathbf{B}(U_{\alpha}))$, where $\mathbf{B}(\phi):\,\mathbf{B}(\pi)^{-1}(\mathbf{B}(U_{\alpha}))\rightarrow\mathbf{B}(U_{\alpha})\times\mathbf{B}(\mathcal{F})$ is the corresponding bundle chart on the body which, in particular, defines a homeomorphism. Let then $V_p,V_q\subset\mathbf{B}(\mathcal{F})$ be disjoint open subsets with $\mathbf{B}(\phi)(p)\in V_p$ and $\mathbf{B}(\phi)(q)\in V_q$. This then finally yields disjoint open subsets $p\in\mathbf{B}(\phi)^{-1}(\mathbf{B}(U_{\alpha})\times V_p)$ and $q\in\mathbf{B}(\phi)^{-1}(\mathbf{B}(U_{\alpha})\times V_q)$ in $\mathbf{B}(\mathcal{E})$ seperating $p$ and $q$. This shows that $\mathbf{B}(\mathcal{E})$ is indeed Hausdorff. That $\mathbf{B}(\mathcal{E})$ is also second countable follows similarly using that $\mathbf{B}(\mathcal{M})\times\mathbf{B}(\mathcal{F})$ is second countable.\\
Hence, this proves that $\mathcal{E}$ indeed defines a $H^{\infty}$-supermanifold and that $(\mathcal{E},\pi,\mathcal{M},\mathcal{F})$ is a super fiber bundle. 
\end{proof}

\begin{example}[Pullback bundle]\label{prop:2.5}
Given a smooth map $f:\,\mathcal{N}\rightarrow\mathcal{M}$ between supermanifolds and a super fiber bundle $(\mathcal{E},\pi,\mathcal{M},\mathcal{F})$, one can construct a new bundle by setting
\begin{equation}
f^*\mathcal{E}:=\{(x,p)\in\mathcal{N}\times\mathcal{E}|\,f(x)=\pi(p)\}\subseteq\mathcal{N}\times\mathcal{E}
\label{eq:2.9}
\end{equation}
as total space together with the projection
\begin{equation}
\pi_f:\,f^*\mathcal{E}\rightarrow\mathcal{N},\,(x,p)\mapsto x
\label{eq:2.10}
\end{equation}
that is, fibers over $\mathcal{M}$ are pulled back w.r.t. $f$ to fibers over $\mathcal{N}$. Let $\{(U_{\alpha},\phi_{\alpha})\}_{\alpha\in\Upsilon}$ be a smooth bundle atlas on $\mathcal{E}$. For $\alpha\in\Upsilon$, define the map $\psi_{\alpha}:\,f^*\mathcal{E}\supseteq\pi^{-1}_f(U_{\alpha})\rightarrow U_{\alpha}\times\mathcal{F}$ via
\begin{equation}
\psi_{\alpha}(x,p):=(x,\mathrm{pr}_2\circ\phi_{\alpha}(p))
\label{eq:2.11}
\end{equation}
It is clear by definition that $\psi_{\alpha}$ is bijective and fiber-preserving. Moreover, for $\alpha,\beta\in\Upsilon$ with $U_{\alpha}\cap U_{\beta}\neq\emptyset$, we compute
\begin{equation}
\psi_{\beta}\circ\psi_{\alpha}^{-1}(x,p)=\psi_{\beta}(x,\phi_{\alpha}^{-1}(f(x),p))=(x,\mathrm{pr}_2\circ\phi_{\beta}\circ\phi_{\alpha}^{-1}(f(x),p))
\label{eq:2.12}
\end{equation}
$\forall (x,p)\in (U_{\alpha}\cap U_{\beta})\times\mathcal{F}$ and thus is smooth. It follows that $\{(U_{\alpha},\psi_{\alpha})\}_{\alpha\in\Upsilon}$ satisfies the properties of a formal bundle atlas so that, by Theorem \ref{prop:2.4}, $(f^*\mathcal{E},\pi_f,\mathcal{N},\mathcal{F})$ has the structure of a super fiber bundle called the \emph{pullback bundle} of $\mathcal{E}$ w.r.t. $f$.\\
\\
The pullback bundle defines the \emph{pullback} in the category of super fiber bundles. More precisely, it satisfies the following universal property: Given super fiber bundles $\mathcal{E}\stackrel{\pi_{\mathcal{E}}}{\rightarrow}\mathcal{M}$ and $\mathcal{Q}\stackrel{\pi_{\mathcal{Q}}}{\rightarrow}\mathcal{N}$ as well as a smooth bundle morphism $(\phi,f):\,(\mathcal{Q},\mathcal{N})\rightarrow(\mathcal{E},\mathcal{M})$, there exists a unique amooth bundle morphism $(\hat{\phi},\mathrm{id}_{\mathcal{N}}):\,(\mathcal{Q},\mathcal{N})\rightarrow(f^*\mathcal{E},\mathcal{N})$ such that the following diagram commutes  
\begin{equation}
\begin{tikzpicture}
  \node (fP) {$f^{*}\mathcal{E}$};
  \node (N) [node distance=1.7cm, right of=fP] {$\mathcal{E}$};
  \node (P) [node distance=1.7cm, below of=fP] {$\mathcal{N}$};
  \node (M) [node distance=1.7cm, below of=N] {$\mathcal{M}$};
  \node (Q) [node distance=1.7cm, left of=fP, above of=fP] {$\mathcal{Q}$};
  \draw[->] (fP) to node [above] {$\mathrm{pr}_2$} (N);
  \draw[->] (fP) to node [left] {$\pi_f$} (P);
  \draw[->] (P) to node [above] {$f$} (M);
  \draw[->] (N) to node [right]{$\pi_{\mathcal{E}}$} (M);
  \draw[->, bend left] (Q) to node [above] {$\phi$} (N);
  \draw[->, bend right] (Q) to node [below] {$\pi_{\mathcal{Q}}\,\,\,\,\,$} (P);
  \draw[->, dashed] (Q) to node  {$\hat{\phi}$} (fP);
\end{tikzpicture}
\label{eq:2.13}
\end{equation}
In fact, $\hat{\phi}$ has to be of the form $\hat{\phi}:\,Q\ni q\mapsto(\pi_{\mathcal{Q}}(q),\phi(q))$ and therefore is smooth.
\end{example}
We next consider a particular subclass of super fiber bundles whose typical fiber carries the structure of a super $\Lambda$-vector space. Therefore, recall that a (real) super $\Lambda$-vector space $\mathcal{V}$ is a free super $\Lambda$-module together with a distinguished equivalence class $[(e_i)_i]$ of homogeneous bases, where two homogeneous bases $(e_i)_i$ and $(f_j)_j$ of $\mathcal{V}$ are called equivalent, if they are related by scalar coefficients. A representative $(e_i)_i\in[(e_i)_i]$ is then called a real homogeneous basis. Correspondingly, the equivalence class generates a real super vector space $V$ such that $\mathcal{V}=\Lambda\otimes V$. 
\begin{definition}[Super vector bundle]\label{prop:2.6}
A \emph{super vector bundle of rank $m|n$} is a super fiber bundle $\mathcal{V}\rightarrow\mathcal{E}\stackrel{\pi}{\rightarrow}\mathcal{M}$ such that
\begin{enumerate}[label=(\roman*)]
	\item the typical fiber is a super $\Lambda$-vector space $\mathcal{V}$ of dimension $\mathrm{dim}\,\mathcal{V}=m|n$.
	\item for each $x\in\mathcal{M}$, the fibers $\mathcal{E}_x:=\pi^{-1}(\{x\})$ have the structure of free super $\Lambda$-modules.
	\item there exists a smooth bundle atlas $\{(U_{\alpha},\phi_{\alpha})\}_{\alpha\in\Upsilon}$ of $\mathcal{E}$ such that, for each $\alpha\in\Upsilon$, the induced map
	\begin{align}
	\phi_{\alpha,x}:\,\mathcal{E}_x&\rightarrow\mathcal{V}\label{eq:2.14}\\
	v_x&\mapsto\mathrm{pr}_2\circ\phi_{\alpha}(v_x)\nonumber
	\end{align}
$\forall x\in U_{\alpha}$ is a (right-linear) isomorphism of super $\Lambda$-modules.
\end{enumerate}
A super vector bundle of rank $1|1$ is also called a \emph{super line bundle}.
\end{definition}
\begin{lemma}[after \cite{Tuynman:2004}]\label{prop:2.7}
Let $\mathcal{V}$ and $\mathcal{W}$ be super $\Lambda$-vector spaces, $\mathcal{M}$ be a supermanifold and $\phi:\,\mathcal{M}\times\mathcal{V}\rightarrow\mathcal{W}$ be a smooth map such that $\phi(p,\,\cdot\,):\,\mathcal{V}\rightarrow\mathcal{W}\in\mathrm{End}_R(\mathcal{V},\mathcal{W})$ is a right-linear map for any $p\in\mathcal{M}$. Then, there exists a unique smooth map $\psi:\,\mathcal{M}\rightarrow\mathrm{End}_R(\mathcal{V},\mathcal{W})$ such that $\psi(p)(v)=\phi(p,v)$ for any $p\in\mathcal{M}$, $v\in\mathcal{V}$.
\end{lemma}
\begin{proof}
Let $(e_i)_{i=1,\ldots,n}$ and $(f_j)_{j=1,\ldots,m}$ be real homogeneous bases of $\mathcal{V}$ and $\mathcal{W}$, respectively, and $(f^j)_{j=1,\ldots,m}$ the corresponding right-dual basis of $\mathcal{W}^*:=\underline{\mathrm{Hom}}_R(\mathcal{W},\Lambda)$. Since $e_i$ are body points and $f^j:\,\mathcal{W}\rightarrow\Lambda$ are smooth, the maps $\phi^j_i:=f^j\circ\phi(\,\cdot\,,e_i):\,\mathcal{M}\rightarrow\Lambda$ are of class $H^{\infty}$ for any $i=1,\ldots,n$ and $j=1,\ldots,m$. Hence, we can define a smooth map $\psi:\,\mathcal{M}\rightarrow\mathrm{End}_R(\mathcal{V},\mathcal{W})$ via $\psi(p):=f_j\otimes\psi^j_i(p)e^i$ $\forall p\in\mathcal{M}$ which, by construction, satisfies $\psi(p)(v)=\phi(p,v)$ for any $v\in\mathcal{V}$. That $\psi$ is unique is immediate.
\end{proof}
\begin{remark}\label{prop:2.8}
Given a super vector bundle $\mathcal{V}\rightarrow\mathcal{E}\rightarrow\mathcal{M}$ and local trivializations $(U_{\alpha},\phi_{\alpha})$ and $(U_{\beta},\phi_{\beta})$, this yields the map
\begin{equation}
\mathrm{pr}_2\circ(\phi_{\alpha}\circ\phi_{\beta}^{-1}):\,(U_{\alpha}\cap U_{\beta})\times\mathcal{V}\rightarrow\mathcal{V}
\label{eq:2.15}
\end{equation}
which, in particular, is linear in the second argument. Thus, using Lemma \ref{prop:2.7}, this yields a smooth map
\begin{equation}
g_{\alpha\beta}:\,U_{\alpha}\cap U_{\beta}\rightarrow\mathrm{GL}(\mathcal{V})
\label{eq:2.16}
\end{equation}
satisfying $g_{\alpha\beta}(x)v=\mathrm{pr}_2\circ(\phi_{\alpha}\circ\phi_{\beta}^{-1})(x,v)$, that is, $g_{\alpha\beta}(x)=\phi_{\alpha,x}\circ\phi_{\beta,x}^{-1}$, which we will call \emph{transition maps}. 
\end{remark}
\begin{prop}\label{prop:2.10}
There is a one-to-one correspondence between local trivializations of a super vector bundle $\mathcal{V}\rightarrow\mathcal{E}\stackrel{\pi}{\rightarrow}\mathcal{M}$  and families $(X_i)_{i}$ of smooth local sections $X_i:\,U\rightarrow\mathcal{E}$ of $\mathcal{E}$, $U\subset\mathcal{M}$ open, such that $(X_{ix})_i$ defines a homogeneous basis of the super $\Lambda$-module $\mathcal{E}_x$ $\forall x\in U$.
\end{prop}
\begin{proof}
Let $(U,\phi_U)$ be a local trivialization of $\mathcal{E}$. If $(e_i)_i$ is a real homogeneous of the super $\Lambda$-vector space $\mathcal{V}$, define $X_i(p):=\phi^{-1}_U(p,e_i)$ $\forall p\in U$. Since $e_i\in\mathcal{V}$ are body points, it follows that $X_i$ define smooth sections of $\mathcal{E}$ and $(X_i(p))_i$ is a homogeneous basis of $\mathcal{E}_p$ $\forall p\in U$.\\
Conversely, if $(X_i)_{i}$ is a family of smooth local sections $X_i:\,U\rightarrow\mathcal{E}$ of $\mathcal{E}$ such that, for any $p\in\mathcal{M}$, $(X_i(p))_i$ defines a homogeneous basis of $\mathcal{E}_p$ $\forall p\in U$, consider the map 
\begin{equation}
\psi:\,U\times\mathcal{V}\rightarrow\pi^{-1}(U)\subseteq\mathcal{E},\,(x,v^ie_i)\mapsto v^iX_i
\label{eq:2.17}
\end{equation}
By definition, $\psi$ is bijective, smooth and an isomorphism of super $\Lambda$-modules on each fiber. Hence, it remains to show that has a smooth inverse. Therefore, following \cite{Tuynman:2004}, let $(W,\phi_W)$ be a local trivialization of $\mathcal{E}$ with $U\cap W\neq\emptyset$. Then, $\phi_W\circ\psi:\,(U\cap W)\times\mathcal{V}\rightarrow(U\cap W)\times\mathcal{V}$ is of the form 
\begin{equation}
\phi_W\circ\psi(p,v)=(p,v^i\tensor{\widetilde{\phi}}{_i^j}(p)e_j)
\label{eq:2.18}
\end{equation}
with $\tensor{\widetilde{\phi}}{_i^j}:=\braket{\mathrm{pr}_2\circ\phi_W(\,\cdot\,,e_i)|\mathrel{^i\! e}}$ which is of class $H^{\infty}$ and invertible. Hence, it has a smooth inverse yielding a smooth inverse of $\phi_W\circ\psi$. But, as $\phi_W$ is a diffeomorphism, this implies that $\psi$ is a local diffeomorphism and hence admits a smooth inverse $\psi^{-1}:\,\pi^{-1}(U)\rightarrow U\times\mathcal{V}$ providing a local trivialization of $\mathcal{E}$.
\end{proof}
\begin{cor}\label{prop:2.11}
A super vector bundle $\mathcal{V}\rightarrow\mathcal{E}\stackrel{\pi}{\rightarrow}\mathcal{M}$ is trivial if and only if there exists a family $(X_i)_{i}$ of smooth global sections $X_i\in\Gamma(\mathcal{E})$ of $\mathcal{E}$ such that $(X_{ix})_i$ defines a homogeneous basis of the super $\Lambda$-module $\mathcal{E}_x$ $\forall x\in\mathcal{M}$.\qed
\end{cor}

Since, it will appear quite frequently in what follows, let us introduce a new symbol for the composition of two left linear morphisms. More precisely, we define the following.
\begin{definition}\label{definitionNEU:2.11}
Let $\phi:\,\underline{\mathrm{Hom}}_L(\mathcal{V},\mathcal{W})$ and $\psi:\,\underline{\mathrm{Hom}}_L(\mathcal{W},\mathcal{W}')$ be two left linear morphisms between super $\Lambda$-modules. Then, the composition $\phi\diamond\psi\in\underline{\mathrm{Hom}}_L(\mathcal{V},\mathcal{W}')$ is defined as the left linear morphism given by
\begin{equation}
\braket{v|\phi\diamond\psi}:=\braket{\braket{v|\phi}|\psi}
\end{equation}
for all $v\in\mathcal{V}$.
\end{definition}

\begin{example}[The dual vector bundle]\label{prop:2.12}
Recall that, to a super $\Lambda$-module $\mathcal{V}$, one can associate its corresponding left-dual $\tensor[^*]{\mathcal{V}}{}$ defined as the super $\Lambda$-module  $\tensor[^*]{\mathcal{V}}{}:=\underline{\mathrm{Hom}}_L(\mathcal{V},\Lambda)$. Analogously, one defines the right-dual $\mathcal{V}^*:=\underline{\mathrm{Hom}}_R(\mathcal{V},\Lambda)$. Thus, given a super vector bundle $\mathcal{V}\rightarrow\mathcal{E}\stackrel{\pi}{\rightarrow}\mathcal{M}$, we can construct the corresponding \emph{left-dual bundle} as follows. As total space, we set
\begin{equation}
\tensor[^*]{\mathcal{E}}{}:=\coprod_{x\in\mathcal{M}}{\tensor[^*]{\mathcal{E}}{_x}}
\label{eq:2.19}
\end{equation}
together with the surjective map 
\begin{align}
\pi_{\tensor[^*]{\mathcal{E}}{}}:\tensor[^*]{\mathcal{E}}{}\rightarrow\mathcal{M},\tensor[^*]{\mathcal{E}}{_x}\ni T_x\mapsto x
\label{eq:2.20}
\end{align}
Let $\{(U_{\alpha},\phi_{\alpha})\}_{\alpha\in\Upsilon}$ be a smooth bundle atlas on $\mathcal{E}$. Then, for any $\alpha\in\Upsilon$, define the map
\begin{align}
\phi^*_{\alpha}:\tensor[^*]{\mathcal{E}}{}\supseteq\!\pi^{-1}_{\tensor[^*]{\mathcal{E}}{}}(U_{\alpha})&\rightarrow U_{\alpha}\times\!\tensor[^*]{\mathcal{V}}{}\\
T_x&\mapsto (x,\phi_{\alpha}^*(T_x))\nonumber
\label{eq:2.21}
\end{align}
where $\phi_{\alpha}^*(T_x)\in\!\tensor[^*]{\mathcal{V}}{}$ is defined as $\braket{v|\phi_{\alpha}^*(T_x)}:=\braket{\braket{v|\phi_{\alpha,x}^{-1}}|T_x}$ $\forall v\in\mathcal{V}$. It follows that $\phi_{\alpha}^*$ is bijective, fiber-preserving and an isomorphism of super $\Lambda$-modules on each fiber. The inverse is given by
\begin{equation}
\phi^{*-1}_{\alpha}:\,U_{\alpha}\times\!\tensor[^*]{\mathcal{V}}{}\rightarrow\pi^{-1}_{\tensor[^*]{\mathcal{E}}{}}(U_{\alpha}),\,(x,\ell)\mapsto(\phi_{\alpha}^{-1})^*\ell_x\in\!\tensor[^*]{\mathcal{E}}{_x}
\label{eq:2.21.1}
\end{equation}
where $\braket{v_x|(\phi_{\alpha}^{-1})^*\ell_x}:=\braket{\braket{v_x|\phi_{\alpha,x}}|\ell}$ $\forall v_x\in\mathcal{E}_x$, $x\in\mathcal{M}$. For $\alpha,\beta\in\Upsilon$ with $U_{\alpha}\cap U_{\beta}\neq\emptyset$, the transition function $\phi^*_{\beta}\circ\phi^{*-1}_{\alpha}:\,(U_{\alpha}\cap U_{\beta})\times\!\tensor[^*]{\mathcal{V}}{}\rightarrow(U_{\alpha}\cap U_{\beta})\times\!\tensor[^*]{\mathcal{V}}{}$ takes the form
\begin{equation}
\phi^*_{\beta}\circ\phi^{*-1}_{\alpha}(x,\ell)=\phi^*_{\beta}((\phi_{\alpha}^{-1})^*\ell_x)=(x,\phi_{\beta}^*((\phi_{\alpha}^{-1})^*\ell_x))
\label{eq:2.22}
\end{equation}
Choosing a real homogeneous basis $(e_i)_i$ of $\mathcal{V}$, the r.h.s. of \eqref{eq:2.22} becomes
\begin{align}
\braket{e_i|\phi_{\beta}^*((\phi_{\alpha}^{-1})^*\ell_x)}&=\braket{\braket{e_i|\phi_{\beta,x}^{-1}}|(\phi_{\alpha}^{-1})^*\ell_x)}=\braket{\braket{e_i|\phi_{\beta,x}^{-1}\diamond\phi_{\alpha,x}}|\ell}\\
&=\tensor{(g_{\alpha\beta})}{^{sT}_{\! i}^j}(x)\ell_j\nonumber
\label{eq:2.23}
\end{align}
with '$sT$' denoting \emph{super transposition}, for any $x\in U_{\alpha}$ and thus is smooth. Hence, it follows that $\{(U_{\alpha},\phi^*_{\alpha})\}_{\alpha\in\Upsilon}$ defines a formal bundle atlas of $\tensor[^*]{\mathcal{E}}{}$ so that, by Theorem \ref{prop:2.3}, $\tensor[^*]{\mathcal{V}}{}\rightarrow\!\tensor[^*]{\mathcal{E}}{}\rightarrow\mathcal{M}$ has the structure of a super vector bundle called the \emph{left-dual super vector bundle} of $\mathcal{E}$. Analogously, one defines the \emph{right-dual super vector bundle} $\mathcal{V}^*\rightarrow\mathcal{E}^*\rightarrow\mathcal{M}$ of $\mathcal{E}$. 
\end{example}
Suppose $\mathcal{V}\rightarrow\mathcal{E}\stackrel{\pi}{\rightarrow}\mathcal{M}$ is a super vector bundle and $(X_i)_i$ a family of smooth local sections $X_i:\,U\rightarrow\mathcal{E}$ of $\mathcal{E}$ over $U\subseteq\mathcal{M}$ open such that, for any $x\in U$, $(X_{ix})_i$ defines a homogeneous basis of the super $\Lambda$-module $\mathcal{E}_x$. According to Prop. \ref{prop:2.10}, this yields a local trivialization $(U,\phi_U)$ of $\mathcal{E}$ with the inverse given by
\begin{equation}
\phi_U^{-1}:\,U\times\mathcal{V}\rightarrow\pi^{-1}(U)\subseteq\mathcal{E},\,(x,v^ie_i)\mapsto v^iX_{ix}
\label{eq:2.24}
\end{equation}
with $(e_i)_i$ a real homogeneous basis of $\mathcal{V}$. By \eqref{eq:2.21.1}, this in turn induces a local local trivialization $(U,\phi^*_U)$ of the left-dual bundle $\tensor[^*]{\mathcal{E}}{}$ with inverse
\begin{equation}
\phi^{*-1}_{U}:\,U\times\!\tensor[^*]{\mathcal{V}}{}\rightarrow\pi^{-1}_{\tensor[^*]{\mathcal{E}}{}}(U),\,(x,\ell)\mapsto(\phi^{-1})^*\ell_x\in\!\tensor[^*]{\mathcal{E}}{_x}
\label{eq:2.25}
\end{equation}
Hence, if $(\tensor[^i]{e}{})_i$ denotes the corresponding left-dual basis of $\tensor[^*]{\mathcal{V}}{}$, again by Prop. \ref{prop:2.10}, this yields a family $(\tensor[^i]{\omega}{})_i$ of smooth local sections $\tensor[^i]{\omega}{}\in\Gamma_U(\tensor[^*]{\mathcal{E}}{})$ of $\tensor[^*]{\mathcal{E}}{}$ given by
\begin{equation}
\tensor[^i]{\omega}{}:=\phi^{*-1}_{U}(\,\cdot\,,\tensor[^i]{e}{})
\label{eq:2.26}
\end{equation}
such that, for any $x\in U$, $(\tensor[^i]{\omega}{_x})_i$ defines a homogeneous basis of the corresponding left-dual super $\Lambda$-module $\tensor[^*]{\mathcal{E}}{_x}$. Evaluation on the $X_i$ then yields
\begin{align}
\braket{X_{ix}|\tensor[^j]{\omega}{_x}}&=\braket{X_{ix}|(\phi^{-1}_U)^*\tensor[^j]{e}{_x}}=\braket{\braket{X_{ix}|\phi_{U,x}}|\tensor[^j]{e}{}}\nonumber\\
&=\braket{\braket{\braket{e_i|\phi_{U,x}^{-1}}|\phi_{U,x}}|\tensor[^j]{e}{}}\nonumber\\
&=\braket{\braket{e_i|\phi_{U,x}^{-1}\diamond\phi_{U,x}}|\tensor[^j]{e}{}}\nonumber\\
&=\braket{e_i|\tensor[^j]{e}{}}=\delta{_i^j}
\label{eq:2.27}
\end{align}
$\forall x\in U$, that is, $\braket{X_{i}|\tensor[^j]{\omega}{}}=\delta{_i^j}$ $\forall i,j$. Hence, we have established the following.
\begin{prop}\label{prop:2.13}
Let $\mathcal{V}\rightarrow\mathcal{E}\stackrel{\pi}{\rightarrow}\mathcal{M}$ be a super vector bundle, $(X_i)_{i=1,\ldots,n}$ a family of local sections $X_i:\,U\rightarrow\mathcal{E}$ of $\mathcal{E}$ over $U\subseteq\mathcal{M}$ open such that, for any $x\in U$, $(X_{ix})_i$ defines a homogeneous basis of the super $\Lambda$-module $\mathcal{E}_x$. Then, there exists a family $(\tensor[^i]{\omega}{})_{i,\ldots,n}$ of smooth local sections $\tensor[^i]{\omega}{}\in\Gamma_U(\tensor[^*]{\mathcal{E}}{})$ of the corresponding left-dual super vector bundle $\tensor[^*]{\mathcal{V}}{}\rightarrow\!\tensor[^*]{\mathcal{E}}{}\rightarrow\mathcal{M}$ such that $(\tensor[^i]{\omega}{_x})_{i}$ defines a homogeneous basis of $\tensor[^*]{\mathcal{E}}{_x}$ $\forall x\in U$ and are dual in the sense that
\begin{equation}
\braket{X_{i}|\tensor[^j]{\omega}{}}=\delta{_i^j}
\label{eq:2.28}
\end{equation}
$\forall i,j=1,\ldots,n$\qed
\end{prop}
\begin{example}[Maurer-Cartan form]\label{example:MC}
Given a super Lie group $\mathcal{G}$, one can choose a real homogeneous basis $(X_{ie})_i$ of the super Lie module $\mathrm{Lie}(\mathcal{G})\equiv T_e\mathcal{G}=\Lambda\otimes\mathfrak{g}$. It then follows that the corresponding left-invariant vector fields $(X_{i})_i$ induce a homogeneous basis of the tangent module $T_g\mathcal{G}$ at any $g\in\mathcal{G}$. Thus, via Prop. \ref{prop:2.13}, this in turn induces a basis $(\tensor[^i]{\omega}{})_{i}$ of smooth 1-forms $\tensor[^i]{\omega}{}\in\Omega^1(\mathcal{M}):=\Gamma(\tensor[^*]{T\mathcal{M}}{})$, that is, smooth sections of the \emph{left-dual cotangent bundle} $\tensor[^*]{T\mathcal{M}}{}$. It follows immediately from the left invariance of the $X_i$ that the 1-forms $\tensor[^i]{\omega}{}$ are also left invariant.

The \emph{Maurer-Cartan form} on $\mathcal{G}$ is defined as the $\mathrm{Lie}(\mathcal{G})$-valued 1-form $\theta_{\mathrm{MC}}\in\Omega^1(\mathcal{G},\mathfrak{g}):=\Omega^1(\mathcal{G})\otimes\mathfrak{g}$ given by
\begin{equation}
    \theta_{\mathrm{MC}}:=\tensor[^i]{\omega}{}\otimes X_i
\end{equation}
By definition, the Maurer-Cartan form is left invariant. Moreover, using the generalized tangent map (see Definition \ref{definition:3.9} in section \ref{section:relative}), it follows that one can also write $(\theta_{\mathrm{MC}})_g=L_{g^{-1}*}$ $\forall g\in\mathcal{G}$ with $L_g:=\mu_{\mathcal{G}}(g,\cdot)$ the \emph{left translation} on $\mathcal{G}$.
\end{example}

\begin{example}[Tensor product of super vector bundles]\label{prop:2.14}
Let $\mathcal{V}\rightarrow\mathcal{E}\rightarrow\mathcal{M}$ and $\widetilde{\mathcal{V}}\rightarrow\widetilde{\mathcal{E}}\rightarrow\mathcal{M}$ be super vector bundles. Set
\begin{equation}
\mathcal{E}\otimes\widetilde{\mathcal{E}}:=\coprod_{x\in\mathcal{M}}{\mathcal{E}_x\otimes_{\Lambda}\widetilde{\mathcal{E}_x}}
\end{equation}
together with the surjective map
\begin{equation}
\pi_{\otimes}:\,\mathcal{E}\otimes\widetilde{\mathcal{E}}\rightarrow\mathcal{M},\,\mathcal{E}_x\otimes\widetilde{\mathcal{E}}_x\ni v_x\mapsto x
\label{eq:2.30}
\end{equation}
Let $\{(U_{\alpha},\phi_{\alpha})\}_{\alpha\in\Upsilon}$ and $\{(V_{\alpha'},\widetilde{\phi}_{\alpha'})\}_{\alpha'\in\Sigma}$ be smooth bundle atlases of $\mathcal{E}$ and $\widetilde{\mathcal{E}}$, respectively. For $(\alpha,\alpha')\in\Upsilon\times\Sigma$ with $U_{\alpha,\alpha'}:=U_{\alpha}\cap V_{\alpha'}\neq\emptyset$, define
\begin{align}
\psi_{\alpha,\alpha'}:\,\mathcal{E}\otimes\widetilde{\mathcal{E}}\supseteq\pi_{\otimes}^{-1}(U_{\alpha,\alpha'})&\rightarrow U_{\alpha,\alpha'}\times(\mathcal{V}\otimes_{\Lambda}\widetilde{\mathcal{V}})\label{eq:2.31}\\
v_x\otimes w_x&\mapsto(x,(\phi_{\alpha}\otimes\widetilde{\phi}_{\alpha'})(v_x\otimes w_x))\nonumber
\end{align}
Then, $\{(U_{\alpha,\alpha'},\psi_{\alpha,\alpha'})\}_{(\alpha,\alpha')\in\Upsilon\times\Sigma}$ defines a formal bundle atlas of $\mathcal{E}\otimes\widetilde{\mathcal{E}}$ turning $\mathcal{V}\otimes\widetilde{\mathcal{V}}\rightarrow\mathcal{E}\otimes\widetilde{\mathcal{E}}\stackrel{\pi_{\otimes}}{\rightarrow}\mathcal{M}$ into a super vector bundle called the \emph{tensor product super vector bundle}.
\end{example}

\begin{definition}[Principal super fiber bundle]\label{prop:2.15}
A \emph{principal super fiber bundle} is a super fiber bundle $\mathcal{G}\rightarrow\mathcal{P}\stackrel{\pi}{\rightarrow}\mathcal{M}$ such that
\begin{enumerate}[label=(\roman*)]
	\item the typical fiber is a super Lie group $\mathcal{G}$.
	\item the total space $\mathcal{P}$ is equipped with a smooth $\mathcal{G}$-right action $\Phi:\,\mathcal{P}\times\mathcal{G}\rightarrow\mathcal{P}$ preserving the fibers, that is, $\pi\circ\Phi=\pi\circ\mathrm{pr}_1$, or explicitly, $p\cdot g:=\Phi(p,g)\in\mathcal{P}_x$ $\forall g\in\mathcal{G}$ and $p\in\mathcal{P}$ with $\pi(p)=x\in\mathcal{M}$.
	\item there exists a smooth bundle atlas $\{(U_{\alpha},\phi_{\alpha})\}_{\alpha\in\Upsilon}$ of $\mathcal{P}$ such that, for each $\alpha\in\Upsilon$, the bundle chart $\phi_{\alpha}:\,\pi^{-1}(U_{\alpha})\rightarrow U_{\alpha}\times\mathcal{G}$ is $\mathcal{G}$-equivariant, i.e.
	\begin{equation}
\phi_{\alpha}\circ\Phi=(\mathrm{id}\times\mu_{\mathcal{G}})\circ(\phi_{\alpha}\times\mathrm{id}_{\mathcal{G}})
	\label{eq:2.32}
	\end{equation}
where $U_{\alpha}\times\mathcal{G}$ is equipped with the standard $\mathcal{G}$-right action $\mathrm{id}\times\mu_{\mathcal{G}}:\,(U_{\alpha}\times\mathcal{G})\times\mathcal{G}\rightarrow U_{\alpha}\times\mathcal{G},\,((x,g),g')\mapsto(x,gg')$.
\end{enumerate}
\end{definition}
\begin{prop}\label{prop:2.16}
Let $\mathcal{G}\rightarrow\mathcal{P}\stackrel{\pi_{\mathcal{P}}}{\rightarrow}\mathcal{M}$ be a principal super fiber bundle, then the orbit space $\mathcal{P}/\mathcal{G}$ equipped with the quotient topology can be given the structure a supermanifold such that $\mathcal{P}/\mathcal{G}$ is canonically isomorphic to $\mathcal{M}$ and the body $\mathbf{B}(\mathcal{P}/\mathcal{G})$ is isomorphic to $\mathbf{B}(\mathcal{P})/\mathbf{B}(\mathcal{G})$ in the sense of ordinary smooth manifolds.\\
Moreover, the sheaf $H^{\infty}_{\mathcal{P}/\mathcal{G}}$ of smooth functions on $\mathcal{P}/\mathcal{G}$ is canonically isomorphic to the quotient sheaf
\begin{equation}
H^{\infty}_{\mathcal{P}}/H^{\infty}_{\mathcal{G}}:\,\mathcal{P}/\mathcal{G}\supset U\rightarrow (H^{\infty}_{\mathcal{P}}/H^{\infty}_{\mathcal{G}})(U):=\{f\in H^{\infty}(\pi^{-1}(U))|\,\Phi^*(f)=f\otimes 1_{\mathcal{G}}\}
\label{eq:2.34}
\end{equation} 
where $\pi:\,\mathcal{P}\rightarrow\mathcal{P}/\mathcal{G}$ is the canonical projection.
\end{prop}
\begin{proof}
Since $\pi_{\mathcal{P}}\circ\Phi=\pi_{\mathcal{P}}\circ\mathrm{pr}_1$, $\pi_{\mathcal{P}}$ is constant on $\mathcal{G}$-orbits. Hence, by universal property of the quotient, there exists a unique continuous map $\hat{\pi}:\,\mathcal{P}/\mathcal{G}\rightarrow\mathcal{M}$ such that the following diagram commutes
\begin{displaymath}
	 \xymatrix{
         \mathcal{P}\ar[rr]^{\pi_{\mathcal{P}}} \ar[d]_{\pi}  & &   \mathcal{M} \\
            \mathcal{P}/\mathcal{G} \ar[rru]_{\hat{\pi}} &          
     }
		\label{eq:2.35}
 \end{displaymath}
It follows immediately that $\hat{\pi}$ is bijective. Moreover, since $\pi_{\mathcal{P}}$ is open as a bundle map, it follows, by definition of the quotient topology, that $\hat{\pi}$ is a homoemorphism. Choosing an atlas $\{(U_{\alpha},\phi_{\alpha})\}_{\alpha\in\Upsilon}$ on $\mathcal{M}$, this yields a corresponding atlas $\{(V_{\alpha},\psi_{\alpha})\}_{\alpha\in\Upsilon}$ of $\mathcal{P}/\mathcal{G}$ by setting $V_{\alpha}:=\hat{\pi}^{-1}(U_{\alpha})$ and $\psi_{\alpha}:=\phi_{\alpha}\circ\hat{\pi}:\,V_{\alpha}\rightarrow\phi_{\alpha}(U_{\alpha})\subset\Lambda^{m,n}$ $\forall\alpha\in\Upsilon$, where $\mathrm{dim}\,\mathcal{M}=(m,n)$. In this way, $\mathcal{P}/\mathcal{G}$ becomes a supermanifold diffeomorphic to $\mathcal{M}$ via $\hat{\pi}$. By Prop. \ref{prop:2.2}, $\mathbf{B}(\mathcal{G})\rightarrow\mathbf{B}(\mathcal{P})\rightarrow\mathbf{B}(\mathcal{M})$ defines a ordinary smooth principal fiber bundle over $\mathbf{B}(\mathcal{M})$ with structure group $\mathbf{B}(\mathcal{G})$. Hence, arguing as above, one concludes that $\mathbf{B}(\mathcal{P})/\mathbf{B}(\mathcal{G})$ is canonically diffeomorphic to $\mathbf{B}(\mathcal{M})$. To summarize, we have the isomorphism
\begin{equation}
\mathbf{B}(\mathcal{P}/\mathcal{G})\stackrel{\mathbf{B}(\hat{\pi})}{\longrightarrow}\mathbf{B}(\mathcal{M})\stackrel{\cong}{\longrightarrow}\mathbf{B}(\mathcal{P})/\mathbf{B}(\mathcal{G})
\label{eq:2.36}
\end{equation}
Finally, let $U\subseteq\mathcal{P}/\mathcal{G}$ be an open subset and $f\in H^{\infty}(\pi^{-1}(U))$ a smooth map with $\Phi^*(f)=f\otimes 1_{\mathcal{G}}$. Then, $f$ is constant on $\mathcal{G}$-orbits such that, by the universal property of the quotient, there exists a unique continuous map $\tilde{f}:\,U\rightarrow\Lambda$ with $\tilde{f}\circ\pi=f$. By definition of the differential structure on $\mathcal{P}/\mathcal{G}$, it follows immediately that $\tilde{f}$ is smooth, i.e., $\tilde{f}\in H^{\infty}(U)$. Conversely, if $g\in H^{\infty}(U)$, then $g':=\pi^*g=g\circ\pi|_{\pi^{-1}(U)}\in H^{\infty}(\pi^{-1}(U))$ satisfying $\Phi^*(g')=\Phi^*\circ\pi^*(g)=(\pi\circ\Phi)^*(g)=(\pi\circ\mathrm{pr}_1)^*(g)=\mathrm{pr}_1^*\circ\pi^*(g)=g'\otimes 1_{\mathcal{G}}$, i.e., $g'\in(H^{\infty}_{\mathcal{P}}/H^{\infty}_{\mathcal{G}})(U)$. This closes the prove of the above proposition.
\end{proof}

\begin{prop}\label{prop:2.17}
There is a one-to-one correspondence between local trivializations of a principal super fiber bundle $\mathcal{G}\rightarrow\mathcal{P}\stackrel{\pi}{\rightarrow}\mathcal{M}$ and smooth local sections $s\in\Gamma_{U}(\mathcal{P})$ of $\mathcal{P}$ for any $U\subseteq\mathcal{M}$ open.
\end{prop}
\begin{proof}
If $(U,\phi_U)$ is a local trivialization of $\mathcal{P}$, the map $s:\,\pi^{-1}(U)\ni p\mapsto s(x):=\phi_{U}^{-1}(x,e)\in\mathcal{P}_x$ defines, as $e\in\mathbf{B}(\mathcal{G})$, a smooth local section of $\mathcal{P}$ over $U$.\\
Conversely, if $s\in\Gamma_{U}(\mathcal{P})$ is a smooth local section of $\mathcal{P}$ over $U\subseteq\mathcal{M}$, consider the map
\begin{equation}
\psi:\,U\times\mathcal{G}\rightarrow\pi^{-1}(U),\,(x,g)\mapsto\Phi(s(x),g)
\label{eq:2.37}
\end{equation}
Then, $\psi$ defines a smooth fiber-preserving and $\mathcal{G}$-equivariant map which, in particular, is bijective, as the $\mathcal{G}$-right action $\Phi:\,\mathcal{P}\times\mathcal{G}\rightarrow\mathcal{P}$ on $\mathcal{P}$ is free and transitive. 
That the inverse is also smooth follows as in proof of Prop. \ref{prop:2.10}. 
\end{proof}


\begin{example}[The frame bundle $\mathscr{F}(\mathcal{E})$]\label{prop:2.19}
Given a super vector bundle $\mathcal{V}\rightarrow\mathcal{E}\stackrel{\pi}{\rightarrow}\mathcal{M}$, one can construct a new bundle as follows. For any $x\in\mathcal{M}$ we define a \emph{frame at} $x$ as an isomorphism of super $\Lambda$-modules $p_x:\,\mathcal{V}\rightarrow\mathcal{E}_x$. Let $\mathscr{F}(\mathcal{E})_x$ denote the set of all frames at $x\in\mathcal{M}$. We set
\begin{equation}
\mathscr{F}(\mathcal{E}):=\coprod_{x\in\mathcal{M}}{\mathscr{F}(\mathcal{E})_x}
\label{eq:2.38}
\end{equation}
together with the surjective map 
\begin{equation}
\pi_{\mathscr{F}}:\,\mathscr{F}(\mathcal{E})\rightarrow\mathcal{M},\,\mathscr{F}(\mathcal{E})_x\ni p_x\mapsto x\in\mathcal{M}
\label{eq:2.39}
\end{equation}
Furthermore, we introduce a $\mathrm{GL}(\mathcal{V})$-right action on $\mathscr{F}(\mathcal{E})$ via
\begin{align}
\Phi:\,\mathscr{F}(\mathcal{E})\times\mathrm{GL}(\mathcal{V})&\mapsto\mathscr{F}(\mathcal{E})\\
(p_x,g)&\mapsto(p_x\circ g)_{x}\nonumber
\label{eq:2.40}
\end{align}
It follows that $\Phi$ is fiber-preserving, i.e., $\Phi(p_x,g)\in\mathscr{F}(\mathcal{E})_x$ $\forall p_x\in\mathscr{F}(\mathcal{E})_x, g\in\mathrm{GL}(\mathcal{V})$ and $x\in\mathcal{M}$, and therefore is also free and transitive on each fiber. Let $\{(U_{\alpha},\phi_{\alpha})\}_{\alpha\in\Upsilon}$ be a smooth bundle atlas on $\mathcal{E}$. Then, for any $\alpha\in\Upsilon$ and $x\in U_{\alpha}$, the induced map $\phi_{\alpha,x}^{-1}:\,\mathcal{V}\rightarrow\mathcal{E}_x$ canonically yields a frame at $x$. Hence, let us define the new map 
\begin{align}
\psi_{\alpha}:\,U_{\alpha}\times\mathrm{GL}(\mathcal{V})&\rightarrow\pi_{\mathscr{F}}^{-1}(U_{\alpha})\subseteq\mathscr{F}(\mathcal{E})\\
(x,g)&\mapsto\Phi(\phi_{\alpha,x}^{-1},g)\nonumber
\label{eq:2.41}
\end{align}
This map is bijective with inverse $\psi^{-1}_{\alpha}(p_x)=(x,\phi_{\alpha,x}\circ p_x)$ $\forall p_x\in\mathscr{F}(\mathcal{E})_x$, $x\in U_{\alpha}$ (note that indeed $\phi_{\alpha,x}\circ p_x\in\mathrm{GL}(\mathcal{V})$). Let $\alpha,\beta\in\Upsilon$ with $U_{\alpha}\cap U_{\beta}\neq\emptyset$, the transition function $\psi^{-1}_{\beta}\circ\psi_{\alpha}:\,(U_{\alpha}\cap U_{\beta})\times\mathrm{GL}(\mathcal{V})\rightarrow(U_{\alpha}\cap U_{\beta})\times\mathrm{GL}(\mathcal{V})$ then takes the form
\begin{align}
\psi^{-1}_{\beta}\circ\psi_{\alpha}(x,g)&=\psi^{-1}_{\beta}(\Phi(\phi_{\alpha,x}^{-1},g))=(x,\phi_{\beta,x}\circ\phi_{\alpha,x}^{-1}\circ g)\nonumber\\
&=(x,g_{\beta\alpha}(x)\circ g)
\end{align}
$\forall (x,g)\in(U_{\alpha}\cap U_{\beta})\times\mathrm{GL}(\mathcal{V})$ and thus is smooth.\\
Hence, we have constructed an appropriate formal bundle of $\mathscr{F}(\mathcal{E})$ such that $\mathrm{GL}(\mathcal{V})\rightarrow\mathcal{\mathscr{F}(\mathcal{E})}\stackrel{\pi}{\rightarrow}\mathcal{M}$ turns into a principal super fiber bundle with structure group $\mathrm{GL}(\mathcal{V})$ and $\mathrm{GL}(\mathcal{V})$-right action $\Phi:\,\mathscr{F}(\mathcal{E})\times\mathrm{GL}(\mathcal{V})\mapsto\mathscr{F}(\mathcal{E})$, called the \emph{frame bundle} of $\mathcal{E}$. 
\end{example}
\begin{definition}\label{prop:2.20}
Given a supermanifold, the \emph{frame bundle} $\mathscr{F}(\mathcal{M})$ of $\mathcal{M}$ is defined as the frame bundle $\mathscr{F}(\mathcal{M})\equiv\mathscr{F}(T\mathcal{M})$ of the associated tangent bundle $T\mathcal{M}$.
\end{definition}

\begin{prop}[Associated fiber bundle]\label{prop:2.21}
Let $\mathcal{G}\rightarrow\mathcal{P}\stackrel{\pi_{\mathcal{P}}}{\rightarrow}\mathcal{M}$ be a principal super fiber bundle with $\mathcal{G}$-right action $\Phi:\,\mathcal{P}\times\mathcal{G}\rightarrow\mathcal{P}$. Let furthermore $\rho:\,\mathcal{G}\times\mathcal{F}\rightarrow\mathcal{F}$ be a smooth left action of $\mathcal{G}$ on a supermanifold $\mathcal{F}$. On $\mathcal{P}\times\mathcal{F}$ consider the map
\begin{equation}
\Phi^{\times}:\,(\mathcal{P}\times\mathcal{F})\times\mathcal{G}\rightarrow\mathcal{P}\times\mathcal{F},\,((p,v),g)\mapsto(\Phi(p,g),\rho(g^{-1},v))
\label{eq:2.42}
\end{equation} 
Then, $\Phi^{\times}$ defines an effective smooth $\mathcal{G}$-right action on $\mathcal{P}\times\mathcal{F}$. Let $\mathcal{E}:=\mathcal{P}\times_{\rho}\mathcal{F}:=(\mathcal{P}\times\mathcal{F})/\mathcal{G}$ be the corresponding coset space and $\pi_{\mathcal{E}}:\,\mathcal{E}\rightarrow\mathcal{M}$ be defined as
\begin{align}
\pi_{\mathcal{E}}:\,\mathcal{E}&\rightarrow\mathcal{M}\\
[p,v]&\mapsto\pi_{\mathcal{P}}(p)\nonumber
\label{eq:2.43}
\end{align}
Then, $\mathcal{E}$ can be equipped with the structure of a supermanifold such that $\pi_{\mathcal{E}}$ is a smooth surjective map and $(\mathcal{E},\pi_{\mathcal{E}},\mathcal{M},\mathcal{F})$ turns into a super fiber bundle, called the \emph{super fiber bundle associated to} $\mathcal{P}$ w.r.t. the $\mathcal{G}$-left action $\rho$ on $\mathcal{F}$. 
\end{prop}
\begin{proof}
It is immediate that $\Phi^{\times}:\,(\mathcal{P}\times\mathcal{F})\times\mathcal{G}\rightarrow\mathcal{P}\times\mathcal{F}$ defines a smooth $\mathcal{G}$-right action. That it also is effective follows from the fact that $\Phi:\,\mathcal{P}\times\mathcal{G}\rightarrow\mathcal{P}$ is effective.\\
Let $\{(U_{\alpha},\phi_{\alpha})\}_{\alpha\in\Upsilon}$ be a smooth bundle atlas of $\mathcal{G}\rightarrow\mathcal{P}\stackrel{\pi_{\mathcal{P}}}{\rightarrow}\mathcal{M}$. For $\alpha\in\Upsilon$ define 
\begin{align}
\phi^{\times}_{\alpha}:\,\pi_{\mathcal{E}}^{-1}(U_{\alpha})&\rightarrow U_{\alpha}\times\mathcal{F}\\
[p,v]&\rightarrow(\pi_{\mathcal{P}}(p),\rho(\mathrm{pr}_2\circ\phi_{\alpha}(p),v))\nonumber
\label{eq:2.44}
\end{align}
which is well-defined as $\rho$ is a $\mathcal{G}$-left action on $\mathcal{F}$. Moreover, $\phi^{\times}_{\alpha}$ is a bijection with inverse $(\phi^{\times}_{\alpha})^{-1}:\,U_{\alpha}\times\mathcal{F}\rightarrow\pi_{\mathcal{E}}^{-1}(U_{\alpha}),\,(x,v)\mapsto[\phi_{\alpha}^{-1}(x,e),v]$. For $\alpha,\alpha'\in\Upsilon$ with $U_{\alpha}\cap U_{\alpha'}\neq\emptyset$, it follows
\begin{equation}
\phi^{\times}_{\beta}\circ(\phi^{\times}_{\alpha})^{-1}(x,v)=(x,\rho(\mathrm{pr}_2\circ\phi_{\beta}\circ\phi_{\alpha}^{-1}(x,e),v))
\label{eq:2.45}
\end{equation}
$\forall (x,v)\in (U_{\alpha}\cap U_{\beta})\times\mathcal{F}$, and thus $\phi^{\times}_{\beta}\circ(\phi^{\times}_{\alpha})^{-1}$ is smooth as $\mathrm{pr}_2\circ\phi_{\beta}\circ\phi_{\alpha}^{-1}$ is smooth and $e\in\mathbf{B}(\mathcal{G})$ is a body point. Thus, the family $\{(U_{\alpha},\phi_{\alpha}^{\times})\}_{\alpha\in\Upsilon}$ defines a formal bundle atlas of $\mathcal{E}$ (w.r.t. $\pi_{\mathcal{E}}$) and hence induces a topology and supermanifold structure on $\mathcal{E}$ such that $(\mathcal{E},\pi_{\mathcal{E}},\mathcal{M},\mathcal{F})$ becomes a super fiber bundle.
\end{proof}
\begin{prop}\label{prop:2.22}
Under the conditions of Prop. \ref{prop:2.21}, the canonical projection
\begin{equation}
\pi:\,\mathcal{P}\times\mathcal{F}\rightarrow\mathcal{E}=\mathcal{P}\times_{\rho}\mathcal{F}
\label{eq:2.46}
\end{equation}
is a smooth bundle map, i.e., $\mathcal{P}\times\mathcal{F}\stackrel{\pi}{\rightarrow}\mathcal{P}\times_{\rho}\mathcal{F}$ carries the structure of a super fiber bundle with typical fiber $\mathcal{G}$. As a consequence, $\pi$ is an open submersion and the topology on $\mathcal{E}$ defined via the construction in the proof of Theorem \ref{prop:2.4} coincides with the quotient topology.
\end{prop}
\begin{proof}
First, let us show that $\pi$ is continuous. Therefore, if $U\subseteq\mathcal{E}$ is open, by definition of the topology on $\mathcal{E}$, for any $\alpha\in\Upsilon$, $U\cap\pi_{\mathcal{E}}^{-1}(U_{\alpha})$ is of the form $U\cap\pi_{\mathcal{E}}^{-1}(U_{\alpha})=(\phi_{\alpha}^{\times})^{-1}(\widetilde{U}_{\alpha})$ with $\widetilde{U}_{\alpha}\subseteq U_{\alpha}\times\mathcal{F}$ open. Then,
\begin{align}
\pi^{-1}(U)&=\pi^{-1}\left(\bigcup_{\alpha\in\Upsilon}{U\cap\pi_{\mathcal{E}}^{-1}(U_{\alpha})}\right)=\bigcup_{\alpha\in\Upsilon}{\pi^{-1}((\phi_{\alpha}^{\times})^{-1}(\widetilde{U}_{\alpha}))}\nonumber\\
&=\bigcup_{\alpha\in\Upsilon}{(\phi_{\alpha}^{\times}\circ\pi)^{-1}(\widetilde{U}_{\alpha})}
\label{eq:2.47}
\end{align}
and hence $\pi^{-1}(U)\subseteq\mathcal{P}\times\mathcal{F}$ is open proving that $\pi$ is continuous. That it is smooth can be checked by direct computation.\\
Next, to see that is in fact a bundle map, i.e., $\mathcal{P}\times\mathcal{F}\stackrel{\pi}{\rightarrow}\mathcal{P}\times_{\rho}\mathcal{F}$ carries the structure of a super fiber bundle, let us define an atlas $\{(\pi_{\mathcal{E}}^{-1}(U_{\alpha}),\widetilde{\phi}_{\alpha})\}_{\alpha\in\Upsilon}$ of locally trivializing bundle charts as follows. For $\alpha\in\Upsilon$, the map $\widetilde{\phi}_{\alpha}:\,\pi^{-1}(\pi_{\mathcal{E}}^{-1}(U_{\alpha}))=\pi_{\mathcal{P}}^{-1}(U_{\alpha})\times\mathcal{F}\rightarrow\pi_{\mathcal{E}}^{-1}(U_{\alpha})\times\mathcal{G}$ is obtained via the following commutative diagram
	\begin{displaymath}
	 \xymatrix{
         \pi_{\mathcal{P}}^{-1}(U_{\alpha})\times\mathcal{F}\ar[rr]^{\phi_{\alpha}\times\mathrm{id}_{\mathcal{F}}}_{\cong} \ar[d]_{\widetilde{\phi}_{\alpha}}  &  &  U_{\alpha}\times\mathcal{G}\times\mathcal{F} \ar[d]^{\mathrm{id}_{U_{\alpha}}\times\Theta}_{\cong}\\
            \pi_{\mathcal{E}}^{-1}(U_{\alpha})\times\mathcal{G}  & &   U_{\alpha}\times\mathcal{F}\times\mathcal{G} \ar[ll]^{(\phi^{\times})^{-1}\times\mathrm{id}_{\mathcal{G}}}_{\cong}     
     }
		\label{eq:2.48}
	\end{displaymath}
where $\Theta:\,\mathcal{G}\times\mathcal{F}\rightarrow\mathcal{F}\times\mathcal{G}$ is the diffeomorphism given by $\Theta(g,v):=(\rho(g,v),g)$ $\forall(g,v)\in\mathcal{G}\times\mathcal{F}$ and thus $\widetilde{\phi}_{\alpha}(p,v)=([p,v],\mathrm{pr}_2\circ\phi_{\alpha}(p))$ $\forall(p,v)\in\pi^{-1}(\pi_{\mathcal{E}}^{-1}(U_{\alpha}))$. It follows that $\widetilde{\phi}_{\alpha}$ is a diffeomorphism preserving the fibers, i.e., $\mathrm{pr}_{1}\circ\widetilde{\phi}_{\alpha}=\pi$ and thus indeed defines a local trivialization.\\
That $\pi:\,\mathcal{P}\times\mathcal{F}\rightarrow\mathcal{E}=\mathcal{P}\times_{\rho}\mathcal{F}$ is an open map follows by a standard argument using the fact that it is a bundle map and thus locally coincides with the projection $\mathrm{pr}_1$ which is open. It remains to show that the topology on $\mathcal{E}$ coincides with the quotient topology. Obviously, any open subset $U\subseteq\mathcal{E}$ is also open w.r.t. the quotient topology as $\pi$ is continuous. Conversely, if $U\subseteq\mathcal{E}$ is a subset of $\mathcal{E}$ such that $\pi^{-1}(U)\subseteq\mathcal{P}\times\mathcal{F}$ is open, it follows from $\pi(\pi^{-1}(U))=U$ that $U$ is open, as well. Hence, this proves the proposition.
\end{proof}

\begin{cor}\label{prop:2.23}
Let $\mathcal{G}\rightarrow\mathcal{P}\stackrel{\pi_{\mathcal{P}}}{\rightarrow}\mathcal{M}$ be a principal super fiber bundle with structure group $\mathcal{G}$ and $\rho:\,\mathcal{G}\rightarrow\mathrm{GL}(\mathcal{V})$ be a representation of $\mathcal{G}$ on a finite dimensional super $\Lambda$-vector space $\mathcal{V}$ which induces the smooth $\mathcal{G}$-left action $\rho:\,\mathcal{G}\times\mathcal{V}\rightarrow\mathcal{V},\,(g,v)\mapsto\rho(g,v)\equiv\rho(g)v$ on $\mathcal{V}$. Then, the associated fiber bundle $\mathcal{E}:=\mathcal{P}\times_{\rho}\mathcal{V}$ can be given the structure of a super vector bundle where, for each $x\in\mathcal{M}$, the fiber $\mathcal{E}_x=\mathcal{P}_x\times_{\rho}\mathcal{V}$ carries the structure of a free super $\Lambda$-module via
\begin{align}
a[p,v]+b[p,w]:=[p,av+bw]
\label{eq:2.49}
\end{align}
$\forall [p,v],[p,w]\in\mathcal{E}_x$, $a,b\in\Lambda$ and $\mathbb{Z}_2$-grading
\begin{equation}
(\mathcal{E}_x)_{0}:=\{[p,v]\in\mathcal{E}_x|\,v\in\mathcal{V}_{0}\},\quad(\mathcal{E}_x)_{1}:=\{[p,v]\in\mathcal{E}_x|\,v\in\mathcal{V}_1\}
\label{eq:2.50}
\end{equation} 
The bundle $\mathcal{V}\rightarrow\mathcal{E}\rightarrow\mathcal{M}$ is called the \emph{super vector bundle} associated to $\mathcal{P}$ w.r.t. the representation $\rho$.\qed
\end{cor}
\begin{cor}\label{prop:2.24}
Let $\mathcal{G}\rightarrow\mathcal{P}\stackrel{\pi_{\mathcal{P}}}{\rightarrow}\mathcal{M}$ be a principal super fiber bundle with $\mathcal{G}$-right action $\Phi:\,\mathcal{P}\times\mathcal{G}\rightarrow\mathcal{P}$ and let $\rho:\,\mathcal{G}\rightarrow\mathrm{GL}(\mathcal{V})$ be a representation of $\mathcal{G}$ on a finite dimensional super $\Lambda$-vector space $\mathcal{V}$. Let $H^{\infty}(\mathcal{P},\mathcal{V})^{\mathcal{G}}$ be the subspace of smooth $\mathcal{G}$-equivariant functions on $\mathcal{P}$ with values in $\mathcal{V}$ defined as
\begin{equation}
H^{\infty}(\mathcal{P},\mathcal{V})^{\mathcal{G}}:=\{f\in H^{\infty}(\mathcal{P},\mathcal{V})|\,f(p\cdot g)=\rho(g^{-1})(f(p)),\,\forall p\in\mathcal{P},g\in\mathcal{G}\}
\label{eq:2.52}
\end{equation}
Let $\mathcal{E}:=\mathcal{P}\times_{\rho}\mathcal{V}$ be the associated super vector bundle. Then, there exists an isomorphism of super vector spaces given by
\begin{equation}
\Psi:\,H^{\infty}(\mathcal{P},\mathcal{V})^{\mathcal{G}}\rightarrow\Gamma(\mathcal{E}),\,f\mapsto(\Psi(f):\,\pi_{\mathcal{P}}(p)\mapsto[p,f(p)])
\label{eq:2.53}
\end{equation}
\end{cor}
\begin{proof}
For $f\in H^{\infty}(\mathcal{P},\mathcal{V})^{\mathcal{G}}$ define the smooth map $\hat{f}:\,\mathcal{P}\rightarrow\mathcal{E}$ via $\hat{f}(p):=[p,f(p)]$, $\forall p\in\mathcal{P}$. By construction, this yields $\Phi^*(\hat{f})=\hat{f}\otimes 1_{\mathcal{G}}$, that is, $\hat{f}$ is constant on $\mathcal{G}$-orbits. Thus, according to Prop. \ref{prop:2.16}, this induces a smooth map $\Psi(f):\,\mathcal{M}\rightarrow\mathcal{E}$ satisfying $\Psi(f)\circ\pi_{\mathcal{P}}=\hat{f}$ and, in particular, is fiber-preserving. Hence, $\Psi(f)\in\Gamma(\mathcal{E})$.\\
Conversely, suppose one has given a smooth section $X\in\Gamma(\mathcal{E})$. For any $p\in\mathcal{P}$, we define the bijective map $[p]:\,\mathcal{V}\rightarrow\mathcal{E}_{x},\,v\mapsto[p,v]$ with $x:=\pi_{\mathcal{P}}(p)$. Hence, let us define a map $f:\,\mathcal{P}\rightarrow\mathcal{V}$ via $f(p)=[p]^{-1}(X(x))$, $\forall p\in\mathcal{P}_x$. By construction, it then follows $f(p\cdot g)=\rho(g^{-1})(f(p))$ for any $g\in\mathcal{G}$. To see that it is smooth, let us choose a local section $s:\,U\rightarrow\mathcal{P}$ of the principal super fiber bundle $\mathcal{P}$ with $U\subseteq\mathcal{M}$ open which, in turn, induces a local trivialization $\psi:\,\mathcal{P}\supseteq\pi_{\mathcal{P}}^{-1}(U)\rightarrow U\times\mathcal{G}$ of $\mathcal{P}$ as well as smooth sections $\hat{s}_i:\,U\rightarrow\mathcal{E}$ of the associated super vector bundle $\mathcal{E}$ via $\hat{s}(x):=[s(x),e_i]$ $\forall x\in U$ where $(e_i)_i$ is a real homogeneous basis of $\mathcal{V}$. By Prop. \ref{prop:2.10}, this induces a local trivialization $\phi:\,\mathcal{E}\supseteq\pi_{\mathcal{E}}^{-1}(U)\rightarrow U\times\mathcal{V}$ of $\mathcal{E}$ over $U$. It is then easy to see that there exists a smooth map $v:\,U\rightarrow\mathcal{\mathcal{V}}$ such that $X(x)=[s(x),v(x)]$ $\forall x\in U$. It then follows $f\circ\psi^{-1}(x,g)=\rho(g^{-1},v(x))$ $\forall(x,g)\in U\times\mathcal{G}$ which is obviously smooth proving that indeed $f\in H^{\infty}(\mathcal{P},\mathcal{V})^{\mathcal{G}}$. This shows that $\Psi$ is bijective. That $\Psi$ is linear and preserves the grading is immediate. 
\end{proof}
\begin{example}[vector bundles as associated bundles]\label{prop:2.25}
Let $\mathcal{V}\rightarrow\mathcal{E}\stackrel{\pi}{\rightarrow}\mathcal{M}$ be a super vector bundle. The general linear group $\mathrm{GL}(\mathcal{V})$ acts in a trivial way on $\mathcal{V}$ via the so-called \emph{fundamental representation}
\begin{equation}
\rho\equiv\mathrm{id}:\,\mathrm{GL}(\mathcal{V})\rightarrow\mathrm{GL}(\mathcal{V}),\,g\mapsto g
\label{eq:2.54}
\end{equation}
This yields the associated super vector bundle $\mathscr{F}(\mathcal{E})\times_{\rho}\mathcal{V}$. Consider the map
\begin{align}
\mathscr{F}(\mathcal{E})\times_{\rho}\mathcal{V}&\rightarrow\mathcal{E}\label{eq:2.55}\\
[p_x,v]&\rightarrow(v^iX_{ix})_x\nonumber
\end{align}
where we have chosen a real homogeneous basis $(e_i)_i$ of $\mathcal{V}$ and set $X_{ix}:=p_x(e_i)$ $\forall i$. It is immediate that \eqref{eq:2.55} is smooth and well-defined, fiber-preserving as well as an isomorphism of super $\Lambda$-modules on each fiber. Hence, it indeed defines an isomorphism of super vector bundles. This shows that each super vector bundle is associated to a principal super fiber bundle.\\ 
\\
This construction also yields a new characterization of the (left-)dual super vector bundle $\tensor[^*]{\mathcal{V}}{}\rightarrow\!\tensor[^*]{\mathcal{E}}{}\rightarrow\mathcal{M}$ as follows. A representation $\rho:\,\mathcal{G}\rightarrow\mathrm{GL}(\mathcal{V})$ of a super Lie group on a super $\Lambda$-vector space yields a representation $\tensor[^*]{\rho}{}:\,\mathcal{G}\rightarrow\mathrm{GL}(\tensor[^*]{\mathcal{V}}{})$ on the corresponding left-dual super $\Lambda$-vector space $\tensor[^*]{\mathcal{V}}{}$ via
\begin{equation}
\tensor[^*]{\rho}{}(g)\ell:\,v\mapsto\braket{v|\tensor[^*]{\rho}{}(g)\ell}:=\braket{\braket{v|\rho(g^{-1})}|\ell}
\label{eq:2.56}
\end{equation}
$\forall g\in\mathcal{G}$, $\ell\in\!\tensor[^*]{\mathcal{V}}{}$ and $v\in\mathcal{V}$, called the \emph{left-dual representation} of $\mathcal{G}$. \\
In our case, i.e., $\mathcal{G}=\mathrm{GL}(\mathcal{V})$ and $\rho\equiv\mathrm{id}$, this yields the associated super vector bundle $\mathscr{F}(\mathcal{E})\times_{\tensor[^*]{\rho}{}}\!\tensor[^*]{\mathcal{V}}{}$. In fact, in turns out that this bundle is isomorphic to the left-dual super vector bundle via
\begin{align}
\mathscr{F}(\mathcal{E})\times_{\tensor[^*]{\rho}{}}\!\tensor[^*]{\mathcal{V}}{}&\rightarrow\tensor[^*]{\mathcal{E}}{}\\
[p_x,\ell]&\rightarrow(\mathrel{^i\! X_x}\ell_{i})_x\nonumber
\label{eq:2.57}
\end{align}
where, for any $x\in\mathcal{M}$, $(\tensor[^i]{X}{_x})_i$ denotes the left-dual basis of $(X_{ix})_i$ in $\tensor[^*]{\mathcal{E}}{_x}$. 
\end{example}


\begin{cor}\label{prop:2.27}
Let $\mathcal{H}\rightarrow\mathcal{P}\stackrel{\pi_{\mathcal{P}}}{\rightarrow}\mathcal{M}$ be a principal super fiber bundle with structure group $\mathcal{H}$ and $\mathcal{H}$-right action $\Phi:\mathcal{P}\times\mathcal{H}\rightarrow\mathcal{H}$. Let $\lambda:\,\mathcal{H}\rightarrow\mathcal{G}$ be a morphism of super Lie groups and $\rho_{\lambda}:=\mu_{\mathcal{G}}\circ(\lambda\times\mathrm{id}_{\mathcal{G}}):\,\mathcal{H}\times\mathcal{G}\rightarrow\mathcal{G}$ the induced smooth left action of $\mathcal{H}$ on $\mathcal{G}$. Then, the associated fiber bundle $\mathcal{P}\times_{\mathcal{H}}\mathcal{G}:=\mathcal{P}\times_{\rho_{\lambda}}\mathcal{G}$ can be given the structure of a principal super fiber bundle with structure group $\mathcal{G}$ and $\mathcal{G}$-right action
\begin{align}
\widetilde{\Phi}:\,(\mathcal{P}\times_{\mathcal{H}}\mathcal{G})\times\mathcal{G}&\rightarrow\mathcal{P}\times_{\mathcal{H}}\mathcal{G}\\
([p,g],h)&\mapsto[p,\mu_{\mathcal{G}}(g,h)]\nonumber
\label{eq:2.59}
\end{align}
Furthermore, let $\iota:\,\mathcal{P}\rightarrow\mathcal{P}\times_{\mathcal{H}}\mathcal{G}$ be defined as $\iota(p):=[p,e]$ $\forall p\in\mathcal{P}$, then $\iota$ is smooth, fiber-preserving and $\mathcal{H}$-equivariant in the sense that $\iota\circ\Phi=\widetilde{\Phi}\circ(\iota\times\lambda)$. Moreover, if $\lambda:\,\mathcal{H}\hookrightarrow\mathcal{G}$ is an embedding, then $\iota$ is an embedding.
\end{cor}
\begin{proof}
First, let us show that $\widetilde{\Phi}:\,(\mathcal{P}\times_{\mathcal{H}}\mathcal{G})\times\mathcal{G}\rightarrow\mathcal{P}\times_{\mathcal{H}}\mathcal{G}$ is a smooth $\mathcal{G}$-right action with respect to which the local trivializations $\phi_{\alpha}^{\times}$ are $\mathcal{G}$-equivariant. To see that is continuous, consider the smooth map $\Phi':\,(\mathcal{P}\times\mathcal{G})\times\mathcal{G}\rightarrow\mathcal{P}\times_{\mathcal{H}}\mathcal{G},\,((p,g),h)\mapsto[p,\mu_{\mathcal{G}}(g,h)]$. Since $\Phi'$ is constant on $\mathcal{G}$-orbits it follows from Prop. \eqref{prop:2.22}, i.e., the topology on $\mathcal{E}$ coincides with the quotient topology, that this induces a continuous map $\widetilde{\Phi}:\,(\mathcal{P}\times_{\mathcal{H}}\mathcal{G})\times\mathcal{G}\rightarrow\mathcal{P}\times_{\mathcal{H}}\mathcal{G}$ such that the following diagram commutes
\begin{displaymath}
	 \xymatrix{
         (\mathcal{P}\times\mathcal{G})\times\mathcal{G}\ar[rr]^{\Phi'} \ar[d]_{\pi\times\mathrm{id}_{\mathcal{G}}}  & &   \mathcal{P}\times_{\mathcal{H}}\mathcal{G} \\
            (\mathcal{P}\times_{\mathcal{H}}\mathcal{G})\times\mathcal{G} \ar[rru]_{\widetilde{\Phi}} &          
     }
		\label{eq:2.60}
 \end{displaymath}
That $\widetilde{\Phi}$ is smooth can be checked by direct computation. Furthermore, for a locally trivializing bundle chart $\phi^{\times}_{\alpha}:\,\pi_{\mathcal{E}}^{-1}(U_{\alpha})\rightarrow U_{\alpha}\times\mathcal{G}$, we compute
\begin{align}
\phi^{\times}_{\alpha}\circ\widetilde{\Phi}([p,g],h)&=(\pi_{\mathcal{P}}(p),\rho(\mathrm{pr}_2\circ\phi_{\alpha}(p),\mu_{\mathcal{G}}(g,h)))\nonumber\\
&=(\mathrm{id}\times\mu_{\mathcal{G}})(\pi_{\mathcal{P}}(p),\rho(\mathrm{pr}_2\circ\phi_{\alpha}(p),g)),h)\nonumber\\
&=(\mathrm{id}\times\mu_{\mathcal{G}})\circ\phi^{\times}_{\alpha}([p,g],h)
\label{eq:2.61}
\end{align}
and hence $\phi^{\times}_{\alpha}\circ\widetilde{\Phi}=(\mathrm{id}\times\mu_{\mathcal{G}})\circ\phi^{\times}_{\alpha}$ as required.\\
Finally, let us consider the map $\iota:\,\mathcal{P}\rightarrow\mathcal{P}\times_{\mathcal{H}}\mathcal{G},\,p\mapsto[p,e]$. Since $\iota=\pi\circ\iota_{\mathcal{P}}$ is a composition of the canonical projection $\pi$ as well as the embedding $\iota_{\mathcal{P}}:\,\mathcal{P}\rightarrow\mathcal{P}\times\{e\}\subset\mathcal{P}\times\mathcal{G}$, it follows that $\iota$ is smooth. That $\iota$ is fiber-preserving is immediate and for $(p,h)\in\mathcal{P}\times\mathcal{H}$ it follows $\iota\circ\Phi(p,h)=[\Phi(p,h),e]=[p,\rho_{\lambda}(h,e)]=[p,\mu_{\mathcal{G}}(e,\lambda(h))]=\widetilde{\Phi}\circ(\iota\times\lambda)(p,h)$, that is, $\iota$ is $\mathcal{H}$-equivariant. To prove the last assertion, if $\lambda:\,\mathcal{H}\hookrightarrow\mathcal{G}$ is an embedding, it follows that $\iota$ is injective. Moreover, since $\pi$ is an open submersion, $\iota$ is a homeomorphism onto its image. To see that it is an immersion and thus an embedding, let $\{(U_{\alpha},\phi^{\times}_{\alpha})\}_{\alpha\in\Upsilon}$ be a smooth bundle atlas of $\mathcal{P}\times_{\rho_{\lambda}}\mathcal{G}$ induced by a smooth bundle atlas of $\mathcal{P}$. Then, for $\alpha\in\Upsilon$, this yields
\begin{equation}
\phi_{\alpha}^{\times}\circ\iota(p)=(\pi_{\mathcal{P}}(p),\lambda(\mathrm{pr}_2\circ\phi_{\alpha}(p)))=(\mathrm{id}\times\lambda)\circ\phi_{\alpha}(p)
\label{eq:2.62}
\end{equation}
$\forall p\in\mathcal{P}$. But, since $\phi_{\alpha}$ is a diffeomorphism and $\mathrm{id}\times\lambda:\,\mathcal{P}\times\mathcal{H}\rightarrow\mathcal{P}\times\mathcal{G}$ an embedding the claim follows.
\end{proof}
\begin{definition}\label{prop:2.28}
Under the assumptions of Corollary \ref{prop:2.27}, the associated principal super fiber bundle $\mathcal{P}\times_{\mathcal{H}}\mathcal{G}:=\mathcal{P}\times_{\rho_{\lambda}}\mathcal{G}$ is called the \emph{$\lambda$-extension} of $\mathcal{P}$. If $\lambda:\,\mathcal{H}\hookrightarrow\mathcal{G}$ is an embedding, $\mathcal{P}\times_{\mathcal{H}}\mathcal{G}$ will also simply be called the \emph{$\mathcal{G}$-extension} of $\mathcal{P}$.
\end{definition}
\begin{definition}\label{prop:2.29}
Let $\mathcal{G}\rightarrow\mathcal{P}\stackrel{\pi_{\mathcal{P}}}{\rightarrow}\mathcal{M}$ be a principal super fiber bundle with structure group $\mathcal{G}$ and $\lambda:\,\mathcal{H}\longrightarrow\mathcal{G}$ a morphism of super Lie groups. A $\mathcal{H}$-bundle $\mathcal{H}\rightarrow\mathcal{Q}\stackrel{\pi_{\mathcal{Q}}}{\rightarrow}\mathcal{M}$ over $\mathcal{M}$ is called a \emph{$\lambda$-reduction} of $\mathcal{P}$, if there exists a smooth map $\Lambda:\,\mathcal{Q}\longrightarrow\mathcal{P}$ such that the following diagram commutes
\begin{equation} 
\begin{xy}
 \xymatrix{
         \mathcal{Q}\times\mathcal{H} \ar[rr] \ar[dd]_{\Lambda\times\lambda}   & &   \mathcal{Q}\ar[dd]^{\Lambda} \ar[rd]^{\pi_{\mathcal{Q}}} & \\
																&			&			&	\mathcal{M} \\
         \mathcal{P}\times\mathcal{G} \ar[rr]       &       &  \mathcal{P}\ar[ru]_{\pi_{\mathcal{P}}}  & 
     }
 \end{xy}
\end{equation}
i.e. $\Lambda$ is fiber-preserving and $\mathcal{H}$-equivariant in the sense that $\Lambda\circ\Phi_{\mathcal{Q}}=\Phi_{\mathcal{P}}\circ(\Lambda\times\lambda)$ with $\Phi_{\mathcal{Q}}:\,\mathcal{Q}\times\mathcal{H}\rightarrow\mathcal{Q}$ and $\Phi_{\mathcal{P}}:\,\mathcal{P}\times\mathcal{G}\rightarrow\mathcal{P}$ the super Lie group actions on $\mathcal{Q}$ and $\mathcal{P}$, respectively.
\end{definition}
\begin{cor}\label{prop:2.30}
Under the assumptions of Corollary \ref{prop:2.27}, the $\mathcal{H}$-bundle $\mathcal{P}$ is a $\lambda$-reduction of the $\mathcal{G}$-bundle $\mathcal{P}\times_{\mathcal{H}}\mathcal{G}$.\qed
\end{cor}

\section{$\mathcal{S}$-relative super connection forms}\label{section:relative}
One of the main issues when working in the standard category of supermanifolds, both in the $H^{\infty}$ or the algebraic category, is that superfields on the body of a supermanifold only contain commuting (bosonic) degrees of freedom, that is, there are no anticommuting (fermionic) field configurations on the body.\\
This, however, turns out to be incompatible with various constructions in physics, such as in the geometric approach to supergravity. In the Castellani-D'Auria-Fré approach to supergraviy \cite{DAuria:1982uck,Castellani:1991et}, for instance, one has the so-called \emph{rheonomy principle} stating that physical fields are completely fixed by their pullback to the body manifold.\\
This can be cured by factorising a given supermanifold $\mathcal{M}$ by an additional parametrizing supermanifold $\mathcal{S}$ and studying superfields on $\mathcal{S}\times\mathcal{M}$. One is then interested in a certain subclass of such superfields which depend as little as possible on this additional supermanifold making them covariant in a specific sense under a change of parametrization. This idea is based on a proposal formulated already by Schmitt in \cite{Schmitt:1996hp} which is motivated by the functorial approach to supermanifold theory according to Molotkov \cite{Mol:10} and Sachse \cite{Sac:09}.
This approach also recently found application in context of superconformal field theories on super Riemannian surfaces \cite{Jost:2014wfa,Kessler:2019bwp} as well as in context of the local approach to super quantum field theories (QFT) \cite{Hack:2015vna}. Moreover, as will be explained in more detail in section \ref{sec:Param}, the description of fermionic fields turns out to be quite similar to considerations in perturbative algebraic QFT \cite{Rejzner:2011au,Rejzner:2016hdj}. \\
In the following, we will adopt the terminology of \cite{Hack:2015vna} introducing the notion of a relative supermanifold. However, unlike as in \cite{Hack:2015vna}, in order to study fermionic fields, we will not restrict on superpoints as parametrizing supermanifolds. We will then define principal connections and connection 1-forms on parametrizing supermanifolds. These results will then be applied in section \ref{section:parallelNEU} for the construction of the parallel tansport map.

\begin{definition}\label{def:9.1}
Let $\mathcal{S}$ and $\mathcal{M}$ be supermanifolds. The pair $(\mathcal{S}\times\mathcal{M},\mathrm{pr}_{\mathcal{S}})$ with $\mathrm{pr}_{\mathcal{S}}:\,\mathcal{S}\times\mathcal{M}\rightarrow\mathcal{S}$ the projection onto the first factor is called a \emph{$\mathcal{S}$-relative supermanifold} also denoted by $\mathcal{M}_{/\mathcal{S}}$.\\
A morphism $\phi:\,\mathcal{M}_{/\mathcal{S}}\rightarrow\mathcal{N}_{/\mathcal{S}}$ between $\mathcal{S}$-relative supermanifolds is a morphism $\phi:\,\mathcal{S}\times\mathcal{M}\rightarrow\mathcal{S}\times\mathcal{N}$ of supermanifolds preserving the projections, i.e. the following diagram is commutative
\begin{displaymath}
	 \xymatrix{
         \mathcal{S}\times\mathcal{M}\ar[rr]^{\phi} \ar[dr]_{\mathrm{pr}_{\mathcal{S}}}  & &    \mathcal{S}\times\mathcal{N} \ar[dl]^{\mathrm{pr}_{\mathcal{S}}}\\
            &    \mathcal{S}   &   
     }
		\label{eq:9.1}
 \end{displaymath}
Hence, $\phi(s,p)=(s,\tilde{\phi}(s,p))$ $\forall (s,p)\in\mathcal{S}\times\mathcal{M}$ with $\tilde{\phi}:=\mathrm{pr}_{\mathcal{N}}\circ\phi:\,\mathcal{S}\times\mathcal{M}\rightarrow\mathcal{N}$. This yields a category $\mathbf{SMan}_{/\mathcal{S}}$ called the \emph{category of $\mathcal{S}$-relative supermanifolds}.
\end{definition}
The following proposition gives a different characterization of morphism between $\mathcal{S}$-relative supermanifolds.
\begin{prop}[after \cite{Hack:2015vna}]\label{prop:9.2}
Let $\mathcal{M}_{/\mathcal{S}},\mathcal{N}_{/\mathcal{S}}\in\mathbf{Ob}(\mathbf{SMan}_{/\mathcal{S}})$ be $\mathcal{S}$-relative supermanifolds. Then, the map
\begin{align}
\alpha_{\mathcal{S}}:\,\mathrm{Hom}_{\mathbf{SMan}_{/\mathcal{S}}}(\mathcal{M}_{/\mathcal{S}},\mathcal{N}_{/\mathcal{S}})&\rightarrow\mathrm{Hom}_{\mathbf{SMan}_{H^{\infty}}}(\mathcal{S}\times\mathcal{M},\mathcal{S}\times\mathcal{N})\label{eq:9.2}\\
(\phi:\,\mathcal{S}\times\mathcal{M}\rightarrow\mathcal{S}\times\mathcal{N})&\mapsto(\mathrm{pr}_{\mathcal{N}}\circ\phi:\,\mathcal{S}\times\mathcal{M}\rightarrow\mathcal{N})\nonumber
\end{align}
is a bijection with the inverse given by
\begin{align}
\alpha_{\mathcal{S}}^{-1}:\,\mathrm{Hom}_{\mathbf{SMan}_{H^{\infty}}}(\mathcal{S}\times\mathcal{M},\mathcal{N})&\rightarrow\mathrm{Hom}_{\mathbf{SMan}_{/\mathcal{S}}}(\mathcal{M}_{/\mathcal{S}},\mathcal{N}_{/\mathcal{S}})\label{eq:9.3}\\
(\psi:\,\mathcal{S}\times\mathcal{M}\rightarrow\mathcal{N})&\mapsto((\mathrm{id}_{\mathcal{S}}\times\psi)\circ(d_{\mathcal{S}}\times\mathrm{id}_{\mathcal{M}}):\,\mathcal{S}\times\mathcal{M}\rightarrow\mathcal{S}\times\mathcal{N})\nonumber
\end{align}
with $d_{\mathcal{S}}:\,\mathcal{S}\rightarrow\mathcal{S}\times\mathcal{S}$ the diagonal map.\qed
\end{prop}
Let $\lambda:\,\mathcal{S}\rightarrow\mathcal{S}'$ be a morphism between parametrizing supermanifolds, we will also call such a morphism a \emph{change of parametrization}. Then, any smooth map $\phi:\,\mathcal{S'}\times\mathcal{M}\rightarrow\mathcal{N}$ can be pulled back via $\lambda$ to a morphism $\lambda^*\phi:=\phi\circ(\lambda\times\mathrm{id}_{\mathcal{M}}):\,\mathcal{S}\times\mathcal{M}\rightarrow\mathcal{N}$. Using \ref{eq:9.3}, this yields the map \cite{Hack:2015vna}
\begin{align}
\lambda^{*}:\,\mathrm{Hom}_{\mathbf{SMan}_{/\mathcal{S}'}}(\mathcal{M}_{/\mathcal{S}'},\mathcal{N}_{/\mathcal{S'}})&\rightarrow\mathrm{Hom}_{\mathbf{SMan}_{/\mathcal{S}}}(\mathcal{M}_{/\mathcal{S}},\mathcal{N}_{/\mathcal{S}})\label{eq:9.4}\\
\phi&\mapsto\alpha_{\mathcal{S}}^{-1}(\alpha_{\mathcal{S}'}(\phi)\circ(\lambda\times\mathrm{id}_{\mathcal{M}}))\nonumber
\end{align}
Hence, explicitly, for $\phi:\,\mathcal{M}_{\mathcal{S}'}\rightarrow\mathcal{N}_{\mathcal{S'}}$, $\lambda^*(\phi)$ reads
\begin{align}
\alpha_{\mathcal{S}}^{-1}(\phi)(s,p)=(s,\mathrm{pr}_{\mathcal{N}}\circ\phi(\lambda(s),p))
\label{eq:9.5}
\end{align}
$\forall(s,p)\in\mathcal{S}\times\mathcal{M}$. The following proposition demonstrates that the set of morphisms between relative supermanifolds is functorial in the parametrizing supermanifold and thus indeed has the required properties under change of parametrization.
\begin{prop}[after \cite{Hack:2015vna}]
The assignment
\begin{align}
\mathbf{SMan}\rightarrow\mathbf{Set}:\,\mathbf{Ob}(\mathbf{SMan})\ni\mathcal{S}&\mapsto\mathrm{Hom}_{\mathbf{SMan}_{/\mathcal{S}}}(\mathcal{M}_{/\mathcal{S}},\mathcal{N}_{/\mathcal{S}})\in\mathbf{Ob}(\mathbf{Set})\nonumber\\
(\lambda:\,\mathcal{S}\rightarrow\mathcal{S}')&\mapsto\lambda^{*}
\label{eq:9.6}
\end{align}
defines a contravariant functor on the category $\mathbf{SMan}_{H^{\infty}}$ of $H^{\infty}$-supermanifolds. Moreover, the map $\lambda^{*}$ associated to the morphism $\lambda:\,\mathcal{S}\rightarrow\mathcal{S}'$ preserves compositions, i.e., $\lambda^*(\phi\circ\psi)=\lambda^*(\phi)\circ\lambda^{*}(\psi)$ for any $\psi:\,\mathcal{M}\rightarrow\mathcal{N}$ and $\phi:\,\mathcal{N}\rightarrow\mathcal{L}$
\end{prop}
\begin{proof}
This is an immediate consequence of the identities \eqref{eq:9.2}, \eqref{eq:9.3} and \eqref{eq:9.4}. Alternatively, one may directly prove this proposition using the explicit formula \eqref{eq:9.5} valid in the $H^{\infty}$-category.
\end{proof}
\begin{definition}
\begin{enumerate}[label=(\roman*)]
	\item Let $\mathcal{M}_{/\mathcal{S}}\in\mathbf{Ob}(\mathbf{SMan}_{/\mathcal{S}})$ be a $\mathcal{S}$-relative supermanifold and $p\in\mathcal{M}_{/\mathcal{S}}$ a point, that is, a tuple $p\equiv(s,x)$ in $\mathcal{S}\times\mathcal{M}$. A \emph{tangent vector} $X_p$ at $p$ is defined as a tangent vector $X_p\in T_p(\mathcal{S}\times\mathcal{M})$ satisfying
\begin{equation}
X_p(f\otimes 1)=0,\,\forall f\in H^{\infty}(\mathcal{S})
\label{eq:9.7}
\end{equation}
The collection of tangent vectors at $p$ defines a super $\Lambda$-sub module of $T_p(\mathcal{S}\times\mathcal{M})$ denoted by $T_p(\mathcal{M}_{/\mathcal{S}})$ which we call the \emph{tangent module} of $\mathcal{M}_{/\mathcal{S}}$ at $p$.\\
A \emph{smooth vector field} $X\in\mathfrak{X}(\mathcal{M}_{/\mathcal{S}})$ on $\mathcal{M}_{/\mathcal{S}}$ is a smooth section $X\in\mathfrak{X}(\mathcal{S}\times\mathcal{M})$ of the tangent bundle of $\mathcal{S}\times\mathcal{M}$ such that $X_p\in T_p(\mathcal{M}_{/\mathcal{S}})$ for any $p\in\mathcal{S}\times\mathcal{M}$. The vector fields on $\mathcal{M}_{/\mathcal{S}}$ form a super $H^{\infty}(\mathcal{S}\times\mathcal{M})$-sub module $\mathfrak{X}(\mathcal{M}_{/\mathcal{S}})$ of $\mathfrak{X}(\mathcal{S}\times\mathcal{M})$ isomorphic to $H^{\infty}(\mathcal{S})\otimes\mathfrak{X}(\mathcal{M})$.
	\item A \emph{co-tangent vector} $\omega_p$ at $p\in\mathcal{M}_{/\mathcal{S}}$ is defined a left-linear morphism $\omega_p:\,T_p(\mathcal{M}_{/\mathcal{S}})\rightarrow\Lambda$. A \emph{$1$-form} $\omega$ on $\mathcal{M}_{/\mathcal{S}}$ is a left-linear morphism of super $H^{\infty}(\mathcal{S}\times\mathcal{M})$-modules $\omega\in\underline{\mathrm{Hom}}_L(\mathfrak{X}(\mathcal{M}_{/\mathcal{S}}),H^{\infty}(\mathcal{S}\times\mathcal{M}))$. The set $\Omega^1(\mathcal{M}_{/\mathcal{S}})$ of $\mathcal{S}$-relative $1$-forms on $\mathcal{M}_{/\mathcal{S}}$ defines a super $H^{\infty}(\mathcal{S}\times\mathcal{M})$-sub module of $\Omega^1(\mathcal{S}\times\mathcal{M})$ isomorphic to $\Omega^1(\mathcal{M})\otimes H^{\infty}(\mathcal{S})$.
	\item Analogously, one defines $k$-forms $\omega\in\Omega^k(\mathcal{M}_{/\mathcal{S}})$ on $\mathcal{M}_{/\mathcal{S}}$, $k\in\mathbb{N}$, as skew-symmetric $k$-left linear morphisms $\omega\in\underline{\mathrm{Hom}}_{a,L}(\mathfrak{X}(\mathcal{M}_{/\mathcal{S}})^k,H^{\infty}(\mathcal{S}\times\mathcal{M}))$ of super $H^{\infty}(\mathcal{S}\times\mathcal{M})$-modules. For $k=0$, we set $\Omega^0(\mathcal{M}_{/\mathcal{S}})\equiv H^{\infty}(\mathcal{S}\times\mathcal{M})$.\\
	Operations on forms, such as the \emph{exterior} and \emph{interior derivative} as well as \emph{Lie derivative} can be defined as in the non-relative setting (see e.g. \cite{Deligne:1999qp,Bartocci,Tuynman:2004,Stavracou:1996qb}). For instance, given a 1-form $\omega\in\Omega^1(\mathcal{M}_{/\mathcal{S}})$ and homogeneous $X\in\mathfrak{X}(\mathcal{M}_{/\mathcal{S}})$, the interior derivative is given by $\iota_X(\omega):=\braket{X|\omega}$ and similarly for arbitrary $k$-forms. The Lie derivative is then defined via the \emph{(graded) Cartan formula} $L_X:=\mathrm{d}\circ\iota_X+\iota_X\circ\mathrm{d}$. Moreover, one has the important identity
	\begin{equation}
	[L_X,\iota_Y]\equiv L_X\circ\iota_Y-(-1)^{|X||Y|}\iota_Y\circ L_X=\iota_{[X,Y]}
	\label{eq:identity}
	\end{equation}
for any homogeneous $X,Y\in\mathfrak{X}(\mathcal{M}_{/\mathcal{S}})$.
\end{enumerate}
\end{definition}

\begin{definition}\label{definition-bundle}
\begin{enumerate}[label=(\roman*)]
	\item Consider a principal super fiber bundle $\mathcal{G}\rightarrow\mathcal{P}\stackrel{\pi}{\rightarrow}\mathcal{M}$ with $\mathcal{G}$-right action $\Phi:\,\mathcal{P}\times\mathcal{G}\rightarrow\mathcal{P}$ as well as a parametrizing supermanifold $\mathcal{S}$. Taking products, this yields a fiber bundle 
\begin{displaymath}
	 \xymatrix{
            \mathcal{S}\times\mathcal{P} \ar[d]^{\pi_{\mathcal{S}}} & \mathcal{G}\ar[l]\\
              \mathcal{S}\times\mathcal{M} & 
     }
		\label{eq:9.8}
 \end{displaymath}
with projection $\pi_{\mathcal{S}}:=\mathrm{id}_{\mathcal{S}}\times\pi$ and $\mathcal{G}$-right action $\Phi_{\mathcal{S}}:=\mathrm{id}_{\mathcal{S}}\times\Phi:\,(\mathcal{S}\times\mathcal{P})\times\mathcal{G}\rightarrow\mathcal{S}\times\mathcal{P}$. By construction, $\pi_{\mathcal{S}}$ defines a morphism of $\mathcal{S}$-relative supermanifolds. Moreover, $\Phi_{\mathcal{S}}$ satisfies $\pi_{\mathcal{S}}\circ\Phi_{\mathcal{S}}=\pi_{\mathcal{S}}\circ\mathrm{pr}_{\mathcal{S}\times\mathcal{P}}$ as well as 
\begin{equation}
\Phi_{\mathcal{S}}\circ(\Phi_{\mathcal{S}}\times\mathrm{id})=\Phi_{\mathcal{S}}\circ(\mathrm{id}\times\mu)
\label{eq:0.1}
\end{equation}
We will call such a group action a $\mathcal{G}$-right action of $\mathcal{S}$-relative supermanifolds.\\
Hence, this yields a fiber bundle $\mathcal{G}\rightarrow\mathcal{P}_{/\mathcal{S}}\stackrel{\pi_{\mathcal{S}}}{\rightarrow}\mathcal{M}_{/\mathcal{S}}$ in the category $\mathbf{SMan}_{/\mathcal{S}}$ of $\mathcal{S}$-relative supermanifolds which will be called a \emph{$\mathcal{S}$-relative principal super fiber bundle}.
\item Let $\mathcal{V}\rightarrow\mathcal{E}\stackrel{\pi}{\rightarrow}\mathcal{M}$ be a super vector bundle. Similarly as above, taking products, this yields a fiber bundle $\mathcal{V}\rightarrow\mathcal{E}_{/\mathcal{S}}\stackrel{\pi_{\mathcal{S}}}{\rightarrow}\mathcal{M}_{/\mathcal{S}}$ in the category $\mathbf{SMan}_{/\mathcal{S}}$ with typical fiber given by a super $\Lambda$-vector space $\mathcal{V}$ which will be called a \emph{$\mathcal{S}$-relative super vector bundle}.\\
A smooth section $X\in\Gamma(\mathcal{E}_{/\mathcal{S}})$ of the $\mathcal{S}$-relative super vector bundle $\mathcal{E}_{/\mathcal{S}}$ is given by morphisms $X:\,\mathcal{M}_{/\mathcal{S}}\rightarrow\mathcal{E}_{/\mathcal{S}}$ of $\mathcal{S}$-relative supermanifolds satisfying $\pi_{\mathcal{S}}\circ X=\mathrm{id}$.
\end{enumerate}
\end{definition}

\begin{remark}\label{prop:3.6}
Let $\mathcal{M}_{/\mathcal{S}}\in\mathbf{Ob}(\mathbf{SMan}_{/\mathcal{S}})$ be a $\mathcal{S}$-relative supermanifold. Choosing a local coordinate neighborhood of $\mathcal{M}$, it is immediate to see that for any $(s,p)\in\mathcal{M}_{/\mathcal{S}}$, the tangent module $T_{(s,p)}(\mathcal{M}_{/\mathcal{S}})$ is isomorphic to $T_{p}\mathcal{M}$ via $T_p\mathcal{M}\ni X_p\mapsto\mathds{1}\otimes X_p\in T_{(s,p)}(\mathcal{M}_{/\mathcal{S}})$.\\
On $\mathcal{S}\times\mathcal{M}$ consider the assignment $T(\mathcal{M}_{/\mathcal{S}}):\,\mathcal{S}\times\mathcal{M}\ni(s,p)\mapsto T_{(s,p)}(\mathcal{M}_{/\mathcal{S}})$ which defines a subbundle of $T(\mathcal{S}\times\mathcal{M})$ and which we call the \emph{tangent bundle} of $\mathcal{M}_{/\mathcal{S}}$. On the other hand, according to Definition \ref{definition-bundle} (ii), one can also consider the $\mathcal{S}$-relative super vector bundle $(T\mathcal{M})_{/\mathcal{S}}$. Together with the previous observation, we obtain an isomorphism
\begin{align}
(T\mathcal{M})_{/\mathcal{S}}&\stackrel{\sim}{\rightarrow}T(\mathcal{M}_{/\mathcal{S}}),\,(s,X_p)\mapsto(\mathds{1}\otimes X)_{(s,p)}\in T_{(s,p)}(\mathcal{M}_{/\mathcal{S}})
\end{align}
Moreover, via this identification, it follows that smooth vector fields on $\mathcal{M}_{/\mathcal{S}}$ can be identified with smooth sections $X\in\Gamma((T\mathcal{M})_{/\mathcal{S}})$ of the $\mathcal{S}$-relative super vector bundle $(T\mathcal{M})_{/\mathcal{S}}$.\\
In a similar way, it follows that $\mathcal{S}$-relative 1-forms on $\mathcal{M}_{/{\mathcal{S}}}$ can be identified with smooth sections $\omega\in\Gamma((\tensor[^*]{T\mathcal{M}}{})_{/\mathcal{S}})$ of the $\mathcal{S}$-relative super vector bundle $(\tensor[^*]{T\mathcal{M}}{})_{/\mathcal{S}}$.
\end{remark}
As in the ordinary theory of principal fiber bundles and gauge theory in physics, connections and connection 1-forms play a very central role. To implement these notions in the category of relative supermanifolds, we first have to introduce the notion of a geometric distribution.
\begin{definition}
A \emph{smooth geometric distribution $\mathcal{E}$ of rank $k|l$} on a $\mathcal{S}$-relative supermanifold $\mathcal{M}_{/\mathcal{S}}$ is an assignment
\begin{equation}
\mathcal{E}:\,\mathcal{M}_{/\mathcal{S}}\ni p\mapsto\mathcal{E}_p\subseteq T_p(\mathcal{M}_{/\mathcal{S}})
\end{equation} 
mapping each point $p\in\mathcal{M}_{/\mathcal{S}}$ to a super $\Lambda$-sub module $\mathcal{E}_p$ of $T_p(\mathcal{M}_{/\mathcal{S}})$ of dimension $k|l$ such that, for any $p\in\mathcal{M}_{/\mathcal{S}}$, there exists $p\in U\subseteq\mathcal{S}\times\mathcal{M}$ open as well as a family $(X_i)_{i=1,\ldots,k+l}$ of smooth vector fields $X\in\mathfrak{X}(U)$ on $\mathcal{S}\times\mathcal{M}$ such that $X_i(q)\in T_q(\mathcal{M}_{/\mathcal{S}})$ $\forall q\in U$ and $(X_i(q))_i$ defines a homogeneous basis of $\mathcal{E}_q$.
\end{definition}
\begin{lemma}[a reformulation of \cite{Tuynman:2004}]\label{prop:3.8}
Let $\mathcal{M}_{/\mathcal{S}}$ be a $\mathcal{S}$-relative supermanifold of dimension $(m,n)$ and $\mathcal{W}\rightarrow\mathcal{F}_{/\mathcal{S}}\rightarrow\mathcal{N}_{/\mathcal{S}}$ be a $\mathcal{S}$-relative super vector bundle of rank $k|l$. Let $(\phi,g):\,T\mathcal{M}_{/\mathcal{S}}\rightarrow\mathcal{F}_{/\mathcal{S}}$ a (right linear) morphism of $\mathcal{S}$-relative super vector bundles such that $\phi_p:\,T_p\mathcal{M}_{/\mathcal{S}}\rightarrow(\mathcal{F}_{/\mathcal{S}})_{g(p)}$ is surjective $\forall p\in\mathcal{M}_{/\mathcal{S}}$. Then
\begin{equation}
\mathrm{ker}(\phi):=\coprod_{p\in\mathcal{M}_{/\mathcal{S}}}{\mathrm{ker}(\phi_{p}:\,T_p\mathcal{M}_{/\mathcal{S}}\rightarrow(\mathcal{F}_{/\mathcal{S}})_{g(p)})}
\end{equation}
defines a smooth geometric distribution of rank $(m-k)|(n-l)$ on $\mathcal{M}_{/\mathcal{S}}$.
\end{lemma}
\begin{proof}
The following proof is a generalization of the proof in the ordinary $H^{\infty}$-category as given in \cite{Tuynman:2004}. For an arbitrary but fixed $p_0\in\mathcal{M}_{/\mathcal{S}}$, let $(U_{/\mathcal{S}},\psi_U)$ and $(V_{/\mathcal{S}},\psi'_V)$ be local trivializations of $T\mathcal{M}_{/\mathcal{S}}$ and $\mathcal{F}_{/\mathcal{S}}$, respectively, such that $p_0\in U_{/\mathcal{S}}$ and $\mathcal{S}\times U\subset g^{-1}(\mathcal{S}\times V)$. Consider the map $\phi_{VU}:=\psi'_V\circ\phi\circ\psi^{-1}_U:\,U_{/\mathcal{S}}\times\Lambda^{m|n}\rightarrow V_{/\mathcal{S}}\times\mathcal{W}$. Since $\phi$ is a right linear super vector bundle morphism, $\mathrm{pr}_2\circ\phi_{VU}(p,\,\cdot\,):\,\Lambda^{m|n}\rightarrow\mathcal{W}\in\mathrm{End}_R(\Lambda^{m|n},\mathcal{W})$ is a right linear map $\forall p\in U_{/\mathcal{S}}$. Hence, there exists a (unique) smooth map $\widetilde{\phi}_{VU}:\,\mathcal{S}\times U\rightarrow\mathrm{End}_R(\Lambda^{m|n},\mathcal{W})$ such that $\widetilde{\phi}_{VU}(p)v=\mathrm{pr}_2\circ\phi_{VU}(p,v)$ $\forall (p,v)\in(\mathcal{S}\times U)\times\Lambda^{m|n}$. Thus, w.r.t. a real homogeneous basis $(e_i)_{i=1,\ldots,m+n}$ and $(f_j)_{j=1,\ldots,k+l}$ of $\Lambda^{m|n}$ and $\mathcal{W}$, respectively, this yields
\begin{equation}
\phi_{VU}(p,v)=(g(p),f_j\tensor{\widetilde{\phi}}{_{VU}^j_i}(p)v^i)
\end{equation}
Since $\phi_p:\,T_p\mathcal{M}_{/\mathcal{S}}\rightarrow(\mathcal{F}_{/\mathcal{S}})_{g(p)}$ is surjective, the rank of the super matrix $\tensor{\widetilde{\phi}}{_{VU}^j_i}(p)$ has to be $k|l$ $\forall p\in\mathcal{S}\times U$. Hence, after reordering the real homogeneous basis $(e_i)_i$, we may assume that the sub super matrix $(\tensor{\widetilde{\phi}}{_{VU}^r_i}(p_0))_{i,r=1,\ldots,k+l}$ is invertible. Hence, there exists $U'\subset\mathcal{S}\times U$ open such that $(\tensor{\widetilde{\phi}}{_{VU}^r_i}(p))_{i,r=1,\ldots,k+l}$ is invertible $\forall p\in U'$. As taking the inverse of super matrices is smooth, the inverse matrix $(\tensor{h}{^i_r}(p))_{i,r=1,\ldots,k+l}$ of $(\tensor{\widetilde{\phi}}{_{VU}^r_i}(p))$ $\forall p\in U'$ defines a smooth right-linear map from $\mathcal{W}$ to $\Lambda^{m|n}$.\\
For $i'=k+l,\ldots,m+n$, set
\begin{equation}
s_{i'}(p):=\psi_U^{-1}\left(p,e_{i'}-e_i\tensor{h}{^i_r}(p)\tensor{\widetilde{\phi}}{_{VU}^r_{i'}}(p)\right)
\end{equation}
$\forall p\in U'$. Then, since $\phi$ is homogeneous, it follows that $s_{i'}\in\mathfrak{X}(U)$, $i'=k+l,\ldots,m+n$ define smooth vector fields on $U'$ such that $s_{i'}(p)\in T_p\mathcal{M}_{/\mathcal{S}}$ $\forall p\in U'$ and $(s_{i'}(p))_{i'}$ is a homogeneous basis of $\mathrm{ker}(\phi)_p$. Hence, this proves that $\mathrm{ker}(\phi)$ defines a smooth geometric distribution of rank $(m-k)|(n-l)$ on $\mathcal{M}_{/\mathcal{S}}$.  
\end{proof}
Before we proceed with the definition of vertical and horizontal distributions as well as connection forms on principal super fiber bundles, let us first note a very important fact concerning partial evaluation of smooth maps defined on supermanifolds. More precisely, note that, in the $H^{\infty}$-category, the space of smooth functions defines a $\mathbb{R}$-vector space. Thus, it follows that, given a smooth map $\phi:\,\mathcal{M}\times\mathcal{N}\rightarrow\mathcal{L}$ between $H^{\infty}$-supermanifolds as well as a point $q\in\mathcal{N}$, the map
\begin{equation}
\phi_q:=\phi(\cdot,q):\,\mathcal{M}\rightarrow\mathcal{L}
\end{equation}
in general, will not be of class $H^{\infty}$, unless $g\in\mathbf{B}(\mathcal{N})$. However, following \cite{Tuynman:2004}, one can still associate a tangent map to $\phi_q$ even if $q$ is not an element of the body. More precisely, one makes the following definition.
\begin{definition}\label{definition:3.9}
Let $\phi:\,\mathcal{M}\times\mathcal{N}\rightarrow\mathcal{L}$ be a smooth map between supermanifolds and $q\in\mathcal{N}$ be an arbitrary but fixed point. Using the identification $T(\mathcal{M}\times\mathcal{N})\cong T(\mathcal{M})\times T\mathcal{N}$, the \emph{generalized tangent map} $\phi_{q*}:\,T_p\mathcal{M}\rightarrow T_{\phi(p,q)}\mathcal{L}$ of $\phi_q\equiv\phi(\cdot,q)$ at $p\in\mathcal{M}$ is defined as
\begin{equation}
\phi_{q*}(X_{p})\equiv D_p\phi_q(X_p):=D_{(p,g)}\phi(X_{p},0_q)
\end{equation}
for any $X_{p}\in T_{p}\mathcal{M}$. Similarly, one defines $\phi_{p*}$ for $p\in\mathcal{M}$.
\end{definition}

\begin{definition}\label{definition:3.10}
Let $\psi:\,\mathcal{M}\times\mathcal{N}\rightarrow\mathcal{L}$ and $\phi:\,\mathcal{K}\times\mathcal{L}\rightarrow\mathcal{P}$ be smooth maps between supermanifolds and $p\in\mathcal{K}$ and $q\in\mathcal{M}$ arbitrary fixed points. Then, the generalized tangent map $(\phi_p\circ\psi_q)_{*}$ of the continuous map $\phi_p\circ\psi_q=\phi(p,\psi(q,\,\cdot\,)):\,\mathcal{N}\rightarrow\mathcal{P}$ is defined, according to \ref{definition:3.9}, as the generalized tangent map $(\phi\times(\mathrm{id}\times\psi))_{(p,q)*}$ associated to $(\phi\times(\mathrm{id}\times\psi))_{(p,q)}:\,\mathcal{N}\rightarrow\mathcal{P}$.
\end{definition}

\begin{prop}\label{definition:3.11}
Let $\psi:\,\mathcal{M}\times\mathcal{N}\rightarrow\mathcal{L}$ and $\phi:\,\mathcal{K}\times\mathcal{L}\rightarrow\mathcal{P}$ be smooth maps between supermanifolds and $p\in\mathcal{K}$ and $q\in\mathcal{M}$ arbitrary fixed points. Then,
\begin{equation}
(\phi_p\circ\psi_q)_{*}=\phi_{p*}\circ\psi_{q*}
\end{equation}
\end{prop}
\begin{proof}
For $X_g\in T_g{\mathcal{N}}$, $g\in\mathcal{N}$, we compute
\begin{align}
(\phi_p\circ\psi_q)_{*}(X_g)&:=(\phi\times(\mathrm{id}\times\psi))_{(p,q)*}(X_g)=D(\phi\times(\mathrm{id}\times\psi))(0_{(p,q)},X_g)\nonumber\\
&=D\phi(D(\mathrm{id}\times\psi)(0_{(p,q)},X_g))\nonumber\\
&=D\phi(0_p,D\psi(0_{q},X_g))\nonumber\\
&=\phi_{p*}\circ\psi_{q*}(X_g)
\end{align}
\end{proof}

\begin{definition}
Let $\Phi_{\mathcal{S}}:\,\mathcal{M}_{/\mathcal{S}}\times\mathcal{G}\rightarrow\mathcal{M}_{/\mathcal{S}}$ be a smooth right action of a super Lie group $\mathcal{G}$ on a $\mathcal{S}$-relative supermanifold $\mathcal{M}_{/\mathcal{S}}$. A \emph{fundamental tangent vector} $\widetilde{X}_p\in T_p\mathcal{M}_{/\mathcal{S}}$ at $p\in\mathcal{M}_{/\mathcal{S}}$ associated to $X\in\mathrm{Lie}(\mathcal{G})=T_e\mathcal{G}$ is defined as
\begin{equation}
\widetilde{X}_p:=(\Phi_{\mathcal{S}})_{p*}(X)\in T_{p}\mathcal{M}_{/\mathcal{S}}
\end{equation}
where we made use of the generalized tangent map. In case $X\in\mathfrak{g}$, this yields a smooth vector field $\widetilde{X}\in\mathfrak{X}(\mathcal{M}_{/\mathcal{S}})$ which can equivalently be written as
\begin{equation}
\widetilde{X}=(\mathds{1}\otimes X_e)\circ\Phi_{\mathcal{S}}^{*}
\end{equation}
and which is called the \emph{fundamental vector field} generated by $X$.
\end{definition}

\begin{definition}
Let $\mathcal{G}\rightarrow\mathcal{P}_{/\mathcal{S}}\stackrel{\pi_{\mathcal{S}}}{\rightarrow}\mathcal{M}_{/\mathcal{S}}$ be a $\mathcal{S}$-relative principal super fiber bundle. The \emph{vertical tangent module} $\mathscr{V}_p$ of $\mathcal{P}_{/\mathcal{S}}$ at a point $p\in\mathcal{P}_{/\mathcal{S}}$ is a super $\Lambda$-sub module of the tangent module $T_p\mathcal{P}_{/\mathcal{S}}$ defined as
\begin{equation}
\mathscr{V}_p:=\mathrm{ker}(D_p\pi_{\mathcal{S}})
\end{equation}
\end{definition}
\makeatletter
\newcommand{\raisemath}[1]{\mathpalette{\raisem@th{#1}}}
\newcommand{\raisem@th}[3]{\raisebox{#1}{$#2#3$}}
\makeatother

\begin{lemma}\label{prop:3.14}
Let $\Phi_{\mathcal{S}}:\,\mathcal{M}_{/\mathcal{S}}\times\mathcal{G}\rightarrow\mathcal{M}_{/\mathcal{S}}$ be a smooth right action of a super Lie group $\mathcal{G}$ on a $\mathcal{S}$-relative supermanifold $\mathcal{M}_{/\mathcal{S}}$. Then, for $X\in\mathrm{Lie}(\mathcal{G})=T_e\mathcal{G}$, one has
\begin{equation}
(\Phi_{\mathcal{S}})_{g*}\widetilde{X}_p=\widetilde{\mathrm{Ad}_{g^{-1}}X}_{\!\raisemath{2pt}{p\cdot g}},\quad\forall p\in\mathcal{P}_{/\mathcal{S}},\,g\in\mathcal{G}
\end{equation}
with $\widetilde{\mathrm{Ad}_{g^{-1}}X}_{\!\raisemath{2pt}{p\cdot g}}$ the fundamental tangent vector associated to $\mathrm{Ad}_{g^{-1}}X\in\mathrm{Lie}(\mathcal{G})$ at $p\cdot g\in\mathcal{M}_{/\mathcal{S}}$.
\end{lemma}
\begin{proof}
By Prop. \ref{definition:3.11}, we find
\begin{align}
(\Phi_{\mathcal{S}})_{g*}\widetilde{X}_p&=(\Phi_{\mathcal{S}})_{g*}\circ(\Phi_{\mathcal{S}})_{p*}(X)=((\Phi_{\mathcal{S}})_{g}\circ(\Phi_{\mathcal{S}})_{p})_{*}(X)\nonumber\\
&=(\Phi_{\mathcal{S}}\circ(\Phi_{\mathcal{S}}\times\mathrm{id}))_{(p,\,\cdot\,,g)*}(X)\nonumber\\
&=(\Phi_{\mathcal{S}}\circ(\mathrm{id}\times\mu))_{(p,\,\cdot\,,g)*}(X)\nonumber\\
&=(\Phi_{\mathcal{S}})_{p*}\circ R_{g*}(X)
\end{align}
Since $R_{g*}=L_{g^*}\circ L_{g^{-1}*}\circ R_{g*}=L_{g^*}\circ\mathrm{Ad}_{g^{-1}}$, it thus follows 
\begin{align}
(\Phi_{\mathcal{S}})_{g*}\widetilde{X}_p&=(\Phi_{\mathcal{S}})_{p*}\circ L_{g^*}\circ L_{g^{-1}*}\circ R_{g*}(X)=(\Phi_{\mathcal{S}})_{p\cdot g*}\circ\mathrm{Ad}_{g^{-1}}(X)=\widetilde{\mathrm{Ad}_{g^{-1}}X}_{\!\raisemath{2pt}{p\cdot g}}
\end{align}
as claimed.
\end{proof}

\begin{prop}
Let $\mathcal{G}\rightarrow\mathcal{P}_{/\mathcal{S}}\rightarrow\mathcal{M}_{/\mathcal{S}}$ be a $\mathcal{S}$-relative principal super fiber bundle with right-action $\Phi_{\mathcal{S}}:\,\mathcal{P}_{/\mathcal{S}}\times\mathcal{G}\rightarrow\mathcal{P}_{/\mathcal{S}}$. For any $p\in\mathcal{P}_{/\mathcal{S}}$, one has 
\begin{equation}
\mathscr{V}_p=\{\widetilde{X}_p|\,X\in\mathrm{Lie}(\mathcal{G})\}
\label{eq:3.0.1}
\end{equation}
i.e., the vertical tangent module $\mathscr{V}_p$ is generated by the fundamental tangent vectors at $p$. In particular, the assignment $\mathcal{V}:\,\mathcal{M}_{/\mathcal{S}}\ni p\mapsto\mathscr{V}_p$ defines a smooth geometric distribution of rank $\mathrm{dim}\,\mathfrak{g}$ called the \emph{vertical tangent bundle} which is right-invariant in the sense that $(\Phi_{\mathcal{S}})_{g*}\mathscr{V}_p=\mathscr{V}_{p\cdot g}$ $\forall g\in\mathcal{G}$.
\end{prop}
\begin{proof}
For any $p\in\mathcal{P}_{/\mathcal{S}}$, let $s:\,U_{/\mathcal{S}}\rightarrow\mathcal{P}_{/\mathcal{S}}$ be a smooth local section of $\mathcal{P}_{/\mathcal{S}}$ with $\pi_{\mathcal{S}}(p)\in\mathcal{S}\times U\subseteq\mathcal{S}\times\mathcal{M}$ open. Then, $\pi_{\mathcal{S}}\circ s=\mathrm{id}$ implies $\pi_{\mathcal{S}*}\circ s_{*}=\mathrm{id}$ and thus $D_p\pi_{\mathcal{S}}:\,T_p\mathcal{M}_{/\mathcal{S}}\rightarrow T_{\pi_{\mathcal{S}}(p)}\mathcal{P}_{/\mathcal{S}}$ is surjective. Since, $D_p\pi_{\mathcal{S}}$ is homogeneous, it follows that $\mathscr{V}_p=\mathrm{ker}\,D_p\pi_{\mathcal{S}}$ is a super $\Lambda$-sub module of $T_p(\mathcal{P}_{/\mathcal{S}})$ of dimension $\mathrm{dim}\,\mathscr{V}_p=\mathrm{dim}\,T_p\mathcal{P}_{/\mathcal{S}}-\mathrm{dim}\,T_{\pi_{\mathcal{S}}(p)}\mathcal{M}_{/\mathcal{S}}=\mathrm{dim}\,\mathfrak{g}$ $\forall p\in\mathcal{P}_{/\mathcal{S}}$.\\
For $X\in T_e\mathcal{G}=\mathrm{Lie}(\mathcal{G})$, the associated fundamental tangent vector $\widetilde{X}_p$ at $p\in\mathcal{P}_{/\mathcal{S}}$ is given by $\widetilde{X}_p=(\Phi_{\mathcal{S}})_{p*}X$. By definition of the generalized tangent map, this yields 
\begin{align}
D_p\pi_{\mathcal{S}}(\widetilde{X}_p)&=D_p\pi_{\mathcal{S}}\circ D_{(p,e)}\Phi_{\mathcal{S}}(0_p,X)\nonumber\\
&=D_{(p,e)}(\pi_{\mathcal{S}}\circ\Phi_{\mathcal{S}})(0_p,X)\nonumber\\
&=D_{(p,e)}(\pi_{\mathcal{S}}\circ\mathrm{pr}_1)(0_p,X)=0
\end{align}
i.e., $\widetilde{X}_p\in\mathscr{V}_p$. As $\Phi_{p*}:\,T_e\mathcal{G}\mapsto T_p\mathcal{P}_{/\mathcal{S}}$ is an even and injective map of super $\Lambda$-modules, it follows from the observation above that it is an isomorphism onto $\mathscr{V}_p$ proving \eqref{eq:3.0.1}.\\
To prove the last assertion, let $(X_i)_i$ be a real homogeneous basis of $T_e\mathcal{G}$ and $\widetilde{X}=\mathds{1}\otimes X\circ\Phi_{\mathcal{S}}^*$ the associated smooth fundamental vector fields on $\mathcal{P}_{/\mathcal{S}}$. Then, $(\widetilde{X}_i(p))_i$ is a homogeneous basis of $\mathscr{V}_p$ $\forall p\in\mathcal{P}_{/\mathcal{S}}$ and thus $\mathscr{V}:\,\mathcal{M}_{/\mathcal{S}}\ni p\mapsto\mathscr{V}_p$ defines a smooth geometric distribution of rank $\mathrm{dim}\,\mathfrak{g}$. It remains to show that $\mathscr{V}$ is indeed right-invariant, that is, $(\Phi_{\mathcal{S}})_{g*}\mathscr{V}_p=\mathscr{V}_{p\cdot g}$ $\forall g\in\mathcal{G}$. Therefore, if $\widetilde{X}_p\in\mathscr{V}_p$ with $X\in\mathrm{Lie}(\mathcal{G})$, it follows from Lemma \ref{prop:3.14}
\begin{equation}
(\Phi_{\mathcal{S}})_{g*}\widetilde{X}_p=\widetilde{\mathrm{Ad}_{g^{-1}}X}_{\!\raisemath{2pt}{p\cdot g}}\in\mathscr{V}_{p\cdot g}.
\end{equation} 
Since $(\Phi_{\mathcal{S}})_{g*}\circ(\Phi_{\mathcal{S}})_{g^{-1}*}=\mathrm{id}$, the claim follows.
\end{proof}
\begin{definition}
Let $\mathcal{G}\rightarrow\mathcal{P}_{/\mathcal{S}}\stackrel{\pi}{\rightarrow}\mathcal{M}_{/\mathcal{S}}$ be a $\mathcal{S}$-relative principal super fiber bundle with right-action $\Phi_{\mathcal{S}}:\,\mathcal{P}_{/\mathcal{S}}\times\mathcal{G}\rightarrow\mathcal{P}_{/\mathcal{S}}$. A \emph{principal connection} (à la Ehresmann) $\mathscr{H}$ on $\mathcal{P}_{/\mathcal{S}}$ is a smooth geometric distribution $\mathscr{H}:\,\mathcal{P}_{/\mathcal{S}}\ni p\mapsto\mathscr{H}_p\subset T_p(\mathcal{P}_{/\mathcal{S}})$ of horizontal tangent modules on $\mathcal{P}_{/\mathcal{S}}$ of rank $\mathrm{dim}\,\mathcal{M}$ such that $\mathscr{H}_p\oplus\mathscr{V}_p=T_p(\mathcal{P}_{/\mathcal{S}})$ and $\mathscr{H}$ is right-invariant in the sense that
\begin{equation}
(\Phi_{\mathcal{S}})_{g*}\mathscr{H}_p=\mathscr{H}_{p\cdot g}
\end{equation}
$\forall p\in\mathcal{P}_{/\mathcal{S}}$, $g\in\mathcal{G}$.
\end{definition}
\begin{definition}
Let $\mathcal{G}\rightarrow\mathcal{P}_{/\mathcal{S}}\rightarrow\mathcal{M}_{/\mathcal{S}}$ be a $\mathcal{S}$-relative principal super fiber bundle and $\mathscr{H}:\,\mathcal{P}_{/\mathcal{S}}\rightarrow T(\mathcal{P}_{/\mathcal{S}})$ a principal connection on $\mathcal{P}_{/\mathcal{S}}$. A tangent vector $X_p\in T_p(\mathcal{P}_{/\mathcal{S}})$ at $p\in\mathcal{P}_{/\mathcal{S}}$ is called \emph{horizontal} or \emph{vertical} if $X_p\in\mathscr{H}_p$ or $X_p\in\mathscr{V}_p$, respectively. Analogously, one defines horizontal and vertical vector fields.
\end{definition} 
\begin{remark}\label{prop:3.18}
Since, $\mathscr{H}_p\oplus\mathscr{V}_p=T_p(\mathcal{P}_{/\mathcal{S}})$ $\forall p\in\mathcal{P}_{/\mathcal{S}}$, this induces projections $\mathrm{pr}_h$ and $\mathrm{pr}_v$ on $T(\mathcal{P}_{/\mathcal{S}})$ onto the horizontal and vertical tangent modules. As $\mathscr{H}$ and $\mathscr{V}$ define smooth horizontal distributions, it clear that, for any smooth vector field $X\in\Gamma(T\mathcal{P}_{/\mathcal{S}})$, the projections $\mathrm{pr}_h\circ X$ and $\mathrm{pr}_v\circ X$ define smooth horizontal and vertical vector fields, respectively.
\end{remark}
We finally come to an equivalent characterization of principal connections in terms of kernels of particular 1-forms defined on (relative) principal super fiber bundles. These so-called super connection 1-forms yield a generalization of the well-known gauge fields playing a prominent role in ordinary gauge theory in physics.
\begin{definition}\label{prop:3.19}
A \emph{super connection $1$-form} $\mathcal{A}$ on the $\mathcal{S}$-relative principal super fiber bundle $\mathcal{G}\rightarrow\mathcal{P}_{/\mathcal{S}}\stackrel{\pi_{\mathcal{S}}}{\rightarrow}\mathcal{M}_{/\mathcal{S}}$ is an even $\mathrm{Lie}(\mathcal{G})$-valued 1-form $\mathcal{A}\in\Omega^1(\mathcal{P}_{/\mathcal{S}},\mathfrak{g}):=\Omega^1(\mathcal{P}_{/\mathcal{S}})\otimes\mathfrak{g}$ such that
\begin{enumerate}[label=(\roman*)]
	\item $\braket{\widetilde{X}|\mathcal{A}}=X$ $\forall X\in\mathfrak{g}$
	\item $(\Phi_{\mathcal{S}})_g^*\mathcal{A}=\mathrm{Ad}_{g^{-1}}\circ\mathcal{A}$ $\forall g\in\mathcal{G}$.
\end{enumerate}
where, in condition (ii), the generalized tangent map was used (see Def. \ref{definition:3.9}).
\end{definition}

\begin{theorem}
For a $\mathcal{S}$-relative principal super fiber bundle $\mathcal{G}\rightarrow\mathcal{P}_{/\mathcal{S}}\rightarrow\mathcal{M}_{/\mathcal{S}}$, there is a one-to-one correspondence between principal connections and connection 1-forms on $\mathcal{P}_{/\mathcal{S}}$. More precisely,
\begin{enumerate}[label=(\roman*)]
	\item if $\mathscr{H}:\,\mathcal{P}_{/\mathcal{S}}\ni p\mapsto\mathscr{H}_p$ is a principal connection on $\mathcal{P}_{/\mathcal{S}}$, then $\mathcal{A}\in\Omega^1(\mathcal{P}_{/\mathcal{S}},\mathfrak{g})_0$ defined via 
\begin{equation}
\braket{(\widetilde{X}_p,Y_p)|\mathcal{A}_p}:=X
\label{eq:3.0.3}
\end{equation}
$\forall(\widetilde{X}_p,Y_p)\in\mathscr{H}_p\oplus\mathscr{V}_p=T_p\mathcal{P}$, $p\in\mathcal{P}_{/\mathcal{S}}$ and $X\in\mathrm{Lie}(\mathcal{G})$, defines a connection 1-form on $\mathcal{P}_{/\mathcal{S}}$.
	\item if $\mathcal{A}\in\Omega^1(\mathcal{P}_{/\mathcal{S}},\mathfrak{g})_0$ is a connection 1-form on $\mathcal{P}_{/\mathcal{S}}$, then the assignment
	\begin{equation}
	\mathscr{H}:\,\mathcal{P}_{/\mathcal{S}}\ni p\mapsto\mathrm{ker}(\mathcal{A}_p)\subset T_p(\mathcal{P}_{/\mathcal{S}})
	\label{eq:3.0.4}
	\end{equation}
	defines a principal connection on $\mathcal{P}_{/\mathcal{S}}$
\end{enumerate}
\end{theorem}
\begin{proof}
\begin{enumerate}[label=(\roman*)]
	\item We have to show that $\mathcal{A}$ as defined via \eqref{eq:3.0.4} satisfies the conditions (i) and (ii) in \ref{prop:3.19} of a connection 1-form on $\mathcal{P}_{/\mathcal{S}}$. First, to see that $\mathcal{A}$ indeed defines a smooth even $\mathrm{Lie}(\mathcal{G})$-valued 1-form on $\mathcal{P}_{/\mathcal{S}}$, i.e., $\mathcal{A}\in\Omega^1(\mathcal{P}_{/\mathcal{S}},\mathfrak{g})_0$, let $(X_i)_i$ be a homogeneous basis of $\mathfrak{g}\subset\mathrm{Lie}(\mathcal{G})=T_e\mathcal{G}$. Consider the components $\tensor[^i]{\mathcal{A}}{}:=\mathcal{A}\diamond\tensor[^i]{X}{}$ $\forall i$ with $(\tensor[^i]{X}{})_i$ the corresponding left-dual basis of $\tensor[^*]{\mathrm{Lie}(\mathcal{G})}{}$. According to Remark \ref{prop:3.18}, if $X\in\mathfrak{X}(\mathcal{P}_{/\mathcal{S}})$ is a smooth vector field, we can decompose $X$ into vertical and horizontal parts via $X=\mathrm{pr}_v\circ X+\mathrm{pr}_h\circ X=:X_v+X_h$. Since $X_v$ and $X_h$ are smooth, it follows that $\tensor[^i]{\mathcal{A}}{}$ is smooth iff $\braket{X_v|\tensor[^i]{\mathcal{A}}{}}$ and $\braket{X_h|\tensor[^i]{\mathcal{A}}{}}$ are smooth for any $X$ and thus iff $\tensor[^i]{\mathcal{A}}{}$ is smooth when restricted to smooth vertical and horizontal vector fields. The fundamental vector fields $\widetilde{X}_i$ generated by $X_i$ define global smooth vertical vector fields such that $(\widetilde{X}_i(p))_i$ is a homogeneous basis of $\mathscr{V}_p$ $\forall p\in\mathcal{P}_{/\mathcal{S}}$. By definition of $\mathcal{A}$, we have
\begin{equation}
\braket{\widetilde{X}_i|\tensor[^j]{\mathcal{A}}{}}=\delta_i^j\quad\forall i,j
\end{equation}
which is smooth and thus $\mathcal{A}$ is smooth on vertical vector fields. Finally, since $\mathscr{H}$ is a smooth geometric distribution, for any $p\in\mathcal{P}_{/\mathcal{S}}$, there exists a $p\in U_p\subseteq\mathcal{S}\times\mathcal{P}$ open as well as a family $(Y_i)_i$ of smooth horizontal vector fields on $U_p$ such that $(Y_i(q))_i$ is a homogeneous basis of $\mathscr{H}_q$ $\forall q\in U_p$. Since $\braket{Y_i|\mathcal{A}}=0$ and thus is smooth, it follows that $\mathcal{A}$ is also smooth on horizontal vector fields. Hence, indeed $\mathcal{A}\in\Omega^1(\mathcal{P}_{/\mathcal{S}},\mathfrak{g})$. That $\mathcal{A}$ has to be even is immediate.\\
It remains to show that $\mathcal{A}$ in fact satisfies the conditions (i) and (ii). By definition, (i) is immediate. Moreover, by right-invariance of $\mathscr{H}$, it suffices to show (ii) for vertical tangent vectors. Hence, let $\widetilde{X}_p\in\mathscr{V}_p\subset T_p(\mathcal{P}_{\mathcal{S}})$ with $X\in\mathrm{Lie}(\mathcal{G})$. Using Lemma \ref{prop:3.14}, we compute 
\begin{equation}
\braket{(\Phi_{\mathcal{S}})_{g*}\widetilde{X}_p|\mathcal{A}_{pg}}=\braket{\widetilde{\mathrm{Ad}_{g^{-1}}X}_{p\cdot g}|\mathcal{A}_{pg}}=\mathrm{Ad}_{g^{-1}}X=\mathrm{Ad}_{g^{-1}}\!\braket{\widetilde{X}_p|\mathcal{A}_p}
\end{equation}
This shows that $\mathcal{A}$ defines a principal connection 1-form on $\mathcal{P}_{/\mathcal{S}}$.
	\item Conversely, for $\mathcal{A}\in\Omega^1(\mathcal{P}_{/\mathcal{S}},\mathfrak{g})_0$ a principal connection 1-form on $\mathcal{P}_{/\mathcal{S}}$, we have to show that $\mathscr{H}:\,\mathcal{P}_{/\mathcal{S}}\ni p\mapsto\mathrm{ker}(\mathcal{A}_p)\subset T_p(\mathcal{P}_{/\mathcal{S}})$ defines a principal connection on $\mathcal{P}_{/\mathcal{S}}$. Therefore, similar as in \cite{Tuynman:2004}, consider the map
\begin{equation}
T(\mathcal{P}_{/\mathcal{S}})\rightarrow(\mathcal{P}\times\mathrm{Lie}(\mathcal{G}))_{/\mathcal{S}},\,T_p\mathcal{P}\ni X_p\mapsto(p,\braket{X_p|\mathcal{A}_p})
\label{eq:3.0.5}
\end{equation}
from $T(\mathcal{P}_{/\mathcal{S}})$ to the trivial $\mathcal{S}$-relative super vector bundle $\mathrm{Lie}(\mathcal{G})\rightarrow(\mathcal{P}\times\mathrm{Lie}(\mathcal{G}))_{/\mathcal{S}}\rightarrow\mathcal{P}_{/\mathcal{S}}$. By Remark \ref{prop:3.6}, we can identify $T(\mathcal{P}_{/\mathcal{S}})$ with the $\mathcal{S}$-relative super vector bundle $(T\mathcal{P})_{/\mathcal{S}}$. Hence, it follows that \eqref{eq:3.0.5} defines a smooth even and surjective, as $\mathcal{A}$ is even and surjective, morphism of $\mathcal{S}$-relative super vector bundles. Hence, by Lemma \ref{prop:3.8}, the kernel of \eqref{eq:3.0.5}, which coincides with $\mathscr{H}$, defines a smooth geometric distribution. To see that it is right-invariant, note that, by condition (ii), for $p\in\mathcal{P}_{/\mathcal{S}}$ and $X_p\in\mathscr{H}_p=\mathrm{ker}(\mathcal{A}_p)$, we have
\begin{equation}
\braket{(\Phi_{\mathcal{S}})_{g*}X_p|\mathcal{A}_{p\cdot g}}=\mathrm{Ad}_{g^{-1}}\!\braket{X_p|\mathcal{A}_p}=0
\end{equation} 
and thus $(\Phi_{\mathcal{S}})_{g*}X_p\in\mathscr{H}_{p\cdot g}$. 
\end{enumerate}
\end{proof}
We finally want to define the notion of a covariant derivative and curvature 2-forms corresponding to super connection forms defined on relative principal super fiber bundles. 
\begin{definition}
Let $\mathcal{G}\rightarrow\mathcal{P}_{/\mathcal{S}}\rightarrow\mathcal{M}_{/\mathcal{S}}$ be a $\mathcal{S}$-relative principal super fiber bundle and $\mathcal{A}\in\Omega^1(\mathcal{P}_{/\mathcal{S}},\mathfrak{g})_0$ a super connection 1-form on $\mathcal{P}_{/\mathcal{S}}$. The linear map $D^{(\mathcal{A})}:\,\Omega^k(\mathcal{P}_{/\mathcal{S}},\mathcal{V})\rightarrow\Omega^{k+1}(\mathcal{P}_{/\mathcal{S}},\mathcal{V})$ defined as
\begin{equation}
\braket{X_0,\ldots,X_k|D^{(\mathcal{A})}\omega}:=\braket{\mathrm{pr}_{h}\circ X_0,\ldots,\mathrm{pr}_{h}\circ X_k|\mathrm{d}\omega}
\end{equation}
for smooth vector fields $X_i\in\mathfrak{X}(\mathcal{P}_{/\mathcal{S}})$, $i=0,\ldots,k$, is called the \emph{covariant derivative induced by $\mathcal{A}$} where $\mathrm{pr}_h:\,T\mathcal{P}_{/\mathcal{S}}\rightarrow\mathscr{H}:=\mathrm{ker}\,\mathcal{A}$ denotes the projection onto the horizontal tangent modules induced by $\mathcal{A}$.
\end{definition}
An important subclass of vector-valued forms on (relative) principal super fiber bundles is provided by forms that transform covariantly in a specific sense under gauge transformaions. This is the content of the following definition. One then asks the question, whether one can define a derivative on such forms so that the transformation property is preserved. As we will see, it follows that the covariant derivative induced by super connections forms indeed has the right properties.  
\begin{definition}
Let $\mathcal{G}\rightarrow\mathcal{P}_{/\mathcal{S}}\rightarrow\mathcal{M}_{/\mathcal{S}}$ be a $\mathcal{S}$-relative principal super fiber bundle with $\mathcal{G}$-right action $\Phi:\,\mathcal{P}_{/\mathcal{S}}\times\mathcal{G}\rightarrow\mathcal{P}_{/\mathcal{S}}$ and $\rho:\,\mathcal{G}\rightarrow\mathrm{GL}(\mathcal{V})$ be a representation of $\mathcal{G}$ on a super $\Lambda$-vector space $\mathcal{V}$. A $k$-form $\omega$ on $\mathcal{P}_{/\mathcal{S}}$ with values in $\mathcal{V}$ is called \emph{horizontal of type $(\mathcal{G},\rho)$}, symbolically, $\omega\in\Omega^k(\mathcal{P}_{/\mathcal{S}},\mathcal{V})^{(\mathcal{G},\rho)}$ if $\omega$ vanishes on vertical tangent vectors and
\begin{equation}
\Phi_{g}^*\omega=\rho(g)^{-1}\circ\omega,\,\forall g\in\mathcal{G}
\end{equation}
\end{definition}
\begin{prop}\label{prop:3.23}
Let $\mathcal{G}\rightarrow\mathcal{P}_{/\mathcal{S}}\rightarrow\mathcal{M}_{/\mathcal{S}}$ be a $\mathcal{S}$-relative principal super fiber bundle and $\mathcal{A}\in\Omega^1(\mathcal{P}_{/\mathcal{S}},\mathfrak{g})_0$ a super connection 1-form on $\mathcal{P}_{/\mathcal{S}}$. Let $\omega\in\Omega^k_{hor}(\mathcal{P}_{/\mathcal{S}},\mathcal{V})^{(\mathcal{G},\rho)}$ be a horizontal $k$-form on $\mathcal{P}_{/\mathcal{S}}$ of type $(\mathcal{G},\rho)$. Then, the induced covariant derivative $D^{(\mathcal{A})}$ takes the form
\begin{equation}
D^{(\mathcal{A})}\omega=\mathrm{d}\omega+\rho_{*}(\mathcal{A})\wedge\omega
\label{eq:4.17.1}
\end{equation}
where the $(k+1)$-form $\rho_{*}(\mathcal{A})\wedge\omega$ on $\mathcal{P}_{/\mathcal{S}}$ is defined as
\begin{equation}
\braket{X_0,\ldots,X_k|\rho_{*}(\mathcal{A})\wedge\omega}:=\sum_{i=0}^k{(-1)^{k-i+\sum_{l=0}^{i-1}|X_i||X_l|}\rho_{*}(\iota_{X_i}\mathcal{A})}(\braket{X_0,\ldots,\widehat{X}_i,\ldots,X_k|\omega})
\label{eq:4.17.2}
\end{equation}
In particular, $D^{(\mathcal{A})}\omega\in\Omega^{k+1}_{hor}(\mathcal{P}_{/\mathcal{S}},\mathcal{V})^{(\mathcal{G},\rho)}$, i.e., $D^{(\mathcal{A})}$ induces a covariant derivative on the subspace of horizontal forms of type $(\mathcal{G},\rho)$. 
\end{prop}
\begin{proof}
That $D^{(\mathcal{A})}\omega$ defines a horizontal form of type $(\mathcal{G},\rho)$ if $\omega$ does, follows immediately from the right-invariance of the horizontal distribution which implies $\mathrm{pr}_h\circ(\Phi_{\mathcal{S}})_g=(\Phi_{\mathcal{S}})_g\circ\mathrm{pr}_h$ $\forall g\in\mathcal{G}$ where $\Phi_{\mathcal{S}}:\,\mathcal{P}_{/\mathcal{S}}\times\mathcal{G}\rightarrow\mathcal{P}_{/\mathcal{S}}$ denotes the $\mathcal{G}$-right action on $\mathcal{P}_{/\mathcal{S}}$. \\
To prove \eqref{eq:4.17.1}, it is sufficient to show equality when evaluating \eqref{eq:4.17.1} on smooth horizontal and vertical vector fields $X_i$, $i=0,\ldots,k$. Moreover, if $X$ is vertical, it suffices to assume that is a fundamental vector field $X=\widetilde{Y}$ generated by some $Y\in\mathfrak{g}$. In case all $X_i$ are vertical, then it is clear that both sides in \eqref{eq:4.17.1} trivially vanish by horizontality. If all $X_i$ are horizontal, then \eqref{eq:4.17.2} simply vanishes as $\mathcal{A}$ vanishes on horizontal vector fields proving equality also in this case. Next, assume that at least two vector fields are vertical. Therefore, note that if $\widetilde{X}$ for some $X\in\mathfrak{g}$ is a fundamental vector field and $Y^{*}$ is horizontal then $[\widetilde{X},Y^{*}]$ is also horizontal. In fact, according to Prop. \ref{prop:5.10}, we have $L_{\widetilde{X}}\mathcal{A}=-\mathrm{ad}_X\circ\mathcal{A}$ such that, following \cite{Tuynman:2004}, this yields
\begin{align}
\braket{[\widetilde{X},Y^{*}]|\mathcal{A}}&=\widetilde{X}\braket{Y^{*}|\mathcal{A}}-(-1)^{|X||Y|}\braket{Y^{*}|L_{\widetilde{X}}\mathcal{A}}\nonumber\\
&=\mathrm{ad}_X\circ\braket{Y^{*}|\mathcal{A}}=0
\end{align}
Hence, again, both sides in \eqref{eq:4.17.1} vanish as $\omega$ is horizontal and $\mathcal{A}$ vanishes on horizontal vector fields. It thus remains to consider the case where at least one vector field is vertical. Thus, suppose $\widetilde{X}$ is a fundamental vector field generated by $X\in\mathfrak{g}$ and $Y_i$ are horizontal $\forall i=1,\dots,k$. By definition, the left hand side of \eqref{eq:4.17.1} simply vanishes. On the other hand, following precisely the same steps as in the proof of Prop. \ref{prop:5.10}, it follows that $L_{\widetilde{X}}\omega=-\mathrm{\rho}_{*}(X)\circ\omega$. This yields
\begin{align}
(-1)^k\braket{\widetilde{X},Y_1\ldots,Y_k|d\omega}&=(-1)^{\sum_{i=1}^k{|X||Y_i|}}\braket{Y_1,\ldots,Y_k|L_{\widetilde{X}}\omega}\nonumber\\
&=-\rho_{*}(X)(\braket{Y_1,\ldots,Y_k|\omega})
\end{align}
where identity \eqref{eq:identity} was used. Moreover, using Definition \eqref{eq:4.17.2}, we get
\begin{align}
\braket{\widetilde{X},Y_1\ldots,Y_k|\rho_{*}(\mathcal{A})\wedge\omega}=(-1)^{k}\rho_{*}(\iota_{\widetilde{X}}\mathcal{A})(\braket{Y_1,\ldots,Y_k|\omega})=\rho_{*}(X)(\braket{Y_1,\ldots,Y_k|\omega})
\end{align}
proving that the right hand side of \eqref{eq:4.17.1} vanishes as well.
\end{proof}

\begin{definition}
Let $\alpha\in\Omega^{k}(\mathcal{M}_{/\mathcal{S}},\mathfrak{g})$ and $\beta\in\Omega^{l}(\mathcal{M}_{/\mathcal{S}},\mathfrak{g})$ be a $k$- resp. $l$-form on a $\mathcal{S}$-relative supermanifold $\mathcal{M}_{/\mathcal{S}}$ with values in a super Lie module $\mathrm{Lie}(\mathcal{G})$ corresponding to a super Lie group $\mathcal{G}$. The $k+l$-form $[\alpha\wedge\beta]\in\Omega^{k+l}(\mathcal{M}_{/\mathcal{S}},\mathfrak{g})$ is then defined via
\begin{equation}
[\alpha\wedge\beta]:=\alpha^i\wedge\mathfrak{C}^{|e_i|}(\beta^j)\otimes[e_i,e_j]
\end{equation}
where we have expanded $\alpha=\alpha^i\otimes e_i$ and $\beta=\beta^i\otimes e_i$ w.r.t. a homogeneous basis $(e_i)_i$ of $\mathfrak{g}$. Here, the involution $\mathfrak{C}:\,\mathcal{V}\rightarrow\mathcal{V}$ is defined as $\mathfrak{C}(v)=(-1)^{|v|}v$ for any homogeneous $v\in\mathcal{V}$.
\end{definition}

\begin{remark}\label{prop:3.0.25}
For $\omega\in\Omega^k_{hor}(\mathcal{P}_{/\mathcal{S}},\mathfrak{g})^{(\mathcal{G},\mathrm{Ad})}$, it follows that
\begin{equation}
D^{(\mathcal{A})}\omega=\mathrm{d}\omega+[\mathcal{A}\wedge\omega]
\end{equation}
To see this, let us consider the case $k=1$. By direct computation, it then follows
\begin{align}
\braket{X,Y|[\mathcal{A}\wedge\omega]}&=\iota_X\circ\iota_Y\left(\mathcal{A}^i\wedge\mathfrak{C}^{|e_i|}(\omega^j)\right)\otimes[e_i,e_j]\nonumber\\
&=\left((-1)^{(|e_i|+|Y|)|X|}\braket{Y|\mathcal{A}^i}\braket{X|\mathfrak{C}^{|e_i|}(\omega^j)}-(-1)^{|e_i||Y|}\braket{X|\mathcal{A}^i}\braket{Y|\mathfrak{C}^{|e_i|}(\omega^j)}\right)\otimes[e_i,e_j]\nonumber\\
&=-[\braket{X|\mathcal{A}},\braket{Y|\omega}]+(-1)^{|X||Y|}[\braket{Y|\mathcal{A}},\braket{X|\omega}]=\braket{X,Y|\mathrm{ad}(\mathcal{A})\wedge\omega}
\end{align}
Hence, indeed, $[\mathcal{A}\wedge\omega]=\mathrm{ad}(\mathcal{A})\wedge\omega$.
\end{remark}

\begin{definition}
Let $\mathcal{A}\in\Omega^1(\mathcal{P}_{/\mathcal{S}},\mathfrak{g})_0$ be a super connection 1-form on a $\mathcal{S}$-relative principal super fiber bundle $\mathcal{G}\rightarrow\mathcal{P}_{/\mathcal{S}}\rightarrow\mathcal{M}_{/\mathcal{S}}$. The horizontal $2$-form 
\begin{equation}
F(\mathcal{A}):=D^{(\mathcal{A})}\mathcal{A}\in\Omega^2_{hor}(\mathcal{P}_{/\mathcal{S}},\mathfrak{g})^{(\mathcal{G},\mathrm{Ad})}
\end{equation}
is called the \emph{curvature} of $\mathcal{A}$.
\end{definition}

\begin{prop}
Let $\mathcal{A}\in\Omega^1(\mathcal{P}_{/\mathcal{S}},\mathfrak{g})_0$ be a super connection 1-form on a $\mathcal{S}$-relative principal super fiber bundle $\mathcal{G}\rightarrow\mathcal{P}_{/\mathcal{S}}\rightarrow\mathcal{M}_{/\mathcal{S}}$. Then,
\begin{enumerate}[label=(\roman*)]
	\item $F(\mathcal{A})=\mathrm{d}\mathcal{A}+[\mathcal{A}\wedge\mathcal{A}]$
	\item $D^{(\mathcal{A})}F(\mathcal{A})=0$ (Bianchi identity)
\end{enumerate}
\end{prop}
\begin{proof}
The first identity can be shown similarly as in the proof of Prop. \ref{prop:3.23} by applying both sides separately on horizontal and vertical vector fields.\\
To show the Bianchi identity, let us decompose $\mathcal{A}=\mathcal{A}^i\otimes e_i$ with $(e_i)_i$ a real homogeneous basis of $\mathfrak{g}$. Then, 
\begin{align}
\mathrm{d}\mathrm{F}(\mathcal{A})&=\frac{1}{2}\mathrm{d}[\mathcal{A}\wedge\mathcal{A}]=(-1)^{|e_i||e_j|}\frac{1}{2}\left(\mathrm{d}\mathcal{A}^i\wedge\mathcal{A}^j-\mathcal{A}^i\wedge \mathrm{d}\mathcal{A}^j\right)\otimes[e_i,e_j]\nonumber\\
&=\frac{1}{2}\left(\mathcal{A}^j\wedge\mathrm{d}\mathcal{A}^i-(-1)^{|e_i||e_j|}\mathcal{A}^i\wedge\mathrm{d}\mathcal{A}^j\right)\otimes[e_i,e_j]=-[\mathcal{A}\wedge \mathrm{d}\mathcal{A}]
\label{eq:3.10.1}
\end{align}
Thus, it follows that $\mathrm{d}\mathrm{F}(\mathcal{A})$ vanishes when evaluating it on horizontal vector fields. But, since $D^{(\mathcal{A})}F(\mathcal{A})=\mathrm{pr}_h\diamond \mathrm{d}\mathrm{F}(\mathcal{A})$ (cf. Definition \ref{definitionNEU:2.11}), the claim follows.
\end{proof}

\section{Super Cartan geometry}\label{section:Cartan}
Super Cartan geometries play a very important role in the geometric approach to supergravity. It is based on the observation that gravity has the equivalent description in terms of ordinary Cartan geometry. In \cite{Eder:2020erq}, we have given a definition of \emph{super Cartan geometries} in the context of enriched categories. In this section, we extend the definition to metric reductive super Cartan geometries and also give a detailed proof for its relation to Ehresmann connections (for its application to supergravity see \cite{Eder:2020erq} or section \ref{section:Application} below). Before we state the main definition, we first need to recall the notion of a \emph{super Klein geometry}.  
\begin{definition}
A \emph{super Klein geometry} is a pair $(\mathcal{G},\mathcal{H})$ consisting of a super Lie Group $\mathcal{G}$ and an embedded super Lie subgroup $\mathcal{H}\hookrightarrow\mathcal{G}$.
\end{definition}
Super Cartan geometries can be regarded as deformed Klein geometries. The following definition provides a generalization of Cartan geometries to the super category (for a physical account on ordinary Cartan geometries with application to gravity see e.g. \cite{Wise:2006sm}). A first approach in this direction in context of relative categories has been studied in \cite{Hack:2015vna} introducing the notion of \emph{Cartan structures}. A generalization to super Cartan geometries which are defined on arbitrary (not necessarily trivial) relative principal super fiber bundles and which also encode non-trivial fermionic degrees of freedom on the body of a supermanifold has been given in \cite{Eder:2020erq}.       
\begin{definition}[\cite{Eder:2020erq}]\label{Def:4.2}
A \emph{super Cartan geometry} $(\pi_{\mathcal{S}}:\,\mathcal{P}_{/\mathcal{S}}\rightarrow\mathcal{M}_{/\mathcal{S}},\mathcal{A})$ modeled on a super Klein geometry $(\mathcal{G},\mathcal{H})$ is a $\mathcal{S}$-relative principal super fiber bundle $\mathcal{H}\rightarrow\mathcal{P}_{/\mathcal{S}}\rightarrow\mathcal{M}_{/\mathcal{S}}$ with structure group $\mathcal{H}$ together with a smooth even $\mathrm{Lie}(\mathcal{G})$-valued 1-form $\mathcal{A}\in\Omega^1(\mathcal{P}_{/\mathcal{S}},\mathfrak{g})_0$ on $\mathcal{P}_{/\mathcal{S}}$ called \emph{super Cartan connection} such that
\begin{enumerate}[label=(\roman*)]
	\item $\braket{\widetilde{X}|\mathcal{A}}=X$,\quad $\forall X\in\mathfrak{h}$
	\item $(\Phi_{\mathcal{S}})^*_{h}\mathcal{A}=\mathrm{Ad}_{h^{-1}}\circ\mathcal{A}$,\quad $\forall h\in\mathcal{H}$
	\item for any $s\in\mathbf{B}(\mathcal{S})$, the pullback of $\mathcal{A}$ w.r.t. the induced embedding $\iota_{\mathcal{P}}:\,\mathcal{P}\hookrightarrow\mathcal{S}\times\mathcal{P},\,p\mapsto(s,p)$ yields an isomorphism $\iota_{\mathcal{P}}^*\mathcal{A}_p:\,T_p\mathcal{P}\rightarrow\mathrm{Lie}(\mathcal{G})$ of free super $\Lambda$-modules for any $p\in\mathcal{P}$
\end{enumerate}
where the last condition will be called the \emph{super Cartan condition}. If (iii) is not satisfied, $\mathcal{A}$ will be called a \emph{generalized super Cartan connection}.
\end{definition}
\begin{remark}
Note that, by definition, it follows that condition (iii) for a super Cartan connection $\mathcal{A}$ is preserved under change of parametrization. In fact, let $\lambda:\,\mathcal{S'}\rightarrow\mathcal{S}$ be a smooth map. Then, since $\lambda|_{\mathbf{B}(\mathcal{S}')}\subseteq\mathbf{B}(\mathcal{S})$, it follows that the pullback $\lambda^*\mathcal{A}$ also satisfies (iii).
\end{remark}

\begin{definition}[after \cite{Tuynman:2004,Tuynman:2018}]\label{bilinear map}
A (homogeneous) \emph{super bilinear form} $\mathscr{S}$ of parity $|\mathscr{S}|\in\mathbb{Z}_2$ on a free super $\Lambda$-module $\mathcal{V}$ is a right bilinear map $\mathscr{S}:\,\mathcal{V}\times\mathcal{V}\rightarrow\Lambda^{\mathbb{C}}$, $\Lambda^{\mathbb{C}}:=\Lambda\otimes\mathbb{C}$, which satisfies $\mathscr{S}(\mathcal{V}_{i},\mathcal{V}_{j})\subseteq(\Lambda^{\mathbb{C}})_{|\mathscr{S}|+i+j}$ and is graded symmetric, i.e.,
\begin{equation}
\mathscr{S}(v,w)=(-1)^{|v||w|}\mathscr{S}(w,v),\quad\forall\text{ homogeneous }v,w\in\mathcal{V}
\label{eq:0.2.3.1}
\end{equation}
Let $V:=\mathcal{V}/\mathcal{N}$ with $\mathcal{N}:=\{v\in\mathcal{V}|\,\exists a\in\Lambda:\,a\neq 0\text{ and }ax=0\}$ the subset of nilpotent vectors. A super bilinear form $\mathscr{S}$ is called \emph{smooth}, if $\mathscr{S}(V,V)\subseteq\mathbb{C}$. An even super bilinear form $\mathscr{S}:\,\mathcal{V}\times\mathcal{V}\rightarrow\Lambda^{\mathbb{C}}$ is called a \emph{super metric} if it is non-degenerate, that is, for any $v\in V$ with $v\neq 0$, there exists $w\in V$ such that $\epsilon(\mathscr{S}(v,w))\neq 0$ where $\epsilon:\,\Lambda^{\mathbb{C}}\rightarrow\mathbb{C}$ is the body map.

\end{definition}

\begin{remark}
If a super bilinear form $\mathscr{S}$ is smooth, it follows immediately that $\mathscr{S}|_{V\times V}$ defines a graded symmetric bilinear form on the super vector space $V:=\mathcal{V}/\mathcal{N}$ in the sense of \cite{Cheng:2012}, i.e., $\mathscr{S}(V_{i},V_{j})=0$ unless $i+j=|\mathscr{S}|$. Moreover, $\mathscr{S}|_{V_0\times V_0}$ is symmetric and $\mathscr{S}|_{V_1\times V_1}$ is antisymmetric. If $\mathscr{S}$ is furthermore even and non-degenerate then so is $\mathscr{S}|_{V\times V}$ which implies that $V_1$ is necessarily even dimensional.
\end{remark}

\begin{definition}
A super Cartan geometry $(\pi_{\mathcal{S}}:\,\mathcal{P}_{/\mathcal{S}}\rightarrow\mathcal{M}_{/\mathcal{S}},\mathcal{A})$ modeled on a super Klein geometry $(\mathcal{G},\mathcal{H})$ is called 
\begin{enumerate}[label=(\roman*)]
	\item \emph{reductive} if the super Lie algebra $\mathfrak{g}$ of $\mathcal{G}$ admits a decomposition of the form $\mathfrak{g}=\mathfrak{g}/\mathfrak{h}\oplus\mathfrak{h}$ with $\mathfrak{h}$ the super Lie algebra of $\mathcal{H}$ and $\mathfrak{g}/\mathfrak{h}$ a super vector space such that the corresponding super $\Lambda$-vector space $\Lambda\otimes\mathfrak{g}/\mathfrak{h}$ is invariant w.r.t. the Adjoint action of $\mathcal{H}$ on $\mathrm{Lie}(\mathcal{G})$.
	\item \emph{metric} if it is reductive and if the super $\Lambda$-vector space $\Lambda\otimes\mathfrak{g}/\mathfrak{h}$ admits a smooth super metric that is invariant w.r.t. the Adjoint action of $\mathcal{H}$. 
\end{enumerate}

\begin{definition}
For $\mathcal{M}$ a supermanifold, consider the right-dual tensor product super vector bundle $(T\mathcal{M}\otimes T\mathcal{M})^*$. A smooth section $g\in\Gamma^{\mathbb{C}}((T\mathcal{M}\otimes T\mathcal{M})^*)\cong\mathrm{Hom}_R(\Gamma(T\mathcal{M})^2,H^{\infty}(\mathcal{M})\otimes\mathbb{C})$ is called a \emph{super metric} on $\mathcal{M}$, if $g_x$ for any $x\in\mathcal{M}$ defines a super metric on the tangent module $T_x\mathcal{M}$. Thus, for any homogeneous smooth vector fields $X,Y\in\Gamma(T\mathcal{M})$, $g(X,Y)=(-1)^{|X||Y|}g(Y,X)$ and the map $\Gamma(T\mathcal{M})\ni X\mapsto g(X,\cdot)\in\Gamma(T\mathcal{M}^*)$ is an isomorphism. Thus, $g$ defines a super metric in the sense of \cite{Goertsches:2008}.
\end{definition}

\begin{remark}
Let $g$ be a super metric on a supermanifold $\mathcal{M}$. Similar as in the proof of Prop. \ref{prop:superLie}, it follows that, if $p\in\mathbf{B}(\mathcal{M})$ is a body point, the tangent module $\mathcal{V}_p:=T_p\mathcal{M}$ has the structure of a super $\Lambda$-vector space $\mathcal{V}_p=\Lambda\otimes V_p$ with the super vector space $V_p$ consisting of derivations $X_p\in T_p\mathcal{M}$ satisfying 
\begin{equation}
X_p(f)\in\mathbb{R},\,\forall f\in H^{\infty}(\mathcal{M})
\end{equation}
Hence, it follows that $V_p$ can be identified with the tangent space $T_p\mathbf{A}(\mathcal{M})$ of the corresponding algebro-geometric supermanifold $\mathbf{A}(\mathcal{M})$ (see section \ref{section:graded}). By definition, any $X_p\in V_p$ arises from the restriction of a local smooth vector field on $\mathcal{M}$ to $p$. Consequently, as $g$ is smooth and $p$ is a body point, it follows that $g_p(V_p,V_p)\subseteq\mathbb{C}$, that is, $g_p$ is smooth according to Definition \ref{bilinear map}.
\end{remark}

\begin{prop}\label{prop:4.7.0}
Let $(\pi_{\mathcal{S}}:\,\mathcal{P}_{/\mathcal{S}}\rightarrow\mathcal{M}_{/\mathcal{S}},\mathcal{A})$ be a super Cartan geometry modeled on a super Klein geometry $(\mathcal{G},\mathcal{H})$. Let $\mathrm{pr}_{\mathfrak{h}}:\,\mathrm{Lie}(\mathcal{G})\rightarrow\mathrm{Lie}(\mathcal{H})$ denote the projection of $\mathrm{Lie}(\mathcal{G})$ onto the super Lie sub module $\mathrm{Lie}(\mathcal{H})$. Then, $\mathrm{pr}_{\mathfrak{h}}\circ\mathcal{A}\in\Omega^1(\mathcal{P}_{/\mathcal{S}},\mathfrak{h})_0$ defines a super connection 1-form on $\mathcal{P}_{/\mathcal{S}}$.\\
Let the super Cartan geometry, in addition, be reductive and $\mathrm{pr}_{\mathfrak{g}/\mathfrak{h}}$ denote the projection of $\mathrm{Lie}(\mathcal{G})$ onto the super $\Lambda$-vector space $\Lambda\otimes\mathfrak{g}/\mathfrak{h}$. Then, $E:=\mathrm{pr}_{\mathfrak{g}/\mathfrak{h}}\circ\mathcal{A}$, called \emph{super soldering form} or \emph{supervielbein}, defines an even horizontal $\Lambda\otimes\mathfrak{g}/\mathfrak{h}$-valued 1-form on $\mathcal{P}_{/\mathcal{S}}$ of type $(\mathcal{H},\mathrm{Ad})$.   
\end{prop}
\begin{proof}
That $\omega:=\mathrm{pr}_{\mathfrak{h}}\circ\mathcal{A}$ defines a super connection 1-form in the sense of Ehresmann is immediate by condition (i) and (ii) for a super Cartan connection. Furthermore, if the Cartan geometry is reductive, $\Lambda\otimes\mathfrak{g}/\mathfrak{h}$ defines a $\mathrm{Ad}(\mathcal{H})$ super $\Lambda$-module. Hence, by condition (ii) for a super Cartan connection, the super soldering form $E:=\mathrm{pr}_{\mathfrak{g}/\mathfrak{h}}\circ\mathcal{A}$ yields a well-defined $\mathcal{H}$-equivariant 1-form on $\mathcal{P}_{/\mathcal{S}}$. To see that is horizontal, let $\widetilde{X}$ be a fundamental vector field on $\mathcal{P}_{/\mathcal{S}}$ generated by $X\in\mathfrak{h}$. Then, by condition (i), it follows
\begin{equation}
X=\braket{\widetilde{X}|\mathcal{A}}=\braket{\widetilde{X}|\omega}+\braket{\widetilde{X}|E}=X+\braket{\widetilde{X}|E}
\end{equation}
Hence, $\braket{\widetilde{X}|E}=0$ proving that $E$ is horizontal.
\end{proof}
\begin{prop}
Let $(\pi_{\mathcal{S}}:\,\mathcal{P}_{/\mathcal{S}}\rightarrow\mathcal{M}_{/\mathcal{S}},\mathcal{A})$ be a reductive super Cartan geometry modeled on a super Klein geometry $(\mathcal{G},\mathcal{H})$ with a super Cartan connection $\mathcal{A}$. Let $E:=\mathrm{pr}_{\mathfrak{g}/\mathfrak{h}}\circ\mathcal{A}$ be the super soldering form as defined in Prop. \ref{prop:4.7.0}. Let furthermore $\iota_{\mathcal{P}}:\,\mathcal{P}\hookrightarrow\mathcal{S}\times\mathcal{P}$ be an embedding. Then, the pullback $\theta:=\iota_{\mathcal{P}}^*E\in\Omega^1_{hor}(\mathcal{P},\mathfrak{g}/\mathfrak{h})^{(\mathcal{H},\mathrm{Ad})}$ defines a non-degenerate 1-form and induces a smooth map 
\begin{equation}
\iota:\,\mathcal{P}\rightarrow\mathscr{F}(\mathcal{M}),\,p\mapsto D_p\pi\circ\theta_p^{-1}
\label{eq:4.0.1}
\end{equation}
which is fiber preserving and $\mathcal{H}$-equivariant in the sense that $\iota\circ\Phi=\Psi\circ(\iota\times\mathrm{Ad})$ with $\mathrm{Ad}:\,\mathcal{H}\rightarrow\mathrm{GL}(\Lambda\otimes\mathfrak{g}/\mathfrak{h})$ the Adjoint action and $\Phi$ and $\Psi$ the group right actions on $\mathcal{P}$ and $\mathscr{F}(\mathcal{M})$, respectively. In particular, $\mathcal{P}$ defines a $\mathcal{H}$-reduction of the frame bundle $\mathscr{F}(\mathcal{M})$. 
\end{prop}
\begin{proof}
By the previous proposition, it is clear that $\theta$ defines a horizontal 1-form on $\mathcal{P}$ of type $(\mathcal{H},\mathrm{Ad})$. Let $\mathscr{H}:=\mathrm{ker}(\iota_{\mathcal{P}}^*\omega)$ be the horizontal distribution induced by the pullback of the super connection 1-form $\omega:=\mathrm{pr}_{\mathfrak{h}}\circ\mathcal{A}$. Then, by the Cartan condition, it follows that $\theta_p:\,\mathscr{H}_p\rightarrow\Lambda\otimes\mathfrak{g}/\mathfrak{h}$, for any $p\in\mathcal{P}$, yields an isomorphism of free super $\Lambda$-modules. Moreover, the pushforward of the bundle projection induces an isomorphism $D_p\pi:\,\mathscr{H}_p\stackrel{\sim}{\rightarrow}T_p\mathcal{M}$. Hence, this in turn induces an isomorphism
\begin{equation}
D_p\pi\circ\theta_p^{-1}:\,\Lambda\otimes\mathfrak{g}/\mathfrak{h}\rightarrow T_p\mathcal{M}
\end{equation}
that is, a linear frame at $p$. It thus remains to show that \eqref{eq:4.0.1} indeed defines a $\mathcal{H}$-reduction of $\mathscr{F}(\mathcal{M})$. Therefore, note that, by condition (ii) for a super Cartan connection, we have $\theta(\Phi_{h*}(Y_p))=\Phi_h^*\theta(Y_p)=\mathrm{Ad}_{h^{-1}}(\theta(Y_p))$ $\forall Y_p\in T_p\mathcal{P}$ and $h\in\mathcal{H}$. Hence, this yields
\begin{equation}
\iota(p\cdot h)=D_{ph}\pi\circ\theta_{ph}^{-1}=D_{ph}\pi\circ D_p\Phi_{h}\circ\theta_p^{-1}\circ\mathrm{Ad}_h=\iota(p)\circ\mathrm{Ad}_h
\end{equation}
which proves that $\iota$ is $\mathcal{H}$-equivariant.
\end{proof}

\begin{cor}\label{Corollary:4.9}
Let $(\pi_{\mathcal{S}}:\,\mathcal{P}_{/\mathcal{S}}\rightarrow\mathcal{M}_{/\mathcal{S}},\mathcal{A})$ be a metric reductive super Cartan geometry modeled on a super Klein geometry $(\mathcal{G},\mathcal{H})$ with a super Cartan connection $\mathcal{A}$ and smooth super metric $\mathscr{S}$ on $\Lambda\otimes\mathfrak{g}/\mathfrak{h}$. Let $\theta:=\iota_{\mathcal{P}}^*E\in\Omega^1(\mathcal{P},\mathfrak{g}/\mathfrak{h})^{(\mathcal{H},\mathrm{Ad})}$ be the pullback of the super soldering form on the principal super fiber bundle $\mathcal{P}$ w.r.t. an embedding $\iota_{\mathcal{P}}:\,\mathcal{P}\hookrightarrow\mathcal{S}\times\mathcal{P}$. For any $x\in\mathcal{M}$, consider the map $g_x:\,T_x\mathcal{M}\times T_x\mathcal{M}\rightarrow\Lambda^{\mathbb{C}}$ defined as
\begin{equation}
g(X_x,Y_x):=\mathscr{S}(\theta_p(X^*_p),\theta(Y^*_p))
\end{equation}
for any $p\in\mathcal{P}_x$ and $X^*_p,Y^*_p\in T_p\mathcal{P}$ the unique horizontal tangent vectors such that $D_p\pi(X^*_p)=X_x$ and $D_p\pi(Y^*_p)=Y_x$. Then, $g_x$ is a well-defined super metric on $T_x\mathcal{M}$ for any $x\in\mathcal{M}$. In particular, the assignment $g:\,\mathcal{M}\ni x\mapsto g_x$ defines a smooth super metric on $\mathcal{M}$.
\end{cor}
\begin{proof}
Since $D_p\pi:\,\mathscr{H}_p\stackrel{\sim}{\rightarrow}T_p\mathcal{M}$ is an isomorphism of super $\Lambda$-modules for any $p\in\mathcal{P}$, where $\mathscr{H}:=\mathrm{ker}(\iota_{\mathcal{P}}^*\omega)$ is the horizontal distribution induced by the pullback of the super connection 1-form $\omega:=\mathrm{pr}_{\mathfrak{h}}\circ\mathcal{A}$., it is clear that, for any tangent vectors $X_x,Y_x\in T_x\mathcal{M}$, the horizontal lifts $X^*_p,Y^*_p\in T_p\mathcal{P}$ exist and are unique. Moreover, as $\theta_p$ is a right linear isomorphism of super $\Lambda$-modules, it is clear that $g_x$ defines a super metric on $T_x\mathcal{M}$ once we have shown that $g_x$ is well-defined. Therefore, let $p'\in\mathcal{P}_x$ be another point on the fiber over $x$. Then, there exists $g\in\mathcal{P}$ such that $p'=p\cdot g$. By uniqueness, it follows $X^*_{p\cdot g}=\Phi_{g*}X_p$ and $Y^*_{p\cdot g}=\Phi_{g*}Y_p$. Thus, 
\begin{equation}
\theta_{p'}(X^*_{p'})=\theta_{p\cdot g}(\Phi_{g*}X_p)=\mathrm{Ad}_{g^{-1}}(\theta_p(X^*_p))
\end{equation}
and similarly for $Y^*$. Thus, as $\mathscr{S}$ is $\mathrm{Ad}(\mathcal{H})$-invariant, it follows that $g_x$ is indeed well-defined. To see that $g:\,\mathcal{M}\ni x\mapsto g_x$ is smooth, let $s:\,\mathcal{M}\supseteq U\rightarrow\mathcal{P}$ be a local section and $(e_i)_i$ be a homogeneous basis of $\mathfrak{g}/\mathfrak{h}$. Then, on $U$, the super metric is given by 
\begin{equation}
g(X,Y)=\mathscr{S}(s^*\theta(X),s^*\theta(Y))=(-1)^{(|e_i|+|X|)|e_j|}\mathscr{S}_{ij}(s^*\theta)(X)^i(s^*\theta)(Y)^j
\end{equation}
where $\mathscr{S}_{ij}:=\mathscr{S}(e_i,e_j)\in\mathbb{C}$ as $\mathscr{S}$ is smooth and we have made the expansion $s^*\theta(X)=e_i\otimes(s^*\theta)(X)^i$ with $(s^*\theta)(X)^i$ smooth functions on $U$ and similarly for $s^*\theta(Y)$. Thus, it follows $g(X,Y)\in H^{\infty}(U)\otimes\mathbb{C}$ proving that $g$ is smooth.
\end{proof}
The following proposition demonstrates the strong link between Cartan connections and Ehresmann connections defined on associated bundles. In physics, this thus provides a concrete relation between Cartan geometries and Yang-Mills gauge theories. This is due to the fact that, by definition, both type of connections turn out to be already fixed uniquely on vertical vector fields, i.e., vector fields tangent to the fibers of the underlying principal super fiber bundles. Indeed, on those type of vertical fields, both connections are related to the Maurer-Cartan forms on the respective structure groups. Using this observation, one then arrives at the following.  

\end{definition}
\begin{prop}\label{prop:4.6}
Let $\mathcal{H}\rightarrow\mathcal{P}_{/\mathcal{S}}\rightarrow\mathcal{M}_{/\mathcal{S}}$ be a $\mathcal{S}$-relative principal super fiber bundle with structure group $\mathcal{H}$ as well as $(\mathcal{G},\mathcal{H})$ a super Klein geometry. Then, there is a bijective correspondence between generalized super Cartan connections in $\Omega^1(\mathcal{P}_{/\mathcal{S}},\mathfrak{g})_0$ and super connection 1-forms in $\Omega^1(\mathcal{P}_{/\mathcal{S}}\times_{\mathcal{H}}\mathcal{G},\mathfrak{g})_0$ with $\mathcal{P}_{/\mathcal{S}}\times_{\mathcal{H}}\mathcal{G}:=(\mathcal{P}\times_{\mathcal{H}}\mathcal{G})_{/\mathcal{S}}$ the $\mathcal{G}$-extension of $\mathcal{P}_{/\mathcal{S}}$.
\end{prop}
\begin{proof}
One direction is immediate, i.e., given a $\mathcal{S}$-relative super connection 1-form $\mathcal{A}$ on $\mathcal{P}_{/\mathcal{S}}\times_{\mathcal{H}}\mathcal{G}$, the pullback $\hat{\iota}^*\mathcal{A}$ w.r.t. the embedding $\hat{\iota}:=\mathrm{id}\times\iota:\,\mathcal{P}_{/\mathcal{S}}\rightarrow\mathcal{P}_{/\mathcal{S}}\times_{\mathcal{H}}\mathcal{G}$, with $\iota$ as defined in \ref{prop:2.27}, yields a generalized super Cartan connection on $\mathcal{P}_{/\mathcal{S}}$ according to Definition \ref{Def:4.2}.\\
Conversely, suppose $\mathcal{A}\in\Omega^1(\mathcal{P}_{/\mathcal{S}},\mathfrak{g})_0$ is a generalized super Cartan connection. Let $\hat{\pi}:\,\mathcal{P}_{/\mathcal{S}}\times\mathcal{G}\rightarrow\mathcal{P}_{/\mathcal{S}}\times_{\mathcal{H}}\mathcal{G}$ be the canonical projection. If $\hat{\Phi}_{\mathcal{S}}$ denotes the $\mathcal{G}$-right action on $\mathcal{P}_{/\mathcal{S}}\times_{\mathcal{H}}\mathcal{G}$, it follows that the fundamental vector fields are given by
\begin{align}
\widetilde{Y}_{[p,g]}&=(\hat{\Phi}_{\mathcal{S}})_{[p,g]*}(Y_e)=
D_{([p,g],e)}\hat{\Phi}_{\mathcal{S}}(0_{[p,g]},Y)=D_{([p,g],e)}\hat{\Phi}_{\mathcal{S}}(D_{(p,g)}\hat{\pi}(0_{p},0_g),Y)\nonumber\\
&=D_{(p,g,e)}(\hat{\Phi}_{\mathcal{S}}\circ(\hat{\pi}\times\mathrm{id}))(0_{p},0_g,Y)=D_{(p,g,e)}(\hat{\pi}\circ(\mathrm{id}\times\mu_{\mathcal{G}}))(0_{p},0_g,Y)\nonumber\\
&=D_{(p,g)}\hat{\pi}(0_p,\mu_{\mathcal{G}}(0_g,Y))=D_{(p,g)}\hat{\pi}(0_p,L_{g*}Y)=D_{(p,g)}\hat{\pi}(0_{p},L_{g*}Y)
\end{align}
for any $Y\in\mathrm{Lie}(\mathcal{G})$ and $p\in\mathcal{P}_{/\mathcal{S}}$, $g\in\mathcal{G}$, where the generalized tangent map was used at various stages. Furthermore, for any $X_{p}\in T_{p}(\mathcal{P}_{/\mathcal{S}})$, one has
\begin{align}
D_{(p,g)}\hat{\pi}(X_{p},0_g)&=D_{(p,g)}\hat{\pi}(X_p,D_{(g,e)}\mu_{\mathcal{G}}(0_g,0_e))=D_{(p,g,e)}(\hat{\pi}\circ(\mathrm{id}\times\mu_{\mathcal{G}}))(X_p,0_g,0_e)\nonumber\\
&=D_{(p,g,e)}(\hat{\pi}\circ(\hat{\Phi}_{/\mathcal{S}}\times\mathrm{id}))(X_p,0_g,0_e)=D_{(p\cdot g,e)}\hat{\pi}((\hat{\Phi}_{/\mathcal{S}})_{g*}X_p,0_e)\nonumber\\
&=\hat{\iota}_{*}((\Phi_{\mathcal{S}})_{g*}X_{p})
\end{align}
$\forall (p,g)\in\mathcal{P}_{/\mathcal{S}}\times\mathcal{G}$. Hence, this yields
\begin{align}
D_{(p,g)}\hat{\pi}(X_{p},Y_g)&=D_{(p,g)}\hat{\pi}(X_{p},0_g)+D_{(p,g)}\hat{\pi}(0_{p},Y_g)\nonumber\\
&=\hat{\iota}_{*}((\Phi_{\mathcal{S}})_{g*}X_p)+D_{(p,g)}\hat{\pi}(0_p,L_{g*}\circ L_{g^{-1}*}(Y_g))\nonumber\\
&=\iota_{*}((\Phi_{\mathcal{S}})_{g*}X_p)+\widetilde{\braket{Y_g|\theta_{\mathrm{MC}}}}_{[p,g]}
\end{align}
where $\theta_{\mathrm{MC}}\in\Omega^1(\mathcal{G},\mathfrak{g})$ is the Maurer-Cartan form on $\mathcal{G}$ (see Example \ref{example:MC}). Therefore, if there exists a super connection 1-form $\hat{\iota}_{*}\mathcal{A}$ whose pullback under $\hat{\iota}$ is given by $\mathcal{A}$, then it necessarily has to be of the form
\begin{equation}
\braket{D_{(p,g)}\hat{\pi}(X_p,Y_g)|\hat{\iota}_{*}\mathcal{A}_{[p,g]}}=\mathrm{Ad}_{g^{-1}}\braket{X_p|\mathcal{A}_p}+\braket{Y_g|\theta_{\mathrm{MC}}}
\label{eq:p5.7.1}
\end{equation}
In particular, as $\hat{\pi}$ is a submersion, it is uniquely determined by (\ref{eq:p5.7.1}). Hence, it remains to show that $\hat{\iota}_{*}\mathcal{A}$ as defined via (\ref{eq:p5.7.1}) is in fact well-defined and a super connection 1-form on $\mathcal{P}_{/\mathcal{S}}\times_{\mathcal{H}}\mathcal{G}$.\\
To see that it is well-defined, note that, for any $(p,g)\in\mathcal{P}_{/\mathcal{S}}\times\mathcal{G}$, the kernel of $D_{(p,g)}\hat{\pi}$ is given by $\{(\widetilde{Y}_p,-R_{g*}Y)|\,Y\in\mathrm{Lie}(\mathcal{H})\}\subset T_p(\mathcal{P}_{/\mathcal{S}})\times T_g\mathcal{G}$ which may be checked by direct computation using the explicit form of the local trivializations of the bundle $\pi:\,\mathcal{P}\times\mathcal{G}\rightarrow\mathcal{P}\times_{\mathcal{H}}\mathcal{G}$ as defined in the proof of Prop. \ref{prop:2.22}. Hence, this yields
\begin{align}
\braket{(\widetilde{Y}_p,-R_{g*}Y)|\hat{\iota}_{*}\mathcal{A}_{[p,g]}}&=\mathrm{Ad}_{g^{-1}}\braket{\widetilde{Y}_p|\mathcal{A}_p}-\braket{R_{g*}Y|\theta_{\mathrm{MC}}}=\mathrm{Ad}_{g^{-1}}(Y)-L_{g^{-1}*}\circ R_{g*}(Y)\nonumber\\
&=\mathrm{Ad}_{g^{-1}}(Y)-\mathrm{Ad}_{g^{-1}}(Y)=0
\end{align}
$\forall Y\in\mathrm{Lie}(\mathcal{H})$. Finally, to see that is independent of the choice of a representative of $[p,g]\in\mathcal{P}_{/\mathcal{S}}\times_{\mathcal{H}}\mathcal{G}$, we compute
\begin{align}
\braket{D_{(ph,h^{-1}g)}\hat{\pi}((\hat{\Phi}_{\mathcal{S}})_{h*}X_p,L_{h^{-1}*}Y_g)|\hat{\iota}_{*}\mathcal{A}_{[ph,h^{-1}g]}}&=\mathrm{Ad}_{(h^{-1}g)^{-1}}\braket{(\hat{\Phi}_{\mathcal{S}})_{h*}X_p|\mathcal{A}_{ph}}+\braket{L_{{h^{-1}}*}Y_g|\theta_{\mathrm{MC}}}\nonumber\\
&=\mathrm{Ad}_{g^{-1}}\braket{X_p|\mathcal{A}_{p}}+L_{(h^{-1}g)^{-1}*}\circ L_{h^{-1}*}(Y_g)\nonumber\\
&=\mathrm{Ad}_{g^{-1}}\braket{X_p|\mathcal{A}_{p}}+\braket{Y_g|\theta_{\mathrm{MC}}}\nonumber\\
&=\braket{D_{(p,g)}\hat{\pi}(X_p,Y_g)|\hat{\iota}_{*}\mathcal{A}_{[p,g]}}
\end{align}
$\forall h\in\mathcal{H}$. This shows that $\hat{\iota}_{*}\mathcal{A}$ is in fact well-defined. To see that is $\mathcal{G}$-equivariant, we compute
\begin{align}
\braket{(\hat{\Phi}_{\mathcal{S}})_{h*}D_{(p,g)}\pi(X_p,Y_g)|\hat{\iota}_{*}\mathcal{A}_{[p,gh]}}&=\braket{D_{(p,g,h)}(\hat{\Phi}_{\mathcal{S}}\circ(\hat{\pi}\times\mathrm{id}))(X_p,Y_g,0_h)|\hat{\iota}_{*}\mathcal{A}_{[p,g]}}\nonumber\\
&=\braket{D_{(p,g,h)}(\hat{\pi}\circ(\mathrm{id}\times\mu_{\mathcal{G}}))(X_p,Y_g,0_h)|\hat{\iota}_{*}\mathcal{A}_{[p,gh]}}\nonumber\\
&=\braket{D_{(p,gh)}\hat{\pi}(X_p,R_{h*}Y_g)|\hat{\iota}_{*}\mathcal{A}_{[p,gh]}}\nonumber\\
&=\mathrm{Ad}_{(gh)^{-1}}\braket{X_p|\mathcal{A}_{p}}+\braket{R_{h*}Y_g|\theta_{\mathrm{MC}}}\nonumber\\
&=\mathrm{Ad}_{h^{-1}}(\mathrm{Ad}_{g^{-1}}\braket{X_p|\mathcal{A}_{p}})+L_{(gh)^{-1}*}\circ R_{h*}(Y_g)\nonumber\\
&=\mathrm{Ad}_{h^{-1}}(\mathrm{Ad}_{g^{-1}}\braket{X_p|\mathcal{A}_{p}}+\braket{Y_g|\theta_{\mathrm{MC}}})\nonumber\\
&=\mathrm{Ad}_{h^{-1}}\braket{D_{(p,g)}\hat{\pi}(X_p,Y_g)|\hat{\iota}_{*}\mathcal{A}_{[p,g]}}
\end{align}
Finally, for $Y\in\mathrm{Lie}(\mathcal{G})$, it follows
\begin{align}
\braket{\widetilde{Y}_{[p,g]}|\hat{\iota}_{*}\mathcal{A}_{[p,g]}}&=\braket{D_{(p,g)}\hat{\pi}(X_p,L_{g*}Y)|\hat{\iota}_{*}\mathcal{A}_{[p,g]}}=\braket{L_{g*}Y|\theta_{\mathrm{MC}}}=Y
\end{align}
Hence, this proves that $\hat{\iota}_{*}\mathcal{A}\in\Omega^1(\mathcal{P}_{/\mathcal{S}}\times_{\mathcal{H}}\mathcal{G},\mathfrak{g})_0$ is a well-defined super connection 1-form on $\mathcal{P}_{/\mathcal{S}}\times_{\mathcal{H}}\mathcal{G}$.
\end{proof}

\section{Graded principal bundles and graded connections}\label{section:graded}
As already outlined in the introduction, there exist various different approaches to supermanifold theory. The precise link between the algebro-geometric approach and Rogers-DeWitt approach has been discussed in various references (see e.g. \cite{Batchelor:1980,Rogers:1980,Eder:2020erq} as well as \cite{Mol:10,Sac:08} in context of the categorial approach; for our choice of the definition of algebraic supermanifolds see \cite{Eder:2020erq}). In this section, we want to show in which sense super principal fiber bundles and connection forms in the category of algebro-geometric supermanifolds $\mathbf{SMan}_{\mathrm{Alg}}$ as given in \cite{Stavracou:1996qb} can be related to the respective formulation in the $H^{\infty}$-category. Just for notational simplication, only in this section, we will call objects and morphisms in the category $\mathbf{SMan}_{\mathrm{Alg}}$ with the addition \emph{graded} to distinguish them from their respective $H^{\infty}$ counterparts. However, we have to emphasize that the definition of graded manifolds as chosen here is different to the original definition as given by Berezin-Kostant-Leites \cite{Berezin:1976,Kostant:1975qe} using the notion of \emph{finite duals}, the latter being much more general (see also \cite{Carmeli:2011} for a comparison). Also, in this section, we are considering trivial parametrizing supermanifolds $\mathcal{S}=\{*\}$. The generalization to non-trivial parametrizing supermanifolds is straightforward, though with the addition graded is a bit cumbersome.
\begin{remark}
Before, we introduce the definition of a group action on graded manifolds, recall that a morphism $f:\,(M,\mathcal{O}_{\mathcal{M}})\rightarrow(N,\mathcal{O}_{\mathcal{N}})$ between graded manifolds consists of a tuple $f=(|f|,f^{\sharp})$ with $f:\,M\rightarrow N$ a smooth map between ordinary smooth manifolds as well as a sheaf morphism $f^{\sharp}:\,\mathcal{O}_{\mathcal{N}}\rightarrow f_{*}\mathcal{O}_{\mathcal{M}}$ called \emph{pullback} which is local in the sense that it maps the maximal ideal of the stalk $\mathcal{O}_{\mathcal{N},f(p)}$ to the maximal ideal of $f_{*}\mathcal{O}_{\mathcal{M},p}$, $\forall p\in M$. 

Finally, by the 'Global Chart Theorem' (see e.g. \cite{Carmeli:2011}), it follows that a morphism between graded manifolds is uniquely determined by its pullback. Hence, in the following, we will often state certain properties of morphisms that only involve the corresponding pullback. 
\end{remark}
\begin{definition}
A \emph{right action} of a graded Lie group $\mathcal{G}=(G,\mathcal{O}_{\mathcal{G}})$ on a graded manifold $\mathcal{M}=(M,\mathcal{O}_{\mathcal{M}})$ is a morphism $\Phi:\,(M,\mathcal{O}_{\mathcal{M}})\times(G,\mathcal{O}_{\mathcal{G}})\rightarrow(M,\mathcal{O}_{\mathcal{M}})$ of graded manifolds such that
\begin{equation}
\Phi^{\sharp}\circ(\Phi^{\sharp}\otimes\mathds{1})=\Phi^{\sharp}\circ(\mathds{1}\otimes\mu_{\mathcal{G}}^{\sharp}),\quad(\mathds{1}\otimes e_{\mathcal{G}}^{\sharp})\circ\Phi^{\sharp}=\mathds{1}
\end{equation} 
where $\mu_{\mathcal{G}}^{\sharp}$ and $e_{\mathcal{G}}^{\sharp}$ denote the pullback of the group multiplication $\mu_{\mathcal{G}}:\,\mathcal{G}\times\mathcal{G}\rightarrow\mathcal{G}$ on $\mathcal{G}$ as well as its \emph{neutral element} $e_{\mathcal{G}}:\,\mathbb{R}^{0|0}\rightarrow\mathcal{G}$. Analogously, one defines a \emph{left action} of $(G,\mathcal{O}_{\mathcal{G}})$ on $(M,\mathcal{O}_{\mathcal{M}})$ as a morphism $\Phi:\,(G,\mathcal{O}_{\mathcal{G}})\times(M,\mathcal{O}_{\mathcal{M}})\rightarrow(M,\mathcal{O}_{\mathcal{M}})$ of graded manifolds such that
\begin{equation}
\Phi^{\sharp}\circ(\mathds{1}\otimes\Phi^{\sharp})=\Phi^{\sharp}\circ(\mu_{\mathcal{G}}^{\sharp}\otimes\mathds{1}),\quad(e_{\mathcal{G}}^{\sharp}\otimes\mathds{1})\circ\Phi^{\sharp}=\mathds{1}
\end{equation} 
\end{definition}
The following definition is a particular variant of the definition of graded principal bundles given in \cite{Stavracou:1996qb}.
\begin{definition}[after \cite{Stavracou:1996qb}]
A \emph{graded principal bundle} over a graded manifold $(M,\mathcal{O}_{\mathcal{M}})$ consists of a graded manifold $(P,\mathcal{O}_{\mathcal{P}})$ as well as a right action $\Phi:\,(P,\mathcal{O}_{\mathcal{P}})\times(G,\mathcal{O}_{\mathcal{G}})\rightarrow(P,\mathcal{O}_{\mathcal{P}})$ of a graded Lie group $(G,\mathcal{O}_{\mathcal{G}})$ on $(P,\mathcal{O}_{\mathcal{P}})$ such that
 \begin{enumerate}[label=(\roman*)]
	\item the quotient $(P/G,\mathcal{O}_{\mathcal{P}}/\mathcal{O}_{\mathcal{G}})$ exists as a graded manifold isomorphic to $(M,\mathcal{O}_{\mathcal{M}})$ and the canonical projection $\pi:\,(P,\mathcal{O}_{\mathcal{P}})\rightarrow(P/G,\mathcal{O}_{\mathcal{P}}/\mathcal{O}_{\mathcal{G}})$ is a submersion
	\item $(P,\mathcal{O}_{\mathcal{P}})$ satisfies the local triviality property: For any $p\in M$, there exists an open neighborhood $U\subseteq M$ of $p$ as well as an isomorphism $\phi:\,(V,\mathcal{O}_{\mathcal{M}}|_V)\rightarrow(U\times G,\mathcal{O}_{\mathcal{M}}|_{U}\hat{\otimes}_{\pi}\mathcal{O}_{\mathcal{G}})$ of graded manifolds, where $V:=|\iota\circ\pi|^{-1}(U)\subseteq P$ with $\iota$ the isomorphism $\iota:\,(P/G,\mathcal{O}_{\mathcal{P}}/\mathcal{O}_{\mathcal{G}})\stackrel{\sim}{\rightarrow}(M,\mathcal{O}_{\mathcal{M}})$, such that $\phi$ is $(G,\mathcal{O}_{\mathcal{G}})$-equivariant, that is,
\begin{equation}
\Phi^{\sharp}\circ\phi^{\sharp}=(\phi^{\sharp}\otimes\mathds{1})\circ(\mathds{1}\otimes\mu_{\mathcal{G}}^{\sharp})
\end{equation}
\end{enumerate} 
\end{definition}

\begin{remark}\label{prop:5.3}
In \cite{Eder:2020erq}, it was shown, using the \emph{functor of points} technique, that there exists an equivalence of categories given by two functors $\mathbf{A}:\,\mathbf{SMan}_{H^{\infty}}\rightarrow\mathbf{SMan}_{\mathrm{Alg}}$ and $\mathbf{H}:\,\mathbf{SMan}_{\mathrm{Alg}}\rightarrow\mathbf{SMan}_{\mathrm{Alg}}$. Here, for a given Grassmann algebra $\Lambda$, both categories have to be restricted to the subcategories of supermanifolds with odd dimensions bounded by the number of generators in $\Lambda$ (this can be avoided choosing instead the infinite dimensional Grassmann-algebra $\Lambda_{\infty}$, see e.g. \cite{Batchelor:1980,Rogers:1980}). If $\mathcal{M}$ is a $H^{\infty}$-supermanifold, the corresponding graded manifold $\mathbf{A}(\mathcal{M})$ is defined as
\begin{equation}
\mathbf{A}(\mathcal{M}):=(\mathbf{B}(\mathcal{M}),\mathbf{B}_{*}H^{\infty}_{\mathcal{M}})
\end{equation}
with $\mathbf{B}_{*}H^{\infty}_{\mathcal{M}}$ the pushforward sheaf $\mathbf{B}(\mathcal{M})\supseteq U\mapsto\mathbf{B}_{*}H^{\infty}_{\mathcal{M}}(U):=H^{\infty}(\mathbf{B}^{-1}(U))$. On the other hand, for any graded manifold $\mathcal{K}=(K,\mathcal{O}_{\mathcal{K}})$, the corresponding $H^{\infty}$-supermanifold is defined as the $\Lambda$-point
\begin{equation}
\mathbf{H}(\mathcal{K}):=\mathrm{Hom}_{\mathbf{SAlg}}(\mathcal{O}(\mathcal{K}),\Lambda)
\end{equation}
As it turns out, these functors, in particular, preserve products. In fact, for instance, consider two $H^{\infty}$-supermanifolds $\mathcal{M}$ and $\mathcal{N}$, then $H^{\infty}(\mathcal{M}\times\mathcal{N})\cong H^{\infty}(\mathcal{M})\hat{\otimes}_{\pi} H^{\infty}(\mathcal{N})$ where the completion is taken w.r.t. the \emph{Grothendiek's $\pi$-topology}. Moreover, $\mathbf{B}(\mathcal{M}\times\mathcal{N})=\mathbf{B}(\mathcal{M})\times\mathbf{B}(\mathcal{N})$ which yields $\mathbf{A}(\mathcal{M}\times\mathcal{N})\cong\mathbf{A}(\mathcal{M})\times\mathbf{A}(\mathcal{N})$. On the other hand, for given graded manifolds $\mathcal{K}$ and $\mathcal{L}$, the corresponding $\Lambda$-point is given by $\mathbf{H}(\mathcal{K}\times\mathcal{L})=\mathrm{Hom}_{\mathbf{SAlg}}(\mathcal{O}(\mathcal{K})\hat{\otimes}_{\pi}\mathcal{O}(\mathcal{L}),\Lambda)$. Given morphisms $\phi:\,\mathcal{O}(\mathcal{K})\rightarrow\Lambda$ and $\psi:\,\mathcal{O}(\mathcal{L})\rightarrow\Lambda$, this yields a morphism $\braket{\phi\otimes\psi}:\,\mathcal{O}(\mathcal{K})\hat{\otimes}_{\pi}\mathcal{O}(\mathcal{L})\rightarrow\Lambda$ which on elementary tensors is defined as
\begin{equation}
\braket{\phi\otimes\psi}(f\otimes g):=\phi(f)\psi(g)
\end{equation}
Conversely, given a morphism $\Psi:\,\mathcal{O}(\mathcal{K})\hat{\otimes}_{\pi}\mathcal{O}(\mathcal{L})\rightarrow\Lambda$, we can define morphims $\phi:\,\mathcal{O}(\mathcal{K})\rightarrow\Lambda$ and $\psi:\,\mathcal{O}(\mathcal{L})\rightarrow\Lambda$ setting $\phi(f):=\Psi(f\otimes 1)$ and $\psi(g):=\Psi(1\otimes g)$. Hence, this yields an isomorphism between $\Lambda$-points $\mathbf{H}(\mathcal{K}\times\mathcal{L})\cong\mathbf{H}(\mathcal{K})\times\mathbf{H}(\mathcal{L})$. To summarize, it follows that $\mathbf{A}$ and $\mathbf{H}$ can be extended to \emph{monoidal functors} between \emph{monoidal categories}.
\end{remark}

\begin{prop}\label{prop:5.4}
There is a bijective correspondence between graded principal bundles and $H^{\infty}$-principal super fiber bundles. More precisely,
\begin{enumerate}[label=(\roman*)]
	\item if $\mathcal{G}\rightarrow\mathcal{P}\stackrel{\pi_{\mathcal{P}}}{\rightarrow}\mathcal{M}$ is a $H^{\infty}$-principal super fiber bundle with structure group $\mathcal{G}$, then $\mathbf{A}(\mathcal{P})=(\mathbf{B}(\mathcal{P}),\mathbf{B}_{*}H^{\infty}_{\mathcal{P}})$ has the structure of a graded principal bundle over the graded manifold $\mathbf{A}(\mathcal{M})=(\mathbf{B}(\mathcal{M}),\mathbf{B}_{*}H^{\infty}_{\mathcal{M}})$.
	\item if $(P,\mathcal{O}_{\mathcal{P}})$ is a graded principal bundle over a graded manifold $(M,\mathcal{O}_{\mathcal{M}})$, then $\mathbf{H}(G,\mathcal{O}_{\mathcal{G}})\rightarrow\mathbf{H}(P,\mathcal{O}_{\mathcal{P}})\rightarrow\mathbf{H}(M,\mathcal{O}_{\mathcal{M}})$ has the structure of a $H^{\infty}$-principal super fiber bundle with bundle map $\tilde{\pi}:=\mathbf{H}(\iota\circ\pi)=\mathbf{H}(\iota)\circ\mathbf{H}(\pi):\,\mathbf{H}(P,\mathcal{O}_{\mathcal{P}})\rightarrow\mathbf{H}(M,\mathcal{O}_{\mathcal{M}})$, where $\pi:\,(P,\mathcal{O}_{\mathcal{P}})\rightarrow(P/G,\mathcal{O}_{\mathcal{P}}/\mathcal{O}_{\mathcal{G}})$ is the canonical projection and $\iota$ the isomorphism $\iota:\,(P/G,\mathcal{O}_{\mathcal{P}}/\mathcal{O}_{\mathcal{G}})\stackrel{\sim}{\rightarrow}(M,\mathcal{O}_{\mathcal{M}})$.
\end{enumerate} 
\end{prop}
\begin{proof}
This is an immediate consequence of Remark \ref{prop:5.3} as well as Prop. \ref{prop:2.16}.
\end{proof}
Hence, this proposition demonstrates that the categorical equivalence between graded and $H^{\infty}$ supermanifolds even carries over to principal super fiber bundles. Next, we want to show that connection 1-forms defined on these bundles are in fact in one-to-one correspondence.
\begin{definition}
Let $(M,\mathcal{O}_{\mathcal{M}})$ be a graded manifold and $\mathfrak{g}$ be a super Lie algebra. A $\mathfrak{g}$-valued differential form $\omega$ on $(M,\mathcal{O}_{\mathcal{M}})$ is an element of $\Omega(M,\mathcal{O}_{\mathcal{M}},\mathfrak{g})\equiv\Omega(M,\mathcal{O}_{\mathcal{M}})\otimes\mathfrak{g}$.
\end{definition}
\begin{definition}
Let $(G,\mathcal{O}_{\mathcal{G}})$ be a graded Lie group with super Lie algebra $\mathfrak{g}$. The morphism
\begin{equation}
\mathrm{Ad}:\,G\rightarrow\mathrm{GL}(\mathfrak{g}),\,\mathrm{Ad}_g(X):=(\mathrm{ev}_g\otimes X\otimes\mathrm{ev}_{g^{-1}})\circ(\mathds{1}\otimes\mu^{\sharp})\circ\mu^{\sharp}
\end{equation}
$\forall X\in\mathfrak{g}$, is called the Adjoint representation of $G$ on $\mathfrak{g}$.
\end{definition}
\begin{definition}
Let $(P,\mathcal{O}_{\mathcal{P}})$ be a graded principal bundle over a graded manifold $(M,\mathcal{O}_{\mathcal{M}})$ and right action $\Phi:\,(P,\mathcal{O}_{\mathcal{P}})\times(G,\mathcal{O}_{\mathcal{G}})\rightarrow(P,\mathcal{O}_{
\mathcal{P}})$ of a graded Lie group $(G,\mathcal{O}_{\mathcal{G}})$ on $(P,\mathcal{O}_{\mathcal{P}})$. A \emph{graded connection 1-form} $\omega$ on $(P,\mathcal{O}_{\mathcal{P}})$ is an even $\mathfrak{g}$-valued 1-form $\omega\in\Omega^1(P,\mathcal{O}_{\mathcal{P}},\mathfrak{g})_0$ such that
\begin{enumerate}[label=(\roman*)]
  \item $\braket{\widetilde{X}|\omega}=X$ $\forall X\in\mathfrak{g}$ and $\widetilde{X}:=\mathds{1}\otimes X\circ\Phi^*\in\mathrm{Der}(\mathcal{O}(\mathcal{P}))$ the associated fundamental vector field
	\item $\Phi_g^*\omega=\mathrm{Ad}_{g^{-1}}\circ\omega$ $\forall g\in G$ and $L_{\widetilde{X}}\omega=-\mathrm{ad}_X\circ\omega$ $\forall X\in\mathfrak{g}$
\end{enumerate}
Here, for $g\in G$, $\Phi_g:\,(P,\mathcal{O}_{\mathcal{P}})\rightarrow(P,\mathcal{O}_{\mathcal{P}})$ is the isomorphism of graded manifolds induced by the pullback morphism $\Phi_g^{\sharp}:=(\mathds{1}\otimes\mathrm{ev}_g)\circ\Phi^{\sharp}:\,\mathcal{O}_{\mathcal{P}}\rightarrow\mathcal{O}_{\mathcal{P}}$. Moreover, for homogeneous $X\in\mathfrak{g}$, $\mathrm{ad}_X\circ\omega$ is defined as
	\begin{equation}
	\braket{Z|\mathrm{ad}_X\circ\omega}=(-1)^{|Z||X|}\mathrm{ad}_X\braket{Z|\omega}
	\end{equation}
for any homogeneous $Z\in\mathrm{Der}(\mathcal{O}(\mathcal{P}))$.
\end{definition}
In order to provide a link between graded and $H^{\infty}$-super connection 1-forms, the following lemma will play a central role. It is based on the equivalent characterization of super Lie groups in terms of the corresponding \emph{super Harish-Chandra pair} $(\mathbf{B}(\mathcal{G}),\mathfrak{g})$ (see e.g. \cite{Tuynman:2004,Tuynman:2018} or \cite{Carmeli:2011} in the pure algebraic setting as well as \cite{Schuett:2018} in case of infinite dimensions). 
\begin{lemma}\label{prop:5.8}
Given a $H^{\infty}$-super Lie group as a well as a smooth map $F\in H^{\infty}(\mathcal{G})$. Then, $F$ vanishes identically on $\mathcal{G}$ if and only if $XF|_{\mathbf{B}(\mathcal{G})}\equiv 0$ for all $X\in\mathcal{U}(\mathfrak{g})$ (and thus in particular $F|_{\mathbf{B}(\mathcal{G})}\equiv 0$) with $\mathcal{U}(\mathfrak{g})$ the universal enveloping algebra of $\mathfrak{g}$.
\end{lemma}
\begin{proof}
One direction is clear, so suppose that for some smooth function $F\in H^{\infty}(\mathcal{G})$, $XF|_{\mathbf{B}(\mathcal{G})}\equiv 0$ for all $X\in\mathcal{U}(\mathfrak{g})$. By the \emph{super Harish Chandra theorem}, $\mathcal{G}$ can be identified with the split super Lie group $\mathbf{S}(\mathfrak{g},\mathbf{B}(\mathcal{G}))\cong\mathcal{G}_0\times(\mathfrak{g}_1\otimes\Lambda_{1})$ associated to the trivial vector bundle $\mathbf{B}(\mathcal{G})\times\mathfrak{g}\rightarrow\mathbf{B}(\mathcal{G})$. Hence, there exist odd functions $\theta^{\alpha}\in H^{\infty}(\mathcal{G})$, $\alpha=1,\ldots,n$, $n=\mathrm{dim}\,\mathfrak{g}_1$, such that any $f\in H^{\infty}(\mathcal{G})$ is of the form $f=\mathbf{S}(f)_I\theta^I$ with $\mathbf{S}(f)_I$ the (generalized) Grassmann extensions of smooth functions $f_I\in C^{\infty}(\mathbf{B}(\mathcal{G}))$, $I\in M_n$. Hence, $F$ can be written in the form $F=\mathbf{S}(F)_I\theta^I$ for some $F_I\in C^{\infty}(\mathbf{B}(\mathcal{G}))$. It then follows from the assumptions that for $1\in\mathcal{U}(\mathfrak{g})$, $F(g)=0=F_{\emptyset}(g)$ for any body point $g\in\mathbf{B}(\mathcal{G})$, that is, $F_{\emptyset}\equiv 0$.\\
Let $\partial_{\alpha}$ be the derivations on $H^{\infty}(\mathcal{G})$ satisfying $\partial_{\alpha}\theta^{\beta}=\delta^{\beta}_{\alpha}$, $(X_i)_{i}$ be a homogeneous basis of smooth left-invariant vector fields $X_i\in\mathfrak{g}$ $\forall i$ on $\mathcal{G}$ and, according to \ref{prop:2.13}, $(\tensor[^i]{\omega}{})_i$ the corresponding smooth left-invariant 1-forms satisfying $\tensor[^i]{\omega}{}(X_j)=\delta_{j}^i$. Then, $\partial_{\alpha}$ can be written in the form $\partial_{\alpha}=\braket{\partial_{\alpha}|\tensor[^i]{\omega}{}}X_{i}$ with $\braket{\partial_{\alpha}|\tensor[^i]{\omega}{}}\in H^{\infty}(\mathcal{G})$ $\forall\alpha=1,\ldots,n$ and $i$. As a consequence, for each multi-index $I\in M_n$, $\partial_I:=\partial_{\alpha_1}\cdots\partial_{\alpha_k}$ with $k=|I|$ is a $H^{\infty}(\mathcal{G})$-linear expansion of elements in $\mathcal{U}(\mathfrak{g})$. Hence, by hypothesis, this implies $0=\partial_IF(g)=(-1)^{k(k-1)/2}F_I(g)$ for any body point $g\in\mathbf{B}(\mathcal{G})$, i.e., $F_I\equiv 0$, and therefore $F\equiv 0$ as claimed.
\end{proof}
\begin{remark}
It is clear that Lemma \ref{prop:5.8} equally holds if one replaces $\mathcal{U}(\mathfrak{g})$ by the respective right-invariant counterpart $\mathcal{U}(\mathfrak{g}^R)$, i.e., the universal enveloping algebra of the super Lie algebra $\mathfrak{g}^R$ of smooth right-invariant vector fields on $\mathcal{G}$.
\end{remark}

\begin{prop}\label{prop:5.10}
Let $\mathcal{G}\rightarrow\mathcal{P}\stackrel{\pi_{\mathcal{P}}}{\rightarrow}\mathcal{M}$ be a $H^{\infty}$-principal super fiber bundle with $\mathcal{G}$-right action $\Phi:\,\mathcal{P}\times\mathcal{G}\rightarrow\mathcal{P}$. A smooth even $\mathrm{Lie}(\mathcal{G})$-valued 1-form $\mathcal{A}\in\Omega^1(\mathcal{P},\mathfrak{g})_0$ is a connection 1-form on $\mathcal{P}$ if and only if
\begin{enumerate}[label=(\roman*)]
  \item $\braket{\widetilde{X}|\mathcal{A}}=X$ $\forall X\in\mathfrak{g}$ and $\widetilde{X}:=\mathds{1}\otimes X\circ\Phi^*\in\Gamma(T\mathcal{P})$ the associated smooth fundamental vector field
	\item $\Phi_g^*\mathcal{A}=\mathrm{Ad}_{g^{-1}}\circ\mathcal{A}$ $\forall g\in\mathbf{B}(\mathcal{G})$ and $L_{\widetilde{X}}\mathcal{A}=-\mathrm{ad}_X\circ\mathcal{A}$ $\forall X\in\mathfrak{g}$.
\end{enumerate}
\end{prop}
\begin{proof}
The proof of this proposition is a bit lengthy and technical as one always has to care about smoothness in the various construction since $H^{\infty}$-smoothness is not preserved under partial evaluation. Therefore, we have moved it to the appendix \ref{section:Proof}. 
\end{proof}
\begin{prop}
There is a bijective correspondence between graded connection 1-forms on graded principal bundles and $H^{\infty}$-super connection 1-forms on $H^{\infty}$-principal super fiber bundles.
\end{prop}
\begin{proof}
This is an immediate consequence of Prop. \ref{prop:5.4} and \ref{prop:5.10} as well as Remark \ref{prop:5.3}.
\end{proof}


\section{Parallel transport map}
\subsection{Preliminaries and first construction}\label{section:parallel}
In this section, we want to derive the parallel transport map corresponding to super connection forms. Therefore, at the beginning, we want restrict to trivial parametrizing supermanifolds $\mathcal{S}=\{*\}$. This will provide us already with the main ideas behind the construction and, at the same time, points out the necessity of the parametrization. The generalization to the relative category will then be considered in the subsequent section. The following proposition plays a central role. 
\begin{prop}\label{prop:6.1}
Let $\Phi:\,\mathcal{M}\times\mathcal{G}\rightarrow\mathcal{M}$ be a smooth right action of a super Lie group $\mathcal{G}$ on a supermanifold $\mathcal{M}$. Then, 
\begin{equation}
D_{(p,g)}\Phi(X_p,Y_g)=\Phi_{g*}(X_p)+\widetilde{\theta_{\mathrm{MC}}(Y_g)}_{p\cdot g}
\end{equation}
at any $(p,q)\in\mathcal{M}\times\mathcal{G}$ and tangent vectors $X_p\in T_p\mathcal{M}$, $Y_g\in T_g\mathcal{G}$, where $\theta_{\mathrm{MC}}\in\Omega^1(\mathcal{G},\mathfrak{g})$ is the \emph{Maurer-Cartan form} on $\mathcal{G}$ (see Example \ref{example:MC}).
\end{prop}
\begin{proof}
The proof is very similar to the classical theory of ordinary smooth manifolds. One only has to care about smoothness. We will therefore employ the notion of the generalized tangent map. By linearity, we have
\begin{align}
D_{(p,g)}\Phi(X_p,Y_g)&=D_{(p,g)}\Phi(X_p,0_g)+D_{(p,g)}\Phi(0_p,Y_g)=\Phi_{g*}(X_p)+\Phi_{p*}(Y_g)
\end{align}
Since $Y_g=L_{g*}\circ L_{g^{-1}*}(Y_g)=L_{g*}(\theta_{\mathrm{MC}}(Y_g))$ by definition of the Maurer-Cartan form, this yields, using Prop. \ref{definition:3.11} and that $\Phi$ defines a right action, 
\begin{align}
\Phi_{p*}(Y_g)&=\Phi_{p*}\circ L_{g*}(\theta_{\mathrm{MC}}(Y_g))=(\Phi_{p}\circ L_{g})_{*}(\theta_{\mathrm{MC}}(Y_g))\nonumber\\
&=(\Phi\circ(\mathrm{id}\times\mu))_{(p,g)*}(\theta_{\mathrm{MC}}(Y_g))=(\Phi\circ(\Phi\times\mathrm{id}))_{(p,g)*}(\theta_{\mathrm{MC}}(Y_g))\nonumber\\
&=\Phi_{p\cdot g*}(\theta_{\mathrm{MC}}(Y_g))
\end{align}
which implies $\Phi_{p*}(Y_g)=\widetilde{\theta_{\mathrm{MC}}(Y_g)}_{p\cdot g}$ by definition of the fundamental tangent vector.
\end{proof}

Before we state the main definition of this section concerning the horizontal lift of smooth paths on supermanifolds, let us briefly recall the notion of a local flow of a smooth vector field. In the super category, one needs to distinguish between even and odd vector fields whose corresponding local flows turn out to possess different properties. In what follows, we want to focus on odd vector fields as these seem to be rarely discussed in the literature. In contrast to the classical theory, it follows that the corresponding local flow depends on two parameters $(t,\theta)$ given by both an even and odd parameter $t$ and $\theta$, respectively. These define elements of the superspace $\Lambda^{1,1}$ which can be given the structure of a super Lie group also called the \emph{super translation group} with multiplication defined via $(t,\theta)\cdot(s,\eta)=(t+s+\theta\eta,\theta+\eta)$ $\forall(t,\theta),(s,\eta)\in\Lambda^{1,1}$.  

\begin{definition}
Let $X\in\Gamma(T\mathcal{M})_1$ be an odd smooth vector field on a supermanifold $\mathcal{M}$, $f:\,\mathcal{M}\rightarrow\mathcal{M}$ a smooth map and $t_0\in\Lambda^{1,1}$ a body point. A smooth map $\phi:\,\mathcal{I}\times U\rightarrow V$, with $U,V\subset\mathcal{M}$ open and $t_0\in\mathcal{I}\subset\Lambda^{1,1}$ an open connected super interval, is called a \emph{local flow} of $X$ around $t_0$ with initial condition $f$, if $\phi$ satisfies
\begin{equation}
(\mathcal{D}\otimes\mathds{1})\circ\phi^*=X\circ\phi
\end{equation}
as well as $\phi_{(t_0,0)}:=\phi(t_0,0,\,\cdot\,)=f$ on $U$, where $\mathcal{D}$ denotes the right-invariant vector field $\mathcal{D}:=\partial_{\theta}+\theta\partial_t$ on $\Lambda^{1,1}$. If, in particular, $\mathcal{I}=\Lambda^{1,1}$ and $U=V=\mathcal{M}$, $\phi$ is called a \emph{global flow}.
\end{definition}
 
\begin{prop}\label{prop:6.3}
Let $X\in\Gamma(T\mathcal{M})_1$ be an odd smooth vector field on a supermanifold $\mathcal{M}$ and $f:\,\mathcal{M}\rightarrow\mathcal{M}$ a smooth map, then, for any body point $t_0\in\Lambda^{1,1}$, $X$ admits a local flow $\phi$ around $t_0$ with initial condition $f$. 
\end{prop}
\begin{proof}
See for instance \cite{Florin:2008} for a proof in the pure algebraic setting using the concept of \emph{functor of points}.
\end{proof}

\begin{cor}
If $\phi^X$ is a local flow of an odd smooth vector field $X$, then
\begin{equation}
\phi^X_{(t,\theta)}\circ\phi^X_{(s,\eta)}=\phi^X_{(t+s+\theta\eta,\theta+\eta)}
\end{equation}
whenever both sides are defined.
\end{cor}
\begin{proof}
We give an algebraic proof of this proposition. Therefore, consider the smooth maps  
\begin{align}
\Phi_1:\,(t,\theta,s,\eta,p)&\mapsto (t+s+\theta\eta,\theta+\eta,\phi^X_{(t+s+\theta\eta,\theta+\eta)}(p))\\
\Phi_2:\,(t,\theta,s,\eta,p)&\mapsto (t+s+\theta\eta,\phi^X_{(t,\theta)}(\phi_{(s,\eta)}(p)))
\end{align}
defined on some open subsets of $\Lambda^{1,1}\times\Lambda^{1,1}\times\mathcal{M}$. It then follows that both maps satsify 
\begin{equation}
(\mathcal{D}\otimes\mathds{1})\circ\Phi_i^*=\Phi_i^*\circ(\mathcal{D}\otimes\mathds{1}+\mathds{1}\otimes X),\quad\forall i=1,2
\end{equation}
For instance, since $\Phi_1=(\mathrm{id}\times\phi^X)\circ(d\times\mathrm{id})\circ(\mu\times\mathrm{id})$, with $d$ the diagonal map on $\Lambda^{1,1}$, it follows from the right-invariance of $\mathcal{D}$ as well as $\mathcal{D}\circ d^*=d^*\circ(\mathcal{D}\otimes\mathds{1}+\mathds{1}\otimes\mathcal{D})$
\begin{align}
(\mathcal{D}\otimes\mathds{1})\circ\Phi_i^*&=(\mathcal{D}\otimes\mathds{1})\circ(\mu^*\otimes\mathds{1})\circ(d^*\otimes\mathds{1})\circ(\mathds{1}\otimes\phi^{X*})\nonumber\\
&=(\mu^*\circ\mathcal{D}\otimes\mathds{1})\circ(d^*\otimes\mathds{1})\circ(\mathds{1}\otimes\phi^{X*})\nonumber\\
&=(\mu^*\otimes\mathds{1})\circ(d^*\otimes\mathds{1})\circ(\mathcal{D}\otimes\mathds{1}+\mathds{1}\otimes\mathcal{D})\circ(\mathds{1}\otimes\phi^{X*})\nonumber\\
&=(\mu^*\otimes\mathds{1})\circ(d^*\otimes\mathds{1})\circ(\mathcal{D}\otimes\phi^{X*}+\mathds{1}\otimes(\phi^{X*}\circ X))\nonumber\\
&=\Phi_1^*\circ(\mathcal{D}\otimes\mathds{1}+\mathds{1}\otimes X)
\end{align} 
and similarly for $\Phi_2$. Hence, $\Phi_i$ for $i=1,2$ both define local flows of the odd vector field $\mathcal{D}\otimes\mathds{1}+\mathds{1}\otimes X$ with the initial condition $\Phi_i(0,0,\,\cdot\,)=(\mathrm{id}\times\phi^X)\circ(d\times\mathrm{id})$. By uniqueness, it thus follows that they have to coincide on the intersection of their domains.
\end{proof}

\begin{definition}
Let $\mathcal{P}\stackrel{\pi}{\rightarrow}\mathcal{M}$ be a principal super fiber bundle and $\mathscr{H}\subset T\mathcal{P}$ a principal connection on $\mathcal{P}$. Let $\mathcal{I}\subseteq\Lambda^{1,1}$ be a super interval and $\gamma:\,\mathcal{I}\rightarrow\mathcal{M}$ be a map also called a \emph{path} on $\mathcal{M}$. Then, a smooth path $\gamma^{hor}:\,\Lambda^{1,1}\supseteq\mathcal{I}\rightarrow\mathcal{P}$ on $\mathcal{P}$ is called a \emph{horizontal lift} of $\gamma$, if $\pi\circ\gamma^{hor}=\gamma$ and $(\gamma^{hor})_*\mathcal{D}\subset(\gamma^{hor})^*\mathscr{H}$. 
\end{definition}
The following theorem provides the existence of horizontal lifts of paths defined on supermanifolds. Since, the zero element $0\in\mathcal{I}$ of a super interval defines a body point and paths are supposed to be smooth, it is important to note that the initial values of both the path and its horizontal lift need to be body points, as well.\\
Moreover, as will be proven below, in order to obtain a non-trivial parallel transport map, one necessarily needs to consider smooth paths depending on both even and odd parameters. However, as will be discussed in the subsequent sections, this can be remedied considering instead relative supermanifolds.

\begin{theorem}\label{thm:8.0}
Let $\mathcal{A}\in\Omega^{1}(\mathcal{P},\mathfrak{g})_0$ be a super connection 1-form on a principal super fiber bundle $\mathcal{G}\rightarrow\mathcal{P}\stackrel{\pi}{\rightarrow}\mathcal{M}$ over $\mathcal{M}$ defining a principal connection $\mathscr{H}\subset T\mathcal{P}$ on $\mathcal{P}$. Then, for any smooth path $\gamma:\,\Lambda^{1,1}\supseteq\epsilon_{1,1}^{-1}([0,1])\rightarrow\mathcal{M}$ on the supermanifold $\mathcal{M}$ and any body point $p\in\mathcal{P}$, there exists a unique horizontal lift $\gamma^{hor}_p:\,\epsilon_{1,1}^{-1}([0,1])\rightarrow\mathcal{P}$ of $\gamma$ such that $\gamma^{hor}_p(0)=p$.\\ 
If, in particular, the path is bosonic, i.e., $\gamma:\,\Lambda_0\supseteq\epsilon_{1,0}^{-1}([0,1])\rightarrow\mathcal{M}$ the  horizontal lift $\gamma^{hor}_p:\,\epsilon_{1,0}^{-1}([0,1])\rightarrow\mathcal{P}$ of $\gamma$ is given by
	\begin{equation}
	\gamma^{hor}_p=\mathbf{S}(\mathbf{B}(\gamma)^{hor}_{p})
	\end{equation}
where $\mathbf{B}(\gamma)^{hor}_{p}$ is the (unique) horizontal lift of $\mathbf{B}(\gamma)$ through $p\in P:=\mathbf{B}(\mathcal{P})$ to the ordinary principal bundle $P$ where the principal connection $H\subset TP$ on $P$ is induced by the ordinary connection 1-form $A:=\mathbf{B}(\mathcal{A})\in\Omega^1(P,\mathfrak{g}_0)$.
\end{theorem}
\begin{proof}
Using a gluing argument, it suffices to show that, for any local trivialization neighborhood $\pi^{-1}(U)\subset\mathcal{P}$ of $\mathcal{P}$ with $U\subset\mathcal{M}$ open and $\mathrm{im}\,\gamma\cap U\neq\emptyset$, there exists a horizontal lift $\gamma^{hor}:\,\mathcal{I}\longrightarrow\mathcal{P}$ of $\gamma$ with $\mathcal{I}=\gamma^{-1}(U)$.\\
Hence, in the following, let us assume that $\gamma$ is contained within a local trivialization neighborhood of $\mathcal{P}$, i.e., there is a local section $s:\,U\rightarrow\mathcal{P}$ of $\mathcal{P}$ such that $\mathrm{im}\,\gamma\subset U$. Let $\mathcal{I}=\epsilon_{1,1}^{-1}([0,1])$ and $\delta:=s\circ\gamma$. A horizontal lift then has to be of the form $\gamma^{hor}:=\Phi\circ(\delta\times g)$ for some smooth function $g:\,\tilde{\mathcal{I}}\rightarrow\mathcal{G}$ defined on some open subset $\tilde{\mathcal{I}}\subseteq\mathcal{I}$. By Prop. \ref{prop:6.1}, this yields
\begin{equation}
\mathcal{D}\gamma^{hor}(t,\theta)=D_{(\delta(t,\theta),g(t,\theta))}\Phi(\mathcal{D}\delta,\mathcal{D}g)=\Phi_{g(t,\theta)*}(\mathcal{D}\delta)+\widetilde{\theta_{\mathrm{MC}}(\mathcal{D}g)}_{\gamma^{hor}(t,\theta)}
\end{equation}
Hence, $\gamma^{hor}$ defines a horizontal lift of $\gamma$ if and only if
\begin{align}
0=\braket{\mathcal{D}\gamma^{hor}(t,\theta)|\mathcal{A}}&=\mathrm{Ad}_{g(t,\theta)^{-1}}\braket{\mathcal{D}\delta|\mathcal{A}}+\braket{\mathcal{D}g|\theta_{\mathrm{MC}}}\nonumber\\
&=L_{g(t,\theta)^{-1}*}(R_{g(t,\theta)*}\braket{\mathcal{D}\delta|\mathcal{A}}+\mathcal{D}g)
\end{align}
But, since the pushforward of the left-translation is an isomorphism of tangent modules, this equivalent to
\begin{equation}
\mathcal{D}g(t,\theta)=-R_{g(t,\theta)*}\braket{\mathcal{D}\delta(t,\theta)|\mathcal{A}}=-R_{g(t,\theta)*}\mathcal{A}^{\gamma}(t,\theta)
\label{eq:6.0.1}
\end{equation}
where we set $\mathcal{A}^{\gamma}(t,\theta):=\braket{\mathcal{D}\delta(t,\theta)|\mathcal{A}}=\braket{\mathcal{D}\gamma(t,\theta)|s^*\mathcal{A}}$. Hence, the claim follows if we can show that \eqref{eq:6.0.1} admits a smooth solution on all of $\mathcal{I}$, i.e., $\tilde{\mathcal{I}}=\mathcal{I}$. To see this, let us define the vector field
\begin{align}
Z:\,\mathcal{I}\times\mathcal{G}&\rightarrow T(\mathcal{I}\times\mathcal{G})\cong\Lambda^2\times T\mathcal{G}\nonumber\\
(s,\eta,g)&\mapsto (\eta,1,-D_{(e,g)}\mu_{\mathcal{G}}(\mathcal{A}^{\gamma}(s,\eta),0_g))
\end{align}
Since $\mathcal{A}^{\gamma}$ and the zero section $\mathbf{0}:\,\mathcal{G}\rightarrow T\mathcal{G},\,g\mapsto 0_g$ are both of class $H^{\infty}$, it follows that $Z$ defines a smooth section of the tangent bundle of the supermanifold $\mathcal{I}\times\mathcal{G}$. In particular, as $\mathcal{A}$ is even and $\mathcal{D}$ is odd, $\mathcal{A}^{\gamma}(t,\theta)$ defines an odd derivation $\forall(t,\theta)\in\mathcal{I}$ and therefore $Z$ is an odd vector field. Hence, by Prop. \eqref{prop:6.3}, there exists a smooth function $F:\,\Lambda^{1,1}\supseteq\epsilon_{1,1}^{-1}(\tilde{I})\rightarrow\Lambda^{1,1}\times\mathcal{G},\,(t,\theta)\mapsto (h(t,\theta),g(t,\theta))$, with smooth functions $h$ and $g$ defined on some interval $\epsilon_{1,1}^{-1}(\tilde{I})$ with $\tilde{I}=[0,\delta')$, $0<\delta'<1$, such that
\begin{equation}
\mathcal{D}F(t,\theta)=Z_{F(t,\theta)}
\label{eq:6.0.2}
\end{equation}
with the initial condition $F(0,0)=(0,0,e)$. Since $h:\,(t,\theta)\mapsto h(t,\theta)\in\Lambda^{1,1}$ has to be of the form $h(t,\theta)=(a(t),\theta b(t))$ with smooth functions $a,b:\,\epsilon_{1,0}^{-1}(\tilde{I})\rightarrow\Lambda_{0}$, it follows that $\mathcal{D}h(t,\theta)=(\theta\partial_t a(t),b(t))$. Hence, by definition of the vector field $Z$, the differential equation \eqref{eq:6.0.2} yields $b\equiv 1$ and $\partial_ta=b\equiv 1$ such that, by the initial condition, it follows $h(t,\theta)=(t,\theta)$. Thus, \eqref{eq:6.0.2} becomes
\begin{equation}
\mathcal{D}F(t,\theta)=(\theta,1,\mathcal{D}g(t,\theta))=(\theta,1,-D_{(t,\theta,g(t,\theta))}\mu_{\mathcal{G}}(\mathcal{A}^{\gamma}(t,\theta),0_{g(t,\theta)}))
\end{equation}
That is, $g:\,\epsilon_{1,1}^{-1}(\tilde{I})\rightarrow\mathcal{G}$ defines a smooth solution of \eqref{eq:6.0.1} with initial condition $g(0,0)=e$.\\
It remains to show that $g$ can be extended to all of $\mathcal{I}$. Therefore, since $\mathcal{I}\times\{e\}\subset\mathcal{I}\times\mathcal{G}$ is compact, there exists finitely many open subsets $U_i\subset\mathcal{I}$ and $V_i\subset\mathcal{G}$, $i=1,\ldots,n$, such that $\bigcup_{i=1}^nU_i=\mathcal{I}$ and $e\in V_i$ $\forall i=1,\ldots,n$, as well as smooth maps $\phi_i^Z:\,\epsilon_{1,1}^{-1}((-\delta_i,\delta_i))\times U_i\times V_i\rightarrow\mathcal{I}\times\mathcal{G}$, $0<\delta_i<1$ $\forall i=1,\ldots,n$, such that $\phi^Z_i$ defines a local flow of $Z$. If we set $\delta:=\min\{\delta',\delta_i\}_{i=1,\ldots,n}>0$ and $V:=\bigcap_{i=1}^{n}V_i\ni e$, we may glue the $\phi_i^Z$ to get a local flow $\phi^Z:\,\epsilon_{1,1}^{-1}((-\delta,\delta))\times\mathcal{I}\times V\rightarrow\mathcal{I}\times\mathcal{G}$. Hence, let us choose $t_i\in[0,1]$, $i=0,\ldots,m$, such that $0=:t_0<t_1<\ldots<t_m:=1$ and $|t_{i}-t_{i-1}|<\delta$ $\forall i=1,\ldots,m$. Set $\mathcal{I}_i:=\epsilon_{1,1}^{-1}([t_{i},t_{i-1}])$ for $i=1,\ldots,m$.\\
By assumption, we know that $g$ is well-defined on $\mathcal{I}_1$. Hence, let us next consider the path $G:\,\mathcal{I}_2\rightarrow\mathcal{I}\times\mathcal{G}$ defined as $G(s,\theta)=\phi^Z_{(s,\theta)}(t_1,0,e)$ which is well-defined by definition of $\mathcal{I}_2$. Similarly as above, it follows that $G$ has to be of the form $G(s,\theta)=(s+t_1,\theta,b(s,\theta))$ with $\mathcal{D}b(s,\theta)=-R_{b(s,\theta)*}\mathcal{A}^{\gamma}(s+t_1,\theta)$. If we then define $G'(s,\theta):=(s+t_1,\theta,b(s,\theta)\cdot g(t_1,0))$ on $\mathcal{I}_2$, it follows (note that $G'$ is smooth as $g(t_1,0)$ is a body point) 
\begin{align}
\mathcal{D}G'(s,\theta)&=(\theta,1,R_{g(t_1,0)*}\mathcal{D}b(s,\theta))=(\theta,1,-R_{g(t_1,0)*}R_{b(s,\theta)*}\mathcal{A}^{\gamma}(s+t_1,\theta))\nonumber\\
&=(\theta,1,-R_{b(s,\theta)\cdot g(t_1,0)*}\mathcal{A}^{\gamma}(s+t_1,\theta))=Z_{G'(s,\theta)}
\end{align}
That is, $G'$ defines an integral curve of $Z$ through $G'(0,0)=(t_1,0,g(t_1,0))$. Thus, let us define the smooth path $\tilde{g}:\,\mathcal{I}_1\cup\mathcal{I}_2\rightarrow\mathcal{G}$ via
\begin{equation}
\tilde{g}(t,\theta):=\begin{cases}
    g(t,\theta), & \text{if } (t,\theta)\in\mathcal{I}_1 \\
    b(t-t_1,\theta)\cdot g(t_1,0), & \text{if } (t,\theta)\in\mathcal{I}_2 
  \end{cases}
\end{equation}
It then follows that $\tilde{g}$ is smooth and defines a solution of \eqref{eq:6.0.1}. We have thus found an extension of $g$ to $\mathcal{I}_1\cup\mathcal{I}_2$. Hence, by induction, it follows that $g$ can be extended to all of $\mathcal{I}$.\\
To prove the last claim note that, by pulling back $\mathcal{A}$ to the bosonic sub supermanifold $\mathcal{P}_0:=\mathbf{S}\circ\mathbf{B}(\mathcal{P})$, $\mathcal{A}$ only takes values in the even super Lie sub module $\Lambda\otimes\mathfrak{g}_0=\mathrm{Lie}(\mathcal{G}_0)$. Hence, on $\mathcal{P}_0$, $\mathcal{A}$ can be reduced to a super connection 1-form on the principal super fiber bundle $\mathcal{G}_0\rightarrow\mathcal{P}_0|_{\mathcal{M}_0}\rightarrow\mathcal{M}_0$. The claim now follows immediately.
\end{proof}

\subsection{Parallel transport map revisited}\label{section:parallelNEU}
As we have seen in the last section, in the ordinary theory of principal super fiber bundles, in order to obtain a non-trivial parallel transport map, one necessarily has to consider smooth paths depending on both even and odd parameters. But, still, due to smoothness, the endpoints at which the parallel transport map is constructed have to be points on the body of a supermanifold. Hence, it is important to emphasize that, in the ordinary category of supermanifolds, \emph{one cannot use the parallel transport map to compare points on different fibers of a super fiber bundle!} 

As we will see in what follows, a resolution is given considering instead super connections forms on parametrized super fiber bundles. At the same time, this also allows us to include (anticommutative) fermionic degrees of freedom on the body of a supermanifold which is of utmost importance for supersymmetric field theories in physics. In this framework, it moreover suffices to consider (parametrized) paths depending solely on an even time parameter. The generalization to both even and odd parameters can be obtained along the lines of the previous section. 
\begin{definition}
Let $\mathcal{G}\rightarrow\mathcal{P}_{/\mathcal{S}}\stackrel{\pi_{\mathcal{S}}}{\rightarrow}\mathcal{M}_{/\mathcal{S}}$ be a $\mathcal{S}$-relative principal super fiber bundle and $\gamma:\,\mathcal{S}\times\mathcal{I}\rightarrow\mathcal{M}_{/\mathcal{S}}$ a smooth path on $\mathcal{M}_{/\mathcal{S}}$ with $\mathcal{I}\subseteq\Lambda^{1,0}$. Given a super connection $1$-form $\mathcal{A}$ on $\mathcal{P}_{/\mathcal{S}}$, a smooth path $\gamma^{hor}:\,\mathcal{S}\times\mathcal{I}\rightarrow\mathcal{P}$ on $\mathcal{P}_{/\mathcal{S}}$ is called a horizontal lift of $\gamma$ w.r.t. $\mathcal{A}$ if $\pi\circ\gamma^{hor}=\gamma$ and $\braket{(\mathds{1}\otimes\partial_t)\alpha_{\mathcal{S}}^{-1}(\gamma^{hor})(s,t)|\mathcal{A}}=0$ $\forall (s,t)\in\mathcal{S}\times\mathcal{I}$.  
\end{definition}
\begin{prop}\label{prop:9.7}
Let $\mathcal{A}\in\Omega^1(\mathcal{P}_{/\mathcal{S}},\mathfrak{g})_0$ be a super connection $1$-form on the $\mathcal{S}$-relative principal super fiber bundle $\mathcal{G}\rightarrow\mathcal{P}_{/\mathcal{S}}\stackrel{\pi_{\mathcal{S}}}{\rightarrow}\mathcal{M}_{/\mathcal{S}}$ as well as $\gamma:\,\mathcal{S}\times\mathcal{I}\rightarrow\mathcal{M}$ a smooth path. Let furthermore $f:\,\mathcal{S}\rightarrow\mathcal{P}$ be a smooth map. Then, there exists a unique horizontal lift $\gamma^{hor}:\,\mathcal{S}\times\mathcal{I}\rightarrow\mathcal{M}$ of $\gamma$ w.r.t. $\mathcal{A}$ such that $\gamma^{hor}(\,\cdot\,,0)=f$.
\end{prop}
\begin{proof}
The proof of this proposition is similar as in Theorem \ref{thm:8.0}. Let us therefore only sketch the most important steps. Again, it suffices to assume that $\gamma$ is contained within a local trivialization neighborhood of $\mathcal{P}_{/\mathcal{S}}$, i.e., there exists an open subset $U\subseteq\mathcal{M}$ and a smooth morphism $\tilde{s}:\,U_{/\mathcal{S}}:=\mathcal{M}_{/\mathcal{S}}|_{\mathcal{S}\times U}\rightarrow\mathcal{P}_{/\mathcal{S}}$ of $\mathcal{S}$-relative supermanifolds such that $\pi_{\mathcal{S}}\circ\tilde{s}=\mathrm{id}_{U_{/\mathcal{S}}}$ and $\mathrm{im}\,\gamma\subseteq\pi^{-1}_{\mathcal{S}}(\mathcal{S}\times U)$.\\
Furthermore, w.l.o.g., we can assume that $\tilde{s}(\,\cdot\,,\gamma(\,\cdot\,,0))=f$. In fact, suppose this would not be the case. Then, since the super Lie group $\mathcal{G}$ acts transitively on each fiber of the underlying principal super fiber bundle $\mathcal{P}$, there exists a unique map $g':\,\mathcal{S}\times U\rightarrow\mathcal{G}$ with $\Phi_{\mathcal{S}}(\tilde{s}(\,\cdot\,,\gamma(\,\cdot\,,0)),g')=f$. Since $\tilde{s}(\,\cdot\,,\gamma(\,\cdot\,,0))$, $f$ as well as the inverse operation in the super Lie group $\mathcal{G}$ are of class $H^{\infty}$, it follows immediately that $g'$ is also smooth. Hence, replacing $\tilde{s}$ by $\Phi_{\mathcal{S}}(\tilde{s},g')$ the map thus obtained will have the required properties.\\
Set $\delta:=\tilde{s}\circ(\mathrm{id}\times\gamma):\,\mathcal{I}_{/\mathcal{S}}\rightarrow\mathcal{P}_{/\mathcal{S}}$. It follows that a horizontal lift has to be of the form $\gamma^{hor}:=\Phi_{\mathcal{S}}\circ(\delta\times g)$ for some smooth function $g:\,\mathcal{S}\times\tilde{\mathcal{I}}\rightarrow\mathcal{G}$ defined on some open subset $\tilde{\mathcal{I}}\subseteq\mathcal{I}$. Using Prop. \ref{prop:6.1}, one finds that $\braket{(\mathds{1}\otimes\partial_t)\alpha_{\mathcal{S}}^{-1}(\gamma^{hor})(s,t)|\mathcal{A}}=0$ if and only if
\begin{equation}
(\mathds{1}\otimes\partial_t)g(s,t)=-R_{g(s,t)*}\braket{(\mathds{1}\otimes\partial_t)\delta(s,t)|\mathcal{A}}=-R_{g(s,t)*}\mathcal{A}^{\gamma}(s,t)
\label{eq:9.9}
\end{equation}
where $\mathcal{A}^{\gamma}(s,t):=\braket{(\mathds{1}\otimes\partial_t)\delta(s,t)|\mathcal{A}}=\braket{(\mathds{1}\otimes\partial_t)\alpha_{\mathcal{S}}^{-1}(\gamma)(s,t)|\tilde{s}^*\mathcal{A}}$ with the initial condition $g(\,\cdot\,,0)=e$. Hence, the claim follows if one can show that \eqref{eq:9.9} admits a smooth solution with $\tilde{\mathcal{I}}=\mathcal{I}$. Therefore, consider the even smooth vector field
\begin{align}
\tilde{Z}:\,(\mathcal{S}\times\mathcal{I})\times\mathcal{G}&\rightarrow T(\mathcal{S}\times\mathcal{I})\times T\mathcal{G}\nonumber\\
(s',t',g)&\mapsto (0_{s'},1,-D_{(e,g)}\mu_{\mathcal{G}}(\mathcal{A}^{\gamma}(s',t'),0_g))
\label{eq:9.10}
\end{align}
It follows that there exists a smooth local solution $F:\,U'\rightarrow(\mathcal{S}\times\mathcal{I})\times\mathcal{G}$, $U'\subset\mathcal{S}\times\mathcal{I}$ open, of the equation
\begin{equation}
\partial_tF(s,t)=\tilde{Z}_{F(s,t)}
\label{eq:9.11}
\end{equation}
with the initial condition $F(\,\cdot\,,0)=(\,\cdot\,,0,e)$. Moreover, $F$ has to be of the form $F(s,t)=(s,t,g(s,t))$ for some smooth function $g:\,U'\rightarrow\mathcal{G}$ such that $g(\,\cdot\,,0)=e$ and
\begin{equation}
\partial_tF(s,t)=(0_s,1,(\mathds{1}\otimes\partial_t)g(s,t))=(0_s,1,-D_{(e,g)}\mu_{\mathcal{G}}(\mathcal{A}^{\gamma}(s,t),0_g))
\label{eq:9.12}
\end{equation}
that is, $g$ is a solution of \eqref{eq:9.9} proving that a local smooth solution indeed exists. Remains to show that $g$ can be extended to all of $\mathcal{S}\times\mathcal{I}$.\\
Therefore, one can proceed as in the proof of Theorem \ref{thm:8.0}. In fact, restricting on compact subsets and gluing local solutions together, it follows that, for any $s_0\in\mathcal{S}$, there exists an open neighborhood $s_0\in V\subset\mathcal{S}$ as well as a smooth map $g_{s_0}:\,V\times\mathcal{I}\rightarrow\mathcal{G}$ such that $V$ is contained in a compact subset and $g_{s_0}$ is a solution of \eqref{eq:9.9} with $g(\,\cdot\,,0)=e$. Thus, by uniqueness of solutions of differential equations, these maps can be glued together yielding a global solution $g:\,\mathcal{S}\times\mathcal{I}\rightarrow\mathcal{G}$ of \eqref{eq:9.9}. 
\end{proof}
\begin{remark}\label{prop:7.3}
For a smooth map $f:\,\mathcal{S}\rightarrow\mathcal{M}$, one can consider the pullback super fiber bundle 
\begin{equation}
f^*\mathcal{P}=\{(s,p)|\,f(s)=\pi(p)\}\subset\mathcal{S}\times\mathcal{P}
\label{eq:9.13}
\end{equation}
over $\mathcal{S}$. A smooth section $\tilde{\phi}:\mathcal{S}\rightarrow f^*\mathcal{P}$ of the pullback bundle is then of the form $\tilde{\phi}(s)=(s,\phi(s))$ $\forall s\in\mathcal{S}$ with $\phi:\,\mathcal{S}\rightarrow\mathcal{P}$ a smooth map satsifying $\pi\circ\phi=f$. Hence, we can identify
\begin{equation}
\Gamma(f^*\mathcal{P})=\{\phi:\,\mathcal{S}\rightarrow\mathcal{P}|\,\pi\circ\phi=f\}
\label{eq:9.14}
\end{equation}
\end{remark}
\begin{definition}\label{def:9.9}
Under the conditions of Prop. \ref{prop:9.7}, the \emph{parallel transport in $\mathcal{P}_{/\mathcal{S}}$ along $\gamma$} w.r.t. the connection $\mathcal{A}$ is defined as
\begin{align}
\mathscr{P}^{\mathcal{A}}_{\mathcal{S},\gamma}:\,\Gamma(\gamma_0^*\mathcal{P})&\rightarrow\Gamma(\gamma_1^*\mathcal{P})\label{eq:9.15}\\
\phi&\mapsto\gamma^{hor}_{\phi}(\,\cdot\,,1)\nonumber
\end{align}
where, for $\Gamma(\gamma_0^*\mathcal{P})\ni\phi:\mathcal{S}\rightarrow\mathcal{P}$, $\gamma^{hor}_{\phi}$ is the unique horizontal lift of $\gamma$ with respect to $\mathcal{A}$ such that $\gamma^{hor}_{\phi}(\,\cdot\,,0)=\phi$.
\end{definition}
Given a morphism $\lambda:\,\mathcal{S}\rightarrow\mathcal{S}'$, this induces the pullback $\lambda^*:\,H^{\infty}(\mathcal{S}')\rightarrow H^{\infty}(\mathcal{S}),\,f\mapsto\lambda^*f=f\circ\lambda$ on the function sheaves. Since the super $H^{\infty}(\mathcal{S}\times\mathcal{M})$-module $\mathfrak{X}(\mathcal{M}_{/\mathcal{S}})$ of smooth vector fields on $\mathcal{M}_{/\mathcal{S}}$ is isomorphic to $H^{\infty}(\mathcal{S})\otimes\mathfrak{X}(\mathcal{M})$, this yields the morphism
\begin{align}
\lambda^*\equiv\lambda^*\otimes\mathds{1}:\,\mathfrak{X}(\mathcal{M}_{/\mathcal{S}'})&\rightarrow\mathfrak{X}(\mathcal{M}_{/\mathcal{S}})\label{eq:9.16}\\
f\otimes X&\mapsto\lambda^*f\otimes X\nonumber
\end{align}
Moreover, from $\Omega^1(\mathcal{M}_{/\mathcal{S}})\cong\Omega^1(\mathcal{M})\otimes H^{\infty}(\mathcal{S})$ we obtain the morphism
\begin{align}
\lambda^*\equiv\mathds{1}\otimes\lambda^*:\,\Omega^1(\mathcal{M}_{/\mathcal{S}'})&\rightarrow\Omega^1(\mathcal{M}_{/\mathcal{S}})\\
\omega\otimes f&\mapsto\omega\otimes\lambda^*f\nonumber
\label{eq:9.17}
\end{align}
By definition, it then follows 
\begin{equation}
\braket{\lambda^*X|\lambda^*\mathcal{A}}=\lambda^*\braket{X|\mathcal{A}}
\label{eq:9.18}
\end{equation}
In fact, since \eqref{eq:9.18} is a local property, let us choose a local coordinate neighborhood such that $X$ and $\omega$ can be locally expanded in the form $X=f^i\otimes X_i$ and $\omega=\omega_j\otimes g^j$ with $X_i$ and $\omega_j$ smooth vector fields and 1-forms on $\mathcal{M}$, respectively. We then compute
\begin{align}
\braket{\lambda^*X|\lambda^*\mathcal{A}}&=\braket{\lambda^*f^i\otimes X_i|\omega_j\otimes\lambda^*g^j}=\lambda^*f^i\braket{X_i|\omega_j}\lambda^*g^j\nonumber\\
&=\lambda^*\braket{f^i\otimes X_i|\omega_j\otimes g^j}=\lambda^*\braket{X|\mathcal{A}}
\label{eq:9.19}
\end{align}
The following proposition summarizes some important properties of the parallel transport map such as the functoriality under composition of paths as well as covariance under change of parametrization demonstrating the independence of the choice of a particular parametrizing supermanifold.

\begin{prop}\label{prop:9.10}
The parallel tansport map enjoys the following properties:
\begin{enumerate}[label=(\roman*)]
	\item $\mathscr{P}^{\mathcal{A}}_{\mathcal{S}}$ is functorial under compositions of paths, that is, for smooth paths $\gamma:\,\mathcal{S}\times\mathcal{I}\rightarrow\mathcal{M}$ and $\delta:\,\mathcal{S}\times\mathcal{I}\rightarrow\mathcal{M}$ on $\mathcal{M}_{/\mathcal{S}}$, one has
\begin{equation}
\mathscr{P}^{\mathcal{A}}_{\mathcal{S},\gamma\circ\delta}=\mathscr{P}^{\mathcal{A}}_{\mathcal{S},\gamma}\circ\mathscr{P}^{\mathcal{A}}_{\mathcal{S},\delta}
\label{eq:9.20}
\end{equation}
\item $\mathscr{P}^{\mathcal{A}}_{\mathcal{S},\gamma}$ is covariant under change of parametrization in the sense that if $\lambda:\,\mathcal{S}\rightarrow\mathcal{S}'$ is a morphism of supermanifolds, then the diagram
\begin{equation}
	 \xymatrix{
         \Gamma(f^*\mathcal{P})\ar[r]^{\mathscr{P}^{\mathcal{A}}_{\mathcal{S'},\gamma}}\ar[d]_{\lambda^*} &   \Gamma(g^*\mathcal{P}) \ar[d]^{\lambda^*}\\
            \Gamma((f\circ\lambda)^*\mathcal{P})\ar[r]^{\mathscr{P}^{\lambda^*\!\mathcal{A}}_{\mathcal{S},\lambda^*\!\gamma}} &   \Gamma((g\circ\lambda)^*\mathcal{P})  
     }
		\label{eq:9.21}
 \end{equation}
is commutative for any smooth path $\gamma:\,f\Rightarrow g$ on $\mathcal{M}_{/\mathcal{S}'}$.
\end{enumerate}
\end{prop}
\begin{proof}
The functoriality property of the parallel transport map under the composition of paths is an immediate consequence of equation \eqref{eq:9.9} or \eqref{eq:9.11} and the uniqueness of solutions of differential equations once fixing the inital conditions. In fact, this implies $(\gamma\circ\delta)^{hor}=\gamma^{hor}\circ\delta^{hor}$ yielding \eqref{eq:9.20} by Definition \eqref{eq:9.15}.\\
To prove covariance under change of parametrization, notice that for a supermanifold morphism $\lambda:\,\mathcal{S}\rightarrow\mathcal{S}'$, one has $(\mathds{1}\otimes\partial_t)\lambda^*\gamma^{hor}=\lambda^*((\mathds{1}\otimes\partial_t)\gamma^{hor})$ so that, by Definition \eqref{eq:9.4}, it follows
\begin{equation}
(\mathds{1}\otimes\partial_t)\alpha_{\mathcal{S}}^{-1}(\lambda^*\gamma^{hor})=\lambda^*((\mathds{1}\otimes\partial_t)\alpha_{\mathcal{S}'}^{-1}(\gamma^{hor}))
\label{eq:9.22}
\end{equation}
and thus
\begin{align}
\braket{(\mathds{1}\otimes\partial_t)\alpha_{\mathcal{S}}^{-1}(\lambda^*\gamma^{hor})|\lambda^*\!\mathcal{A}}&=\braket{\lambda^*((\mathds{1}\otimes\partial_t)\alpha_{\mathcal{S}'}^{-1}(\gamma^{hor}))|\lambda^*\!\mathcal{A}}\nonumber\\
&=\braket{(\mathds{1}\otimes\partial_t)\alpha_{\mathcal{S}'}^{-1}(\gamma^{hor})|\mathcal{A}}=0
\label{eq:9.23}
\end{align} 
according to \eqref{eq:9.18}. Since $\lambda^*\gamma^{hor}_{\phi}(\,\cdot\,,0)=\phi\circ\lambda=\lambda^*\phi$ and $\pi\circ\lambda^*\gamma^{hor}_{\phi}=\lambda^*\gamma$, by uniqueness, this yields $\lambda^*\gamma^{hor}_{\phi}=(\lambda^*\gamma)^{hor}_{\lambda^*\phi}$ and therefore
\begin{equation}
\mathscr{P}^{\lambda^*\!\mathcal{A}}_{\mathcal{S},\lambda^*\!\gamma}(\lambda^*\phi)=(\lambda^*\gamma)^{hor}_{\lambda^*\phi}(\,\cdot\,,1)=\lambda^*\gamma^{hor}_{\phi}(\,\cdot\,,1)=\lambda^*(\mathscr{P}^{\mathcal{A}}_{\mathcal{S}',\gamma}(\phi))
\label{eq:9.24}
\end{equation}
$\forall\phi\in\Gamma(f^*\mathcal{P})$ proving the commutativity of the diagram \eqref{eq:9.21}.
\end{proof}

\begin{definition}
A \emph{global gauge transformation} $f$ on the $\mathcal{S}$-relative principal super fiber bundle $\mathcal{G}\rightarrow\mathcal{P}_{/\mathcal{S}}\rightarrow\mathcal{M}_{/\mathcal{S}}$ is a morphism $f:\,\mathcal{P}_{/\mathcal{S}}\rightarrow\mathcal{P}_{/\mathcal{S}}$ of $\mathcal{S}$-relative supermanifolds which is fiber-preserving and $\mathcal{G}$-equivariant, i.e., $\pi_{\mathcal{S}}\circ f=\pi_{\mathcal{S}}$ and $f\circ\Phi_{\mathcal{S}}=\Phi_{\mathcal{S}}\circ(f\times\mathrm{id})$. The set of global gauge transformations on $\mathcal{P}_{/\mathcal{S}}$ will be denoted by $\mathscr{G}(\mathcal{P}_{/\mathcal{S}})$.
\end{definition}
\begin{prop}\label{prop:7.0.6}
Their exists a bijective correspondence between the set $\mathscr{G}(\mathcal{P}_{/\mathcal{S}})$ of global gauge transformations on the $\mathcal{S}$-relative principal super fiber bundle $\mathcal{G}\rightarrow\mathcal{P}_{/\mathcal{S}}\rightarrow\mathcal{M}_{/\mathcal{S}}$ and the set 
\begin{equation}
H^{\infty}(\mathcal{S}\times\mathcal{P},\mathcal{G})^{\mathcal{G}}:=\{\sigma:\,\mathcal{S}\times\mathcal{P}\rightarrow\mathcal{G}|\,\sigma\circ\Phi_{\mathcal{S}}=\alpha_{g^{-1}}\circ\sigma\}
\label{eq:9.25}
\end{equation}
via
\begin{equation}
H^{\infty}(\mathcal{S}\times\mathcal{P},\mathcal{G})^{\mathcal{G}}\ni\sigma\mapsto\Phi_{\mathcal{S}}\circ(\mathrm{id}\times\sigma)\circ d_{\mathcal{S}\times\mathcal{P}}\in\mathscr{G}(\mathcal{P}_{\mathcal{S}})
\label{eq:9.26}
\end{equation}
In particular, global gauge transformations are super diffeomorphisms on $\mathcal{P}_{/\mathcal{S}}$ and $\mathscr{G}(\mathcal{P}_{/\mathcal{S}})$ forms an abstract group under composition of smooth maps. 
\end{prop}
\begin{proof}
The proof of this proposition is almost the same as in the classical theory. Hence, let us only show that the map $\sigma_f\in H^{\infty}(\mathcal{S}\times\mathcal{P},\mathcal{G})^{\mathcal{G}}$ corresponding to a global gauge transformation $f\in\mathscr{G}(\mathcal{P}_{/\mathcal{S}})$ such that $f(s,p)=(s,p)\cdot\sigma_f(s,p)$ $\forall(s,p)\in\mathcal{S}\times\mathcal{P}$ is indeed of class $H^{\infty}$.\\
Therefore, choose a local trivialization $(U,\phi_U)$ of $\mathcal{P}$ and set $\tilde{\phi}_U:=\mathrm{id}\times\phi_U:\,\pi_{\mathcal{S}}^{-1}(\mathcal{S}\times U)\rightarrow(\mathcal{S}\times U)\times\mathcal{G}$. On the local trivialization neighborhood, $f$ is then of the form
\begin{equation}
\tilde{\phi}_U\circ f\circ\tilde{\phi}_U^{-1}((s,x),g)=((s,x),\sigma(s,x,g))
\label{eq:9.27}
\end{equation}
for some smooth function $\sigma:\,(\mathcal{S}\times U)\times G\rightarrow\mathcal{G}$. Hence, 
\begin{equation}
\sigma_f\circ\phi_U^{-1}((s,x),g)=\mu_{\mathcal{G}}(g^{-1},\sigma(s,x,g))
\label{eq:9.28}
\end{equation}
on $(\mathcal{S}\times U)\times\mathcal{G}$ proving that $\sigma_f$ is smooth. That global gauge transformations are diffeomorphisms and $\mathscr{G}(\mathcal{P}_{/\mathcal{S}})$ forms an abstract group now follows immediately from the respective properties of $H^{\infty}(\mathcal{S}\times\mathcal{P},\mathcal{G})^{\mathcal{G}}$.
\end{proof}
\begin{prop}\label{Prop:2.26}
Let $\mathcal{A}\in\Omega^1(\mathcal{P}_{/\mathcal{S}},\mathfrak{g})_0$ be a $\mathcal{S}$-relative super connection $1$-form and $f\in\mathscr{G}(\mathcal{P}_{/\mathcal{S}})$ a global gauge transformation on the $\mathcal{S}$-relative principal super fiber bundle $\mathcal{G}\rightarrow\mathcal{P}_{/\mathcal{S}}\stackrel{\pi_{\mathcal{S}}}{\rightarrow}\mathcal{M}_{/\mathcal{S}}$. Then,
\begin{enumerate}[label=(\roman*)]
	\item $f^*\mathcal{A}\in\Omega^1(\mathcal{P}_{/\mathcal{S}},\mathfrak{g})_0$ is a connection 1-form and, in particular,
\begin{equation}
f^*\mathcal{A}=\mathrm{Ad}_{\sigma_f^{-1}}\circ\mathcal{A}+\sigma_f^*\theta_{\mathrm{MC}}
\label{eq:9.29}
\end{equation}
\item the diagram
\begin{displaymath}
	 \xymatrix{
         \Gamma(g^*\mathcal{P})\ar[r]^{\mathscr{P}^{\mathcal{A}}_{\mathcal{S},\gamma}}\ar[d]_{\alpha_{\mathcal{S}}\circ f\circ\alpha_{\mathcal{S}}^{-1}} &   \Gamma(h^*\mathcal{P}) \ar[d]^{\alpha_{\mathcal{S}}\circ f\circ\alpha_{\mathcal{S}}^{-1}}\\
            \Gamma(g^*\mathcal{P})\ar[r]^{\mathscr{P}^{f^*\mathcal{A}}_{\mathcal{S},\gamma}} &   \Gamma(h^*\mathcal{P})  
     }
		\label{eq:9.30}
 \end{displaymath}
is commutative for any smooth path $\gamma:\,g\Rightarrow h$ on $\mathcal{M}_{/\mathcal{S}}$.
\end{enumerate}
\end{prop}
\begin{proof}
First, let us show that $f^*\mathcal{A}$ is a $\mathcal{S}$-relative connection 1-form on $\mathcal{P}_{/\mathcal{S}}$. Therefore, since $f$ is $\mathcal{G}$-equivariant, it follows that fundamental vector fields $\widetilde{X}$ on $\mathcal{P}_{/\mathcal{S}}$ associated to $X\in\mathfrak{g}$ satisfy
\begin{align}
f_{*}\widetilde{X}&=\mathds{1}\otimes X\circ\Phi_{\mathcal{S}}^*\circ f^*=\mathds{1}\otimes X\circ f^*\otimes\mathds{1}\circ\Phi_{\mathcal{S}}\nonumber\\
&=f^*\circ\mathds{1}\otimes X\circ\Phi_{\mathcal{S}}^*=f^*\circ\widetilde{X}
\label{eq:9.31}
\end{align}
which yields
\begin{equation}
\braket{\widetilde{X}|f^*\mathcal{A}}=\braket{f_{*}\widetilde{X}|\mathcal{A}_{f(\,\cdot\,)}}=f^*\braket{\widetilde{X}|\mathcal{A}}=X
\label{eq:9.32}
\end{equation}
$\forall X\in\mathfrak{g}$. Moreover,
\begin{equation}
(\Phi_{\mathcal{S}})_g^*(f^*\mathcal{A})=(f\circ\Phi_{\mathcal{S}})^*\mathcal{A}=f^*((\Phi_{\mathcal{S}})_g^*\mathcal{A})=\mathrm{Ad}_{g^{-1}}\circ f^*\mathcal{A}
\label{eq:9.33}
\end{equation}
This proves that $f^*\mathcal{A}\in\Omega^1(\mathcal{P}_{/\mathcal{S}},\mathfrak{g})_0$ indeed defines a connection 1-form on $\mathcal{P}_{/\mathcal{S}}$. Next, applying Prop. \ref{prop:6.1}, we find
\begin{align}
f_{*}(X_p)=D_{p}(\Phi_{\mathcal{S}}\circ(\mathrm{id}\times\sigma_f))(X_p,X_p)&=D_{(p,\sigma_f(p))}\Phi_{\mathcal{S}}(X_p,D_p\sigma_f(X_p))\nonumber\\
&=(\Phi_{\mathcal{S}})_{\sigma_{f}(p)^*}(X_p)+\widetilde{\theta_{\mathrm{MC}}(D_p\sigma_f(X_p))}
\label{eq:9.34}
\end{align}
$\forall X_p\in T_p(\mathcal{P}_{/\mathcal{S}})$, $p\in\mathcal{P}_{/\mathcal{S}}$, and thus
\begin{align}
\braket{X_p|f^*\mathcal{A}_p}&=\braket{D_pf(X_p)|\mathcal{A}_{f(p)}}=\braket{(\Phi_{\mathcal{S}})_{\sigma_{f}(p)^*}(X_p)|\mathcal{A}}+\braket{D_p\sigma_f(X_p)|\theta_{\mathrm{MC}}}\nonumber\\
&=\mathrm{Ad}_{\sigma_f(p)^{-1}}\braket{X_p|\mathcal{A}_p}+\braket{X_p|\sigma_f^*\theta_{\mathrm{MC}}}
\label{eq:9.35}
\end{align}
which yields \eqref{eq:9.29}. To prove the last assertion, let $\gamma_{f,\phi}^{hor}:\,\mathcal{S}\times\mathcal{P}\rightarrow\mathcal{P}$ for any $\phi\in\Gamma(g^*\mathcal{P})$ be the unique horizontal lift of the smooth path $\gamma:\,g\Rightarrow h$ w.r.t. $\mathcal{A}$ through $\alpha_{\mathcal{S}}(f\circ\alpha_{\mathcal{S}}^{-1}(\phi))\in\Gamma(g^*\mathcal{P})$. Set
\begin{equation}
\tilde{\gamma}_{\phi}:=\alpha_{\mathcal{S}}(f^{-1}\circ\alpha_{\mathcal{S}}^{-1}(\gamma_{f,\phi}^{hor})):\,\mathcal{S}\times\mathcal{P}\rightarrow\mathcal{P}
\label{eq:9.36}
\end{equation}
Then, $\tilde{\gamma}_{\phi}$ is a smooth path on $\mathcal{P}_{/\mathcal{S}}$ with $\tilde{\gamma}_{\phi}(\,\cdot\,,0)=\phi$ and 
\begin{equation}
\braket{(\mathds{1}\otimes\partial_t)\tilde{\gamma}_{\phi}|f^*\mathcal{A}}=\braket{(\mathds{1}\otimes\partial_t)\alpha_{\mathcal{S}}^{-1}(\gamma_{f,\phi}^{hor})|\mathcal{A}}=0
\label{eq:9.37}
\end{equation}
so that $\tilde{\gamma}_{\phi}$ coincides with the unique horizontal lift of $\gamma$ w.r.t. $f^*\mathcal{A}$ through $\phi$. 
\end{proof}
\begin{example}\label{example:7.9}
We want to give an explicit local expression of the parallel transport map as derived above making it more accessible for applications in physics. Therefore, we assume that $\mathcal{G}$ is a super matrix Lie group, i.e., an embedded super Lie subgroup of the general linear group $\mathrm{GL}(\mathcal{V})$ of a super $\Lambda$-vector space $\mathcal{V}$. In this case, the pushforward $R_{g*}$ of the right translation for any $g\in\mathcal{G}$ then just coincides with the right multiplication by the super matrix $g$. Hence, let $\gamma:\,\mathcal{S}\times\mathcal{I}\rightarrow\mathcal{M}$ be a smooth path which is contained within a local trivialization neighborhood of $\mathcal{P}_{/\mathcal{S}}$ and $\tilde{s}:\,U_{/\mathcal{S}}:=\mathcal{M}_{/\mathcal{S}}|_{\mathcal{S}\times U}\rightarrow\mathcal{P}_{/\mathcal{S}}$ the corresponding smooth section. Then, equation \eqref{eq:9.9} in the proof of Prop. \eqref{prop:9.7} reads 
\begin{equation}
(\mathds{1}\otimes\partial_t)g(s,t)=-\mathcal{A}^{\gamma}(s,t)\cdot g(s,t)
\label{eq:9.38}
\end{equation}
with $\mathcal{A}^{\gamma}(s,t):=\braket{(\mathds{1}\otimes\partial_t)\alpha_{\mathcal{S}}^{-1}(\gamma)(s,t)|\tilde{s}^*\mathcal{A}}$. Furthermore, suppose that $U$ defines a local coordinate neighborhood of $\mathcal{M}$. The 1-form $\tilde{s}^*\mathcal{A}$ on $\mathcal{S}\times U$ can then be expanded in the form
\begin{equation}
\tilde{s}^*\mathcal{A}=\mathrm{d}x^{\mu}\mathcal{A}^{(\tilde{s})}_{\mu}+\mathrm{d}\theta^{\alpha}\mathcal{A}^{(\tilde{s})}_{\alpha}
\label{eq:9.39}
\end{equation} 
with smooth even and odd functions $\mathcal{A}^{(\tilde{s})}_{\mu}$ and $\mathcal{A}^{(\tilde{s})}_{\alpha}$ on $\mathcal{S}\times U$, respectively. This yields 
\begin{equation}
\mathcal{A}^{\gamma}(s,t)=:\dot{x}^{\mu}\mathcal{A}^{(\tilde{s})}_{\mu}(s,t)+\dot{\theta}^{\alpha}\mathcal{A}^{(\tilde{s})}_{\alpha}(s,t)
\label{eq:9.40}
\end{equation}
Hence, the solution of equation \eqref{eq:9.38} with the initial condition $g(\,\cdot\,,0)=\mathds{1}$ takes the form
\begin{equation}
g(s,t)=\mathcal{P}\exp\left(-\int_{0}^t{\mathrm{d}t'\,\dot{x}^{\mu}\mathcal{A}^{(\tilde{s})}_{\mu}(s,t')+\dot{\theta}^{\alpha}\mathcal{A}^{(\tilde{s})}_{\alpha}(s,t')}\right)
\label{eq:9.41}
\end{equation}
where $\mathcal{P}\exp(\ldots)$ denotes the usual \emph{path-ordered exponential}. This is the most general local expression of the parallel transport map corresponding to a $\mathcal{S}$-relative super connection 1-form. This form is used for instance in \cite{Mason:2010yk} in the discussion about the relation between super twistor theory and $\mathcal{N}=4$ super Yang-Mills theory (see also \cite{Groeger:2013aja}). Note that in case $\mathcal{S}=\{*\}$ is a single point, the odd coefficients in \eqref{eq:9.41} become zero so that this expression just reduces the parallel transport map of an ordinary connection 1-form on a principal fiber bundle in accordance with Theorem \ref{thm:8.0}.\\
By definition, $g[\mathcal{A}]:=g(\cdot,1)$ defines a smooth map $g[\mathcal{A}]:\,\mathcal{S}\rightarrow\mathcal{G}$ from the parametrizing supermanifold $\mathcal{S}$ to the gauge group $\mathcal{G}$. As explained in detail in \cite{Eder:2020erq} (see also section \ref{section:graded}), there exists an equivalence of categories $\mathbf{A}:\,\mathbf{Man}_{H^{\infty}}\rightarrow\mathbf{Man}_{\mathrm{Alg}}$ between the category $\mathbf{Man}_{H^{\infty}}$ of $H^{\infty}$-supermanifolds and the category $\mathbf{Man}_{\mathrm{Alg}}$ of algebro-geometric supermanifolds. Using this equivalence, it thus follows 
\begin{equation}
H^{\infty}(\mathcal{S},\mathcal{G})\cong\mathrm{Hom}_{\mathbf{SMan}_{\mathrm{Alg}}}(\mathbf{A}(\mathcal{S}),\mathbf{A}(\mathcal{G}))
\end{equation}
Hence, $g[\mathcal{A}]$ can be identified with a $\mathbf{A}(\mathcal{S})$-point of $\mathbf{A}(\mathcal{G})$. This coincides with the results of \cite{Florin:2008} and \cite{Groeger:2013aja} where the parallel transport of super connections on super vector bundles in the pure algebraic setting has been considered. It was found that the parallel transport map has the interpretation in terms of $\mathcal{T}$-points of a general linear group.
\end{example}
As explained in example \ref{example:7.9}, the parallel transport map corresponding to super connection 1-forms on relative principal super fiber bundles shares many properties with the parallel transport map as studied in the pure algebraic setting in \cite{Florin:2008,Groeger:2013aja} in the context of connections on super vector bundles. To make this link even more precise, let us start with an equivalent characterization of horizontal forms in terms of forms with values in the associated bundle.
\begin{prop}\label{prop:7.11}
Let $\mathcal{G}\rightarrow\mathcal{P}_{/\mathcal{S}}\stackrel{\pi_{\mathcal{S}}}{\rightarrow}\mathcal{M}_{/\mathcal{S}}$ be a $\mathcal{S}$-relative principal super fiber bundle and $\rho:\,\mathcal{G}\rightarrow\mathrm{GL}(\mathcal{V})$ be a representation of $\mathcal{G}$ on a super $\Lambda$-vector space $\mathcal{V}$. Then, there exists an isomorphism between $\Omega^k_{hor}(\mathcal{P}_{/\mathcal{S}},\mathcal{V})^{(\mathcal{G},\rho)}$ and $\Omega^k(\mathcal{M}_{/\mathcal{S}},\mathcal{E}_{/\mathcal{S}})\cong\Omega^k(\mathcal{M}_{/\mathcal{S}})\otimes\mathcal{E}_{/\mathcal{S}}$, i.e., $k$-forms on $\mathcal{M}_{/\mathcal{S}}$ with values in the associated $\mathcal{S}$-relative super vector bundle $\mathcal{E}_{/\mathcal{S}}:=(\mathcal{P}\times_{\rho}\mathcal{V})_{/\mathcal{S}}$.
\end{prop}
\begin{proof}
For $k=0$, this is straightforward generalization of Corollary \ref{prop:2.24}. For general $k\in\mathbb{N}$, suppose $\omega$ is a horizontal $k$-form on $\mathcal{P}_{/\mathcal{S}}$ of type $(\mathcal{G},\rho)$. Choose a local trivialization $s:\,U_{/\mathcal{S}}\rightarrow\mathcal{P}_{/\mathcal{S}}$ with $U\subseteq\mathcal{M}$ open. On $U_{/\mathcal{S}}$, we then define a $k$-form $\overline{\omega}\in\Omega^k(\mathcal{M}_{/\mathcal{S}},\mathcal{E}_{/\mathcal{S}})$ as follows
\begin{equation}
\braket{X_1,\ldots,X_k|\overline{\omega}}:=[s,\braket{X_1,\ldots,X_k|s^*\omega}]
\label{eq:7.0.1}
\end{equation}
for any smooth vector fields $X_i$ on $\mathcal{M}_{/\mathcal{S}}$, $i=1,\ldots,k$. By horizontality and $\mathcal{G}$-equivariance of $\omega$, it is then immediate to see that $\overline{\omega}$ is indeed well-defined and independent on the choice of a local section. The inverse direction follows similarly.
\end{proof}
\begin{definition}
Under the assumptions of Prop. \ref{prop:7.11}, on $\Omega(\mathcal{M}_{/\mathcal{S}},\mathcal{E}_{/\mathcal{S}})$, the \emph{exterior covariant derivative} $\mathrm{d}_{\mathcal{A}}:\,\Omega^k(\mathcal{M}_{/\mathcal{S}},\mathcal{E}_{/\mathcal{S}})\rightarrow\Omega^{k+1}(\mathcal{M}_{/\mathcal{S}},\mathcal{E}_{/\mathcal{S}})$ \emph{induced by $\mathcal{A}$} is defined via
\begin{equation}
\mathrm{d}_{\mathcal{A}}\overline{\omega}:=\overline{D^{(\mathcal{A})}\omega}
\end{equation}
for any $\omega\in\Omega^k_{hor}(\mathcal{P}_{/\mathcal{S}},\mathcal{V})^{(\mathcal{G},\rho)}$. For $k=0$, we also write $\mathrm{d}_{\mathcal{A}}\equiv\nabla^{(\mathcal{A})}$.
\end{definition} 
\begin{definition}
Under the assumptions of Prop. \ref{prop:7.11}, let $f:\,\mathcal{S}\rightarrow\mathcal{M}$ be a smooth map. As in Remark \ref{prop:7.3}, for the pullback bundle $f^*\mathcal{E}$, we have
\begin{equation}
\Gamma(f^*\mathcal{E})=\{\phi:\,\mathcal{S}\rightarrow\mathcal{E}|\,\pi_{\mathcal{E}}\circ\phi=f\}
\end{equation}
By definition, any $\phi\in\Gamma(f^*\mathcal{E})$ is of the form $\phi=[\phi_0,v]$ with $\phi_0\in\Gamma(f^*\mathcal{P})$. Hence, let $\gamma:\,\mathcal{S}\times\mathcal{I}\rightarrow\mathcal{M}$ be a smooth path on $\mathcal{M}_{/\mathcal{S}}$. The connection 1-form $\mathcal{A}$ on $\mathcal{P}_{/\mathcal{S}}$ then induces a parallel transport map $\mathscr{P}^{\mathcal{E},\mathcal{A}}_{\mathcal{S},\gamma}$ on $\mathcal{E}_{/\mathcal{S}}$ along $\gamma$ via
\begin{equation}
\mathscr{P}^{\mathcal{E},\mathcal{A}}_{\mathcal{S},\gamma}:\,\Gamma(\gamma_0^*\mathcal{E})\rightarrow\Gamma(\gamma_1^*\mathcal{E}),\,\phi=[\phi_0,v]\mapsto[\mathscr{P}^{\mathcal{A}}_{\mathcal{S},\gamma}(\phi_0),v]
\end{equation}
with $\mathscr{P}^{\mathcal{A}}_{\mathcal{S}}$ the parallel transport map induced by $\mathcal{A}$ as defined via \ref{def:9.9}.
\end{definition}
The following proposition, together with Prop. \ref{prop:3.23} giving an explicit form of the exterior covariant derivative on horizontal forms, provides a link between the parallel transport on associated $\mathcal{S}$-relative super vector bundles and the parallel transport on algebraic super vector bundles as constructed in \cite{Florin:2008,Groeger:2013aja}.
\begin{prop}
Under the assumptions of Prop. \ref{prop:7.11}, let $\mathcal{A}\in\Omega^1(\mathcal{P}_{/\mathcal{S}},\mathfrak{g})_0$ be a super connection 1-form on the $\mathcal{S}$-relative principal super fiber bundle $\mathcal{G}\rightarrow\mathcal{P}_{/\mathcal{S}}\rightarrow\mathcal{M}_{/\mathcal{S}}$ and $\mathscr{P}^{\mathcal{E},\mathcal{A}}_{\mathcal{S}}$ the induced parallel transport map on the associated $\mathcal{S}$-relative super vector bundle. Let furthermore $\gamma:\,\mathcal{S}\times\mathcal{I}\rightarrow\mathcal{M}$ be a smooth path on $\mathcal{M}_{/\mathcal{S}}$ and $e\in\Gamma(\mathcal{E}_{/\mathcal{S}})$ a smooth section which is covariantly constant along $\gamma$ w.r.t. $\mathcal{A}$, i.e.,
\begin{equation}
\braket{(\mathds{1}\otimes\partial_t)\hat{\gamma}|\nabla^{(\mathcal{A})}e}=0
\end{equation}
$\forall (s,t)\in\mathcal{S}\times\mathcal{I}$ with $\hat{\gamma}:=\alpha_{\mathcal{S}}^{-1}(\gamma)$. Then, the pullback of $e$ along the path $\gamma$ is given by $\hat{\gamma}^*e=[\gamma^{hor}_{\phi},v]$ with $[\phi,v]=:e\circ\hat{\gamma}(\cdot,0)\in\Gamma(\gamma_{0}^*\mathcal{E})$ and $\gamma_{\phi}^{hor}$ is the unique horizontal lift through $\phi$. In particular, 
\begin{equation}
\hat{\gamma}^*_1e=\mathscr{P}^{\mathcal{E},\mathcal{A}}_{\mathcal{S},\gamma}(\hat{\gamma}^*_0e)
\end{equation}
\end{prop}
\begin{proof}
By locality, it suffices to assume that the claim holds on a local trivilization neighborhood. Hence, w.l.o.g. suppose that $\gamma$ is contained within a local trivialization neighborhood of $\mathcal{P}_{/\mathcal{S}}$ induced by a local section $\tilde{s}:\,U_{/\mathcal{S}}\rightarrow\mathcal{P}_{/\mathcal{S}}$. With respect to this trivialization, the section $e$ is then of the form $e=[\tilde{s},v]$ with $v:\,\mathcal{S}\times U\rightarrow\mathcal{V}$ a smooth map. Using \eqref{eq:7.0.1}, it then immediately follows by definition of the covariant derivative that
\begin{equation}
\braket{(\mathds{1}\otimes\partial_t)\hat{\gamma}|\nabla^{(\mathcal{A})}e}=[\delta,(\mathds{1}\otimes\partial_t)\hat{v}+\rho_{*}(\mathcal{A}^{\gamma})\hat{v}]
\end{equation}
where $\mathcal{A}^{\gamma}:=\braket{(\mathds{1}\otimes\partial_t)\hat{\gamma}|\tilde{s}^*\mathcal{A}}$, $\delta:=\tilde{s}\circ\hat{\gamma}:\,\mathcal{I}_{/\mathcal{S}}\rightarrow\mathcal{P}_{/\mathcal{S}}$ and $\hat{v}=v\circ\hat{\gamma}$ such that $\hat{v}(\cdot,0)=v_0$. Hence, $e$ is covariantly constant along $\gamma$ iff
\begin{equation}
(\mathds{1}\otimes\partial_t)\hat{v}+\rho_{*}(\mathcal{A}^{\gamma})\hat{v}=0
\end{equation}
On the other hand, consider the smooth path $\tilde{e}(s,t):=[\gamma_{\phi}^{hor}(s,t),v_0]$ $\forall(s,t)\in\mathcal{S}\times\mathcal{I}$. By Prop. \ref{prop:9.7}, w.r.t. to the chosen local trivialization, the horizontal lift takes the form $\gamma_{\phi}=\Phi_{\mathcal{S}}(\delta,g)$ with $g:\,\mathcal{S}\times\mathcal{I}\rightarrow\mathcal{G}$ a smooth map satisfying
\begin{equation}
(\mathds{1}\otimes\partial_t)g(s,t)=-R_{g(s,t)*}\mathcal{A}^{\gamma}(s,t)
\end{equation}
together with the initial condition $g(\cdot,0)=e$. Hence, this yields $\tilde{e}=[\Phi_{\mathcal{S}}(\delta,g),v_0]=[\delta,\hat{v}]$ with $\tilde{v}:=\rho(g)v_0:\,\mathcal{S}\times\mathcal{I}\rightarrow\mathcal{V}$. Taking the partial time derivative of $\tilde{v}$, this then yields, together with $\rho\circ R_g=R_{\rho(g)}\circ\rho$ $\forall g\in\mathcal{G}$,
\begin{align}
(\mathds{1}\otimes\partial_t)\tilde{v}(s,t)&=D_{g(s,t)}\rho((\mathds{1}\otimes\partial_t)g(s,t))v_0=-D_{e}(\rho\circ R_{g(s,t)})(\mathcal{A}^{\gamma})v_0\nonumber\\
&=-R_{\rho(g(s,t))*}(\rho_{*}(\mathcal{A}^{\gamma}))v_0=-\rho_{*}(\mathcal{A}^{\gamma})\rho(g(s,t))v_0\nonumber\\
&=-\rho_{*}(\mathcal{A}^{\gamma})\tilde{v}(s,t)
\end{align}
where, in the second line, we used that the pushforward of the right-translation on $\mathrm{GL}(\mathcal{V})$ can be identified with the ordinary right-group multiplication. Since, $\tilde{v}(\cdot,0)=v_0$ it thus follows from the uniqueness of solutions of differential equations once fixing the initial conditions that $\tilde{v}=\hat{v}$ on $\mathcal{S}\times\mathcal{I}$. This proves the proposition.
\end{proof}

\section{Applications}\label{section:Application}
In this last section, we want to provide some explicit examples in order to demonstrate how supermanifold techniques as studied in this present article find concrete applications in mathematical physics. To this end, at the beginning, we will discuss $\mathcal{N}=1$, $D=4$ Poincaré supergravity in the Cartan geometric framework. In particular, we will study bundle automorphisms as well as local symmetries of the supergravity action. In this context, we will also briefly comment on the Castellani-Castellani-D'Auria-Fré approach to supergravity \cite{DAuria:1982uck,Castellani:1991et,Castellani:2013iq,Castellani:2018zey}
 (for a different approach towards a mathematical rigorous formulation of geometric supergravity using the notion of \emph{chiral triples} see \cite{Cortes:2018lan}). Furthermore, we will sketch a link between superfields on parametrized supermanifolds and the description of anticommutative fermionic fields in pAQFT.\\
Finally, we will study Killing vector fields of super Riemannian manifolds arising from metric reductive super Cartan geometries. As an example, we will classify Killing spinors of homogeneous super anti-de Sitter space.
\subsection{Geometric supergravity}
\subsubsection{The super Poincaré group}
Before we discuss the interpretation of Poincaré supergravity in terms of a super Cartan geometry, let us briefly recall the basic definition of the super translation group as well as the super Poincaré group in $D=4$. We choose the following conventions: The gamma matrices $\gamma_I$, $I=0,\ldots,3$, satisfy the 4D Clifford algebra relations
\begin{equation}
    \{\gamma_I,\gamma_J\}=2\eta_{IJ}
\end{equation}
with $\eta$ the Minkowski metric on $\mathbb{R}^{1,3}$ with signature $(-+++)$. The Clifford algebra $\mathrm{Cl}(\mathbb{R}^{1,3},\eta)$ as a real vector space of dimension $\mathrm{dim}\,\mathrm{Cl}(\mathbb{R}^{1,3},\eta)=16$ is spanned by the unit $1$ together with the elements
\begin{equation}
\gamma_{I_{1}I_2\cdots I_{k}}:=\gamma_{[I_1}\gamma_{I_2}\cdot\ldots\cdot\gamma_{I_k]}    
\end{equation}
for $k=1,\ldots,4$ where the bracket denotes antisymmetrization. The \emph{charge conjugation matrix} $C$ is given by $C=i\gamma^3\gamma^1$ and satisfies the symmetry relations
\begin{equation}
    C^T=-C,\quad (C\gamma_I)^T=C\gamma_I\text{ and }(C\gamma_{IJ})^T=C\gamma_{IJ}
\end{equation}
Moreover, $\epsilon^{IJKL}=-\epsilon_{IJKL}$ denotes the completely antisymmetric symbol in $D=4$ with the convention $\epsilon^{0123}:=1$.

\begin{definition}[The super translation group $\mathcal{T}^{1,3|4}$]
Let $\kappa_{\mathbb{R}}:\,\mathrm{Spin}^{+}(1,3)\rightarrow\mathrm{GL}(\Delta_{\mathbb{R}})$ be the real Majorana representation of $\mathrm{Spin}^{+}(1,3)$. Consider the trivial vector bundle bundle $\mathbb{R}^{1,3}\times\Delta_{\mathbb{R}}\rightarrow\mathbb{R}^{1,3}$ over Minkowski spacetime $(\mathbb{R}^{1,3},\eta)$. Applying the split functor, this then yields a split supermanifold 
\begin{equation}
\mathbb{R}^{1,3|4}=\mathbf{S}(\Delta_{\mathbb{R}},\mathbb{R}^{1,3})\cong\Lambda^{4,4}
\end{equation}
also called \emph{super Minkowski spacetime}. On this supermanifold, we define the map 
\begin{align}
\mu:\,\Lambda^{4,4}\times\Lambda^{4,4}&\rightarrow\Lambda^{4,4}\\
((x,\theta),(y,\eta))&\mapsto (z,\theta+\eta)\nonumber
\end{align}
where $z^I:=x^I+y^I-\frac{1}{4}(C\gamma^I)_{\alpha\beta}\theta^{\alpha}\eta^{\beta}$ for $I=0,\ldots,3$ and $C$. It follows immediately that $\mu$ is smooth and associative. In particular, $\mathbb{R}^{1,3|4}$ equipped with $\mu$ defines a super Lie group with neutral element $e=(0,0)$ and inverse $i:\,\Lambda^{4,4}\rightarrow\Lambda^{4,4},\,(x,\theta)\mapsto (-x,-\theta)$ (note that $C\gamma^I$ is symmetric for $I=0,\ldots,3$). From now on, let us denote this super Lie group by $\mathcal{T}^{1,3|4}$ and call it the \emph{super translation group}.\\
To derive the corresponding super Lie algebra, note that the co-multiplication is given by
\begin{align}
\mu^*(x^I)&=x^I\otimes 1+1\otimes x^I-\frac{1}{4}(C\gamma^I)_{\alpha\beta}\theta^{\alpha}\otimes\theta^{\beta}\\
\mu^*(\theta^{\alpha})&=\theta^{\alpha}\otimes 1+1\otimes\theta^{\alpha}
\end{align} 
The super Lie module of the super translation group takes the form
\begin{equation}
T_e\mathcal{T}^{1,3|4}=:\Lambda\otimes\mathfrak{t}\cong\mathrm{span}_{\Lambda}\left\{\frac{\partial}{\partial x^I}\bigg{|}_e,\frac{\partial}{\partial\theta^{\alpha}}\bigg{|}_e\right\}
\end{equation} 
so that a homogeneous basis of left invariant vector fields is given by $P_I:=\left(\mathds{1}\otimes\frac{\partial}{\partial x^I}\big{|}_e\right)\circ\mu^*$ and $Q_{\alpha}:=\left(\mathds{1}\otimes\frac{\partial}{\partial\theta^{\alpha}}\big{|}_e\right)\circ\mu^*$ for $I=0,\ldots,3$ and $\alpha=1,\ldots,4$. Hence, their action on the coordinate functions yields
\begin{align}
Q_{\alpha}(x^I)&=\frac{1}{4}(C\gamma^I)_{\alpha\gamma}\theta^{\gamma},\quad Q_{\alpha}(\theta^{\beta})=\delta^{\beta}_{\alpha}\nonumber\\
P_I(x^J)&=\delta^J_I,\quad P_I(\theta^{\alpha})=0
\end{align}
so that the vector fields can explicitly be written in the form
\begin{equation}
Q_{\alpha}=\frac{\partial}{\partial\theta^{\alpha}}+\frac{1}{4}(C\gamma^I)_{\alpha\beta}\theta^{\beta}\frac{\partial}{\partial x^I},\quad P_I=\frac{\partial}{\partial x^I}
\end{equation}
Using these identities, we can compute the corresponding commutation relations which yields
\begin{align}
[Q_{\alpha},Q_{\beta}]=\frac{1}{2}(C\gamma^I)_{\alpha\beta}P_I,\quad[Q_{\alpha},P_I]=0\,\text{ and }\,[P_I,P_J]=0 
\label{eq:8.0.1}
\end{align}
\end{definition}
\begin{definition}[Super Poincaré group] 
The \emph{super Poincaré group} in $D=4$, $\mathcal{N}=1$ is defined as the semi-direct product
\begin{equation}
\mathrm{ISO}(\mathbb{R}^{1,3|4})=\mathcal{T}^{1,3|4}\rtimes_{\Phi}\mathbf{S}(\mathrm{Spin}^+(1,3))
\end{equation}
where $\Phi:\,\mathbf{S}(\mathrm{Spin}^+(1,3))\rightarrow\mathrm{GL}(\mathcal{T}^{1,3|4})$ is the representation of the purely bosonic super Lie group $\mathbf{S}(\mathrm{Spin}^+(1,3))$ on the super translation group $\mathcal{T}^{1,3|4}$ obtained by applying the split functor on the group representation
\begin{equation}
\mathrm{Spin}^+(1,3)\ni g\mapsto\mathrm{diag}(\lambda^+(g),\kappa_{\mathbb{R}}(g))\in\mathrm{GL}(\mathbb{R}^{1,3}\oplus\Pi\Delta_{\mathbb{R}})
\end{equation}
of $\mathrm{Spin}^+(1,3)$ on the super vector space $\mathbb{R}^{1,3}\oplus\Pi \Delta_{\mathbb{R}}$ with $\lambda^+:\,\mathrm{Spin}^+(1,3)\rightarrow\mathrm{SO}^+(1,3)$ the universal covering map and $\Pi:\,\mathbf{SVec}\rightarrow\mathbf{SVec}$ the \emph{parity functor}, i.e., $\Pi \Delta_{\mathbb{R}}$ is viewed as a purely odd super vector space.\\
The super Lie algebra $\mathfrak{iso}(\mathbb{R}^{1,3|4})$ is generated by the bosonic momenta and Lorentz boosts $(P_I,M_{IJ})$, $I,J=1,\ldots,3$, and four fermionic Majorana generators $Q_{\alpha}$, $\alpha=1,\ldots, 4$. It follows that, in addition to \eqref{eq:8.0.1}, the non-vanishing (graded) commutation relations are given by
\begin{align}
[M_{IJ},Q_{\alpha}]=\frac{1}{2}Q_{\beta}\tensor{(\gamma_{IJ})}{^{\beta}_{\alpha}},\text{ and }[P_I,Q_{\alpha}]=0
\end{align}
\end{definition}
Let us equip the super translation group $\mathcal{T}^{1,3|4}$ (or super Minkowski spacetime) with a smooth super metric $\mathscr{S}$ setting $\mathscr{S}(P_I,P_J)=\eta_{IJ}$ and $\mathscr{S}(Q_{\alpha},Q_{\beta})=C_{\alpha\beta}$. By definition, it follows that $\mathscr{S}$ is invariant under the Adjoint representation of $\mathrm{Spin}^+(1,3)$ on $\mathcal{T}^{1,3|4}$. Moreover, $\mathrm{ISO}(\mathbb{R}^{1,3|4})$ can be identified with the super isometry group of super Minkowski spacetime.
\subsubsection{Poincaré supergravity and the Castellani-D'Auria-Fré approach}
With these preparations, let us turn next to supergravity. We want to describe $D=4$, $\mathcal{N}=1$ Poincaré supergravity as a metric reductive super Cartan geometry $(\pi_{\mathcal{S}}:\,\mathcal{P}_{/\mathcal{S}}\rightarrow\mathcal{M}_{/\mathcal{S}},\mathcal{A})$ modeled on the super Klein geometry $(\mathcal{G},\mathcal{H})=(\mathrm{ISO}(\mathbb{R}^{1,3|4}),\mathrm{Spin}^+(1,3))$ corresponding to super Minkowski spacetime. Here and in the following, for notational simplification, we will often identify $\mathrm{Spin}^+(1,3)$ with the corresponding bosonic split super Lie group $\mathbf{S}(\mathrm{Spin}^+(1,3))$. It follows that the super Cartan connection decomposes as
\begin{equation}
\mathcal{A}=E+\omega:=\mathrm{pr}_{\mathfrak{g}/\mathfrak{h}}\circ\mathcal{A}+\mathrm{pr}_{\mathfrak{h}}\circ\mathcal{A}=:e^IP_I+\frac{1}{2}\omega^{IJ}M_{IJ}+\psi^{\alpha}Q_{\alpha}
\label{eq:8.223}
\end{equation}
with $E$ is the supervielbein and $\omega$ a Ehresmann connection. The Cartan curvature $F(\mathcal{A})=\mathrm{d}\mathcal{A}+[\mathcal{A}\wedge\mathcal{A}]$ takes the form
\begin{equation}
F(\mathcal{A})=F(\mathcal{A})^IP_I+\frac{1}{2}F(\mathcal{A})^{IJ}M_{IJ}+F(\mathcal{A})^{\alpha}Q_{\alpha}
\end{equation}
where
\begin{align}
F(\mathcal{A})^I=\Theta^{(\omega) I}-\frac{1}{4}\bar{\psi}\wedge\gamma^I\psi,\quad F(\mathcal{A})^{IJ}=F(\omega)^{IJ}\text{ and }F(\mathcal{A})^{\alpha}=D^{(\omega)}\psi^{\alpha}
\label{eq:torsion}
\end{align}
with $F(\omega)$ and $\Theta^{(\omega)}$ the curvature and torsion of $\omega$, respectively, and $D^{(\omega)}\psi=\mathrm{d}\psi+\kappa_{\mathbb{R}*}(\omega)\wedge\psi$ the spinor covariant derivative.\\
The action of $D=4$, $\mathcal{N}=1$ Poincaré supergravity can be obtained from the MacDowell-Mansouri action of anti-de Sitter supergravity as discussed in \cite{MacDowell:1977jt,Castellani:2013iq,Eder:2020erq} performing the Inönü-Wigner contraction. Let $P\rightarrow M$ be the underlying ordinary $\mathrm{Spin}^+(1,3)$-bundle obtained after applying the body functor. Choosing a local section $s:\,M\supset U\rightarrow P\subset\mathcal{P}$, it follows that the action takes the form 
\begin{equation}
    S(\mathcal{A})=\frac{1}{2\kappa}\int_{M}{s^*\mathscr{L}}
\end{equation}
with Langrangian $\mathscr{L}\in\Omega^4_{hor}(\mathcal{P}_{/\mathcal{S}})$ defining a horizontal form on $\mathcal{P}_{/\mathcal{S}}$ given by 
\begin{align}
\mathscr{L}:=\frac{1}{2}F(\omega)^{IJ}\wedge e^K\wedge e^L\epsilon_{IJKL}+
i\bar{\psi}\wedge\gamma_{*}\gamma_ID^{(\omega)}\psi\wedge e^I
\label{eq:Action}
\end{align}
In what follows, we want to study the local symmetries of the Lagrangian of Poincaré supergravity. We will therefore adapt the 'group-geometric' approach of Castellani-D'Auria-Fré. Before we proceed, we need some preparations.
\begin{definition}\label{Prop:8.3}
Let $\mathcal{H}\rightarrow\mathcal{P}_{/\mathcal{S}}\rightarrow\mathcal{M}_{/\mathcal{S}}$ be a $\mathcal{S}$-relative principal super fiber bundle. On $\mathcal{P}_{/\mathcal{S}}$, the set of \emph{infinitesimal automorphisms} is defined as    
\begin{equation}
\mathfrak{aut}(\mathcal{P}_{/\mathcal{S}})=\{X\in\Gamma(T\mathcal{P}_{/\mathcal{S}})|\,(\Phi_{\mathcal{S}})_{h*}X=X,\,\forall h\in\mathcal{H}\}
\end{equation}
Since the generalized tangent map commutes with the commutator between smooth vector fields, it follows that $\mathfrak{aut}(\mathcal{P}_{/\mathcal{S}})$ defines a proper super Lie subalgebra of the super Lie algebra $\Gamma(T\mathcal{P}_{/\mathcal{S}})$ of smooth sections of the tangent bundle of $\mathcal{P}_{/\mathcal{S}}$. 
\end{definition}
Suppose $X\in\Gamma(T\mathcal{M}_{/\mathcal{S}})$ is a smooth vector field. As the restriction of the bundle projection to the horizontal distribution $\mathscr{H}:=\mathrm{ker}(\omega)$ induced by the Ehresmann connection $\omega$ yields an isomorphism on the tangent bundle of $\mathcal{M}_{/\mathcal{S}}$, it follows that there exists a unique horizontal lift $X^*\in\Gamma(T\mathcal{P}_{/\mathcal{S}})$ such that $X^*_p\in\mathscr{H}_p$ and $\pi_{\mathcal{S}*}X^*=X$. Then, by uniqueness, it follows $(\Phi_{\mathcal{S}})_{g^*}X^*=X^*$ since $\pi_{\mathcal{S}}\circ\Phi_{\mathcal{S}}=\pi_{\mathcal{S}}\circ\mathrm{pr}_1$. The horizontal lift thus defines an infinitesimal automorphism on $\mathcal{P}_{/\mathcal{S}}$.  Moreover, it follows that we can identify $\Gamma(T\mathcal{P}_{/\mathcal{S}})\subset\mathfrak{aut}(\mathcal{P}_{/\mathcal{S}})$ with smooth horizontal vector fields on $\mathcal{P}_{/\mathcal{S}}$. Let us next consider the vertical counterpart.
\begin{definition}
The set $\mathfrak{gau}(\mathcal{P}_{/\mathcal{S}})$ of \emph{vertical infinitesimal automorphisms} or \emph{infinitesimal gauge transformations} is defined as the subset of $\mathfrak{aut}(\mathcal{P}_{/\mathcal{S}})$ given by
\begin{equation}
\mathfrak{gau}(\mathcal{P}_{/\mathcal{S}})=\{X\in\mathfrak{aut}(\mathcal{P}_{/\mathcal{S}})|\,X_p\in\mathscr{V}_p,\,\forall p\in\mathcal{P}_{/\mathcal{S}}\}
\end{equation}
Again, since the generalized tangent map preserves the commutator, it follows that the commutator between infinitesimal gauge transformations is again an infinitesimal gauge transformation. Thus, $\mathfrak{gau}(\mathcal{P}_{/\mathcal{S}})$ defines a proper super Lie subalgebra of $\mathfrak{aut}(\mathcal{P}_{/\mathcal{S}})$.
\end{definition}
Together with Proposition \ref{prop:2.24}, we obtain the following.
\begin{prop}
There exists an isomorphism between infinitesimal gauge transformations $\mathfrak{gau}(\mathcal{P}_{/\mathcal{S}})$ and $\mathcal{H}$-equivariant smooth functions on $\mathcal{S}\times\mathcal{P}$ with values in $\mathrm{Lie}(\mathcal{H})$ via
\begin{align}
\mathrm{Ad}(\mathcal{P}_{/\mathcal{S}})\cong H^{\infty}(\mathcal{S}\times\mathcal{P},\mathrm{Lie}(\mathcal{H}))^{\mathcal{H}}&\stackrel{\sim}{\rightarrow}\mathfrak{gau}(\mathcal{P}_{/\mathcal{S}})\\
f&\mapsto X:\,p\mapsto D_{(p,e)}\Phi_{\mathcal{S}}(0_p,f(p))
\end{align}
where $\mathrm{Ad}(\mathcal{P}_{/\mathcal{S}}):=(\mathcal{P}\times_{\mathrm{Ad}}\mathrm{Lie}(\mathcal{H}))_{/\mathcal{S}}$ is the \emph{Adjoint bundle}.
\end{prop}
\begin{remark}
Note that, according to Prop. \ref{prop:7.0.6}, one can identify global gauge transformations on $\mathscr{G}(\mathcal{P}_{/\mathcal{S}})$ with the set $H^{\infty}(\mathcal{S}\times\mathcal{P},\mathcal{H})^{\mathcal{H}}$ which forms an abstract group via pointwise multiplication. Taking pointwise derivatives, this suggests that one may thus interpret $\mathfrak{gau}(\mathcal{P}_{/\mathcal{S}})$ as the super Lie algebra of the group of global gauge transformations $\mathscr{G}(\mathcal{P}_{/\mathcal{S}})$ if one may be able to equip it with a smooth supermanifold structure. An explicit proof, however, requires the study of infinite dimensional supermanifolds (see for instance \cite{Sac:09,Schuett:2018} for recent results in this direction).
\end{remark}
With these preparations, let us discuss the local symmetries of Poincaré supergravity. Since, the Lagrangian $\mathscr{L}$ is pulled back to the bosonic sub supermanifold $\mathcal{M}_0$ (or rather its body $M=\mathbf{B}(\mathcal{M}_0)$), we only have to require that $\mathscr{L}$ is invariant after restriction to $\mathcal{P}_0$. Hence, let $X\in\mathfrak{aut}(\mathcal{P}_{/\mathcal{S}})$ be an infinitesimal automorphism. We say that $X$ defines a \emph{local symmetry} iff
\begin{equation}
\delta_X\mathscr{L}|_{\mathcal{P}_0}:=\iota_{\mathcal{P}_0}^*(L_X\mathscr{L})=\mathrm{d}\alpha
\label{eq:8.0.0}
\end{equation}
for some smooth $3$-form $\alpha\in\Omega^3((\mathcal{P}_0)_{/\mathcal{S}})$. In order to compute the variation of the Lagrangian, one has to determine the variations of the individual field components. Therefore, note that the super Cartan connection induces an isomorphism
\begin{equation}
\mathcal{A}:\,\mathfrak{aut}(\mathcal{P}_{/\mathcal{S}})\rightarrow H^{\infty}(\mathcal{S}\times\mathcal{P},\mathrm{Lie}(\mathcal{G}))^{\mathcal{H}},\,X\mapsto\braket{X|\mathcal{A}}
\end{equation}
from the super Lie algebra of infinitesimal automorphisms to the super Lie algebra of $\mathcal{H}$-equivariant smooth functions with values in $\mathrm{Lie}(\mathcal{G})$. In fact, for any $X\in\mathfrak{aut}(\mathcal{P}_{/\mathcal{S}})$, we have $(\Phi_{\mathcal{S}})_{h*}X=X$ $\forall h\in\mathcal{H}$ which implies
\begin{equation}
\braket{X|\mathcal{A}}(p\cdot h)=\braket{X_p|(\Phi_{\mathcal{S}})_{h}^*\mathcal{A}}=\mathrm{Ad}_{h^{-1}}\braket{X|\mathcal{A}}(p)
\end{equation}
Hence, it follows that the variation of the super Cartan connection takes the form
\begin{equation}
\delta_X\mathcal{A}:=L_X\mathcal{A}=\iota_XF(\mathcal{A})+D^{(\mathcal{A})}(\iota_X\mathcal{A})
\label{eq:8.234}
\end{equation}
where, according to Remark \ref{prop:3.0.25}, $D^{(\mathcal{A})}(\iota_X\mathcal{A})=\mathrm{d}(\iota_X\mathcal{A})+[\mathcal{A}\wedge\iota_X\mathcal{A}]$. In fact, choosing a homogeneous basis $(T_{\underline{A}})_{\underline{A}}$ of $\mathfrak{g}$, a direct calculation yields (setting $|\underline{A}|\equiv|T_{\underline{A}}|$)
\begin{align}
\iota_X[\mathcal{A}\wedge\mathcal{A}]&=(-1)^{|\underline{A}||\underline{B}|}\iota_X(\mathcal{A}^{\underline{A}}\wedge\mathcal{A}^{\underline{B}})\otimes[T_{\underline{A}},T_{\underline{B}}]\nonumber\\
&=(-1)^{|\underline{A}||\underline{B}|}(\iota_X\mathcal{A}^{\underline{A}}\wedge\mathcal{A}^{\underline{B}}-(-1)^{|\underline{A}||X|}\mathcal{A}^{\underline{A}}\wedge\iota_X\mathcal{A}^{\underline{B}})\otimes[T_{\underline{A}},T_{\underline{B}}]\nonumber\\
&=(-1)^{(|\underline{A}|+|X|)|\underline{B}|}\mathcal{A}^{\underline{B}}\wedge\iota_X\mathcal{A}^{\underline{A}}\otimes[T_{\underline{B}},T_{\underline{A}}]-(-1)^{(|\underline{B}|+|X|)|\underline{A}|}\mathcal{A}^{\underline{A}}\wedge\iota_X\mathcal{A}^{\underline{B}}\nonumber\\
&=-2[\mathcal{A}\wedge\iota_X\mathcal{A}]
\end{align}
which immediately gives \eqref{eq:8.234}. Since $\mathscr{L}$ is obviously invariant under global $\mathrm{Spin}^+(1,3)$-gauge transformations by definition, i.e., it is horizontal, it follows that any $X\in\mathfrak{gau}(\mathcal{P}_{/\mathcal{S}})$, in particular, defines a local symmetry of the Lagrangian without even pulling it back to $\mathcal{P}_0$. That is, we have
\begin{equation}
    \delta_X\mathscr{L}=0,\quad\forall X\in\mathfrak{gau}(\mathcal{P}_{/\mathcal{S}})
\end{equation}
Since the Cartan curvature $F(\mathcal{A})$ is horizontal, this furthermore implies that the curvature contribution to the variations \eqref{eq:8.234} of the Cartan connection vanish in general iff $X\in\mathfrak{gau}(\mathcal{P}_{/\mathcal{S}})$. In this case, the variation of the connection takes the form of an ordinary infinitesimal gauge transformation.

On the other hand, this is not immediately the case if $X$ is horizontal and the curvature contributions do not vanish a priori. Let us briefly explain how these are treated in the Castellani-D'Auria-Fré approach to supergravity and how this approach may be related to the present formalism (for a detailed introduction to this fascinating subject see for instance \cite{DAuria:1982uck,Castellani:1991et,Castellani:2018zey} as well as \cite{Castellani:2014goa,Cremonini:2019aao,Catenacci:2018xsv} using the concept of \emph{integral forms}).

The horizontal vector fields can be subdivided into two different categories. In fact, note that the super Poincaré algebra splits into three $\mathcal{H}$-equivariant subspaces $\mathfrak{iso}(\mathbb{R}^{1,3|4})=\mathbb{R}^{1,3}\oplus\mathfrak{spin}^+(1,3)\oplus\Delta_{\mathbb{R}}$ so that one can decompose $H^{\infty}(\mathcal{S}\times\mathcal{P},\mathrm{Lie}(\mathcal{G}))^{\mathcal{H}}\cong\Gamma(\mathcal{E}_{/\mathcal{S}})$ according to
\begin{equation}
\mathcal{E}:=\mathcal{E}^{\mathbb{R}^{1,3}}\oplus\mathcal{E}^{\mathfrak{spin}}\oplus\mathcal{E}^{\Delta_{\mathbb{R}}}
\end{equation}
where, for instance, $\mathcal{E}^{\Delta_{\mathbb{R}}}:=\mathcal{P}\times_{\mathrm{Ad}}(\Lambda\otimes\Delta_{\mathbb{R}})$ denotes the associated super vector bundle corresponding to the real Majorana representation on $\Delta_{\mathbb{R}}$. Horizontal vector fields then correspond to sections of the bundle $\mathcal{E}^{\mathbb{R}^{1,3}}\oplus\mathcal{E}^{\Delta_{\mathbb{R}}}$. In the Castellani-D'Auria-Fré approach, horizontal vector fields corresponding to $\mathcal{E}^{\mathbb{R}^{1,3}}$ are typically referred to as infinitesimal spacetime translations while those corresponding to $\mathcal{E}^{\Delta_{\mathbb{R}}}$ are associated to supersymmetry transformations. In general, they do not provide local symmetries of the Lagrangian. In the Castellani-D'Auria-Fré approach, one then tries to resolve this by appropriately fixing the curvature contributions (or rather their pullback to $\mathcal{P}_0$) to the variations of the connection. These are typically referred to as the \emph{horizontality} and \emph{rheonomy conditions}. However, this cannot be done arbitrarily as, for instance, one has to ensure consistency with the Bianchi identity $D^{(\mathcal{A})}F(\mathcal{A})=0$.

To explain this in a bit more detail, following \cite{Castellani:2018zey}, note that the Lie derivative of the Lagrangian $L_{X}\mathscr{L}=\mathrm{d}(\iota_X\mathscr{L})+\iota_X\mathrm{d}\mathscr{L}$ picks up an additional exact form so that, after pulling back to $\mathcal{P}_0$, condition \eqref{eq:8.0.0} can be written in the equivalent form
\begin{equation}
    (\iota_X\mathrm{d}\mathscr{L})|_{\mathcal{P}_0}=\mathrm{d}\alpha'
    \label{eq:8.0.0.1}
\end{equation}
for some $\alpha'\in\Omega^3((\mathcal{P}_0)_{/\mathcal{S}})$. In order to compute the variation of the Lagrangian, one observes that its exterior derivative can be completely re-expressed in terms of the components \eqref{eq:torsion} of the super Cartan curvature. In fact, exploiting the manifest $\mathrm{Spin}^+(1,3)$-invariance of the Lagrangian, it follows that
\begin{align}
    \mathrm{d}\mathscr{L}=&F(\omega)^{IJ}\wedge D^{(\omega)}e^K\wedge e^L\epsilon_{IJKL}+
iD^{(\omega)}\bar{\psi}\wedge\gamma_{*}\gamma_ID^{(\omega)}\psi\wedge e^I-i\bar{\psi}\wedge\gamma_{*}\gamma_ID^{(\omega)}D^{(\omega)}\psi\wedge e^I\nonumber\\
&-i\bar{\psi}\wedge\gamma_{*}\gamma_ID^{(\omega)}\psi\wedge D^{(\omega)}e^I\nonumber\\
=&F(\mathcal{A})^{IJ}\wedge F(\mathcal{A})^K\wedge e^L\epsilon_{IJKL}+
i\bar{\rho}\wedge\gamma_{*}\gamma_I\rho\wedge e^I+\frac{1}{4}F(\omega)^{IJ}\wedge\bar{\psi}\wedge\gamma^K\psi\wedge e^L\epsilon_{IJKL}\nonumber\\
&-\frac{1}{4}F(\omega)^{JK}\wedge\bar{\psi}\wedge\gamma_{*}\gamma_{JK}\psi\wedge e^I-i\bar{\psi}\wedge\gamma_{*}\gamma_I\rho\wedge F(\mathcal{A})^I-i\bar{\rho}\wedge\gamma_{*}\gamma_I\psi\wedge\bar{\psi}\wedge\gamma^I\psi
 \label{eq:8.0.2}
\end{align}
where we set $\rho^{\alpha}:=F(\mathcal{A})^{\alpha}$ and used the relation $\bar{\psi}\wedge\gamma_{*}\gamma_I\rho=\bar{\rho}\wedge\gamma_{*}\gamma_I\psi$. Moreover, in the second equality, we made use of the Bianchi identity
\begin{equation}
    D^{(\omega)}D^{(\omega)}\psi=\kappa_{\mathbb{R}*}(F(\omega))\wedge\psi=\frac{1}{4}F(\omega)^{IJ}\wedge\gamma_{IJ}\psi
\end{equation}
Using the relation $\bar{\psi}\gamma_*\gamma_I\gamma_{JK}\psi=-\bar{\psi}\gamma_*\gamma_{JK}\gamma_I\psi$ which gives
\begin{equation}
    \bar{\psi}\wedge\gamma_*\gamma_I\gamma_{JK}\psi=\frac{1}{2}\bar{\psi}\wedge\gamma_*\{\gamma_I,\gamma_{JK}\}\psi=i\epsilon_{IJKL}\bar{\psi}\wedge\gamma^L\psi
\end{equation}
it follows that the third and forth term in the second equality of \eqref{eq:8.0.2} exactly cancel. Furthermore, due to the \emph{Fierz identity} $\gamma_I\psi\wedge\bar{\psi}\wedge\gamma^I\psi=0$, it follows that \eqref{eq:8.0.0.1} reduces to
\begin{align}
    \mathrm{d}\mathscr{L}=&F(\mathcal{A})^{IJ}\wedge F(\mathcal{A})^K\wedge e^L\epsilon_{IJKL}+
i\bar{\rho}\wedge\gamma_{*}\gamma_I\rho\wedge e^I-i\bar{\psi}\wedge\gamma_{*}\gamma_I\rho\wedge F(\mathcal{A})^I
 \label{eq:8.0.3}
\end{align}
Thus, indeed, the exterior derivative of the Lagrangian can be solely expressed in terms of the curvature of the super Cartan connection. 

Let us consider variations of the Lagrangian which correspond to supersymmetry transformations. Thus, let
$X\in\Gamma(\mathcal{E}^{\Delta_{\mathbb{R}}}_{/\mathcal{S}})_0$ be even smooth vector field and $\epsilon:=\braket{X|\mathcal{A}}$ be the corresponding $\mathrm{Spin}^+(1,3)$-equivariant function. The contraction of  \eqref{eq:8.0.3} with $X$ then takes the form
\begin{align}
    \iota_X\mathrm{d}\mathscr{L}=&\iota_X F(\mathcal{A})^{IJ}\wedge F(\mathcal{A})^K\wedge e^L\epsilon_{IJKL}+F(\mathcal{A})^{IJ}\wedge\iota_X F(\mathcal{A})^K\wedge e^L\epsilon_{IJKL}\nonumber\\
&+2i\bar{\rho}\wedge\gamma_{*}\gamma_I(\iota_X\rho)\wedge e^I
-i\bar{\epsilon}\wedge\gamma_{*}\gamma_I\rho\wedge F(\mathcal{A})^I+i\bar{\epsilon}\wedge\gamma_{*}\gamma_I(\iota_X\rho)\wedge F(\mathcal{A})^I\nonumber\\
&+i\bar{\psi}\wedge\gamma_{*}\gamma_I\rho\wedge \iota_X F(\mathcal{A})^I
\end{align}
Thus, when pulled back to the bosonic subbundle $\mathcal{P}_0$, we see that condition  \eqref{eq:8.0.0.1} for a local symmetry is satisfied if
\begin{equation}
    \iota_XF(\mathcal{A})^{I}=0,\quad\iota_XF(\mathcal{A})^{\alpha}=0,\quad\iota_XF(\mathcal{A})^{IJ}\wedge e^K\epsilon_{IJKL}=-i\bar{\rho}\gamma_*\gamma_L\epsilon
    \label{eq:8.0.4}
\end{equation}
where, here and in the following, the pullback to $\mathcal{P}_0$ is always implicitly assumed. As can be checked by direct computation, the last condition in \eqref{eq:8.0.4} is solved by
\begin{equation}
    \iota_XF(\mathcal{A})^{IJ}=-\frac{i}{4}\left(\epsilon^{IJKL}\bar{\rho}_{KL}\gamma_*\gamma_M\epsilon\,e^M+\epsilon^{KLM[I}\bar{\rho}_{KL}\gamma_*\gamma_M\epsilon\,e^{J]}\right)=:\bar{\theta}^{IJ}_K\epsilon\,e^K
    \label{eq:8.0.5}
\end{equation}
where we set $\rho=:\frac{1}{2}\rho_{IJ}e^I\wedge e^J$. Thus, to summarize, it follows that even smooth vector fields $X\in\Gamma(\mathcal{E}^{\Delta_{\mathbb{R}}}_{/\mathcal{S}})_0$ define local symmetries of the Lagrangian provided that their contraction with the super Cartan curvature satsifies the conditions \eqref{eq:8.0.4} and \eqref{eq:8.0.5} also called the \emph{rheonomy conditions}. Inserting these conditions into the general formula \eqref{eq:8.234}, it follows that the supersymmetry transformations of the individual components of the super Cartan connection take the form 
\begin{align}
\delta_{X}e^I=\frac{1}{2}\bar{\epsilon}\gamma^I\psi,\quad\delta_{X}\psi=D^{(\omega)}\epsilon\quad\text{and}\quad\delta_{X}\omega^{IJ}=\bar{\theta}^{IJ}_K\epsilon\,e^K
\end{align}
Note again that, by definition, $X\in\Gamma(\mathcal{E}^{\Delta_{\mathbb{R}}}_{/\mathcal{S}})_0$ implies $\braket{X|\omega}=0$, that is, $X$ is horizontal. Hence, in this framework, supersymmetry transformations have the interpretation in terms of super diffeomorphisms on the base supermanifold $\mathcal{M}$. Of course, one still needs to check whether the rheonomy conditions are in fact compatible with the Bianchi identity. As it turns out, this is indeed the case provided that the basic fields satisfy their equations of motion. Thus, in this framework, it follows that supersymmetry transformations can be interpreted in terms of super diffeomorphisms only when applied on solutions of the field equations \cite{Castellani:2018zey}. For general field configurations, this will no longer be the case. In this case, one needs to add additional fields, so-called \emph{auxilary fields}, to the theory. For more details on this subject, the interested reader may be referred to \cite{Castellani:2018zey} and references therein.
\begin{remark}
As discussed in \cite{Eder:2020erq}, in the Cartan geometric framework, there also exists another kind of variation of the Lagrangian arising from the lift of the Cartan connection to an Ehresmann connection on the associated bundle. In fact, using Prop. \ref{prop:2.27}, one can consider the $\mathcal{G}$-extension $\mathcal{P}[\mathcal{G}]_{/\mathcal{S}}:=(\mathcal{P}\times_{\mathcal{H}}\mathcal{G})_{/\mathcal{S}}$ with respect to which $\mathcal{P}_{/\mathcal{S}}$ defines a $\mathcal{H}$-reduction via the embedding
\begin{equation}
\hat{\iota}:\,\mathcal{P}_{/\mathcal{S}}\rightarrow\mathcal{P}[\mathcal{G}]_{/\mathcal{S}},\,p\mapsto[p,e]
\end{equation}
According to Prop. \ref{prop:4.6}, the Cartan connection can be lifted to an Ehresmann connection $\hat{\mathcal{A}}$ on the associated bundle in a unique way such that $\hat{\iota}^*\hat{\mathcal{A}}=\mathcal{A}$. Consider then a gauge transformation $X\in\mathfrak{gau}(\mathcal{P}[\mathcal{G}]_{/\mathcal{S}})$ such that $\epsilon:=\braket{X|\hat{\mathcal{A}}}\in H^{\infty}(\mathcal{S}\times\mathcal{P}[\mathcal{G}],\Lambda\otimes\Delta_{\mathbb{R}})^{\mathcal{G}}$. It then follows that the variations $\delta_X\hat{\mathcal{A}}:=L_X\hat{\mathcal{A}}$ of the individual components (denoted by the same symbol) of the connection take the form 
\begin{align}
\delta_{X}e^I=\frac{1}{2}\bar{\epsilon}\gamma^I\psi,\quad\delta_{X}\psi=D^{(\omega)}\epsilon\quad\text{and}\quad\delta_{X}\omega^{IJ}=0
\end{align}
In contrast to the previous considerations, the curvature now does not enter to the field variations as the infinitesimal automorphisms are vertical, i.e., they are gauge transformations, not super diffeomorphisms. When pulled back to $\mathcal{P}_{/\mathcal{S}}$, this yields precisely the supersymmetry transformation for the supervielbein $E$ as found in the previous prescription while the variation of the spin connection is altered. It follows that these variations indeed lead to local symmetries of the theory provided $\omega$ satisfies its field equations \cite{Freedman:1976xh,vanNieuwenhuizen:2004rh} which equivalent to the supertorsion constraint $F(\mathcal{A})^I=0$.
\end{remark}
\subsubsection{On the role of the parametrizing supermanifold}\label{sec:Param}
As we have frequently observed in the previous sections, among other things, studying parametrized supermanifolds is mandatory in order to incorporate nontrivial fermionic degrees of freedom on the body of a supermanifold. In fact, for the action $S(\mathcal{A})$ corresponding to Poincaré supergravity, one has
\begin{equation}
S(\mathcal{A})=\int_{M}s^*\mathscr{L}\in H^{\infty}(\mathcal{S})_0
\end{equation}
that is, the action defines an even smooth map on the underlying parametrizing supermanifold $\mathcal{S}$. If $\mathcal{S}$ would be trivial, this then implies that the fermionic fields contained in the action would simply drop off, that is, the action reduces the standard action of ordinary Einstein gravity.

It follows from the definition of super connection forms on parametrized super fiber bundles that all physical quantities transform covariantly under change of parametrization. More precisely, let $\lambda:\,\mathcal{S}'\rightarrow\mathcal{S}$ be a change of parametrization. Then, the super connection form $\mathcal{A}$ transforms via $\mathcal{A}\rightarrow\lambda^*\mathcal{A}$ so that for the action of the theory it follows 
\begin{equation}
    S(\mathcal{A})\rightarrow S(\lambda^*\mathcal{A})=\lambda^*S(\mathcal{A})\in H^{\infty}(\mathcal{S}')_0
\end{equation}
This may be regarded as the mathematical realization of the physical requirement that the physical theory should not depend on a particular choice of a parametrizing supermanifold.

In what follows, we want to further analyze the structure of field configuration space of Poincaré supergravity and work out explicitly the relation to pAQFT \cite{Rejzner:2011au,Rejzner:2016hdj}. However, let us emphasize that a full consistent treatment requires the study of infinite dimensional supermanifolds. Hence, in the following, we will only sketch the main ideas behind such a link. A rigorous account with detailed proofs will be reported elsewhere.

Choosing a reference connection, the configuration space of the theory can be identified with $\Omega^1(\mathcal{M}_{/\mathcal{S}},\mathrm{Ad}(\mathcal{P_{/\mathcal{S}}}))_0$. Since, by the \emph{rheonomy principle} \cite{Castellani:1991et}, we are actually only dealing with the pullback of super connections on the body of the supermanifold, in what follows, it suffices restrict on forms defined on the body $M:=\mathbf{B}(\mathcal{M})$ so that, for a specific choice of parametrization $\mathcal{S}$, the configuration space $\mathcal{C}$ of the theory can be taken to be
\begin{equation}
    \mathcal{C}:=(H^{\infty}(\mathcal{S})\otimes\Omega^1(M,P\times_{\mathrm{Ad}}\mathfrak{g}))_0=H^{\infty}(\mathcal{S})_0\otimes F_0\oplus H^{\infty}(\mathcal{S})_1\otimes F_1
\end{equation}
where $F$ denotes the infinite dimensional $\mathbb{Z}_2$-graded vector space given by
\begin{equation}
    F:=\Omega^1(M,E)=\Omega^1(M,E_0)\oplus\Omega^1(M,E_1)
\end{equation}
Here, $E:=P\times_{\mathrm{Ad}}\mathfrak{g}$ is the associated bundle which itself carries the structure of a super vector space with even and odd part respectively given by  $E_0=P\times_{\mathrm{Ad}}\mathfrak{g}_0$ and $E_1=P\times_{\mathrm{Ad}}\mathfrak{g}_1$ with $E_1$ the spinor bundle of Majorana fermions. For $\Phi\in\mathcal{C}$ it follows that, for any $s\in\mathcal{S}$, $\Phi(s)$ defines an element of the superspace
\begin{equation}
    \overline{F}(\Lambda):=(F\otimes\Lambda)_0=F_0\otimes\Lambda_0\oplus F_1\otimes\Lambda_1
\end{equation}
where $\Lambda$ is the Grassmann algebra over which $\mathcal{S}$ is modeled as a $H^{\infty}$ supermanifold. Hence, in this sense, it follows that one can identify the configuration space $\mathcal{C}$ with the space of $H^{\infty}$-smooth functions on $\mathcal{S}$ with values in the superspace $\overline{F}(\Lambda)$.

So far, our consideration was based on the choice of a particular parametrizing supermanifold $\mathcal{S}$. However, as explained at the beginning of this section, by definition of super connection forms defined on parametrized principal super fiber bundles, all the above constructions behave covariantly under change of parametrization. Due to this property, one may ask the question, whether there exists a particular choice of a parametrization $\EuScript{S}$, possibly infinite dimensional, such that any field configuration associated to some finite $\mathcal{S}$ can be obtained via pullback. Hence, $\EuScript{S}$ should be suitably large enough to encode all the configurations associated to $\mathcal{S}$-parametrized field theories. 

In fact, this idea has been first studied by Schmitt in \cite{Schmitt:1996hp} where, in this context, a theory of infinite dimensional \emph{analytic supermanifolds} has been developed. In the following, we want to sketch this idea using the Molotkov-Sachse approach to supermanifold theory \cite{Mol:10,Sac:08}, more precisely locally convex supermanifolds as considered in \cite{Schuett:2018}, as this seems to be the more established approach to this subject. In fact, as already outlined in the introduction, the Moltokov-Sachse approach can be regarded as a generalization of the correspondence between finite dimensional algebro-geometric and Rogers-De Witt supermanifolds via the functor of points prescription to the case of infinite dimensions. This is actually one of the reasons why we have focused on the Rogers-De Witt approach in this paper. 

To find a suitable candidate for $\EuScript{S}$, note that, for any $\Lambda\in\mathbf{Ob}(\mathbf{Gr})$, the superspace $\overline{F}(\Lambda)$ can be endowed with the structure of a locally convex space by choosing a particular locally convex topology on $F$ and extending it to $\overline{F}(\Lambda)$ using the product topology. In this way, the assignment $\mathbf{Ob}(\mathbf{Gr})\ni\Lambda\mapsto\overline{F}(\Lambda)$ induces a functor 
\begin{equation}
    (\overline{F}:\,\mathbf{Gr}\rightarrow\mathbf{Top})\in\mathbf{Top}^{\mathbf{Gr}}
\end{equation}
from the category of superpoints to the category $\mathbf{Top}$ of topological spaces. Hence, $\overline{F}$ defines a supermanifold (more precisely, a superdomain) in the sense of Molotkov-Sachse \cite{Mol:10,Sac:08,Schuett:2018}. 

Next, note that the parametrizing supermanifold $\mathcal{S}$, as a $H^{\infty}$ supermanifold, can be regarded as a $\Lambda$-point $\mathcal{S}\equiv\overline{\mathcal{S}}(\Lambda)$ of a particular algebro-geometric supermanifold $\overline{\mathcal{S}}$. On the other hand, via the functor of points prescription, $\overline{\mathcal{S}}$ itself induces a functor $\overline{\mathcal{S}}:\,\mathbf{Gr}\rightarrow\mathbf{Top}$ and thus yields a Molotkov-Sachse supermanifold (for a proof see e.g. \cite{Sac:08}). Hence, if we do not focus on a particular Grassmann algebra, according to the discussion above, we may identify the configuration space $\mathcal{C}$ with $(\mathcal{SC}^{\infty}(\overline{\mathcal{S}})\hat{\otimes} F)_0$, i.e., smooth functions (in the sense of Molotkov-Sachse) on $\overline{\mathcal{S}}$ with values in $\overline{F}$. We now make the important assumption that this space can be identified with $\mathcal{SC}^{\infty}(\overline{\mathcal{S}},\overline{F})$, that is, the space smooth maps between the infinite dimensional supermanifolds $\overline{\mathcal{S}}$ and $\overline{F}$. Note that this would be trivially the case, if $F$ were finite dimensional. 

Hence, in this case, it follows that any $\Phi\in\mathcal{C}$ can be identified with a morphism $\mu_{\Phi}:\,\overline{\mathcal{S}}\rightarrow\overline{F}$. But, note that $\mu_{\Phi}\equiv\mu_{\Phi}^*(x_{\overline{F}})$ with $x_{\overline{F}}:=\mathrm{id}:\,\overline{F}\rightarrow\overline{F}$ the identity morphism. Hence, any field configuration $\Phi$ associated to a particular parametrizing supermanifold $\overline{\mathcal{S}}$ can be obtained via pullback of $x_{\overline{F}}$ w.r.t. a (unique) morphism $\mu_{\Phi}:\,\overline{\mathcal{S}}\rightarrow\overline{F}$. As in \cite{Schmitt:1996hp}, we may therefore call $x_{\overline{F}}$ a \emph{fundamental coordinate} and the morphism $\mu_{\Phi}$ the \emph{classifying morphism} of the field configuration $\Phi$. Hence, to summarize, based on these observations, this suggests setting $\EuScript{S}:=\overline{F}$. With slight abuse of terminology, in what follows, we will also often refer to $\EuScript{S}$ as a configuration space.

As shown in \cite{Schuett:2018}, any smooth function $f\in\mathcal{SC}^{\infty}(\overline{F})\equiv\mathcal{SC}^{\infty}(\overline{F},\overline{\mathbb{R}^{1|1}})$ on $\overline{F}$ can be uniquely described in terms of a so-called \emph{skeleton} $(f_n)_n$ consisting of smooth maps between locally convex spaces $f_n:\,F_0\rightarrow\mathcal{A}\mathrm{lt}^n(F_1,\mathbb{R}^{1|1})$, $n\in\mathbb{N}_0$, from $F_0$ to $\mathcal{A}\mathrm{lt}^n(F_1,\mathbb{R}^{1|1})$, i.e., the space of graded symmetric $n$-multilinear smooth functionals on $F_1$. Thus, we can make the identification 
\begin{equation}
    \mathcal{SC}^{\infty}(\overline{F},\overline{\mathbb{R}^{1|1}})\cong C^{\infty}(F_0,\mathcal{A}\mathrm{lt}(F_1,\mathbb{R}^{1|1}))
    \label{eq:MolSac}
\end{equation}
It is interesting to note that \eqref{eq:MolSac} is precisely the field configuration space as considered in \cite{Rejzner:2011au,Rejzner:2016hdj} in context of pAQFT. There, among other things, this space has been considered in order to (classically) consistently incorporate the anticommutative nature of fermionic fields. As we see, here, it arises quite naturally studying relative supermanifolds.

To make this link to the description of fermionic fields in pAQFT even more precise, let us choose a mutual local trivialization neighborhood of $M$ and the vector bundle $E$. Let $(e_{\underline{A}})_{\underline{A}}$ with $\underline{A}\in\{I,\alpha\}$ be a corresponding homogeneous basis of local sections of $E$. Decomposing the fundamental coordinate $x_{\overline{F}}$ w.r.t. the homogeneous basis $(e_{\underline{A}})_{\underline{A}}$ and evaluating on coordinate differentials $\partial_{\mu}|_p$ at any point $p\in M$, it follows that the odd components induce smooth functionals $\Psi_{\mu}^{\alpha}(p)\in(F_1)'$, with $(F_1)'=\Gamma'(T^*M\otimes E_1)$ the topological dual of $F_1$, via $\Psi_{\mu}^{\alpha}(p):=\mathrm{pr}_{\alpha}\circ\braket{\partial_{\mu}|_{p}| x_{\overline{F}}(\cdot)}$ such that
\begin{equation}
    \Psi_{\mu}^{\alpha}(p):\,F_1\rightarrow\mathbb{R}^{0|1}\cong\mathbb{R},\,\psi\mapsto\psi^{\alpha}_{\mu}(p)
\end{equation}
Thus, it follows that, in this framework, fermionic fields are described in terms of odd evaluation functionals on the configuration space. This is exactly the interpretation of (classical) anticommutative fermionic fields in pAQFT \cite{Rejzner:2011au,Rejzner:2016hdj}. Given two fermionic fields $\Psi_{\mu}^{\alpha}(p)$ and $\Psi_{\nu}^{\beta}(p)$ at the same point $p\in M$, their product is defined via the ordinary wedge product yielding the bilinear map
\begin{equation}
    \Psi_{\mu}^{\alpha}(p)\Psi_{\nu}^{\beta}(p)\equiv\Psi_{\mu}^{\alpha}(p)\wedge\Psi_{\nu}^{\beta}(p)=-\Psi_{\nu}^{\beta}(p)\wedge\Psi_{\mu}^{\alpha}(p)
\end{equation}
so that, in this sense, the fermionic fields are indeed anticommutative.

\subsection{Killing vector fields and Killing spinors}
We want to describe global symmetries of a super Cartan geometry. Therefore, let $(\pi_{\mathcal{S}}:\,\mathcal{P}_{/\mathcal{S}}\rightarrow\mathcal{M}_{/\mathcal{S}},\mathcal{A})$ be a metric and reductive super Cartan geometry modeled on a super Klein geometry $(\mathcal{G},\mathcal{H})$ with super Cartan connection $\mathcal{A}$ and smooth $\mathrm{Ad}(\mathcal{H})$-invariant super metric $\mathscr{S}$ defined on $\mathfrak{g}/\mathfrak{h}$. For sake of simplicity, let us assume that $\mathcal{S}$ is a \emph{superpoint}, i.e., the body just consists of a single point $\mathbf{B}(\mathcal{S})=\{*\}$.\\
From a physical perspective, the necessity of the choice of a nontrivial parametrizing supermanifold $\mathcal{S}$ is based on the requirement of non-vanishing (anticommuting) fermionic degrees of freedom on the body of a supermanifold. The structure of the underlying base supermanifold $\mathcal{M}$ in the Cartan geometric framework, however, is encoded in the super soldering form $E=\mathrm{pr}_{\mathfrak{g}/\mathfrak{h}}\circ\mathcal{A}$ when restricted on bosonic configurations. In fact, let $\mathscr{H}=\mathrm{ker}(\omega|_{\mathbf{B}(\mathcal{S})})$ be the horizontal distribution induced by pullback of the Ehresmann connection $\omega:=\mathfrak{h}\circ\mathcal{A}$ to $\mathbf{B}(\mathcal{S})$. By the super Cartan condition, it then follows that the restriction $\theta:=E|_{\mathbf{B}(\mathcal{S})}\in\Omega^1(\mathcal{P},\mathfrak{g}/\mathfrak{h})$ yields an isomorphism
\begin{equation}
\theta_p:\,\mathscr{H}_p\rightarrow\Lambda\otimes\mathfrak{g}/\mathfrak{h}
\end{equation}
of free super $\Lambda$-modules at any $p\in\mathcal{P}$. Horizontal vector fields on $\mathcal{P}$ are in one-to-one correspondence with smooth vector fields on the base supermanifold $\mathcal{M}$. If $X\in\Gamma(T\mathcal{M})$ is a smooth vector field and $X^*$ its horizontal lift, then $X^*$, in particular, defines an infinitesimal bundle automorphism. By $\mathcal{H}$-equivariance, it follows that $\braket{X^*|\theta}$ defines a $\mathcal{H}$-equivariant smooth function on $\mathcal{M}$ with values in $\Lambda\otimes\mathfrak{g}/\mathfrak{h}$. Thus, to summarize, we have an isomorphism
\begin{equation}
\theta:\,\Gamma(T\mathcal{M})\stackrel{\sim}{\rightarrow} H^{\infty}(\mathcal{P},\Lambda\otimes\mathfrak{g}/\mathfrak{h})^{\mathcal{H}},\,X\rightarrow\braket{X^*|\theta}
\end{equation}
Since the super Cartan geometry is metric, according to Corollary \ref{Corollary:4.9}, the super soldering form $\theta$ induces a super metric $g$ on $\mathcal{M}$ which, w.r.t. any local section $s\in\Gamma(\mathcal{P})$ of the principal super fiber bundle, takes the form
\begin{equation}
g=(-1)^{|e_i||e_j|}\mathscr{S}_{ij}(s^*\theta)^i\otimes(s^*\theta)^j
\label{eq:8.241}
\end{equation}
where we have chosen a homogeneous basis $(e_i)_i$ of $\mathfrak{g}/\mathfrak{h}$ and set $\mathscr{S}_{ij}:=\mathscr{S}(e_i,e_j)$. With respect to this metric, the base $(\mathcal{M},g)$ has the structure of a \emph{super Riemannian manifold}. We now want to introduce the super analog of an infinitesimal isometry on a super Riemannian manifold.
\begin{definition}
Let $(\mathcal{M},g)$ be a super Riemannian manifold. A smooth vector field $X\in\Gamma(T\mathcal{M})$ is called a \emph{Killing vector field} if $g$ is constant along the flow generated by $X$, that is, 
\begin{equation}
L_Xg=0
\end{equation}
Since $L_{[X,Y]}=[L_X,L_Y]$ for any smooth vector fields $X$ and $Y$, Killing vector fields form a super Lie subalgebra $\mathfrak{k}(\mathcal{M})$ of $\Gamma(T\mathcal{M})$.
\end{definition}
In the supergeometric framework, Killing vector fields $X\in\mathfrak{k}(\mathcal{M})$ fall into two categories depending on their grading. In context of the super Cartan geometry, we will call an odd Killing vector field $X\in\mathfrak{k}(\mathcal{\mathcal{M}})_1$ a \emph{Killing spinor}. To explain its name, let us consider for example the super Cartan geometry corresponding to $\mathcal{N}=1$ Poincaré supergravity and let $X\in\mathfrak{k}(\mathcal{\mathcal{M}})_1$ be a odd Killing vector field. Via the super soldering form, this then corresponds to an odd smooth function 
\begin{equation}
\epsilon:=H^{\infty}(\mathcal{P},\Lambda\otimes\mathfrak{t})^{\mathrm{Spin}^+(1,3)}_1\cong\Gamma(\mathcal{P}\times_{\mathrm{Spin}^+(1,3)}\Lambda\otimes\mathfrak{t})_1
\end{equation}
with $\mathfrak{t}=\mathbb{R}^{1,3}\oplus\Delta_{\mathbb{R}}$ the super Lie algebra of the super translation group $\mathcal{T}^{1,3|4}$. If one restricts to the bosonic sub supermanifold $\mathcal{M}_0$ this then implies that $\epsilon$ defines a section of the associated spinor bundle, that is, it defines a Majorana spinor. Thus, Killing spinors of the super Riemannian geometry, induced by the metric reductive Cartan geometry, are associated to Majorana spinors.

Note that odd Killing vector fields on a supermanifold vanish when pulled back to the underlying bosonic supermanifold $\mathcal{M}_0$. Thus, they define trivial infinitesimal isometries of the bosonic Riemannian manifold $(\mathcal{M}_0,g_0)$ with $g_0$ the even part of the super metric $g$. Nevertheless, existence of Killing spinors, in general, may still impose strong restrictions on the structure of the underlying bosonic geometry. In fact, given two Killing spinors $X,Y\in\mathfrak{k}(\mathcal{M})_1$ their (graded) commutator $[X,Y]\in\mathfrak{k}(\mathcal{M})_0$ defines an even Killing vector field which can be non-vanishing when pulled back the bosonic submanifold.

To illustrate this, let us consider the homogeneous super Cartan geometry $(\pi:\,\mathcal{G}\rightarrow\mathcal{G}/\mathcal{H},\theta_{\mathrm{MC}})$ given by the super Klein geometry $(\mathcal{G},\mathcal{H})=(\mathrm{OSp}(1|4),\mathrm{Spin}^+(1,3))$ corresponding to the homogeneous supermanifold $\mathcal{M}:=\mathcal{G}/\mathcal{H}$ also called the \emph{super anti-de Sitter space} due to $\mathbf{B}(\mathcal{M})\cong\mathrm{Sp}(4)/\mathrm{Spin}^+(1,3)\cong\mathrm{AdS}_4$ (see e.g. \cite{Nicolai:1984hb,Freedman:1983na,Wipf:2016,Cheng:2012} for a definition of the orthosymplectic Lie superalgebra as well as \cite{Eder:2020erq} for our choice of conventions). Again, $\theta_{\mathrm{MC}}$ denotes the Maurer-Cartan form on $\mathcal{G}$. On $\mathfrak{osp}(1|4)$, we can define a smooth bi-invariant super metric $\mathscr{S}$ induced by the \emph{supertrace} $\mathrm{str}$ on $\mathrm{Mat}(1|4,\Lambda^{\mathbb{C}})$ setting    
\begin{equation}
\mathscr{S}(X,Y):=-\mathrm{str}(X\cdot Y),\quad\forall X,Y\in\mathrm{Lie}(\mathcal{G})
\end{equation}  
This yields an orthogonal decomposition of $\mathfrak{osp}(1|4)=\mathbb{R}^{1,3}\oplus\mathfrak{spin}^+(1,3)\oplus\Delta_{\mathbb{R}}$ into $\mathrm{Ad}(\mathcal{H})$-invariant subspaces generated by momenta and Lorentz transformations $P_I$ and $M_{IJ}$, respectively, as well as four Majorana charges $Q_{\alpha}$ such that $\mathscr{S}(P_I,P_J)=\frac{1}{L^2}\eta_{IJ}$ and $\mathscr{S}(Q_{\alpha},Q_{\beta})=\frac{1}{L}C_{\alpha\beta}$ where $L$ is the so-called anti-de Sitter radius. Thus, on $\mathfrak{g}/\mathfrak{h}=\mathbb{R}^{1,3}\oplus\Delta_{\mathbb{R}}$, $\mathscr{S}$ takes the form
\begin{equation}
\mathscr{S}=\begin{pmatrix}
	\frac{1}{L^2}\eta & 0\\
	0 & \frac{1}{L}C
\end{pmatrix}
\end{equation}
Let furthermore,
\begin{equation}
E:=\mathrm{pr}_{\mathfrak{g}/\mathfrak{h}}\circ\theta_{\mathrm{MC}}=:e^IP_I+\psi^{\alpha}Q_{\alpha}
\end{equation}
be the super soldering form which induces a super metric $g$ on $\mathcal{G}/\mathcal{H}$ via \eqref{eq:8.241}. According to Definition \ref{Prop:8.3}, an infinitesimal automorphism $X\in\mathfrak{aut}(\mathcal{G})$ needs to satisfy
\begin{equation}
R_{h*}X=X,\quad\forall h\in\mathcal{H}
\end{equation}
Thus, in particular, it follows that infinitesimal automorphisms are provided by right-invariant vector fields on $\mathcal{G}$. In fact, it follows that right-invariant vector fields are Killing vector fields of the induced super Riemannian geometry $(\mathcal{M},g)$. To see this, note first that left- and right-invariant vector fields commute. In fact, by the algebraic Definition \eqref{eq:Appendix:262}, it follows for, respectively, homogeneous right- and left-invariant vector fields $X$ and $Y$ on $\mathcal{G}$ that 
\begin{equation}
Y\circ X=\mathds{1}\otimes Y_e\circ\mu^*\circ X=\mathds{1}\otimes Y_e\circ X\otimes\mathds{1}\circ\mu^*=(-1)^{|X||Y|}X\circ\mathds{1}\otimes Y_e\circ\mu^*=(-1)^{|X||Y|}X\circ Y
\end{equation}
and therefore $[X,Y]=0$. Thus, this yields
\begin{equation}
\braket{Y|L_XE}=(-1)^{|X||Y|}(X\braket{Y|E}-\braket{[X,Y]|E})=(-1)^{|X||Y|}X\braket{Y|E}=0
\end{equation}  
as $\braket{Y|E}$ is constant by left-invariance. But, since, at any $g\in\mathcal{G}$, the left-invariant vector yield a homogeneous basis of $T_g\mathcal{G}$, this implies $L_XE=0$ so that, by \eqref{eq:8.241}, $X$ is indeed a Killing vector field. According to our definition above, odd right-invariant vector fields define Killing spinors which can be identified with $\mathfrak{g}_1=\Delta_{\mathbb{R}}$, that is, they are Majorana spinors. Let $\epsilon:=\epsilon^{\alpha}Q_{\alpha}\in\mathfrak{g}_1$ be such a spinor. Using \eqref{eq:8.234} and $L_XE=0$ for any right-invariant vector field $X$, it then follows in particular (note that the Cartan curvature vanishes by the Maurer-Cartan equation) 
\begin{equation}
L_{\epsilon^R}\psi=D^{(\omega)}\epsilon'-\frac{1}{2L}e^I\gamma_I\epsilon'=0
\label{eq:8.250}
\end{equation}
where we set $\epsilon':=\iota_{\epsilon^R}E$. Spinors satisfying an equation of the form \eqref{eq:8.250} are called \emph{twistor spinors} \cite{Alekseevsky:1997ri}. Thus, we see that odd Killing vector fields of the super Riemannian manifold $(\mathcal{M},g)$ correspond to twistor spinors.\\
Next, let $\epsilon,\eta\in\mathfrak{g}_1$ be two Majorana spinors and $\epsilon^R,\eta^R$ be the corresponding right-invariant vector fields. Since the commutator of two right-invariant vector fields is again right-invariant, the bilinear $K^*:=[\epsilon^R,\eta^R]$ again defines a Killing vector field. Using, $[\epsilon^R,\eta^R]=-[\epsilon,\eta]^R$, it follows
\begin{equation}
K^*=-\bar{\epsilon}\gamma^I\eta P^R_I-\frac{1}{4L}\bar{\epsilon}\gamma^{IJ}\eta M_{IJ}^R
\end{equation}   
As $K^*$ is purely bosonic, its pushforward $K:=\pi_{*}K^*$ defines, in general, a non-vanishing Killing vector field for the bosonic semi Riemannian manifold $(\mathcal{M}_0,g_0)$, i.e., ordinary $D=4$ AdS spacetime. 
\begin{remark}
Much more generally, in the algebraic framework, in \cite{Alekseevsky:1997ri}, it was shown that for the split supermanifold $\mathbf{S}(E_{\mathbb{R}},M)$ associated to a Majorana bundle $E_{\mathbb{R}}\times M\rightarrow M$ with $M$ spin, odd Killing vector fields of the corresponding super Riemannian geometry precisely correspond to twistor spinors. We suspect that any super Cartan geometry corresponding to Poincaré supergravity is of this form. More precisely, the principal super fiber bundle is of the form $\mathcal{P}\rightarrow\mathbf{S}(E_{\mathbb{R}},M)$ with $\mathcal{P}$ a spin-reduction of the frame bundle $\mathscr{F}(\mathbf{S}(E_{\mathbb{R}},M))$. This is supported by the fact that the super soldering form induces frame fields on the base supermanifold which are of the form as stated in \cite{Alekseevsky:1997ri}. 
\end{remark}
\begin{remark}
In physics, Killing spinors appear by studying consistency conditions for the bosonic background of solutions of supergravity theories. For instance, let us consider a super Cartan geometry corresponding to $\mathcal{N}=1$, $D=4$ AdS supergravity. The super Cartan connection is again of the form \eqref{eq:8.223}. The supersymmetry transformations of the individual field components are given by 
\begin{align}
\delta_{\epsilon}e^I=\frac{1}{2}\bar{\epsilon}\gamma^I\psi,\quad\delta_{\epsilon}\psi=D^{(\omega)}\epsilon-\frac{1}{2L}e^I\gamma_I\epsilon\,\text{ and }\,\delta_{\epsilon}\omega^{IJ}=\frac{1}{2L}\bar{\epsilon}\gamma^{IJ}\psi
\label{eq:8.252}
\end{align}
One may then be interested in the structure of the bosonic background of a general solution of the SUGRA field equations. As argued in section \ref{sec:Param}, the parametrizing supermanifold $\mathcal{S}$ may be regarded as a field configuration space. Bosonic field configurations taking values in the ordinary complex numbers are then encoded in the body $\mathbf{B}(\mathcal{S})$. But, if one then pulls back the field variations \eqref{eq:8.252} to the underlying spacetime manifold $\mathbf{B}(\mathcal{M})$, the fermionic fields, as being anticommutative, simply vanish. Hence, if the bosonic background solutions are required to be consistent with supersymmetry, there has to exist a spinor field $\epsilon$ such that
\begin{equation}
D^{(\omega)}\epsilon-\frac{1}{2L}e^I\gamma_I\epsilon=0
\label{eq:8.253}
\end{equation}
Thus, again, this leads back to the Killing spinor equation.   
\end{remark}
\section{Summary}
In this paper, we have studied the theory of super fiber bundles, in particular, principal super fiber bundles and super connection forms defined on them. We therefore studied these objects mainly in the Rogers-DeWitt category extending the seminal work of Tuynman \cite{Tuynman:2004} to relative supermanifolds defining objects in enriched categories. Moreover, we established a link to the theory of principal super fiber bundles and super connection forms in the algebro-geometric approach \cite{Stavracou:1996qb} and showed that both approaches are in fact equivalent.\\
We then also studied the parallel transport map correspoding to principal connections. A generic issue in both the algebro-geometric and concrete approach is the lack of (anticommutative) fermionic degrees of freedom on the body of a supermanifold. From a mathematical point of view, this implies that the parallel transport map cannot be used to compare points on different fibers of the bundle in contrast to the classical theory. A resolution is given studying relative supermanifolds as introduced in \cite{Hack:2015vna} whose underlying idea was first proposed already by Schmitt in \cite{Schmitt:1996hp} and which is rooted in the Molotkov-Sachse approach to supermanifold theory. In this paper, a rigorous mathematical account on this subject was given. In particular, we studied super connection forms on relative principal super fiber bundles and, in this context, also defined super Cartan geometries and super Cartan connections as introduced in \cite{Eder:2020erq}. Finally, the parallel transport was constructed in this relative category. Moreover, the corresponding parallel transport map on associated vector bundles was considered. In this context, among other things, we also established a link to similar constructions in the algebraic approach as studied for instance by Dumitrescu \cite{Florin:2008} and Groeger \cite{Groeger:2013aja}.\\
In the last section, we gave some concrete examples demonstrating the relevance of supermanifold techniques within mathematical physics. We therefore considered $\mathcal{N}=1$, $D=4$ supergravity as a metric reductive super Cartan geometry and analyzed local symmetries of this model. In this context, we also briefly sketched a possible embedding of Castellani-D'Auria-Fré approach to supergravity \cite{DAuria:1982uck,Castellani:1991et,Castellani:2018zey,Castellani:2014goa,Cremonini:2019aao,Catenacci:2018xsv} into the present formalism. Finally, Killing vector fields of super Riemannian manifolds were defined arising of metric reductive Cartan geometries. In this context, odd Killing vector fields were identified with Killing spinors typically arising as consistency conditions for the bosonic background of solutions of supergravity with supersymmetry.\\
There are many possible and interesting extensions of the present formalism, both from the mathematical and physical perspective. On the one hand, it would be interesting to generalize the formalism to include higher dimensional supergravity theories as well as extended supersymmetries. Higher dimensional supergravity theories typically involve higher gauge fields. Connection forms on higher principal bundles are studied for instance in \cite{Baez:2010ya,Schreiber:2013pra}. It would be very interesting to see how these approaches can be related. \\
On the other hand, there is quite some recent interest in the description of boundary charges \cite{DePaoli:2018erh,Freidel:2020xyx}. This has been addressed for instance in context of AdS supergravity in \cite{Aneesh:2020fcr}. There, among other things, it was found that a consistent treatment of (supersymmetric) boundary charges may be possible in the context of the superspace approach to supergravity. It would therefore be very interesting to generalize the Iyer-Wald's Noether charge formalism \cite{Iyer:1994ys} to supergravity which, in particular, explicitly takes into account the underlying supersymmetry of the theory. This may be achieved generalizing the work of Prahbu \cite{Prabhu:2015vua} to field theories defined on principal super fiber bundles.

\subsection*{Acknowledgments}
I thank Hanno Sahlmann for helpful discussions at various stages of this work. I also thank an anonymous referee for helpful comments and suggestions which helped to improve the manuscript. I thank the German Academic Scholarship Foundation (Studienstiftung des Deutschen Volkes) for financial support.
 
\subsection*{Data Availability}
The data that support the findings of this study are available from the corresponding author
upon reasonable request.

\appendix

\section{Supermanifolds and super Lie groups}\label{appendix:Super}
In this section, we want to briefly review the basic definition of $H^{\infty}$-supermanifolds in the Rogers-DeWitt approach in order to fix some terminology used in the main text (standard references include \cite{DeWitt:1984,Rogers:2007,Tuynman:2004}).
\begin{definition}
Let $\Lambda$ be a Grassmann algebra (finite or infinite dimensional) and $\Lambda^{m,n}:=\Lambda_0^m\times\Lambda_1^n$  be the \emph{superdomain} of dimension $(m,n)$ for any $m,n\in\mathbb{N}_0$. On $\Lambda^{m,n}$ one has the \emph{body map} given by the projection $\epsilon_{m,n}:\,\Lambda^{m,n}\rightarrow\mathbb{R}^m$ onto the real subspace $\mathbb{R}^m$. We equip $\Lambda^{m,n}$ with the \emph{DeWitt-topology} defined as the coarsest topology such that the body map is continuous.\\
For any open subset $U\subseteq\mathbb{R}^m$, a function $f:\,\epsilon_{m,n}^{-1}(U)\rightarrow\Lambda$ is called \emph{($H^{\infty}$-)smooth}, if there exists ordinary real smooth functions $f_I\in C^{\infty}(U)$, $I\in M_n$ a multi-index of length $\leq n$, such that
\begin{equation}
f(x,\theta)=\sum_I\mathbf{G}(f_I)(x)\theta^I
\end{equation}
$\forall(x,\theta)\in\epsilon_{m,n}^{-1}(U)$ where $\mathbf{G}(f_I)$ is the so-called \emph{Grassmann extension} of $f_I$ defined as
\begin{equation}
\mathbf{G}(f)(x):=\sum_{J}{\frac{1}{J!}\partial_J f_I(\epsilon_{m,n}(x))s(x)^I}
\end{equation}
with $s(x):=x-\epsilon_{m,n}(x)$ the \emph{soul} of $x\in\Lambda^m_0$.
\end{definition}

\begin{remark}
Following \cite{Tuynman:2004}, a topological space $M$ together with a $C^{\infty}$-smooth structure will be called a \emph{proto $C^{\infty}$-manifold}. That is, a proto $C^{\infty}$-manifold is just an ordinary smooth manifold without requiring the underlying topological space to be Hausdorff and second countable.
\end{remark}
\begin{definition}
Let $\mathcal{M}$ be a topological space. A local \emph{superchart} on $\mathcal{M}$ is a pair $(U,\phi_U)$ consisting of an open subset $U\subseteq\mathcal{M}$ as well as a homeomorphism $\phi_U:\,U\rightarrow\phi_U(U)\subseteq\Lambda^{m,n}$ onto an open subset of the superdomain $\Lambda^{m,n}$.\\
A $(m,n)$-dimensional $H^{\infty}$-smooth atlas on $\mathcal{M}$ is a family $\{(U_{\alpha},\phi_{\alpha})\}_{\alpha\in\Upsilon}$ of supercharts $\phi_{\alpha}:\,U_{\alpha}\rightarrow\phi_{\alpha}(U_{\alpha})\subseteq\Lambda^{m,n}$, $\alpha\in\Upsilon$, of $\mathcal{M}$ such that $\bigcup_{\alpha\in\Upsilon}{U_{\alpha}}=\mathcal{M}$ and the supercharts are smoothly compatible. That is, for any $\alpha,\beta\in\Lambda$ with $U_{\alpha}\cap U_{\beta}\neq\emptyset$, the \emph{transition functions} 
\begin{equation}
\phi_{\beta}\circ\phi_{\alpha}^{-1}:\,\phi_{\alpha}(U_{\alpha}\cap U_{\beta})\rightarrow\phi_{\beta}(U_{\alpha}\cap U_{\beta})
\end{equation}
are of class $H^{\infty}$. An atlas on $\mathcal{M}$ is called \emph{maximal} if any superchart $(U,\phi_U)$ of $\mathcal{M}$ that is smoothly compatible with any superchart of the given atlas is already contained in this atlas.
\end{definition}
\begin{definition}
A \emph{proto $H^{\infty}$-supermanifold} of dimension $(m,n)$ is a topological space $\mathcal{M}$ furnished with a $(m,n)$-dimensional $H^{\infty}$-smooth atlas. 
\end{definition}
\begin{definition}
Let $\mathcal{M}$ and $\mathcal{N}$ be proto $H^{\infty}$-supermanifolds. A continuous map $f:\,\mathcal{M}\rightarrow\mathcal{N}$ is called \emph{smooth} if, for any local charts $(U,\phi_U)$ and $(V,\psi_V)$ of $\mathcal{M}$ and $\mathcal{N}$, respectively, with $U\cap f^{-1}(V)\neq\emptyset$, the map
\begin{equation}
\psi_V\circ f\circ\phi_U^{-1}:\,\phi_U(U\cap f^{-1}(V))\rightarrow\psi_V(V)
\end{equation}
is of class $H^{\infty}$.
\end{definition}

\begin{definition}\label{def:B.6}
Let $\mathcal{M}$ be a proto $H^{\infty}$-supermanifold of dimension $\mathrm{dim}\,\mathcal{M}=(m,n)$ with maximal atlas $\left\{(V_{\alpha},\psi_{\alpha})\right\}_{\alpha\in\Gamma}$. On $\mathcal{M}$, one can introduce an equivalence relation $\sim$ via (for a proof that this indeed defines an equivalence relation see \cite{Rogers:2007})  
	\begin{equation}
	p\sim q:\Leftrightarrow\text{$\exists\alpha\in\Gamma:\,p,q\in V_{\alpha}$ and $\epsilon_{m,n}(\psi_{\alpha}(p))=\epsilon_{m,n}(\psi_{\alpha}(q))$}
	\end{equation}
It immediately follows from the definition that, for any $p\in\mathcal{M}$, there exists a unique element $\mathbf{B}(p)\in\mathcal{M}$ such that $p\sim\mathbf{B}(p)$ and $f(\mathbf{B}(p))\in\mathbb{R}$ $\forall f\in H^{\infty}(\mathcal{M})$ which we call its \emph{body representative}. This yields a proper subset 
\begin{equation}
\mathbf{B}(\mathcal{M}):=\{p\in\mathcal{M}|\,f(p)\in\mathbb{R}\,,\forall f\in H^{\infty}(\mathcal{M})\}
\end{equation}
called the \emph{body} of $\mathcal{M}$, together with a surjective map $\mathbf{B}:\,\mathcal{M}\rightarrow\mathbf{B}(\mathcal{M}),\,p\mapsto\mathbf{B}(p)$ called the \emph{body map}. For any smooth map $f:\,\mathcal{M}\rightarrow\mathcal{N}$ between $H^{\infty}$-supermanifolds, one has $f(\mathbf{B}(\mathcal{M}))\subseteq\mathbf{B}(\mathcal{N})$. Hence, the body map can be extended to morphisms setting $\mathbf{B}(f):=f|_{\mathbf{B}(\mathcal{M})}:\,\mathbf{B}(\mathcal{M})\rightarrow\mathbf{B}(\mathcal{N})$. This yields a functor $\mathbf{B}:\,\mathbf{SMan}_{proto,H^{\infty}}\rightarrow\mathbf{Set}$ which we call the \emph{body functor}.
\end{definition}

\begin{lemma}\label{eq:B.7}
Let $p\in\mathcal{M}$ be a point on a proto $H^{\infty}$-supermanifold $\mathcal{M}$ and $\sim$ the equivalence relation as in Definition \ref{def:B.6}. Then, any other point $q\in\mathcal{M}$ in the equivalence class of $p$ will be contained in an arbitrary small open neighborhood $U$ of $p$ in $\mathcal{M}$.
\end{lemma}
\begin{proof}
Let $p\in U$ be an open neighborhood of $p$ in $\mathcal{M}$ and $(V_{\alpha},\psi_{\alpha})$ a coordinate neighborhood of $p$ with $q\in V_{\alpha}$. Then, $p\in U\cap V_{\alpha}$ and $\psi_{\alpha}(U\cap V_{\alpha})$ is open in $\Lambda^{m,n}$ so that, by the definition of the DeWitt-topology, there is a $\tilde{U}\subset\mathbb{R}^m$ open with $\psi_{\alpha}(U\cap V_{\alpha})=\epsilon_{m,n}^{-1}(\tilde{U})$. As $p\sim q$ we have $\epsilon_{m,n}(\psi_{\alpha}(q))=\epsilon_{m,n}(\psi_{\alpha}(p))\in\tilde{U}$ and thus $\psi_{\alpha}(q)\in \psi_{\alpha}(U\cap V_{\alpha})$. But, since $\psi_{\alpha}$ is injective this implies $q\in U\cap V_{\alpha}$ and therefore $q\in U$.
\end{proof}

\begin{lemma}\label{eq:B.8}
Let $\mathcal{M}$ be a proto $H^{\infty}$-supermanifold and $U,V\subseteq\mathcal{M}$ open subsets in $\mathcal{M}$. Then $U\subseteq V$ $\Leftrightarrow$ $\mathbf{B}(U)\subseteq\mathbf{B}(V)$ and thus, in particular, $U=V$ $\Leftrightarrow$ $\mathbf{B}(U)=\mathbf{B}(V)$.\\
Moreover, any open subset $U\subseteq\mathcal{M}$ is of the form $U=\mathbf{B}^{-1}(W)$ with $W$ open in $\mathbf{B}(\mathcal{M})$ (in fact $W=\mathbf{B}(U)$) which can be thought of as a generalization of the DeWitt-topology.
\end{lemma}
\begin{proof}
It is clear that $U\subseteq V$ implies $\mathbf{B}(U)\subseteq\mathbf{B}(V)$. Hence, suppose that $\mathbf{B}(U)\subseteq\mathbf{B}(V)$. Then $x\in U$ yields $\mathbf{B}(x)\in\mathbf{B}(U)\subseteq\mathbf{B}(V)$ such that there exists $y\in V$ with $x\sim y$. But, by Lemma \ref{eq:B.7}, this implies $x\in V$ and thus indeed $U\subseteq V$.\\
Next, for any open subset $U\subseteq\mathcal{M}$ define $U':=\mathbf{B}^{-1}(\mathbf{B}(U))$ yielding $U\subseteq U'$.  For any $x\in V$ one has $\mathbf{B}(x)\in\mathbf{B}(V)=\mathbf{B}(U)$. Thus, similarly as above, this implies $x\in U$ as $U$ is open and therefore $U'\subseteq U$, that is, $U'=U$.
\end{proof}

\begin{definition}
A $H^{\infty}$-supermanifold $\mathcal{M}$ is a proto $H^{\infty}$-supermanifold whose body $\mathbf{B}(\mathcal{M})$, equipped with the trace toplogy, defines a second countable Hausdorff topological space and thus defines an ordinary $C^{\infty}$-smooth manifold.\\
Supermanifolds together with smooth maps between them form a category $\mathbf{SMan}_{H^{\infty}}$ called the \emph{category of $H^{\infty}$-supermanifolds}.
\end{definition}

\begin{prop}\label{Prop:A.10}	
Let $f:\,\mathcal{M}\rightarrow\mathcal{N}$ be a smooth map between $H^{\infty}$-supermanifolds $\mathcal{M}$ and $\mathcal{N}$. Then, $f$ is surjective if and only if $\mathbf{B}(f):\,\mathbf{B}(\mathcal{M})\rightarrow \mathbf{B}(\mathcal{N})$ is surjective.
\end{prop}
\begin{proof}
Let $M:=\mathbf{B}(\mathcal{N})$ and $N:=\mathbf{B}(\mathcal{N})$. If $f$ is surjective, it follows $\mathbf{B}(f)(M)=\mathbf{B}(f(\mathcal{M}))=\mathbf{B}(\mathcal{N})=N$ and thus $\mathbf{B}(f)$ is surjective. Conversely, if $\mathbf{B}(f)$ be surjective, it follows from $N=\mathbf{B}(f)(M)=\mathbf{B}(f(\mathcal{M}))$ that $\mathbf{B}(f(\mathcal{M}))$ is open and thus $f(\mathcal{M})$ is open in $\mathcal{N}$ implying $\mathbf{B}(f(\mathcal{M}))=\mathbf{B}(f)(M)=N=\mathbf{B}(\mathcal{N})$. Hence, by Lemma \ref{eq:B.8}, this yields $f(\mathcal{M})=\mathcal{N}$, that is, $f$ is surjective.  
\end{proof}

\begin{example}
To any vector bundle $V\rightarrow E\rightarrow M$ with $\mathrm{dim}\,V=n$ and $\mathrm{dim}\,M=m$, it follows that one can associate a $H^{\infty}$-supermanifold $\mathbf{S}(E,M)$ of dimension $(m,n)$ called the \emph{split supermanifold} (see e.g.  \cite{Tuynman:2004} for an explicit construction via Grassmann extensions of transition functions or \cite{Eder:2020erq} using the functor of points applied to the algebro-geometric supermanifold $(M,\Gamma(\Exterior E^*))$). Moreover, any morphism $(\phi,f):\,(E,M)\rightarrow (F,N)$ between two vector bundles induces a morphism $\mathbf{S}(\phi,f):\,\mathbf{S}(E,M)\rightarrow\mathbf{S}(F,N)$ between the corresponding split supermanifolds. Hence, this yields a functor
\begin{equation}
\mathbf{S}:\,\mathbf{Vect}_{\mathbb{R}}\rightarrow\mathbf{SMan}_{H^{\infty}}
\label{eq:2.2.6}
\end{equation}
between the category of real vector bundles to the category of $H^{\infty}$-supermanifolds which we call the \emph{split functor}. It is a general result due to Batchelor \cite{Batchelor:1979} that any algebro-geometric supermanfold is isomorphic to a split supermanifold, i.e., \eqref{eq:2.2.6} is surjective on objects. However, the split functor is \emph{not} full, i.e., not every morphism $f:\,\mathbf{S}(E,M)\rightarrow\mathbf{S}(F,N)$ between split manifolds arises from a morphism between the respective vector bundles $(E,M),(F,N)\in\mathbf{Ob}(\mathbf{Vect}_{\mathbb{R}})$. Hence, the structure of morphisms between supermanifolds in general turns out to be much richer than for ordinary vector bundles.\\
A smooth manifold $M$ can be identified with the trivial vector bundle $M\times\{0\}\rightarrow M$, i.e., one has the \emph{inclusion functor} $\mathbf{i}:\,\mathbf{Man}\rightarrow\mathbf{Vect}_{\mathbb{R}}$ such that $\mathbf{i}(M):=(M\times\{0\},M)$ and $\mathbf{i}(f):=(\mathbf{0},f)$ for any $M\in\mathbf{Ob}(\mathbf{Man})$ and morphisms $f\in\mathrm{Hom}_{\mathbf{Man}}(M,N)$. Combined with the split functor, this yields another functor
\begin{equation}
\mathbf{S}\equiv\mathbf{S}\circ\mathbf{i}:\,\mathbf{Man}\rightarrow\mathbf{SMan}_{H^{\infty}}
\end{equation}
mapping an ordinary smooth manifold $M$ to a supermanifold $\mathbf{S}(M)$ with trivial odd dimensions, also called a \emph{bosonic supermanifold}. Conversely, given a bosonic supermanifold $\mathcal{M}$, it follows that the corresponding split supermanifold $\mathcal{M}_0:=\mathbf{S}(\mathbf{B}(\mathcal{M}))$ is isomorphic to $\mathcal{M}$. Hence, in this way, one obtains an equivalence of categories between bosonic supermanifolds and ordinary $C^{\infty}$-smooth manifolds.
\end{example}

\begin{definition}
A super Lie group $\mathcal{G}$ is a $H^{\infty}$-supermanifold together with two morphisms $\mu:\,\mathcal{G}\times\mathcal{G}\rightarrow\mathcal{G}$ and $i:\,\mathcal{G}\rightarrow\mathcal{G}$ called \emph{multiplication} and \emph{inverse}, respectively, as well as an element $e\in\mathcal{G}$ called \emph{neutral element} such that, after applying the forgetful functor, $(\mathcal{G},\mu,i,e)$ defines an ordinary group in the category of sets.
\end{definition}

\begin{definition}
A vector field $X\in\mathfrak{X}(\mathcal{G})$ on a super Lie group $\mathcal{G}$ is called \emph{left}- resp. \emph{right- invariant} if
\begin{equation}
\mathds{1}\otimes X\circ\mu^*=\mu^{*}\circ X\quad\text{resp.}\quad X\otimes\mathds{1}\circ\mu^{*}=\mu^*\circ X
\label{eq:Appendix:262}
\end{equation} 
where $\mu^*$ denotes the pullback of the group multiplication.
\end{definition}
\begin{prop}\label{prop:superLie}
Let $\mathcal{G}$ be a super Lie group. Then, the tangent module $T_e\mathcal{G}$ at the identity has the structure of a super Lie module of the form 
\begin{equation}
T_e\mathcal{G}=\Lambda\otimes\mathfrak{g}
\end{equation}
with $\mathfrak{g}$ canonically isomorphic to the super Lie algebra of smooth left-invariant (resp. right-invariant) vector fields on $\mathcal{G}$.
\end{prop} 
\begin{proof}
Let us only prove the proposition for the left-invariant case, the right-invariant case being similar. Note that the neutral element $e\in\mathcal{G}$ of a super Lie group $\mathcal{G}$ is a body point. In fact, since $\mathbf{B}(i)=i|_{\mathbf{B}(\mathcal{G})}$, it follows that, for any $g\in\mathbf{B}(\mathcal{G})$, one has $g^{-1}=i(g)=\mathbf{B}(i)(g)\in\mathbf{B}(\mathcal{G})$. Thus, since $\mathbf{B}(\mu)=\mu|_{\mathbf{B}(\mathcal{G})}$, we have in particular $g\cdot g^{-1}=e\in\mathbf{B}(\mathcal{G})$ (this also shows that $\mathbf{B}(\mathcal{G})$ in fact defines an ordinary Lie group).\\
Suppose, $X\in\mathfrak{X}(\mathcal{G})$ is a left-invariant vector field. Since, $e\in\mathbf{B}(\mathcal{G})$ it then immediately follows that the corresponding tangent vector $X_e\in T_e\mathcal{G}$ at $e$ satisfies
\begin{equation}
X_e(f):=X(f)(e)\in\mathbb{R},\,\forall f\in H^{\infty}(\mathcal{G})
\label{eq:9.266}
\end{equation}
Conversely, suppose $Y_e\in T_e\mathcal{G}$ satisfies \eqref{eq:9.266}. This induces a left-invariant vector field $Y\in\mathfrak{X}(\mathcal{G})$ setting $Y:=\mathds{1}\otimes Y_e\circ\mu^{*}$. In fact, using associativity of the group multiplication, it follows
\begin{align}
\mathds{1}\otimes Y\circ\mu^{*}&=\mathds{1}\otimes(\mathds{1}\otimes Y_e\circ\mu^{*})\circ\mu^{*}=\mathds{1}\otimes\mathds{1}\otimes Y_e\circ\mathds{1}\otimes\mu^{*}\circ\mu^{*}\nonumber\\
&=\mathds{1}\otimes\mathds{1}\otimes Y_e\circ\mu^{*}\otimes\mathds{1}\circ\mu^{*}=\mu^{\sharp}\circ(\mathds{1}\otimes Y_e\circ\mu^{*})=\mu^{*}\circ Y
\end{align}
Thus, we see that their exists an isomorphism between the super vector space of left-invariant vector fields on $\mathcal{G}$ and the super vector space of tangent vectors $X_e\in T_e\mathcal{G}$ at $e\in\mathcal{G}$ satisfying \eqref{eq:9.266} which we denote by $\mathfrak{g}$. Using this isomorphism, this induces a super Lie algebra structure on $\mathfrak{g}$ which, by linearity, then can be extended to all of $T_e\mathcal{G}$ yielding the structure of a super Lie module.
\end{proof}

\begin{definition}
Let $\mathcal{G}$ be a super Lie group. The \emph{super Lie module} $\mathrm{Lie}(\mathcal{G})$ of $\mathcal{G}$ is defined as the tangent module $\mathrm{Lie}(\mathcal{G})\equiv T_e\mathcal{G}=\Lambda\otimes\mathfrak{g}$. If not stated otherwise, we will identify $\mathfrak{g}$ with the super Lie algebra of smooth left-invariant vector fields on $\mathcal{G}$.
\end{definition}

\section{Proof of Proposition \ref{prop:5.10}}\label{section:Proof}

\begin{proof}
Since, for any $p\in\mathcal{P}$, the map $\Phi_{p*}:\,\mathrm{Lie}(\mathcal{G})\rightarrow\mathscr{V}_p\subset T_p\mathcal{P}$ is an isomorphism of free super $\Lambda$-modules and $\mathcal{A}\in\Omega^1(\mathcal{P},\mathfrak{g})_0$ is linear, condition (i) in Definition \ref{prop:3.19} of a connection 1-form and condition (i) of the proposition are equivalent. In the following, we thus can restrict to condition (ii).\\
Hence, suppose $\mathcal{A}\in\Omega^1(\mathcal{P},\mathfrak{g})_0$ is a connection 1-form on $\mathcal{P}$. Then, by restriction on body points, the first part of condition (ii) is immediate, i.e., $\Phi_g^*\mathcal{A}=\mathrm{Ad}_{g^{-1}}\circ\mathcal{A}$ for all $g\in\mathbf{B}(\mathcal{G})$. For the second part, let us extend $\mathcal{A}$ to a smooth $\mathrm{Lie}(\mathcal{G})$-valued 1-form $\tilde{\mathcal{A}}\in\Omega^1(\mathrm{Lie}(\mathcal{G})_0\times\mathcal{P},\mathfrak{g})$ on $\mathrm{Lie}(\mathcal{G})_0\times\mathcal{P}$ by setting $\braket{(Z,Z')|\tilde{\mathcal{A}}}:=\braket{Z'|\mathcal{A}}$ $\forall (Z,Z')\in T(\mathrm{Lie}(\mathcal{G})_0\times\mathcal{P})\cong\mathrm{Lie}(\mathcal{G})\times T\mathcal{P}$. Moreover, we consider the extended $\mathcal{G}$-right action $\tilde{\Phi}:\,(\mathrm{Lie}(\mathcal{G})_0\times\mathcal{P})\times\mathcal{G}\rightarrow\mathrm{Lie}(\mathcal{G})_0\times\mathcal{P}$ on $\mathrm{Lie}(\mathcal{G})_0\times\mathcal{P}$ defined via $\tilde{\Phi}((Y,p),g):=(Y,\Phi(p,g))$. It is then immediate to see that that condition (ii) in Definition \ref{prop:3.19} of a connection 1-form is equivalent to
\begin{equation}
\tilde{\Phi}_g^*\tilde{\mathcal{A}}=\mathrm{Ad}_{g^{-1}}\circ\tilde{\mathcal{A}},\quad\forall g\in\mathcal{G}
\label{eq:5.0.1}
\end{equation}
Consider the even smooth vector field $\widetilde{\mathcal{Z}}\in\Gamma(T(\mathrm{Lie}(\mathcal{G})_0\times\mathcal{P}))$ given by \[\widetilde{\mathcal{Z}}_{(X,p)}:=(0_Y,D_{(p,e)}\Phi(0_p,Y_e))=(0_Y,\widetilde{Y}_p)\] $\forall (Y,p)\in\mathrm{Lie}(\mathcal{G})_0\times\mathcal{P}$ parametrizing the (not necessarily $H^{\infty}$-smooth) fundamental vector fields on $\mathcal{P}$. The flow $\phi^{\widetilde{\mathcal{Z}}}:\,\mathbb{R}_{S,0}\times(\mathrm{Lie}(\mathcal{G})_0\times\mathcal{P})\rightarrow\mathrm{Lie}(\mathcal{G})_0\times\mathcal{P}$ of $\widetilde{\mathcal{Z}}$ is of the form $\phi^{\widetilde{\mathcal{Z}}}_t(Y,p)=(Y,\Phi(p,e^{tY}))=\tilde{\Phi}_{e^{tY}}(Y,p)$. Using (\cite{Tuynman:2004}, Prop. V.7.27), \eqref{eq:5.0.1} then yields
\begin{align}
\braket{(0_Y,Z_p)|L_{\widetilde{\mathcal{Z}}}\tilde{\mathcal{A}}_{(Y,p)}}&=\frac{\partial}{\partial t}\bigg{|}_{t=0}\braket{(0_Y,Z_p)|((\phi^{\widetilde{\mathcal{Z}}}_t)^*\tilde{\mathcal{A}})_{(Y,p)}}=\frac{\partial}{\partial t}\bigg{|}_{t=0}\braket{(0_Y,Z_p)|(\Phi_{e^{tY}}^*\mathcal{A})_{(Y,p)}}\nonumber\\
&=\frac{\partial}{\partial t}\bigg{|}_{t=0}\mathrm{Ad}_{e^{-tY}}\braket{(0_Y,Z_p)|\tilde{\mathcal{A}}_{(Y,p)}}=-\mathrm{ad}_Y\braket{Z_p|\mathcal{A}_{p}}
\end{align}
$\forall (X,p)\in\mathrm{Lie}(\mathcal{G})_0\times\mathcal{P}$ and smooth homogeneous $Z\in\Gamma(T\mathcal{P})$. On the other hand, one has
\begin{align}
\braket{(0_{Y},Z_p)|L_{\widetilde{\mathcal{Z}}}\tilde{\mathcal{A}}_{(Y,p)}}&=\widetilde{\mathcal{Z}}_{(Y,p)}\braket{(0,Z)|\tilde{\mathcal{A}}}-\braket{[\widetilde{\mathcal{Z}},(0,Z)]_{(Y,p)}|\tilde{\mathcal{A}}_{(Y,p)}}\nonumber\\
&=\widetilde{Y}_p\braket{Z|\mathcal{A}}-\braket{[\widetilde{\mathcal{Z}},(0,Z)]_{(Y,p)}|\tilde{\mathcal{A}}_{(Y,p)}}
\end{align}
for any $(Y,p)\in\mathrm{Lie}(\mathcal{G})_0\times\mathcal{P}$. If $Y=X\in\mathfrak{g}_0\subseteq\mathrm{Lie}(\mathcal{G})_0$, it follows $[\widetilde{\mathcal{Z}},(0,Z)]_{(X,p)}=[\widetilde{X},Z]_{p}$ yielding
\begin{equation}
\widetilde{X}_p\braket{Z|\mathcal{A}}-\braket{[\widetilde{X},Z]_{p}|\mathcal{A}_{p}}=\braket{Z_p|L_{\widetilde{X}}\mathcal{A}_p}
\end{equation}
and thus $\braket{Z_p|L_{\widetilde{X}}\mathcal{A}_p}=\braket{Z_p|-\mathrm{ad}_X\circ\mathcal{A}_p}$ $\forall p\in\mathcal{P}$. Since this holds for any smooth homogeneous vector field $Z$, this implies
\begin{equation}
L_{\widetilde{X}}\mathcal{A}=-\mathrm{ad}_X\circ\mathcal{A},\,\text{for }X\in\mathfrak{g}_0
\end{equation}
On the other hand, if $Y=\tau X$ with $X\in\mathfrak{g}_1$ and $\tau\in\Lambda_{1}$, it follows $[\widetilde{\mathcal{Z}},(0,Z)]_{(Y,p)}=\tau[\widetilde{X},Z]_p$ such that
\begin{equation}
\tau\widetilde{X}_p\braket{Z|\mathcal{A}}-\tau\braket{[\widetilde{X},Z]_{p}|\mathcal{A}_{p}}=(-1)^{|Z|}\tau\braket{Z_p|L_{\widetilde{X}}\mathcal{A}_p}
\end{equation}
and therefore $\tau\braket{Z_p|L_{\widetilde{X'}}\mathcal{A}_p}=-(-1)^{|Z|}\tau\mathrm{ad}_{X}\braket{Z_p|\mathcal{A}_p}=\tau\braket{Z_p|-\mathrm{ad}_{X}\circ\mathcal{A}_p}$ $\forall p\in\mathcal{P}$. Since this holds for any $\tau\in\Lambda_{1}$, we thus have 
\begin{equation}
L_{\widetilde{X}}\mathcal{A}=-\mathrm{ad}_X\circ\mathcal{A},\,\text{for }X\in\mathfrak{g}_1
\end{equation}
Conversely, suppose $\mathcal{A}\in\Omega^1(\mathcal{P},\mathfrak{g})_0$ satisfies the conditions (i) and (ii) of the proposition. For any smooth homogeneous vector field $Z\in\Gamma(T\mathcal{P})$ consider the $\mathrm{Lie}(\mathcal{G})$-valued $H^{\infty}$-smooth functions $F_{Z},G_{Z}\in H^{\infty}(\mathcal{G}\times\mathcal{P})\otimes\mathrm{Lie}(\mathcal{G})$ defined as
\begin{align}
F_{Z}(g,p):=\braket{Z_p|\Phi^*_g\mathcal{A}_p}=\braket{D_{(p,g)}\Phi(Z_p,0_g)|\mathcal{A}_{p\cdot g}}
\end{align}
as well as 
\begin{equation}
G_{Z}(g,p):=\mathrm{Ad}_{g^{-1}}\braket{Z_p|\mathcal{A}_p}
\end{equation}
$\forall (g,p)\in\mathcal{G}\times\mathcal{P}$. Since $H^{\infty}(\mathcal{G}\times\mathcal{P})\cong H^{\infty}(\mathcal{G})\hat{\otimes}_{\pi}H^{\infty}(\mathcal{P})$, it follows from Lemma \ref{prop:5.8} that $\mathcal{A}$ defines a connection 1-form on $\mathcal{P}$ if and only if we can show 
\begin{equation}
(X\otimes\mathds{1})F_{Z}(g,p)=(X\otimes\mathds{1})G_Z(g,p)
\label{eq:5.0.2}
\end{equation}
$\forall p\in\mathcal{P}$ and body points $g\in\mathbf{B}(\mathcal{G})$ as well as $X\in\mathcal{U}(\mathfrak{g})$ and smooth homogeneous vector fields $Z\in\Gamma(T\mathcal{P})$. For $X=1\in\mathcal{U}(\mathfrak{g})$, this is an immediate consequence of the first part of condition (ii).\\
Using the extension $\tilde{\mathcal{A}}\in\Omega^1(\mathrm{Lie}(\mathcal{G})_0\times\mathcal{P},\mathrm{Lie}(\mathcal{G}))$ of $\mathcal{A}$ on $\mathrm{Lie}(\mathcal{G})_0\times\mathcal{P}$ as well as the $\mathcal{G}$-right action $\tilde{\Phi}:\,(\mathrm{Lie}(\mathcal{G})_0\times\mathcal{P})\times\mathcal{G}\rightarrow\mathrm{Lie}(\mathcal{G})_0\times\mathcal{P}$ as defined above, we may extend $F_Z$ and $G_Z$ to $H^{\infty}$-smooth functions $\tilde{F}_Z$ and $\tilde{G}_Z$ on $\mathrm{Lie}(\mathcal{G})_0\times\mathcal{G}\times\mathcal{P}$ by setting
\begin{align}
\tilde{F}_{Z}(Y,g,p):&=\braket{\tilde{\Phi}_{g*}(0_Y,Z_p)|\tilde{\mathcal{A}}_{\tilde{\Phi}(Y,p,g)}}=\braket{D_{(Y,p,g)}\tilde{\Phi}(0_Y,Z_p,0_g)|\tilde{\mathcal{A}}_{\tilde{\Phi}(Y,p,g)}}\nonumber\\
&=\braket{D_{(p,g)}\Phi(Z_p,0_g)|\mathcal{A}_{p\cdot g}}=F_Z(g,p)
\end{align}
as well as 
\begin{equation}
\tilde{G}_{Z}(Y,g,p):=\mathrm{Ad}_{g^{-1}}\braket{(0_Y,Z_p)|\tilde{\mathcal{A}}_{(Y,p)}}=\mathrm{Ad}_{g^{-1}}\braket{Z_p|\mathcal{A}_p}=G_Z(g,p)
\end{equation}
$\forall (Y,g,p)\in\mathrm{Lie}(\mathcal{G})_0\times\mathcal{G}\times\mathcal{P}$. Let $\mathcal{Z}^L\in\Gamma(T(\mathrm{Lie}(\mathcal{G})_0\times\mathcal{G}))$ be the $H^{\infty}$-smooth homogeneous vector field on $\mathrm{Lie}(\mathcal{G})_0\times\mathcal{G}$ defined as $\mathcal{Z}^L_{(Y,g)}:=(0_Y,D_{(g,e)}\mu(0_g,Y_e))$ $\forall(Y,g)\in\mathrm{Lie}(\mathcal{G})_0\times\mathcal{G}$. Similarly as above, using the explicit form of the flow $\phi^L$ of $\mathcal{Z}^L$, this yields
\begin{align}
((\phi^L_t)^*\tilde{F}_Z)(Y,g,p)&=\braket{\tilde{\Phi}_{ge^{tY*}}(0_Y,Z_p)|\tilde{\mathcal{A}}_{\tilde{\Phi}(Y,p,g\cdot e^{tY})}}\nonumber\\
&=\braket{\tilde{\Phi}_{e^{tY*}}\circ\tilde{\Phi}_{g*}(0_Y,Z_p)|\tilde{\mathcal{A}}_{\tilde{\Phi}(Y,p,g\cdot e^{tY})}}\nonumber\\
&=\braket{(0_Y,\Phi_{g*}Z_p)|(\phi^{\widetilde{\mathcal{Z}}}_t)^*\tilde{\mathcal{A}}_{\phi^{\widetilde{\mathcal{Z}}}_t(Y,p\cdot g)}}
\end{align}
Taking the derivative, we thus conclude
\begin{align}
(Y\otimes\mathds{1})F_Z(g,p)&=\frac{\partial}{\partial t}\bigg{|}_{t=0}\braket{(0_Y,\Phi_{g*}Z_p)|(\phi^{\widetilde{\mathcal{Z}}}_t)^*\tilde{\mathcal{A}}_{\phi^{\widetilde{\mathcal{Z}}}_t(Y,p\cdot g)}}\nonumber\\
&=\braket{(0_Y,\Phi_{g*}Z_p)|L_{\widetilde{\mathcal{Z}}}\tilde{\mathcal{A}}_{(Y,p\cdot g)}}
\end{align}
Following exactly the same steps as above, one then concludes
\begin{equation}
(X\otimes\mathds{1})F_Z(g,p)=(-1)^{|Z||X|}\braket{\Phi_{g*}Z_p|L_{\widetilde{X}}\mathcal{A}_{p\cdot g}},\quad\forall\text{homogeneous }X\in\mathfrak{g}
\end{equation}
For $G_Z$ we proceed similarly and compute
\begin{align}
((\phi^L_t)^*\tilde{G}_Z)(Y,g,p)&=\mathrm{Ad}_{(ge^{tY})^{-1}}\braket{(0_Y,Z_p)|\tilde{\mathcal{A}}_{(Y,p)}}\nonumber\\
&=\mathrm{Ad}_{e^{-tY}}(\mathrm{Ad}_{g^{-1}}\braket{Z_p|\mathcal{A}_{p}})=\mathrm{Ad}_{e^{-tY}}(G_Z(g,p))
\end{align}
which yields
\begin{align}
(Y\otimes\mathds{1})F_Z(g,p)&=\frac{\partial}{\partial t}\bigg{|}_{t=0}\mathrm{Ad}_{e^{-tY}}(G_Z(g,p))=-\mathrm{ad}_YG_Z(g,p)\nonumber\\
&=-\mathrm{ad}_Y(\mathrm{Ad}_{g^{-1}}\braket{Z_p|\mathcal{A}_{p}})
\end{align}
Hence, it follows for $\forall X\in\mathfrak{g}$
\begin{equation}
(X\otimes\mathds{1})G_Z(g,p)=-\mathrm{ad}_X(\mathrm{Ad}_{g^{-1}}\braket{Z_p|\mathcal{A}_{p}})
\end{equation}
$\forall g\in\mathcal{G},\,p\in\mathcal{P}$. If $g\in\mathbf{B}(\mathcal{G})$ is a body point this, together with condition (ii), yields
\begin{align}
(X\otimes\mathds{1})G_Z(g,p)&=-\mathrm{ad}_X(\mathrm{Ad}_{g^{-1}}\braket{Z_p|\mathcal{A}_{p}})=-\mathrm{ad}_X\braket{\Phi_{g*}Z_p|\mathcal{A}_{p\cdot g}}\nonumber\\
&=(-1)^{|Z||X|}\braket{\Phi_{g^*}Z_p|L_{\widetilde{X}}\mathcal{A}_{p\cdot g}}=(X\otimes\mathds{1})F_Z(g,p)
\end{align}
proving \eqref{eq:5.0.2} in case $X\in\mathfrak{g}$. Next, let $Y\circ X\in\mathcal{U}(\mathfrak{g})$ with homogeneous $Y,X\in\mathfrak{g}$. In a similar way as above, one finds  
\begin{equation}
(Y\circ X\otimes\mathds{1})F_Z(g,p)=(-1)^{|Z|(|X|+|Y|)}\braket{\Phi_{g*}Z_p|L_{\widetilde{Y}}L_{\widetilde{X}}\mathcal{A}_{p\cdot g}}
\end{equation}
as well as
\begin{align}
(Y\circ X\otimes\mathds{1})G_Z(g,p)&=(Y\otimes\mathds{1})\mathrm{ad}_X\circ G_Z(g,p)=(-1)^{|X||Y|}\mathrm{ad}_X((Y\otimes\mathds{1})G_Z(g,p))\nonumber\\
&=(-1)^{|X||Y|}\mathrm{ad}_X\circ\mathrm{ad}_Y\circ G_Z(g,p)\nonumber\\
&=(-1)^{|X||Y|}\mathrm{ad}_X\circ\mathrm{ad}_Y(\mathrm{Ad}_{g^{-1}}\braket{Z_p|\mathcal{A}_{p}})
\label{eq:5.0.3}
\end{align}
$\forall g\in\mathcal{G},\,p\in\mathcal{P}$.
Taking the Lie derivative on both sides of the second part of condition (ii), one obtains
\begin{equation}
L_{\widetilde{Y}}L_{\widetilde{X}}\mathcal{A}=-L_{\widetilde{Y}}(-\mathrm{ad}_X\circ\mathcal{A})=-(-1)^{|X||Y|}\mathrm{ad}_X\circ L_{\widetilde{Y}}\mathcal{A}=(-1)^{|X||Y|}\mathrm{ad}_X\circ\mathrm{ad}_Y\circ\mathcal{A}
\label{eq:5.0.4}
\end{equation}
Hence, inserting \eqref{eq:5.0.3} in \eqref{eq:5.0.4} and restricting on body points $g\in\mathbf{B}(\mathcal{G})$, it follows
\begin{align}
(Y\circ X\otimes\mathds{1})G_Z(g,p)&=(-1)^{|X||Y|}\mathrm{ad}_X\circ\mathrm{ad}_Y(\mathrm{Ad}_{g^{-1}}\braket{Z_p|\mathcal{A}_{p}})\nonumber\\
&=(-1)^{|X||Y|}\mathrm{ad}_X\circ\mathrm{ad}_Y(\braket{\Phi_{g*}Z_p|\mathcal{A}_{p\cdot g}})\nonumber\\
&=(-1)^{|Z|(|X|+|Y|)}(\braket{\Phi_{g*}Z_p|(-1)^{|X||Y|}\mathrm{ad}_X\circ\mathrm{ad}_Y\circ\mathcal{A}_{p\cdot g}})\nonumber\\
&=(-1)^{|Z|(|X|+|Y|)}(\braket{\Phi_{g*}Z_p|L_{\widetilde{Y}}L_{\widetilde{X}}\mathcal{A}_{p\cdot g}})\nonumber\\
&=(Y\circ X\otimes\mathds{1})F_Z(g,p)
\end{align}
proving \eqref{eq:5.0.2} in case $Y\circ X\in\mathcal{U}(\mathfrak{g})$ with $Y,X\in\mathfrak{g}$. Thus, by induction, one concludes that \eqref{eq:5.0.2} holds for any $X\in\mathcal{U}(\mathfrak{g})$ and hence $\mathcal{A}$ indeed defines a connection 1-form on $\mathcal{P}$.
\end{proof}


\begin{thebibliography}{56}


\bibitem{Freedman:1976xh}
D.~Z.~Freedman, P.~van Nieuwenhuizen and S.~Ferrara,
``Progress Toward a Theory of Supergravity,''
Phys. Rev. D \textbf{13} (1976), 3214-3218
doi:10.1103/PhysRevD.13.3214

\bibitem{vanNieuwenhuizen:2004rh}
P.~van Nieuwenhuizen,
``Supergravity as a Yang\textendash{}Mills Theory,''
doi:10.1142/9789812567147\_0018
[arXiv:hep-th/0408137 [hep-th]].

\bibitem{MacDowell:1977jt}
S.~W.~MacDowell and F.~Mansouri,
``Unified Geometric Theory of Gravity and Supergravity,''
Phys. Rev. Lett. \textbf{38} (1977), 739
[erratum: Phys. Rev. Lett. \textbf{38} (1977), 1376]
doi:10.1103/PhysRevLett.38.739

\bibitem{Wise:2006sm}
D.~K.~Wise,
``MacDowell-Mansouri gravity and Cartan geometry,''
Class. Quant. Grav. \textbf{27} (2010), 155010
doi:10.1088/0264-9381/27/15/155010
[arXiv:gr-qc/0611154 [gr-qc]].


\bibitem{DAuria:1982uck}
R.~D'Auria and P.~Fre,
``Geometric Supergravity in d = 11 and Its Hidden Supergroup,''
Nucl. Phys. B \textbf{201} (1982), 101-140
[erratum: Nucl. Phys. B \textbf{206} (1982), 496]
doi:10.1016/0550-3213(82)90281-4

\bibitem{Castellani:1991et}
L.~Castellani, R.~D'Auria and P.~Fre,
``Supergravity and superstrings: A Geometric perspective. Vol. 1: Mathematical foundations,''
World Scientific (1991) 1-603, Singapore. 

\bibitem{Castellani:2013iq}
L.~Castellani,
``OSp(1|4) supergravity and its noncommutative extension,''
Phys. Rev. D \textbf{88} (2013) no.2, 025022
doi:10.1103/PhysRevD.88.025022
[arXiv:1301.1642 [hep-th]].

\bibitem{Castellani:2018zey}
L.~Castellani,
``Supergravity in the Group-Geometric Framework: A Primer,''
Fortsch. Phys. \textbf{66} (2018) no.4, 1800014
doi:10.1002/prop.201800014
[arXiv:1802.03407 [hep-th]].

\bibitem{Castellani:2014goa}
L.~Castellani, R.~Catenacci and P.~A.~Grassi,
``Supergravity Actions with Integral Forms,''
Nucl. Phys. B \textbf{889} (2014), 419-442
doi:10.1016/j.nuclphysb.2014.10.023
[arXiv:1409.0192 [hep-th]].

\bibitem{Cremonini:2019aao}
C.~A.~Cremonini and P.~A.~Grassi,
``Pictures from Super Chern-Simons Theory,''
JHEP \textbf{03} (2020), 043
doi:10.1007/JHEP03(2020)043
[arXiv:1907.07152 [hep-th]].

\bibitem{Catenacci:2018xsv}
R.~Catenacci, P.~A.~Grassi and S.~Noja,
``Superstring Field Theory, Superforms and Supergeometry,''
J. Geom. Phys. \textbf{148} (2020), 103559
doi:10.1016/j.geomphys.2019.103559
[arXiv:1807.09563 [hep-th]].

\bibitem{Cortes:2018lan}
V.~Cort\'es, C.~I.~Lazaroiu and C.~S.~Shahbazi,
``${\mathcal {N}}=1$ Geometric Supergravity and Chiral Triples on Riemann Surfaces,''
Commun. Math. Phys. \textbf{375} (2019) no.1, 429-478
doi:10.1007/s00220-019-03476-7
[arXiv:1810.12353 [hep-th]].

\bibitem{Schreiber:2013pra}
U.~Schreiber,
``Differential cohomology in a cohesive infinity-topos,''
[arXiv:1310.7930 [math-ph]].


\bibitem{Alekseevsky:1997ri}
D.~V.~Alekseevsky, V.~Cortes, C.~Devchand and U.~Semmelmann,
``Killing spinors are Killing vector fields in Riemannian supergeometry,''
J. Geom. Phys. \textbf{26} (1998), 37-50
doi:10.1016/S0393-0440(97)00036-3
[arXiv:dg-ga/9704002 [math.DG]].



\bibitem{Deligne:1999qp}
P.~Deligne, P.~Etingof, D.~Freed, L.~Jeffrey, D.~Kazhdan, J.~Morgan, D.~Morrison and E.~Witten,
``Quantum fields and strings: A course for mathematicians. Vol. 1''

\bibitem{Carmeli:2011}
C.~Carmeli, L.~Caston, R.~Fioresi,
``Mathematical Foundations of Supersymmetry,''
EMS Series of Lectures in Mathematics 15, European Mathematical Society, 2011.

\bibitem{Goertsches:2008}
O.~Goertsches,
``Riemannian supergeometry,''
Mathematische Zeitschrift 260.3 (2008): 557-593. 

\bibitem{Bartocci}
C.~Bartocci, U.~Bruzzo, D.~Hernandez-Ruiperez,
``The Geometry of Supermanifolds,''
Kluwer Academic Publishers, 1991.

\bibitem{Baez:2010ya}
J.~C.~Baez and J.~Huerta,
``An Invitation to Higher Gauge Theory,''
Gen. Rel. Grav. \textbf{43} (2011), 2335-2392
doi:10.1007/s10714-010-1070-9
[arXiv:1003.4485 [hep-th]].

\bibitem{Nicolai:1984hb}
H.~Nicolai,
``REPRESENTATIONS OF SUPERSYMMETRY IN ANTI-DE SITTER SPACE,''
CERN-TH-3882.

\bibitem{Freedman:1983na}
D.~Z.~Freedman and H.~Nicolai,
``Multiplet Shortening in Osp($N$,4),''
Nucl. Phys. B \textbf{237} (1984), 342-366
doi:10.1016/0550-3213(84)90164-0

\bibitem{Freedman:2012zz}
D.~Z.~Freedman and A.~Van Proeyen,
``Supergravity,''
Cambridge Univ. Press, 2012.

\bibitem{Wipf:2016}
A.~Wipf,
``Introduction to Supersymmetry,''
Vorlesungsskript, Universit\"at Jena, 2016. 


\bibitem{Leites:1980}
D.~A.~Leites,
``Introduction to the theory of supermanifolds,''
Russian Math. Surveys 35(1980), 1–64.


\bibitem{Eder:2020erq}
K.~Eder,
``Super Cartan geometry and the super Ashtekar connection,''
[arXiv:2010.09630 [gr-qc]].


\bibitem{Berezin:1976}
F.~Berezin and D.~Leites
``Supermanifolds,''
Soviet Maths Doknaja \textbf{16}, pp. 1218-1222.

\bibitem{Kostant:1975qe}
B.~Kostant,
``Graded Manifolds, Graded Lie Theory, and Prequantization,''
Lect. Notes Math. \textbf{570} (1977), 177-306
doi:10.1007/BFb0087788


\bibitem{Batchelor:1979}
M.~Batchelor, 
``The structure of supermanifolds,''
Transactions of the American Mathematical Society 253 (1979): 329-338.


\bibitem{Batchelor:1980}
M.~Batchelor, 
``Two approaches to supermanifolds,''
Transactions of the American Mathematical Society 258.1 (1980): 257-270.

\bibitem{DeWitt:1984}
B.~DeWitt,
``Supermanifolds,''
Cambridge University Press, Cambridge 1984.

\bibitem{Rogers:1980}
A.~Rogers,
``A global theory of supermanifolds,''
Journal of Mathematical Physics 21.6 (1980): 1352-1365.

\bibitem{Rogers:2007}
A.~Rogers,
``Supermanifolds: Theory and Applications,''
World Scientific (2007).


\bibitem{Tuynman:2004}
G.~M.~Tuynman,
``Supermanifolds and Supergroups-Basic Theory,''
Kluwer Academic Publishers (2004),
doi:10.1007/1-4020-2297-2

\bibitem{Tuynman:2018}
G.~M.~Tuynman,
``Super unitary representations revisited,''
arXiv:1711.00233v2 [math.DG]

\bibitem{Cheng:2012}
S.~J.~Cheng and W.~Wang, 
``Dualities and Representations of Lie Superalgebras,''
Graduate Studies in Mathematics 144, Amer. Math. Soc., Providence, RI, 2012.

\bibitem{Hack:2015vna}
T.~P.~Hack, F.~Hanisch and A.~Schenkel,
``Supergeometry in locally covariant quantum field theory,''
Commun. Math. Phys. \textbf{342} (2016) no.2, 615-673
doi:10.1007/s00220-015-2516-4
[arXiv:1501.01520 [math-ph]].

\bibitem{Jost:2014wfa}
J.~Jost, E.~Ke\ss{}ler and J.~Tolksdorf,
``Super Riemann surfaces, metrics and gravitinos,''
Adv. Theor. Math. Phys. \textbf{21} (2017), 1161-1187
doi:10.4310/ATMP.2017.v21.n5.a2
[arXiv:1412.5146 [math-ph]].

\bibitem{Kessler:2019bwp}
E.~Ke\ss{}ler,
``Supergeometry, Super Riemann Surfaces and the Superconformal Action Functional,''
Lect. Notes Math. \textbf{2230} (2019), pp.
doi:10.1007/978-3-030-13758-8

\bibitem{Mol:10}
V.~Molotkov,
``infinite-dimensional and colored supermanifolds,''
Journal of Nonlinear Mathematical Physics, 17:sup1, 375-446 (2010), doi: 10.1142/S140292511000088X

\bibitem{Sac:08}
C.~Sachse,
``A categorical formulation of superalgebra and supergeometry,''
arXiv:0802.4067v1 [math.AG].

\bibitem{Sac:09}
C.~Sachse,
``Global Analytic Approach to Super Teichmueller Spaces,''
arXiv:0902.3289v1 [math.AG].


\bibitem{Rejzner:2011au}
K.~Rejzner,
``Fermionic fields in the functional approach to classical field theory,''
Rev. Math. Phys. \textbf{23} (2011), 1009-1033
doi:10.1142/S0129055X11004503
[arXiv:1101.5126 [math-ph]].

\bibitem{Rejzner:2016hdj}
K.~Rejzner,
``Perturbative Algebraic Quantum Field Theory: An Introduction for Mathematicians,''
Math. Phys. Stud. (2016)
doi:10.1007/978-3-319-25901-7

\bibitem{Schuett:2018}
J.~Sch\"utt,
``Infinite-Dimensional Supermanifolds via Multilinear Bundles,''
arXiv preprint arXiv:1810.05549 (2018).

\bibitem{Schmitt:1996hp}
T.~Schmitt,
``Supergeometry and quantum field theory, or: What is a classical configuration?,''
Rev. Math. Phys. \textbf{9} (1997), 993-1052
doi:10.1142/S0129055X97000348
[arXiv:hep-th/9607132 [hep-th]].

\bibitem{Stavracou:1996qb}
T.~Stavracou,
``Theory of connections on graded principal bundles,''
Rev. Math. Phys. \textbf{10} (1998), 47-80
doi:10.1142/S0129055X98000033
[arXiv:dg-ga/9605006 [math.DG]].

\bibitem{Florin:2008}
F.~Dumitrescu,
``Superconnections and parallel transport,''
Pacific Journal of Mathematics 236.2 (2008): 307-332.

\bibitem{Groeger:2013aja}
J.~Groeger,
``Super Wilson Loops and Holonomy on Supermanifolds,''
[arXiv:1312.4745 [math-ph]].

\bibitem{DePaoli:2018erh}
E.~De Paoli and S.~Speziale,
``A gauge-invariant symplectic potential for tetrad general relativity,''
JHEP \textbf{07} (2018), 040
doi:10.1007/JHEP07(2018)040
[arXiv:1804.09685 [gr-qc]].

\bibitem{Freidel:2020xyx}
L.~Freidel, M.~Geiller and D.~Pranzetti,
``Edge modes of gravity. Part I. Corner potentials and charges,''
JHEP \textbf{11} (2020), 026
doi:10.1007/JHEP11(2020)026
[arXiv:2006.12527 [hep-th]].

\bibitem{Aneesh:2020fcr}
P.~B.~Aneesh, S.~Chakraborty, S.~J.~Hoque and A.~Virmani,
``First law of black hole mechanics with fermions,''
Class. Quant. Grav. \textbf{37} (2020) no.20, 205014
doi:10.1088/1361-6382/aba5ab
[arXiv:2004.10215 [hep-th]].

\bibitem{Iyer:1994ys}
V.~Iyer and R.~M.~Wald,
``Some properties of Noether charge and a proposal for dynamical black hole entropy,''
Phys. Rev. D \textbf{50} (1994), 846-864
doi:10.1103/PhysRevD.50.846
[arXiv:gr-qc/9403028 [gr-qc]].

\bibitem{Prabhu:2015vua}
K.~Prabhu,
``The First Law of Black Hole Mechanics for Fields with Internal Gauge Freedom,''
Class. Quant. Grav. \textbf{34} (2017) no.3, 035011
doi:10.1088/1361-6382/aa536b
[arXiv:1511.00388 [gr-qc]].

\bibitem{Mason:2010yk}
L.~Mason and D.~Skinner,
``The Complete Planar S-matrix of N=4 SYM as a Wilson Loop in Twistor Space,''
JHEP \textbf{12} (2010), 018
doi:10.1007/JHEP12(2010)018
[arXiv:1009.2225 [hep-th]].

\bibitem{Hamilton:2017}
M.~J.~D.~Hamilton,
``Mathematical Gauge Theory: With Applications to the Standard Model of
Particle Physics,''
Universitext. Springer International Publishing, 2017.

\bibitem{Baum}
H.~Baum,
``Eichfeldtheorie--Eine Einführung in die Differentialgeometrie auf Faserbündeln,''
Springer Spektrum, Berlin Heidelberg, 2014.

\end{thebibliography}
\end{document}